\documentclass{article}
\usepackage{geometry}

\usepackage[latin9]{inputenc}
\usepackage{color}
\usepackage{mathtools}
\usepackage{multibib}
\newcites{sup}{References}
\usepackage[authoryear]{natbib}
\usepackage{todonotes}
\usepackage{amsmath}
\usepackage{mathrsfs}
\usepackage{accents}
\usepackage{float}
\usepackage{graphicx}
\usepackage{caption}
\usepackage{subcaption}
\usepackage{enumitem}  
\usepackage[unicode=true,pdfusetitle, bookmarks=true,bookmarksnumbered=false,bookmarksopen=false, breaklinks=false,pdfborder={0 0 1},backref=false,colorlinks=true]{hyperref}
\usepackage{amsthm}
\usepackage{bbm}
\usepackage{amssymb}
\usepackage{chngcntr}
\hypersetup{linkcolor=blue,citecolor=blue}

\makeatletter

\renewcommand\paragraph{\@startsection{paragraph}{4}{\z@}%
            {-2.5ex\@plus -1ex \@minus -.25ex}%
            {1.25ex \@plus .25ex}%
            {\normalfont\normalsize\bfseries}}

\theoremstyle{plain}

\theoremstyle{plain}

\floatstyle{ruled}
\newfloat{algorithm}{tbp}{loa}
\providecommand{\algorithmname}{Algorithm}
\floatname{algorithm}{\protect\algorithmname}

\usepackage{algorithm}
\usepackage{dsfont}
\usepackage{algpseudocode}
\usepackage{mathtools}

\DeclareMathOperator*{\argmax}{argmax}

\newcounter{daggerfootnote}
\newcommand*{\daggerfootnote}[1]{%
    \setcounter{daggerfootnote}{\value{footnote}}%
    \renewcommand*{\thefootnote}{\fnsymbol{footnote}}%
    \footnote[2]{#1}%
    \setcounter{footnote}{\value{daggerfootnote}}%
    \renewcommand*{\thefootnote}{\arabic{footnote}}%
    }

\newcommand{\setX}{\mathsf{X}}
\newcommand{\setY}{\mathsf{Y}}

\newcommand{\setZ}{\mathsf{Z}}
\newcommand{\setR}{\mathsf{R}}

\newcommand{\F}{\mathcal{F}}
\newcommand{\dd}{\mathrm{d}}
\newcommand{\bigO}{\mathcal{O}}
\newcommand{\smallo}{{\scriptscriptstyle\mathcal{O}}} 
\newcommand{\R}{\mathbb{R}}
\renewcommand{\P}{\mathbb{P}}
\newcommand{\E}{\mathbb{E}}
\newcommand\iid{\stackrel{\mathclap{\normalfont\mbox{\tiny{iid}}}}{\sim}}
\newcommand{\ind}{\mathbbm{1}}
\newcommand\dist{\stackrel{\mathclap{\normalfont\mbox{\tiny{dist}}}}{=}}

\newcommand\PP{\stackrel{\mathclap{\normalfont\mbox{\tiny{$\P$}}}}{\longrightarrow}}

\newcommand{\cess}{c_{\text{\tiny{\(\mathrm{ESS}\)}}}}

\newtheorem{remark}{Remark}
\newtheorem{theorem}{Theorem}
\newtheorem{lemma}{Lemma}
\newtheorem{proposition}{Proposition}
\newtheorem{definition}{Definition}
\newtheorem{corollary}{Corollary}
\newtheorem{condition}{}

\newtheorem{conditionB}{}

\newtheorem{assumption}{}

\newtheorem{assumptionSSM}{}

\newtheorem{assumptionSSMB}{}

\newtheorem{assumptionMLEB}{}

\makeatother

\title{Self-Organizing State-Space Models with Artificial Dynamics}
\date{ }

\author{}

\begin{document}
\maketitle

\begin{center}
\vspace{-1cm}
\large{Yuan Chen$^{(1)}$, Mathieu Gerber$^{(1)}$\daggerfootnote{Corresponding author. Address: School of Mathematics, University of  Bristol, Queens Road, Bristol, BS8 1QU, UK. Email: mathieu.gerber@bristol.ac.uk}, Christophe Andrieu$^{(1)}$, Randal Douc$^{(2)}$}\\
\vspace{0.4cm}

\small{(1) School of Mathematics, University of Bristol, UK }\\
\vspace{0.1cm}

\small{(2) SAMOVAR, Telecom SudParis, Institut Polytechnique de Paris, France}
\vspace{0.2cm}

\end{center}

\begin{abstract}
We consider the problem of performing parameter and state inference 
in a state-space model (SSM) parametrized by a static parameter $\theta$. A popular idea to address this problem consists of incorporating $\theta$ in the state of the system and allowing its time evolution, modelled as a Markov chain $(\theta_t)_{t\geq 1}$. This proxy model defines a so-called self-organizing SSM (SO-SSM) to which one may apply standard particle filters. However, the practical implementation of this idea in a theoretically justified manner has remained an open problem until now. In this paper we fill this gap and in particular show that theoretically consistent SO-SSMs can be defined such that $\|\mathrm{Var}(\theta_{t+1}|\theta_{t})\|\rightarrow 0$ slowly as $t\rightarrow\infty$. This, in turn, leads to  particle filter algorithms for online inference in SSMs which we find to be robust in simulation. 
We also develop constructions of $(\theta_t)_{t\geq 1}$ and associated theoretical guarantees tailored to the application of SO-SSMs to  maximum likelihood estimation in SSMs, leading to novel iterated filtering algorithms. The algorithms developed in this work have the advantage of being simple to implement and  
to require minimal tuning to perform well.
 
\end{abstract}
 
\addtocontents{toc}{\protect\setcounter{tocdepth}{-1}}

\section{Introduction\label{sec:intro}}

\subsection{State-space models}

In this work we consider a  times series $(Y_t)_{t\geq 1}$ taking its values in some  space $\setY$   and whose distribution is modelled  using a parametric  state-space model (SSM). Denoting by $\setX$  the state-space of the model and by   $\Theta\subseteq\R^d$ the   parameter space,   we let $\{\chi_\theta(\dd x_1),\,\theta\in\Theta\}$ be a collection of probability distributions on $\setX$ and, for all $t\geq 1$ and $(\theta,x)\in\Theta\times \setX$, we let $ M_{t+1,\theta}(x,\dd x_{t+1})$ be a   probability distribution  on $\setX$ and $f_{t,\theta}(\cdot|x)$  be a  probability density functions on $\setY$  w.r.t.~some reference measure $\dd y$ (so that $f_{t,\theta}(y|x)\dd y$ is a probability distribution on $\setY$). Then,   throughout this paper we suppose that, for some $\theta\in\Theta$, the data generative process of  $(Y_t)_{t\geq 1}$ is  as follows:  
\begin{align}\label{eq:SSM}
X_1\sim \chi_\theta(\dd x_1),\quad Y_t|X_t \sim f_{t,\theta}(y|X_t)\dd y,\quad   X_{t+1}|X_t\sim M_{t+1,\theta}(X_{t},\dd x_{t+1}),\quad\forall t\geq 1.
\end{align}
 We stress that    we do not assume that the SSM is well-specified, that is that there exists a $\theta\in\Theta$ such that    \eqref{eq:SSM} is indeed the data generative process of $(Y_t)_{t\geq 1}$. To simplify notation,  for any $t\geq 1$ we use the shorthand $Y_{1:t}$ for the random variables $(Y_1,\dots,Y_t)$ while $y_{1:t}\in\setY^t$ denotes  an arbitrary realization of $Y_{1:t}$. In addition, for all   $t\geq 1$ we let $\bar{p}_{t,\theta}(\dd x_t|y_{1:t})$ be  the filtering distribution of $X_t$, that is the conditional distribution of $X_t$ given $Y_{1:t}=y_{1:t}$ under model \eqref{eq:SSM}.
 
In many applications, before using SSM \eqref{eq:SSM} to infer the latent states $(X_t)_{t\geq 1}$ or predict future observations, the model parameter $\theta$ must be learnt from the data, typically through either a Bayesian or a frequentist approach.  We refer the reader to  \citet{kantas2015particle} for a comprehensive overview of off-line and online parameter estimation methods, and to \citet{crisan2018nested} for an online algorithm that jointly performs Bayesian inference on both parameters and states in SSMs. Importantly, to address the intractability of both the likelihood function and its gradient in general SSMs, existing parameter inference approaches typically rely on  sophisticated  Monte Carlo algorithms which may be difficult to implement and use. The focus in this manuscript is on simple to implement, robust and easy to use methodology which can be justified theoretically.

\subsection{Self-organizing state-space models}

A  simple approach to bypass the need to estimate the  static parameter in \eqref{eq:SSM},  dating back at least to \citet{cox1964estimation}, involves incorporating it in the state and modelling it as an unobserved Markov chain $(\theta_t)_{t\geq 1}$  such that $\theta_1\sim \pi_1(\dd \theta_1)$ for some probability distribution $\pi_1(\dd \theta_1)$ on $\Theta$. In its most basic form one may assume such that $\P(\theta_t=\theta_1)=1$ for all $t\geq 2$. The resulting parameter-free SSM, later referred to by   \citet{kitagawa1998self} as a self-organizing SSM (SO-SSM), therefore assumes the following generative model for the observation process $(Y_t)_{t\geq 1}$:  
\begin{align}\label{eq:SO-SSM}
Y_t|(\theta_t,X_t)\sim f_{t,\theta_t}(y|X_t)\dd y,\quad\theta_{t+1}=\theta_t,\quad
    X_{t+1}|(\theta_{t+1},X_t)\sim M_{t+1,\theta_{t+1}}(X_{t},\dd x_{t+1})
\end{align}
for all $t\geq 1$, while $\theta_1\sim\pi_1(\dd\theta_1)$ and $X_1\sim \chi_{\theta_1}(\dd x_1)$.
The rationale behind this idea is simple: for any $t\geq 1$,  the filtering distribution of $\theta_t$ under model  \eqref{eq:SO-SSM}  coincides with the Bayesian posterior of $\theta$ under model \eqref{eq:SSM}, associated with the observations $Y_{1:t}$ and the prior distribution $\pi_1(\dd\theta_1)$; hence, deploying a filtering algorithm on SO-SSM  \eqref{eq:SO-SSM} jointly learns both  parameter $\theta$ and the hidden Markov process $(X_t)_{t\geq 1}$.

Although appealing in theory, the key difficulty of this approach is to compute or approximate numerically the filtering distributions of the  SO-SSM. In particular,  even when \eqref{eq:SSM} is linear Gaussian, the corresponding SO-SSM defined by \eqref{eq:SO-SSM} is usually non-linear, implying that the Kalman filter cannot be deployed on this model and more sophisticated numerical methods must be used.   Following the introduction  of particle filter (PF) algorithms for inference in nonlinear and non-Gaussian SSMs by \citet{gordon1993novel}, \citet{kitagawa1998self} proposed their use on SO-SSM \eqref{eq:SO-SSM} for joint parameter and state inference.   However, since the Markov chain $(\theta_t)_{t\geq 1}$ has no forgetting property, the resulting   PF suffers from sample impoverishment, and its estimate of the filtering distribution degrades quickly as $t$ increases \citep[][Section 6.2.1]{kantas2015particle}.

\subsection{Self-organizing state-space models with artificial dynamics}

A popular approach to  facilitate the deployment of    PF algorithms  on   SO-SSMs   consists of introducing  an artificial dynamics on the static parameter. This strategy, proposed e.g.~by  \citet{kadirkamanathan2000particle} \citep[see also][]{liu2001combined, gustafsson2003particle} requires one to choose a sequence $(K_{t})_{t\geq 2}$ of Markov kernels acting on $\Theta$ (that is, $K_{t}(\theta,\dd\theta_{t})$ is a probability distribution on $\Theta$ for all $t\geq 2$ and all $\theta\in\Theta$)  and then define the following instrumental SO-SSM  as generative model for $(Y_t)_{t\geq 1}$: 
\begin{align}\label{eq:SO-SSM_new2}
\begin{cases}
Y_t|(\theta_t,X_t)\sim f_{t,\theta_t}(y|X_t)\dd y,\\
\theta_{t+1}\sim K_{t+1}(\theta_t,\dd\theta_{t+1}),\\
X_{t+1}|(\theta_{t+1},X_t)\sim M_{t+1,\theta_{t+1}}(X_{t},\dd x_{t+1})
\end{cases} \quad t\geq 1
\end{align}
with, as above,   $\theta_1\sim\pi_1(\dd\theta_1)$ and $X_1\sim \chi_{\theta_1}(\dd x_1)$. For instance,  when $\Theta=\R^d$, a common choice is to let $ K_{t}(\theta_{t-1},\dd\theta_t)=\mathcal{N}_d(\theta_{t-1},\sigma^2_{t} I_d)$ for all $t\geq 2$. Algorithm \ref{algo:Generic_SSM} describes a simple bootstrap PF  applied to SO-SSM  \eqref{eq:SO-SSM_new2}. For the generic choice of $(K_{t})_{t\geq 2}$ above, ensuring asymptotic statistical accuracy, i.e.~that inferences from using \eqref{eq:SO-SSM} or \eqref{eq:SO-SSM_new2} eventually agree, requires $\sigma_t^2$ to vanish as $t\rightarrow\infty$ but, as we shall see, reliable 
PF algorithms require $\sigma_t^2$ not to vanish too quickly. Theoretical support   for this approach  has been obtained by \citet{gerber2022global} for models involving i.i.d.~observations (i.e.~assuming that $(Y_t)_{t\geq 1}$ is a sequence of i.i.d.~observations and that $f_{t,\theta}(y|x)\dd y$ is independent of $x$ for all $t\geq 1$ and all $\theta\in\Theta$) but remains missing for more general SSMs.

\begin{algorithm}[t]
\begin{algorithmic}[1]
\Require constants $N\in\mathbb{N}$ and $\cess\in(0,1]$
\vspace{0.1cm}

\State\label{init1}let $\theta_1^n\sim \pi_1(\dd\theta_1)$, $X_1^n\sim\chi_{\theta_1^n}(\dd x_1)$, $w_1^n= f_{1,\theta_1^n}(Y_1| X_1^n)$  and $W_1^n=w_1^n/\sum_{m=1}^N w_1^m$
\vspace{0.1cm}

\State let $p_{1,\Theta}^N(\dd \theta_1|Y_{1})=\sum_{n=1}^N W_1^n \delta_{ \{\theta_1^n\}}(\dd\theta_1)$ and  $p_{1,\setX}^N(\dd x_1|Y_{1})=\sum_{n=1}^N W_1^n \delta_{\{ X_1^n\}}(\dd x_1)$
\vspace{0.1cm}

\For{$t\geq 2$}
\vspace{0.1cm}

\State let $\mathrm{ESS}_{t-1}= 1/\sum_{m=1}^N (W_{t-1}^m)^2$ and $a_{n}=n$
\vspace{0.1cm}

\If{$\mathrm{ESS}_{t-1}\leq  N\,\cess$}

\State \label{Res}let  $\{a_{1},\dots,a_{N}\}= \text{\texttt{Resampling}}(\{W_{t-1}^{m}\}_{m=1}^N)$  and $w_{t-1}^n=1$

\EndIf
\vspace{0.1cm}

\State\label{Mut1} let $\theta_t^n\sim K_{t}(\theta_{t-1}^{a_{n}},\dd\theta_t)$ and $X_t^n\sim M_{t,\theta_{t}^{n}}(X_{t-1}^{a_{n}},\dd x_t)$

\State \label{W}let $w_t^n= w_{t-1}^n f_{t,\theta_t^n}(Y_t|X_t^n)$ and $W_t^n=w_t^n/\sum_{m=1}^N w_t^m$
\vspace{0.1cm}

\State \label{dist}let $p_{t,\Theta}^N(\dd \theta_t|Y_{1:t})=\sum_{n=1}^N W_t^n \delta_{ \{\theta_t^n\}}(\dd\theta_t)$ and  $p_{t,\setX}^N(\dd x_t|Y_{1:t})=\sum_{n=1}^N W_t^n \delta_{\{ X_t^n\}}(\dd x_t)$
\EndFor
\end{algorithmic}
\caption{Bootstrap particle filter for   SO-SSM \eqref{eq:SO-SSM_new2}
\\   (\small{Operations with index $n$ must be performed for all $n\in \{1,\dots,N\}$)}\label{algo:Generic_SSM}}
\end{algorithm}

\subsection{Contributions of the paper\label{sub:contribution}}


For $\Theta$ a compact parameter space we provide general conditions on model \eqref{eq:SSM} and design practically usable Markov kernels $(K_{t})_{t\geq 2}$ for which SO-SSM  \eqref{eq:SO-SSM_new2} combined with a PF can be used  to perform   joint state and parameter inference in SSMs with theoretical guarantees. 
An advantage of our theory is that it imposes minimal conditions on $(K_{t})_{t\geq 2}$. In particular, for any $\alpha>0$ one can choose $(K_{t})_{t\geq 2}$ such that $\|\mathrm{Var}(\theta_{t+1}|\theta_{t})\|\rightarrow 0$ at speed $t^{-\alpha}$, therefore allowing the use of SO-SSMs with slowly vanishing dynamics on $\theta$. Our  numerical experiments show that this slow decay is beneficial, allowing us to employ simple and easy-to-implement algorithms for joint parameter and state inference in SSMs that perform  well with little tuning. This robustness, however, is expected to come at the cost of reduced statistical efficiency, since adding a dynamics to $\theta$ is likely to lead to a loss of information about the target parameter, all the more so that $\|\mathrm{Var}(\theta_{t+1}|\theta_{t})\|\rightarrow 0$ vanishes slowly.

To alleviate this problem we take advantage of the fact that our theory also covers the scenario where $(K_{t})_{t\geq 2}$ is a sequence of random Markov kernels.
More specifically, we design SO-SSMs and associated PF algorithms for joint state and parameter estimation in SSMs that allow $\theta$ to evolve only at specific time instants determined adaptively. The resulting algorithms are computationally cheaper than those introducing a dynamics on $\theta$ at every iteration, and are found to be more accurate due to the reduced perturbations.

Although online  joint state and parameter inference originally motivated   SO-SSMs, this class of models is also useful for estimating the MLE of a given SSM for a finite dataset. Indeed, this latter problem can be equivalently formulated as parameter inference   in  a specific instance of SSM \eqref{eq:SSM} in which the observation process $(Y_t)_{t\geq 1}$ is defined as the periodic repetition of the available data points  \citep{ionides2015inference}.  In this   setting, Algorithm \ref{algo:Generic_SSM} belongs to the class of  iterated filtering (IF) algorithms introduced by \citet{ionides2015inference} for computing the MLE in SSMs. Iterated filtering is  popular among practitioners due to its simplicity and has proved to be a useful tool for parameter inference in challenging SSMs \citep[see e.g.][]{stocks2020model}. However these algorithms currently lack convincing theoretical guarantees \citep[see the results and assumptions in][]{ionides2015inference}. Our theoretical analysis of SO-SSMs therefore also enables us to develop a rigorous theory for IF, and to  introduce efficient IF algorithms with adaptive dynamics on $\theta$.
 
The theory underpinning our main results is not limited to SSMs and applies to a broader class of models we name Feynman-Kac models in random environments. Other applications and extensions are discussed in Section \ref{sec:conclusion}.

\subsection{Set-up, notation and outline of the paper} 

We assume henceforth that $\Theta$ is a regular compact set (see Definition \ref{def:sharp} in Appendix \ref{sec:theory}), that is $\Theta$ is a compact set having corners which are   not too sharp (e.g.~$\Theta$ is an hypercube), and  that the initial distribution $\pi_1(\dd \theta_1)$ has a strictly  positive density on $\Theta$ with respect to the Lebesgue measure on $\R^d$ (e.g.~$\pi_1(\dd\theta_1)$ is the uniform distribution  on $\Theta$).  For technical reasons,  we assume below that all the random variables are defined on a complete   probability space $(\Omega,\F,\P)$,  that the state-space $\setX$ is a Polish space (e.g.~$\setX=\R^{d_x}$), that $\setY$ is a measurable space (e.g.~$\setY=\R^{d_y}$) and that $\setX$ and $\Theta$  are  equipped with their   Borel $\sigma$-algebra.

For all $t\geq 1$     we denote by $p_{t}(\dd (\theta_t,x_t)|y_{1:t})$    the conditional distribution of $(\theta_t,X_t)$  given $Y_{1:t}=y_{1:t}$   under    SO-SSM \eqref{eq:SO-SSM_new2} and we let
\begin{align*}
p_{t,\Theta}(\dd \theta_t|y_{1:t})=\int_{\setX}p_{t}(\dd (\theta_t,x_t)|y_{1:t}),\quad p_{t,\setX}(\dd x_t|y_{1:t})=\int_{\Theta}p_t(\dd (\theta_t,x_t)|y_{1:t}).
\end{align*}
The symbol $\PP$ is used to denote the convergence in probability while $\E[Z]$ denotes the expectation of a random variable $Z$.  We denote by $\delta_{\{z'\}}(\dd z)$ the Dirac mass centred on $z'$, so that for any   set $A$ we have $\delta_{\{z'\}}(A)=1$ if $z'\in A$ while  $\delta_{\{z'\}}(A)=0$ otherwise. To introduce some further notation  let $k\in\mathbb{N}$,   $\Sigma$ be a $k\times k$ symmetric and positive definite matrix, $m\in\R^k$, $\nu\in(0,\infty)$ and $A\subseteq\R^k$. Then, we denote by $\mathrm{TN}_A(m,\Sigma)$    the $k$-dimensional Gaussian distribution with mean $m$ and covariance matrix $\Sigma$, truncated on the set $A$, and we denote by $\mathrm{TS}_A(m,\Sigma,\nu)$ the $k$-dimensional Student-t distribution with location parameter $m$, scale matrix $\Sigma$ and $\nu$ degrees of freedom, truncated on $A$. Below we use the  convention that  $\mathrm{TS}_A(m,c\Sigma,\nu)=\mathrm{TN}_A(m,c\Sigma)=\delta_{\{m\}}(\dd z)$ if $c=0$.

The rest of the paper is organized as follows.   Appendix \ref{sec:theory} contains our general theory for self-organizing Feynman-Kac models  in  random environments. This theory is applied in Section \ref{sec:SSM} to define SO-SSMs  for online joint parameter and state inference in SSMs, and in Section \ref{sec:MLE}  to propose new IF algorithms for MLE estimation. The behaviour of the algorithms discussed in this work is illustrated in Section \ref{sec:num} through numerical experiments, and Section \ref{sec:conclusion} concludes. All the proofs are gathered in the   \hyperlink{SM}{Supplementary Material}  and the code used for the numerical experiments of Section \ref{sec:num} is available on GitHub\footnote{at \url{github.com/mathieugerber/paperSO-SSM}}.

\section{Online state and parameter inference  in SSMs\label{sec:SSM}}

\subsection{Learning task and  first assumptions  \label{sub:thetstar}}

To formally define the parameter value $\theta_\star$ we aim at estimating in this section,  for all  $t\geq 1$ we let $L_t(\cdot)$ denotes the  likelihood function of SSM \eqref{eq:SSM}, that is we let
\begin{align}\label{eq:log_lik}
L_t(\theta)=   \int_{\setX^{t+1}} \chi_\theta(\dd x_1)\prod_{s=1}^t f_{s,\theta}(Y_s|x_s)M_{s+1,\theta}(x_{s},\dd x_{s+1}),\quad\forall \theta\in\Theta.
\end{align}
Under the assumptions   imposed below  on the SSM  there exists a function $l:\Theta\rightarrow\R$ such that  $t^{-1}\log L_t(\theta)\PP l(\theta)$ for all $\theta\in\Theta$, and we let $\theta_\star=\argmax_{\theta\in\Theta}l(\theta)$. To simplify the discussion we suppose   that the function $l(\cdot)$ has a unique global maximum  but this condition is not needed for our theoretical results to apply (see Appendix \ref{sec:theory}). We remind the reader that  the MLE $\hat{\theta}_t\in\argmax_{\theta\in\Theta}L_t(\theta)$ converges to $\theta_\star$ as $t\rightarrow\infty$ \citep[see e.g.][]{Douc_MLE}  and that, in a Bayesian setting, the  posterior distribution   $\bar{\pi}_t(\dd\theta)\propto L_t(\theta)\pi_1(\dd \theta)$ concentrates on $\theta_\star$ as $t\rightarrow\infty$  \citep{douc2020posterior}.  

For all $t\geq 1$ and $\theta\in\Theta$      the filtering distribution $\bar{p}_{t,\theta}(\dd x_t|Y_{1:t})$ is formally defined by
\begin{align*}
\bar{p}_{t,\theta}(\dd x_t|Y_{1:t})=\frac{f_{t,\theta}(Y_t|x_t)\int_{\setX^{t}}  \chi_\theta(\dd x_1) \prod_{s=1}^{t-1}f_{s,\theta}(Y_s|x_s) M_{s+1,\theta}(x_{s}, \dd x_{s+1})}{L_t(\theta)}
\end{align*}
and, with the above definition of $\theta_\star$, the state estimation problem we consider in this section  is that of estimating the optimal filtering distribution  $\bar{p}_{t,\theta_\star}(\dd x_t|Y_{1:t})$ for $t\geq 1$.

In order for the log-likelihood function $t^{-1}\log L_t(\cdot)$ to converge to some function $l(\cdot)$  we need to impose some stationarity conditions on the observation process $(Y_t)_{t\geq 1}$ as well as some minimal regularity assumptions on the model. To this aim, we  let $\tau\in\mathbb{N}$ and throughout this section we    assume  that   $(Y_t)_{t\geq 1}$ is a $\tau$-stationary process and SSM \eqref{eq:SSM} a $\tau$-periodic SSM, in the sense of the following definition:
\begin{definition}\label{def:periodic}
We say that  $(Y_t)_{t\geq 1}$  is a $\tau$-stationary process if $(Z_t)_{t\geq 1}$ is a stationary and ergodic process, where $Z_t=(Y_{(t-1)\tau+1},\dots,Y_{t\tau})$ for all $t\geq 1$, and we say that  \eqref{eq:SSM} is $\tau$-periodic SSM if, for all $t\geq 1$ and all $i\in\{1,\dots,\tau\}$, we have $f_{t\tau+i,\theta}(y|x')=f_{i,\theta}(y|x')$ and $ M_{t\tau +i+1,\theta}(x',\dd x)=  M_{i+1,\theta}(x',\dd x)$ for all $ (\theta,x',y)\in\Theta\times\setX\times\setY$.
\end{definition}

When $\tau=1$  this assumption is the standard choice for studying the behaviour of the MLE and of the Bayesian posterior distribution in SSMs (see e.g.~the last two references). In this case, $(Y_t)_{t\geq 1}$ is assumed to be a stationary process, an assumption that may not hold for the original data but could be satisfied after an appropriate transformation \citep[see][and references therein]{salles2019nonstationary}. If $(Y_t)_{t\geq 1}$ is $\tau$-stationary process for some $\tau>1$, the  observation process is allowed to have a periodic (or seasonal) behaviour.

\subsection{Outline of the main results of this section\label{sub:res_sketch}}

In this section  we apply the theory presented in Appendix \ref{sec:theory} to obtain conditions on    SSM   \eqref{eq:SSM} and to define Markov kernels $(K_t)_{t\geq 2}$  for which  SO-SSM  \eqref{eq:SO-SSM_new2}    can be used  to consistently perform online joint state and parameter inference with the following theoretical guarantees:
\begin{align}\label{eq:goal_online}
\int_\Theta \theta_t\, p_{t,\Theta}(\dd \theta_t |Y_{1:t})\PP \theta_\star,\quad  \big\|p_{t,\setX}(\dd x_t|Y_{1:t})-\bar{p}_{t,\theta_\star}(\dd x_t|Y_{1:t})\big\|_{\mathrm{TV}}\PP 0\quad\text{ as $t\rightarrow\infty$}.
\end{align}

For all $t\geq 1$ let  $p_{t,\Theta}^N(\dd \theta_t|Y_{1:t})$ and  $p_{t,\setX}^N(\dd x_t| Y_{1:N})$ be the particle approximations of the two distributions   $p_{t,\Theta}(\dd \theta_t |Y_{1:t})$ and $p_{t,\setX}(\dd x_t |Y_{1:t})$ computed by Algorithm \ref{algo:Generic_SSM}. Then, under the technical conditions on \eqref{eq:SSM} imposed later in this section, and by following standard calculations used to derive error bounds for the mean squared error of PF algorithms \citep[see e.g.][Section 11.2.2]{chopin2020introduction},  we readily obtain that if \eqref{eq:goal_online} holds then, for some sequences $(C_t)_{t\geq 1}$ and $(\epsilon_t)_{t\geq 1}$  of $(0,\infty)$-valued random variables such that $ \epsilon_t\PP 0$,  we have $\P$-a.s.
\begin{align}\label{eq:cor_SSM2_1}
&\E\Big[ \big\|\int_\Theta\theta_t\,p_{t,\Theta}^N(\dd \theta_t|Y_{1:t})-\theta_\star\big\|\,\Big|Y_{1:t}\Big]\leq\frac{C_t}{\sqrt{N}}+\epsilon_t, \quad \forall (N,t)\in\mathbb{N}^2
\end{align}
and,  for any function $\varphi:\setX\rightarrow [-1,1]$,
\begin{align}\label{eq:cor_SSM2_2}
&\E\Big[\big|\int_\setX\varphi(x_t) \Big(p^N_{t,\setX}(\dd x_t|Y_{1:t})-\bar{p}_{t,\theta_\star}(\dd x_t|Y_{1:t})\Big)\big|\,\Big| Y_{1:t}\Big]\leq\frac{C_t}{\sqrt{N}}+\epsilon_t, \quad \forall (N,t)\in\mathbb{N}^2.
\end{align}

We stress on the fact that \eqref{eq:cor_SSM2_1}-\eqref{eq:cor_SSM2_2} are not informative about the performance of Algorithm \ref{algo:Generic_SSM}   since the behaviour of $C_t$ as $t$ increases is unknown. In particular, \eqref{eq:cor_SSM2_1}-\eqref{eq:cor_SSM2_2} also hold  for SO-SSM \eqref{eq:SO-SSM} without   dynamics on $\theta$ but the algorithm performs poorly. Indeed if in Algorithm \ref{algo:Generic_SSM}  we have  $K_t(\theta,\dd\theta_t)=\delta_{\{\theta\}}(\dd\theta_t)$ for all $t\geq 2$ and all $\theta \in\Theta$   then,  as $t$ increases,  the distribution $p_{t,\Theta}^N(\dd\theta_t| Y_{1:t})$  converges quickly to a Dirac mass at one of the $N$ particles $\{\theta_1^n\}_{n=1}^N$ sampled from $\pi_1(\dd\theta_1)$ on Line \ref{init1}. It can be shown that for this SO-SSM   \eqref{eq:cor_SSM2_1}-\eqref{eq:cor_SSM2_2} only hold  for a sequence of random variables $(C_t)_{t\geq 1}$ that, in some sense,  diverges to infinity quickly   with $t$ \citep[see][and references therein]{del2010forward}. By contrast, it is expected that  adding a  dynamics on $\theta$ improves the dependence in $t$ of the bounds provided in \eqref{eq:cor_SSM2_1}-\eqref{eq:cor_SSM2_2}. In particular, we conjecture that if the   dynamics on $\theta$ vanishes sufficiently slowly  then \eqref{eq:cor_SSM2_1}-\eqref{eq:cor_SSM2_2} hold for a sequence $(C_t)_{t\geq 1}$ bounded in probability.

\subsection{SO-SSM models with fast vanishing artificial dynamics\label{sub:ssm_cont_fast}}


Let $K_{t}(\theta,\dd\theta_t)=\mathrm{TN}_{\Theta}\big(\theta , h^2_t \Sigma\big)$ for all $ \theta \in\Theta$ and all $t\geq 2$,  with $\Sigma$  a $d\times d$ symmetric and positive definite matrix and $(h_t)_{t\geq 2}$   a sequence in $[0,\infty)$. Then, under suitable conditions on  SSM   \eqref{eq:SSM} (see Section \ref{sub:theorySSM}) and the choice $h_t=\smallo(t^{-1})$, the consistency results in \eqref{eq:goal_online}-\eqref{eq:cor_SSM2_2} hold. However, this choice implies that the Markov chain $(\theta_t)_{t\geq 1}$ is such that $\|\mathrm{Var}(\theta_t|\theta_{t-1})\|=\bigO(h^2_t)=\smallo(t^{-2})$, that is, the lack of flexibility in the choice of $(h_t)_{t\geq 2}$ leads to a rapidly vanishing dynamics on $\theta$, making  estimation of the SO-SSM by a PF algorithm challenging.
 
\subsection{SO-SSM models with slowly vanishing  artificial dynamics\label{sub:ssm_cont_slow}}

We now define Markov kernels $(K_t)_{t\geq 2}$ inducing a more persistent dynamics on $\theta$ while also ensuring that the consistency results in \eqref{eq:goal_online}-\eqref{eq:cor_SSM2_2} hold. Let $\Sigma$ be a $d\times d$ symmetric and positive definite matrix, let
$(\alpha,\nu)\in(0,\infty)^2$  and $(t_1,\Delta)\in\mathbb{N}^2$ be some constants,  let  $t_p=t_{p-1}+\Delta\lceil \log (t_{p-1})^{2}\rceil$ for all $p\geq 2$ and  let $(h_t)_{t\geq 2}$ be a sequence in $[0,\infty)$ such that $h_t=\bigO(t^{-\alpha})$ and such that $ \liminf_{p\rightarrow\infty} t_p^\alpha h_{t_p}>0$ (e.g.~$h_t=t^{-\alpha}$ for all $t\geq 2$). Then,  under suitable conditions on  SSM   \eqref{eq:SSM} (see Section \ref{sub:theorySSM}), the results in \eqref{eq:goal_online}-\eqref{eq:cor_SSM2_2} hold  when  
\begin{align}\label{eq:SSM2}
K_t(\theta,\dd\theta_t)=
\begin{cases}
 \mathrm{TS}_{\Theta}\big(\theta , h^2_t\Sigma, \nu\big), &t\in (t_p)_{p\geq 1}\\
  \mathrm{TN}_{\Theta}\big(\theta , h^2_t\Sigma\big), &t\not\in (t_p)_{p\geq 1}
 \end{cases},\quad  \theta \in\Theta,\quad t\geq 2.
\end{align}
Under this definition of $(K_t)_{t\geq 2}$  the  Markov chain $(\theta_t)_{t\geq 1}$ is   such that   $\|\mathrm{Var}(\theta_t|\theta_{t-1})\|=\bigO(t^{-2\alpha})$, and therefore \eqref{eq:SSM2} can be used to   define a  dynamics on $\theta$ that vanishes at any chosen polynomial rate. It is this flexibility that allows the design of robust PF approximations in the sequel.  The definition \eqref{eq:SSM2} of $(K_{t })_{t\geq 2}$ is inspired by that introduced in \citet{gerber2022global} for models with i.i.d.~observations, and we refer the reader to this reference for an explanation of its rationale.

\subsection{SO-SSMs with adaptive artificial dynamics\label{sub:ssm_adapt}}

As argued earlier, adding a non-degenerate dynamics on $\theta$ facilitates the deployment of PF algorithms on SO-SSMs. However, adding a dynamics on $\theta$ is undesirable from a statistical point  of view (as argued in  Section \ref{sub:contribution}) and makes the deployment of a PF on the resulting SO-SSM computationally more expensive.

It therefore seems sensible to introduce a persistent dynamics on $\theta$ only when needed,  in an adaptive manner. More specifically at iteration $t\geq 2$ of  Algorithm \ref{algo:Generic_SSM} it is sensible to impose a non-degenerate dynamics on $\theta$ using $K_{t}$ only if the support of  the distribution $p_{t-1,\Theta}^N(\dd\theta_{t-1}| Y_{1:(t-1)})$ is, in some sense, too small. Following this idea, Algorithm \ref{algo:online_2} presents a PF scheme  for joint parameter and state inference in SSMs where an adaptive strategy is used to determine at which iterations a dynamics on $\theta$ should be added. In particular, in Algorithm \ref{algo:online_2}, if we let $\alpha>1$, then at iteration $t\geq 2$  the Markov kernel $K_t$ defined in Section \ref{sub:ssm_cont_fast} is used to propagate $\theta$  only when a resampling step is performed. Conversely, when $\alpha\leq 1$, the kernel $K_t$ from  Section \ref{sub:ssm_cont_slow} is applied at iteration $t$ if a resampling step is triggered or if $t\in (t_p)_{p\geq 1}$.


We show in Theorem~\ref{thm:SSM2} that the results in \eqref{eq:cor_SSM2_1}-\eqref{eq:cor_SSM2_2} remain  valid for Algorithm \ref{algo:online_2}. Key to the analysis is to view Algorithm \ref{algo:online_2} as a PF applied to the SO-SSM \eqref{eq:SO-SSM_new2} for which the kernel sequence $(K_t)_{t\geq 2}=(K_t^N)_{t\geq 2}$ is selected at random, according to the resampling events at random times of the PF. Both properties of this SO-SSM and its PF implementation then follow from uniform properties of $(K_t^N)_{t\geq 2}$ in the PF realisation, and we will see that, in \eqref{eq:cor_SSM2_1}-\eqref{eq:cor_SSM2_2},  the sequences $(C_t)_{t\geq 1}$ and $(\epsilon_t)_{t\geq 1}$ are independent of the number of particles $N$ under our assumptions.

\begin{algorithm}[!t]

\begin{algorithmic}[1]
\Require constants $(N,t_1,\Delta)\in\mathbb{N}^3$,  $\cess\in(0,1]$, $\alpha\in(0,\infty)$ and $\nu\in (0,\infty)$, and a $d\times d$ symmetric

\hspace{-1.4cm}   and  positive definite  matrix $\Sigma$

\vspace{0,2cm}

\State let $p=1$, $\theta_1^n\sim \pi_1(\dd\theta_1)$, $X_1^n\sim\chi_{\theta_1^n}(\dd x_1)$,   $w_1^n= f_{1,\theta_1^n}(Y_1| X_1^n)$ and $W_1^n=w_1^n/\sum_{m=1}^N w_1^m$
\vspace{0.1cm}
 
\State   let $p_{1,\Theta}^N(\dd \theta_1|Y_{1})=\sum_{n=1}^N W_1^n \delta_{ \{\theta_1^n\}}(\dd\theta_1)$ and  $p_{1,\setX}^N(\dd x_t|Y_{1 })=\sum_{n=1}^N W_1^n \delta_{\{ X_1^n\}}(\dd x_1)$

\For{$t\geq 2$}
\vspace{0,1cm}

\State let $\mathrm{ESS}^N_{t-1}= 1/\sum_{m=1}^N (W_{t-1}^m)^2$, $a_{n}=n$ and $\theta_t^n=\theta_{t-1}^n$ 
\vspace{0,1cm}

\If{$\mathrm{ESS}^N_{t-1}\leq  N\,\cess$} 
\vspace{0,1cm}

\State   let  $\{a_{1},\dots,a_{N}\}=\text{\texttt{Resampling}}(\{W_{t-1}^{m}\}_{m=1}^N)$ and $w_{t-1}^n=1$
\If{$\alpha> 1$}
\State   let  $\theta_t^n\sim \mathrm{TN}_\Theta(\theta_{t-1}^{a_n},t^{-2\alpha}\Sigma)$
\EndIf
\EndIf
\If{$\alpha\leq 1$ and either $t=t_p$\text{ or }$\mathrm{ESS}^N_{t-1}\leq  N\,\cess$}
\State let  $\theta_t^n\sim \mathrm{TS}_\Theta(\theta_{t-1}^{a_n}, t^{-2\alpha}\Sigma,\nu)$,  $t_{p+1}=t_{p}+\Delta\lceil \log (t_{p})^{2}\rceil$ and $p\gets p+1$
\EndIf
\vspace{0,1cm}

\State let $X_t^n\sim M_{t,\theta_{t}^{n}}(X_{t-1}^{a_{n}},\dd x_t)$, $w_t^n= w_{t-1}^n f_{t,\theta_t^n}(Y_t|X_t^n)$ and $W_t^n=w_t^n/\sum_{m=1}^N w_t^m$

\State   let $p_{t,\Theta}^N(\dd \theta_t|Y_{1:t})=\sum_{n=1}^N W_t^n \delta_{ \{\theta_t^n\}}(\dd\theta_t)$ and  $p_{t,\setX}^N(\dd x_t|Y_{1:t})=\sum_{n=1}^N W_t^n \delta_{\{ X_t^n\}}(\dd x_t)$
\EndFor
\vspace{0,2cm}
\end{algorithmic}
\caption{Online inference in SSMs  with  adaptive artificial dynamics on $\theta$\\
(\small{Operations with index $n$ must be performed for all $n\in \{1,\dots,N\}$)}\label{algo:online_2}}
\end{algorithm}

\subsection{Practical recommendations\label{sub:pratical}}

As a default approach for online joint state and parameter inference in SSMs we recommend to take for $\pi_1(\dd\theta_1)$ the uniform distribution on $\Theta$ (recall that $\Theta$ is assumed to be a compact set), and to use Algorithm \ref{algo:online_2} with  $\alpha=0.5$, $\Delta=1$, $\nu=t_1=100$ and with $\Sigma$ the identity matrix, assuming that the model is parametrized in such a way that all the components of $\theta_\star$ are roughly on a similar scale. This default approach is shown to perform well in Section \ref{sec:num}, even on challenging   SSMs.  We however stress that, by    constraining  the values the model parameters can take, the choice of $\Theta$  may play a  key role in the performance of Algorithm \ref{algo:online_2}. 
It is also worth mentioning that, when estimating the filtering distribution of a SO-SSM, more efficient alternatives to the bootstrap PF are  available, including the guided particle filter in general and, when \eqref{eq:SSM} is a linear Gaussian SSM, the marginalized (or Rao-Blackwellised) PF   \citep[see][Chapter 10, for a discussion of these alternative PF algorithms]{chopin2020introduction}.

To motivate the specific choice  $\alpha=0.5$, assume that $\cess=1$ in Algorithm \ref{algo:online_2}, so that $\|\mathrm{Var}(\theta_{t}|\theta_{t-1})\|\approx t^{-2\alpha}$.  In this case, if $\alpha$ is large, then as $t\rightarrow\infty$,  the resulting distribution $p_{t,\Theta}^N(\dd \theta_t|Y_{1:t})$ will typically  concentrate quickly on some   $\theta'\in\Theta$  such that $\theta'\neq \theta_\star$. From an optimization perspective, premature concentration is likely to occur because, for large values of  $\alpha$,  Algorithm \ref{algo:online_2} does not explore  the parameter space $\Theta$ well. On the other hand, if $\alpha$ is too small, Algorithm \ref{algo:online_2} explores $\Theta$ better, but the distribution $p_{t,\Theta}^N(\dd \theta_t|Y_{1:t})$ may concentrate on $\theta_\star$ too slowly. With this trade-off in mind, we recommend setting $\alpha=0.5$ as the largest exponent value for which $\sum_{t\geq 2} \|\mathrm{Var}(\theta_{t}|\theta_{t-1})\|\approx \sum_{t\geq 2} t^{-2\alpha }=\infty$. Choosing $\alpha=0.5$ therefore maximizes the speed at which $p_{t,\Theta}^N(\dd \theta_t|Y_{1:t})$ can concentrate on $\theta_\star$ while still preserving the ability of Algorithm \ref{algo:online_2} to reach any part of $\Theta$ in finite time.

Although our theoretical results require choosing $\nu<\infty$ when $\alpha=0.5$, in several (unreported) numerical experiments we have observed that the choice of $\nu$, and thus of $\Delta$ and $t_1$,  has  no noticeable impact on the behaviour of  Algorithm \ref{algo:online_2}. The choice $\nu=100$ is  motivated by the fact that, in this case, $\mathrm{TS}_{\Theta}\big(\theta , t^{-2\alpha}\Sigma, \nu\big)\approx \mathrm{TN}_{\Theta}\big(\theta, t^{-2\alpha}\Sigma \big)$ for all $t\geq 2$ and all $\theta\in\Theta$, making the values of the two parameters $\Delta$ and $t_1$ even less relevant to the behaviour of  Algorithm \ref{algo:online_2}; as a result the proposed values $\Delta=1$ and $t_1=100$ are somewhat arbitrary. Finally, choosing the uniform distribution on $\Theta$ for $\pi_1(\dd\theta_1)$ and the identity matrix for $\Sigma$ is natural when no prior information about the profile of the likelihood function or the value of  $\theta_\star$ is   available.

\subsection{Theory \label{sub:theorySSM}}

\subsubsection{Technical assumptions on the state-space model\label{sub:assumeSSM}}

In addition to the assumptions already   stated,    we   assume  that for all $(t,\theta,x)\in\mathbb{N}\times\Theta\times\setX$, the probability measure $M_{t+1,\theta}(x,\dd x_{t+1})$ is absolutely continuous w.r.t.~some $\sigma$-finite measure $\lambda(\dd x)$ on $\setX$, and thus for all $t\geq 1$ and $(\theta,x)\in\Theta\times \setX$, we have $M_{t+1,\theta}(x,\dd x_{t+1})=m_{t+1,\theta}(x_{t+1}|x)\lambda(\dd x_{t+1})$ for some (measurable) function $m_{t+1,\theta}(\cdot|x)$. We also consider a set of  four technical assumptions, referred to as Assumptions  \ref{assumeSSM:K_set}-\ref{assumeSSM:smooth}. These assumptions are stated in Appendix \ref{app:assume_online} and here we only briefly describe their role in our analysis. Assumptions \ref{assumeSSM:K_set}-\ref{assumeSSM:G} are borrowed from \citet{Douc_MLE}, whose results allow  us to   show that the log-likelihood function  of SSM \eqref{eq:SSM} converges to some function $l(\cdot)$ (Proposition \ref{prop:conv_SSM_0} below).  
The additional  condition   \ref{assumeSSM:smooth}
notably requires that,  for all $t\geq 1$ and $(y,x,x')\in\setY\times\setX^2$, the two mappings $\theta\mapsto \log f_{t,\theta}(y|x)$  and $\theta\mapsto \log m_{t+1,\theta}(x|x')$ are smooth. Assumption  \ref{assumeSSM:smooth} is required to control the impact that the dynamics on $\theta$ has on the estimation process.  In Section \ref{sup:examples} of the \hyperlink{SM}{Supplementary Material} we prove that, under weak conditions on $(Y_t)_{t\geq 1}$,
Assumptions \ref{assumeSSM:K_set}-\ref{assumeSSM:smooth} are for instance  satisfied for  a large class of multivariate linear Gaussian SSMs (including the SSM used for the numerical experiments in Section \ref{sub:Gaussian}) and, with a simple stochastic volatility model,  we show that the validity of this assumption is no limited to  linear  Gaussian SSMs.

 \subsubsection{Main results\label{sub:assume_informal}}
 
The next proposition shows that under the   assumptions we consider,  the log-likelihood function of SSM \eqref{eq:SSM} converges to some function $l(\cdot)$ as $t\rightarrow\infty$.
 
\begin{proposition}\label{prop:conv_SSM_0}
Assume that Assumptions \ref{assumeSSM:K_set}-\ref{assumeSSM:G} hold and   that there exists a compact set $E\subset\setX$ such that $\lambda(E)>0$ and such that $\chi_\theta(E)>0$ for all $\theta\in\Theta$. Then, there exists a function $l:\Theta\rightarrow\R$ such that $\P\big(\lim_{t\rightarrow\infty} t^{-1}\log L_{t}(\theta)=l(\theta)\big)=1$ for  all $\theta\in\Theta$, where $L_t(\cdot)$ is defined in \eqref{eq:log_lik}.
\end{proposition}

The following two theorems are the main results of this section:

\begin{theorem}\label{thm:SSM}
 Assume that  Assumptions \ref{assumeSSM:K_set}-\ref{assumeSSM:smooth} hold, that
\begin{align}\label{eq:stability_SSM}
\text{$\exists$  compact sets  $E, E'\subset \setX$ s.t. } \log \bar{p}_{t,\theta_\star}(E|Y_{1:t})=\bigO_\P(1),\quad \log p_{t,\setX}(E'|Y_{1:t})=\bigO_\P(1)
\end{align}
and, with $l:\Theta\rightarrow\R$   as in  Proposition \ref{prop:conv_SSM_0}, assume that  $\{\theta_\star\}=\argmax_{\theta\in\Theta}l(\theta)$ for some  $\theta_\star\in \mathring{\Theta}$.  Finally, let $\chi$ be  a probability measure on $\setX$  such that $\chi(E'')>0$ for any compact set $E''\subset\setX$ for which $\lambda(E'')>0$, and assume that $\chi_\theta(\dd x)=\int_\setX M_{\tau+1,\theta}(x',\dd x)\chi(\dd x')$ for all $\theta\in\Theta$. Then, \eqref{eq:goal_online} holds for   $(K_t)_{t\geq 2}$ defined in Section  \ref{sub:ssm_cont_fast} and Section \ref{sub:ssm_cont_slow}. In addition, if there exists a compact set $E'''\subset\setX$ such that $\inf_{(\theta,x)\in\Theta\times\setX}M_{t+1,\theta}(x,E''')>0$ for all $t\in\{1,\dots,\tau\}$  then \eqref{eq:stability_SSM}   holds.
\end{theorem}

\begin{theorem}\label{thm:SSM2}
 Assume that  Assumptions \ref{assumeSSM:K_set}-\ref{assumeSSM:smooth} hold, that there exists  a compact set  $E\subset\setX$ such that $\inf_{(\theta,x)\in\Theta\times\setX}M_{t+1,\theta}(x,E)>0$ for all $t\in\{1,\dots,\tau\}$ and, with $l:\Theta\rightarrow\R$   as in  Proposition \ref{prop:conv_SSM_0},  that  $\{\theta_\star\}=\argmax_{\theta\in\Theta}l(\theta)$ for some  $\theta_\star\in \mathring{\Theta}$.  Finally,  assume that $\{\chi_\theta(\dd x_1),\,\theta\in\Theta\}$ is as in Theorem \ref{thm:SSM} and for   all $(t,N)\in\mathbb{N}^2$   let  $p^N_{t,\Theta}(\dd \theta_t|Y_{1:t})$  and   $p^N_{t,\setX}(\dd x_t|Y_{1:t})$ be as defined by Algorithm \ref{algo:online_2}.  Then,   there exist two sequences $(C_t)_{t\geq 1}$ and $(\epsilon_t)_{t\geq 1}$ of   $(0,\infty)$-valued random variables, such that $\epsilon_t\PP 0$ and such that \eqref{eq:cor_SSM2_1}-\eqref{eq:cor_SSM2_2} hold.
\end{theorem}

The assumption stated in   \eqref{eq:stability_SSM}   can be interpreted as a stability requirement for both the SSM and the SO-SSM. Although this condition appears mild, showing that it holds without additional strong assumptions   on the SSM   is not currently easy. Remark that \eqref{eq:stability_SSM}  holds when the state space $\setX$ is compact, while the last part of Theorem \ref{thm:SSM} provides a sufficient criterion when $\setX$ is not a compact set. 
The condition  on  $\{\chi_\theta(\dd x_1),\,\theta\in\Theta\}$ ensures that  SO-SSM \eqref{eq:SO-SSM_new2} fits in the general framework covered by our theory (see Appendix \ref{sec:theory}).  To prove Theorem~\ref{thm:SSM2} we need the second part of \eqref{eq:stability_SSM} to hold uniformly in the realisations of the PF and in the number of particles $N$, which is ensured by assuming $\inf_{(\theta,x)\in\Theta\times\setX}M_{t+1,\theta}(x,E)>0$ for all $t\in\{1,\dots,\tau\}$ and some compact set $E\subset\setX$. 
We expect one to be able to relax this condition in future.


\section{Maximum likelihood estimation in SSMs\label{sec:MLE}}

\subsection{Problem statement }

In this section, we let   $(\tilde{y}_1,\dots,\tilde{y}_T)$ with $T\geq 2$ denote elements of $\setY$,  assumed to be a realization of some  random vector  $(\tilde{Y}_1,\dots,\tilde{Y}_T)$  for which the following parametric SSM   is used as  generative model:
\begin{equation}\label{eq:SSM_MLE}
\begin{split}
&\tilde{Y}_1|\tilde{X}_1\sim \tilde{f}_{1,\theta}(y|\tilde{X}_1)\dd y,\quad \tilde{X}_1\sim\tilde{\chi}_\theta(\dd x_1)\\
&\tilde{Y}_s|\tilde{X}_s\sim \tilde{f}_{s,\theta}(y|\tilde{X}_s)\dd y,\quad  \tilde{X}_{s}|\tilde{X}_{s-1}\sim \tilde{M}_{s,\theta}(\tilde{X}_{s-1},\dd x_{s}),\quad s\in\{2,\dots,T\}
\end{split}
\end{equation}
where, as above, $\theta\in\Theta$ and $\setX$ is the state-space. Then,   we consider the problem of using an SO-SSM to compute the MLE  of $\theta$ in   model \eqref{eq:SSM_MLE}; that is, to compute the parameter value $\tilde{\theta}_T:=\argmax_{\theta\in\Theta} \tilde{L}_T(\theta)$ with the function  $\tilde{L}_T(\cdot)$ defined by
\begin{align}\label{eq:lik_IF}
\tilde{L}_T(\theta)=  \int_{\setX^{T+1}} \tilde{\chi}_\theta(\dd x_1)\tilde{f}_{1,\theta}(\tilde{y}_1|x_1)\prod_{s=2}^T \tilde{f}_{s,\theta}(\tilde{y}_s|x_s)\tilde{M}_{s,\theta}(x_{s-1},\dd x_{s}),\quad \theta\in\Theta.
\end{align}
To simplify,    we assume  that the  likelihood function  $\tilde{L}_T(\cdot)$ has a unique global maximizer but we stress that our theoretical results are not limited to this case.

\subsection{MLE estimation as an online parameter inference problem\label{sub:mle_ssm} }

In order  to compute $\tilde{\theta}_T$ with an SO-SSM,    we first define the observation process  $(Y_t)_{t\geq 1}$ and SSM \eqref{eq:SSM} in such a way that $\tilde{\theta}_T$ corresponds to the parameter value $\theta_\star$ specified in the previous section. With $L_t(\cdot)$ as defined in \eqref{eq:log_lik} for all $t\geq 1$, this means specifying  $(Y_t)_{t\geq 1}$ and SSM \eqref{eq:SSM} so that, as $t\rightarrow\infty$, the  function $t^{-1}\log L_t(\cdot)$   converges to some function $l(\cdot)$ for which we have $\tilde{\theta}_T=\argmax_{\theta\in\Theta}l(\theta)$.

To this aim, and following \citet{ionides2015inference},  we define the  observation process $(Y_t)_{t\geq 1}$ by infinitely cloning the $T$ successive data points $\{\tilde{y}_s\}_{s=1}^T$, specifically, we let $(Y_t)_{t\geq 1}$ be such that $Y_{(t-1)T+s}=\tilde{y}_{s}$ for all $t\geq 1$ and all  $s\in\{1,\dots,T\}$. Then, we  model the distribution of $(Y_t)_{t\geq 1}$  with  SSM \eqref{eq:SSM} where, for all $(\theta,x,y)\in\Theta\times\setX\times\setY$, all $s\in\{1,\dots,T\}$ and all $t\geq 1$,
\begin{align}\label{eq:SSM_t}
& f_{(t-1)T+s,\theta}(y|x)=\tilde{f}_{s,\theta}(y|x),\quad  M_{(t-1)T+s,\theta}(x,\dd x_t)=
\begin{cases}
\tilde{M}_{s,\theta}(x,\dd x_t), &s\neq 1 \\
\tilde{\chi}_\theta(\dd x_t), &s=1.
\end{cases}
\end{align}
With the above definitions of $(Y_t)_{t\geq 1}$ and SSM \eqref{eq:SSM},  we  have $t^{-1}\log L_{tT}(\theta)=\log \tilde{L}_T(\theta)$ for all $\theta\in\Theta$ and $t\geq 1$, from which it   trivially follows that  $(tT)^{-1}\log L_{tT}(\theta)$ converges to $T^{-1}\log\tilde{L}_T(\theta)$ for all $\theta\in\Theta$ as $t\rightarrow\infty$.  Under a mild additional assumption on \eqref{eq:SSM_MLE},  we can show the stronger result that for all $\theta\in\Theta$ we have $t^{-1}\log L_{t}(\theta)\rightarrow T^{-1}\log\tilde{L}_T(\theta)$     as $t\rightarrow\infty$,  and therefore, as required,  the  function $t^{-1}\log L_t(\cdot)$   converges to a function $l(\cdot)$  such that $\tilde{\theta}_T=\argmax_{\theta\in\Theta}l(\theta)$.

Since  the so-defined   observation process $(Y_t)_{t\geq 1}$ is $T$-stationary and the particular instance of SSM \eqref{eq:SSM}  defined above is $T$-periodic (see Definition \ref{def:periodic}  with $\tau=T$), we are precisely in the setting of Section \ref{sec:SSM}. The results of that section can therefore be applied to define (IF)  algorithms that can be used to compute an  estimate $\hat{\theta}_t^N$ of the MLE such that, for some (non-random) sequences $(C_t)_{t\geq 1}$ and $(\epsilon_t)_{t\geq 1}$ in $(0,\infty)$ with $\epsilon_t\rightarrow 0$, 
\begin{align}\label{eq:IF_SSM2_1}
&\E\big[  \|\hat{\theta}_t^N-\tilde{\theta}_T \| \big]\leq\frac{C_t}{\sqrt{N}}+\epsilon_t, \quad \forall (N,t)\in\mathbb{N}^2.
\end{align}
Remark that, since the sequence $(Y_t)_{t\geq 1}$ is deterministic in this section, the expectation in \eqref{eq:IF_SSM2_1} is taken with respect to the distribution of the random variables generated by the IF algorithms. We refer the reader to Section \ref{sub:res_sketch} for comments on this type of error bounds.

\subsection{New IF algorithms and practical recommendations \label{sub:IF}}
 
When $(Y_t)_{t\geq 1}$ and SSM \eqref{eq:SSM}  are as defined in Section \ref{sub:mle_ssm}, Algorithm \ref{algo:Generic_SSM} with $(K_{t})_{t\geq 2}$ as defined in either Section  \ref{sub:ssm_cont_fast} or Section \ref{sub:ssm_cont_slow}, as well as   Algorithm \ref{algo:online_2}, are all IF algorithms such that, under suitable conditions on SSM \eqref{eq:SSM_MLE} (see Section \ref{sub:IF_theory}),  the error bound \eqref{eq:IF_SSM2_1} holds  with $\hat{\theta}_t^N=\int_\Theta \theta_t p_{t,\Theta}^N(\dd\theta_t|Y_{1:t})$  (where  $p_{t,\Theta}^N(\dd\theta_t|Y_{1:t})$ is as defined in these algorithms).

A version of the latter IF algorithm is presented in Algorithm \ref{algo:IF_2},
where for all $s\in\{2,\dots,T\}$  a   dynamics is imposed on $\theta$ when processing observation $\tilde{y}_s$ only  if a resampling step is performed. 
Following the discussion in Section \ref{sub:pratical}, our default recommendation for performing maximum likelihood estimation  in SSMs using IF is to use Algorithm \ref{algo:IF_2} with $\pi_1(\dd\theta_1)$ the uniform distribution on $\Theta$, $\Sigma$ the identity matrix, and parameters $\alpha=0.5$, $\Delta=1$ and $\nu=t_1=100$. In the context of this section,  choosing   $\alpha=0.5$ in Algorithm \ref{algo:IF_2} is expected to reduce the risk that the particle system becomes trapped in a local mode of the likelihood function; however, this choice may also lead to slower convergence to the MLE (see Section \ref{sub:pratical}).

\begin{algorithm}[!t]
 
\begin{algorithmic}[1]
\Require constants $(N,t_1,\Delta)\in\mathbb{N}^3$,  $\cess\in(0,1]$, $\alpha\in(0,\infty)$ and $\nu\in (0,\infty)$, and a $d\times d$ symmetric 

\hspace{-1.4cm}  symmetric and  positive definite  matrix $\Sigma$

\vspace{0.1cm}

\State let $t_1\gets 1+T t_1$, $t=0$, $p=1$, $\theta_{0,T}^n\sim \pi_1(\dd\theta_1)$, $w_{0,T}=1$, $W_{0,T}^n=1/N$ and $\mathrm{ESS}^N_{0,T}=N$ 
\vspace{0.1cm}

\For{$k\geq 1$}
\vspace{0.1cm}

\State $t\gets t+1$
\vspace{0.1cm}

\State let  $\theta_{k,1}^n=\theta_{k-1,T}^n$ and $a_n=n$
\vspace{0.1cm}

\If{$\mathrm{ESS}^N_{k-1,T}\leq N c_{\mathrm{ESS}}$}

\State let $(a_1,\dots,a_N)=\text{\texttt{Resampling}}(\{ W_{k-1,T}^{m}\}_{m=1}^N)$ and  $w_{k-1,T}^n=1$

\If{  $\alpha>1$}
\State let $\theta_{k,1}^n\sim \mathrm{TN}_\Theta(\theta^{a_{n}}_{k-1,T}, t^{-2\alpha} \Sigma)$
\EndIf
\EndIf
\If{$\alpha\leq 1$\text{ and either }$t=t_p$ \text{ or }$\mathrm{ESS}^N_{k-1,T}\leq N c_{\mathrm{ESS}}$}
\State let $\theta_{k,1}^n\sim \mathrm{TS}_\Theta(\theta^{a_{n}}_{k-1,T}, t^{-2\alpha} \Sigma,\nu)$, $t_{p+1}= t_p+\Delta T \lceil \log (t_p)^2\rceil$ and $p\gets p+1$

\EndIf

\vspace{0.1cm}

\State  let $X_{k,1}^n\sim\tilde{\chi}_{\theta_{k,1}^n}(\dd x_1)$, $w_{k,1}^n= w_{k-1,T}^n\tilde{f}_{1,\theta_{k,1}^n}(\tilde{y}_1| X_{k,1}^n)$ and $W_{k,1}^n=w_{k,1}^n/\sum_{m=1}^N w_{k,1}^m$
\vspace{0.1cm}

\State\label{Line_MLE1} 
let $\hat{\theta}_t^N=\sum_{n=1}^N W_{k,1}^n \theta_{k,1}^n$

\For{$s\in\{2,\dots,T\}$}
\vspace{0.1cm}

\State $t\gets t+1$
\vspace{0.1cm}

\State let $\mathrm{ESS}^N_{k, s-1}= 1/\sum_{m=1}^N (W_{k,s-1}^m)^2$, $a_{n}=n$ and $\theta_{k,s}^n=\theta_{k,s-1}^n$
\vspace{0.1cm}

\If{$\mathrm{ESS}^N_{k,s-1}\leq  N\,\cess$}

\State  let  $(a_1,\dots,a_N)=\text{\texttt{Resampling}}(\{ W_{k,s-1}^{m}\}_{m=1}^N)$, $\theta_{k,s}^n\sim \mathrm{TN}_\Theta(\theta^{a_{n}}_{k,s-1}, t^{-2\alpha}\Sigma)$ \hspace*{1.75cm} and  $w_{k,s-1}^n=1$  
\EndIf
\vspace{0.1cm}

\State  let $X_{k,s}^n\sim \tilde{M}_{s,\theta_{k,s}^{n}}(X_{k,s-1}^{a_{n}},\dd x_s)$,   $w_{k,s}^n= w_{k,s-1}^n \tilde{f}_{s,\theta_{k,s}^n}(\tilde{y}_s|X_{k,s}^n)$ and 

\hspace{0.3cm} $W_{k,s}^n=w_{k,s}^n/\sum_{m=1}^N w_{k,s}^m$
\vspace{0.1cm}

\State\label{Line_MLE2}   
let $\hat{\theta}_t^N=\sum_{n=1}^N W_{k,s}^n \theta_{k,s}^n$
\EndFor

\EndFor
\end{algorithmic}
\caption{Iterated filtering   with  adaptive artificial dynamics\\(\small{Operations with index $n$ must be performed for all $n\in \{1,\dots,N\}$)}\label{algo:IF_2}}
\end{algorithm}

The extent of this issue depends on the specific problem. Informally, if the observations  $\{\tilde{y}_s\}_{s=1}^{T}$ are an approximate sample from a stationary process then, based on the results of Section \ref{sec:SSM}, we expect some online learning to occur within each pass through the data performed by Algorithm \ref{algo:IF_2}. As a result, the MLE estimate may improve not only from one full pass to the next,  but also progressively within a single pass. In contrast, in situations where  this online learning mechanism does not occur, more passes through the data will be needed for the particle system in Algorithm \ref{algo:IF_2} to concentrate around the MLE (see Section \ref{supp:dis} of the \hyperlink{SM}{Supplementary Material} for an illustration of this phenomenon). In addition, and still speaking informally, the successive passes through the  non-stationarity data $\{\tilde{y}_s\}_{s=1}^{T}$ may cause the MLE estimate $\hat{\theta}_t^N$ to evolve in a sawtooth pattern. To stabilize the estimation of the MLE, and motivated by a well-known averaging technique to accelerate convergence in stochastic gradient methods \citep{polyak1992acceleration},  we suggest  using the time-averaged estimator $\bar{\theta}^N_{b,t}:=\frac{1}{t-bT}\sum_{i=bT+1}^t \hat{\theta}_i^N$ for $t>bT$, where the integer $b\geq 0$ is a user-specified parameter that defines a burn-in period. The effectiveness of this averaging approach for estimating the MLE is illustrated in Section \ref{sub:covid}.

\subsection{Theory\label{sub:IF_theory}}

\subsubsection{Technical assumptions on the state-space model}

As mentioned above, Theorems \ref{thm:SSM}-\ref{thm:SSM2} can be used to derive conditions  on \eqref{eq:SSM_MLE} under which the IF algorithms introduced in Section \ref{sub:IF} are such that \eqref{eq:IF_SSM2_1} holds. However,  in the context of this section,   $(Y_t)_{t\geq 1}$ and   SSM    \eqref{eq:SSM} have  two key properties that
we can exploit when applying the theory developed in Appendix \ref{sec:theory}. As a result, we can provide theoretical justifications for the specific instances of SO-SSM  \eqref{eq:SO-SSM_new2} considered in this section under  weaker assumptions on  model \eqref{eq:SSM_MLE} than those required to apply Theorems \ref{thm:SSM}-\ref{thm:SSM2}.  The first property is that, as mentioned earlier, for the considered observation  process and SSM \eqref{eq:SSM}, establishing the convergence of the  log-likelihood function  $t^{-1}\log L_t(\cdot)$ to some limiting function $l(\cdot)$ requires only a mild assumption  on \eqref{eq:SSM_MLE}. Secondly, the SSM  defined by \eqref{eq:SSM} and \eqref{eq:SSM_t} is such that the   probability distribution $M_{sT+1,\theta}(x',\dd x)$ is independent of $x'$ for all $s\geq 1$ and all $\theta\in\Theta$, meaning that at   time $t=sT+1$, the model  forgets its past. By exploiting this property in the construction of the Markov kernels $(K_{t})_{t\geq 2}$,  we can further relax the assumptions on  model \eqref{eq:SSM_MLE} needed to apply the results of Appendix \ref{sec:theory}.

We assume henceforth that \eqref{eq:SSM_MLE} is such that, for all $s\in\{2,\dots,T\}$ and all $(\theta,\tilde{x})\in\Theta\times\setX$, the probability measure $\tilde{M}_{s,\theta}(\tilde{x},\dd x_s)$ is absolutely continuous w.r.t.~some $\sigma$-finite measure $\lambda(\dd x)$ on $\setX$, and thus  $\tilde{M}_{s,\theta}(\tilde{x},\dd x_s)=\tilde{m}_{s,\theta}(x_s|\tilde{x})\lambda(\dd x_s)$ for some measurable function $\tilde{m}_{s,\theta}(\cdot|\tilde{x})$. Similarly, we  assume  that for all $\theta\in\Theta$ the probability measure $\tilde{\chi}_\theta(\dd x_1)$ is absolutely continuous w.r.t.~$\lambda(\dd x)$, and thus $\tilde{\chi}_\theta(\dd x_1)=\tilde{p}_\theta(x_1)\lambda(\dd x_1)$ for some measurable function $\tilde{p}_\theta(\cdot)$. Then, to control the impact that the  dynamics on $\theta$ has on the estimation process, we consider  Assumption \ref{assumeSSMB:smooth_MLE} (stated in Appendix \ref{app:assume_mle})  which notably imposes that the mappings $\theta\mapsto \log \tilde{f}_{s,\theta}(y|x)$ and $\theta\mapsto \log \tilde{m}_{s,\theta}(x|\tilde{x})$   are smooth for all $(x,\tilde{x},y)\in\setX^2\times\setY$ and all $s\in\{1,\dots,T\}$, with the convention that  $\tilde{m}_{1,\theta}(x|\tilde{x})=\tilde{p}_\theta(x)$ for all $(x,x')\in\setX^2$ when $s=1$.  In Section \ref{supp_MLE} of the \hyperlink{SM}{Supplementary Material} we prove that Assumption \ref{assumeSSMB:smooth_MLE}  is for instance  satisfied for  a large class of multivariate linear Gaussian SSMs and, with a simple stochastic volatility model, we show that the validity of this assumption is no limited to  linear  Gaussian SSMs.

\subsubsection{Main results}

The following two theorems are the main results of this section:

\begin{theorem}\label{thm:MLE}
Let $(Y_t)_{t\geq 1}$ and SSM \eqref{eq:SSM} be  as defined in Section \ref{sub:mle_ssm},  and assume that $\tilde{\theta}_T\in\mathring{\Theta}$. Let  $(K_{t})_{t\geq 2} $ be as defined in either Section  \ref{sub:ssm_cont_fast} or Section \ref{sub:ssm_cont_slow},  with  $(t_p)_{p\geq 1}$  a sequence in $\{mT+1,\,m\in\mathbb{N}\}$. Then,  $\lim_{t\rightarrow\infty}\int_\Theta \theta_t\, p_{t,\Theta}(\dd \theta_t |Y_{1:t})= \tilde{\theta}_T$.
\end{theorem}

\begin{theorem}\label{thm:MLE2}
Let $(Y_t)_{t\geq 1}$ and SSM \eqref{eq:SSM} be  as defined in Section \ref{sub:mle_ssm}. In addition,  assume that $\tilde{\theta}_T\in\mathring{\Theta}$ and that there exists  a compact set  $E\subset\setX$ such that $\inf_{(\theta,x)\in\Theta\times\setX}\tilde{M}_{s,\theta}(x,E)>0$ for all $s\in\{2,\dots,T\}$, and for  all $(t,N)\in\mathbb{N}^2$  let  $\hat{\theta}_t^N$   be as defined by Algorithm \ref{algo:IF_2}. Then,   there exist two  (non-random) sequences $(C_t)_{t\geq 1}$ and $(\epsilon_t)_{t\geq 1}$ in $(0,\infty)$, such that $\lim_{t\rightarrow\infty}\epsilon_t=0$, for which we have $\E[\|\hat{\theta}_t^N- \tilde{\theta}_T\|]\leq C_t/\sqrt{N}+\epsilon_t$ for all $(N,t)\in\mathbb{N}^2$.
\end{theorem}

Remark that, unlike Theorem \ref{thm:SSM}  obtained for the online setting, Theorem \ref{thm:MLE} does not require a stability condition similar to \eqref{eq:stability_SSM}. We refer the reader to   Section \ref{sub:assume_informal} for comments on the  assumption, required by  Theorem \ref{thm:MLE2}, that  $\inf_{(\theta,x)\in\Theta\times\setX}\tilde{M}_{s,\theta}(x,E)>0$ for all $s\in\{2,\dots,T\}$ and some compact set $E\subset\setX$.

\section{Numerical experiments\label{sec:num}}

In the experiments presented below, Algorithms \ref{algo:Generic_SSM}-\ref{algo:IF_2}  are implemented with $\cess=0.7$ and use SSP resampling  \citep{gerber2019negative}. In addition, adhering to the recommendations outlined in Section \ref{sub:pratical} and Section \ref{sub:IF}, throughout this section $\pi_1(\dd\theta_1)$ is the uniform distribution on $\Theta$  and,   in Algorithms \ref{algo:online_2}-\ref{algo:IF_2}, we let  $(\Delta, t_1,\nu)=(1,100,100)$ and $\Sigma$ be the identity matrix.

\subsection{Online  inference in a periodic linear Gaussian SSM\label{sub:Gaussian}}

To introduce the  model discussed in this subsection, which is motivated  by the real data example of Section \ref{sub:LG_real}, we first need to define additional notation. For $q\in\{2,3,4\}$, we define $\{b_{j}\}_{j=1}^q$ as   basis functions for $\mathcal{S}_q(\{\xi_j\}_{j=1}^q)$, the space of natural cubic splines with knots $0= \xi_1<\dots,\xi_{q+1}=24$, excluding non-zero constant functions (and thus  $b_j(0)=0$ for all $j\in\{1,\dots,q\}$). Informally, the ability to approximate a twice continuously differentiable function by a function in $\mathcal{S}_q(\{\xi_j\}_{j=1}^q)$ improves as $q$ increases. We let $\{\xi_j\}_{j=1}^{q+1}=\{0,12,24\}$ when $q=2$,  $\{\xi_j\}_{j=1}^{q+1}=\{0,8,16,24\}$ when  $q=3$ and $\{\xi_j\}_{j=1}^{q+1}=\{0,6,12,18,24\}$ when $q=4$. In addition, for all $q\in\{2,3,4\}$ and $j\in\{1,\dots,q\}$, we let  $\tilde{b}_j(\cdot)$ be such that $\tilde{b}_j(u)=b_j\big(u-24 \lfloor(u-1)/24\rfloor\big)$ for all $u\geq 0$, so that $\tilde{b}_j(u)=\tilde{b}_j(u+24m)$ for all $u\geq 0$ and all integer $m\geq 0$. 

Let $q\in\{2,3,4\}$,  $\Theta\subset\R^{3q+1}$ and  write $\theta\in\Theta$ as $\theta=(\beta,\rho,\sigma)$ with $\beta\in\R^q$, $\rho\in\R^q$ and $\sigma\in\R^{q+1}$. Then, in this section we assume  that the observation process  $(Y_t)_{t\geq 1}$ is a 24-stationary process    and we consider the following 24-periodic SSM as generative model:
\begin{equation}\label{eq:SSM_LG}
X_{1}\sim\mathcal{N}_q(0,4 I_q),\,\,
\begin{cases}
Y_t|X_t\sim\mathcal{N}_1\Big(\sum_{j=1}^q (\beta_j+X_{t,j})\tilde{b}_j(t),\sigma^2_1\Big)\\
X_{t+1}|X_t\sim\mathcal{N}_q\Big(\mathrm{diag}(\rho_1,\dots,\rho_q) X_{t},\mathrm{diag}(\sigma^2_{2},\dots,\sigma_{q+1}^2)\Big)
\end{cases}  t\geq 1.
\end{equation}

The objective in what follows is to use the SO-SSMs introduced in Section \ref{sec:SSM} to perform  online parameter and state inference in  model \eqref{eq:SSM_LG}. Noting that \eqref{eq:SSM_LG} is a linear Gaussian SSM, and following the discussion  in Section  \ref{sub:pratical}, for this experiment we use    the Rao-Blackwellised version of Algorithms \ref{algo:Generic_SSM}-\ref{algo:online_2}, with $N=10^4$ particles. As above, for all $t\geq 1$ we   denote by $p_{t,\Theta}^N(\dd \theta_t|Y_{1:t})$ and $p_{t,\setX}^N(\dd x_t|Y_{1:t})$  the resulting PF estimates of $p_{t,\Theta}(\dd \theta_t|Y_{1:t})$ and $p_{t,\setX}(\dd x_t|Y_{1:t})$, respectively, and we let $\hat{\theta}_t^N=\int_\Theta\theta_t  p_{t,\Theta}^N(\dd \theta_t|Y_{1:t})$.

\subsubsection{Experiments with synthetic data\label{sub:Num_LG_sim}}

In the experiments presented below  $(Y_t)_{t\geq 1}$ follows the distribution specified by the generative model \eqref{eq:SSM_LG} when $\theta=\theta_{\star}:=(\beta_{\star},\rho_{\star},\sigma_{\star})$, with $\beta_{\star}$ a random draw from the $\mathcal{U}([-2,2]^q)$ distribution, $\rho_{\star}$    a random draw from the $\mathcal{U}([-1,1]^q)$ distribution and    $\sigma_{\star}=(0.5,1,\dots, 1)$. The parameter space is   $\Theta=[-10,10]^q\times[-1,1]^q\times[\epsilon,4]^{q+1}$ for some $\epsilon\in [0,4)$, and   Assumptions \ref{assumeSSM:K_set}-\ref{assumeSSM:smooth}  used in Section \ref{sub:theorySSM}  to prove the convergence results \eqref{eq:goal_online}-\eqref{eq:cor_SSM2_2}  are satisfied if $\epsilon>0$ (by Proposition \ref{prop:Assume_LG} in the \hyperlink{SM}{Supplementary Material}). However,  to avoid imposing a lower bound on the variance parameters of the model, we let $\epsilon=0$. In the following we refer to  Algorithm \ref{algo:Generic_SSM} with $(K_t)_{t\geq 2}$ as defined in Section \ref{sub:ssm_cont_fast} with $\Sigma$ the identity matrix and with $(h_t)_{t\geq 2}=(t^{-\alpha})_{t\geq 2}$ as Algorithm A1-$\alpha$, and  to   Algorithm \ref{algo:online_2}  as Algorithm A2-$\alpha$.  In Figure   \ref{fig:LG_sim} we compare, for $q=2$ and   $q=4$, the performance of Algorithms A1-0.5, A1-1.1, A2-0.4, A2-0.5, A2-0.6 and A2-1.1  for online joint state and parameter inference in model \eqref{eq:SSM_LG}. The results presented in this figure are obtained from  50 independent runs of these algorithms, where  in each   run an independent realization of $(Y_t)_{t\geq 1}$ is used.

The boxplots presented in Figure   \ref{fig:LG_sim} show, for $T=10^4$,  the estimation error $d^{-1}\|\hat{\theta}_T^N-\theta_\star\|$ obtained in each run of each algorithm    (with $\|\cdot\|$ denoting the Euclidean norm on $\R^d$). From these boxplots three key observations emerge. Firstly,  the estimation error  tends to be the largest for Algorithms A1-1.1 and A2-1.1, and the performance of the three algorithms with a fast vanishing dynamics on $\theta$, namely of  Algorithms A1-1.1, A2-1.1 and A2-0.6, deteriorates as the number of  parameters to estimate increases from $d=7$ (when $q=2$) to $d=13$ (when $q=4$),  while that of the other algorithms remains essentially identical. This  observation  is not surprising  in the light of the discussion in Section \ref{sub:pratical}, where we argued that  the ability of Algorithms \ref{algo:Generic_SSM}-\ref{algo:online_2} to explore $\Theta$ deteriorates as the rate of decay of the dynamics on $\theta$ increases. Secondly, and as expected from the discussion of Section \ref{sub:ssm_adapt},   Algorithm A2-$0.5$  tends  to outperform Algorithm A1-$0.5$, and has in addition  the advantage of being computationally cheaper (since it requires sampling new $\theta$ particles less often). For instance, for $q=4$ the running time of Algorithm A1-$0.5$ is about three times longer than that of Algorithm A2-$0.5$. Thirdly,  Algorithm A2-0.5, the default algorithm  recommended in  Section \ref{sub:pratical}, performs best for $q=4$ while for $q=2$ it is almost as good as the best performing algorithm, Algorithm A2-0.6.

In the remaining plots in Figure \ref{fig:LG_sim} we study, for each algorithm and as $t$ increases, the evolution of the average estimation error obtained when $\int_{\R^q} x_t \bar{p}_{t,\theta\star}(\dd x_t| Y_{1:t})$ is estimated by $\int_{\R^q} x_t p^N_{t,\setX}(\dd x_t| Y_{1:t})$. From these plots we observe that the ability of all six algorithms to estimate the filtering mean $\int_{\R^q} x_t \bar{p}_{t,\theta\star}(\dd x_t| Y_{1:t})$ improves over time,   and that all the conclusions about the ranking of the algorithms drawn above for parameter inference remain valid for state inference.    
 

\begin{figure}[!t]
\centering
\hspace*{0cm}\includegraphics[scale=0.44]{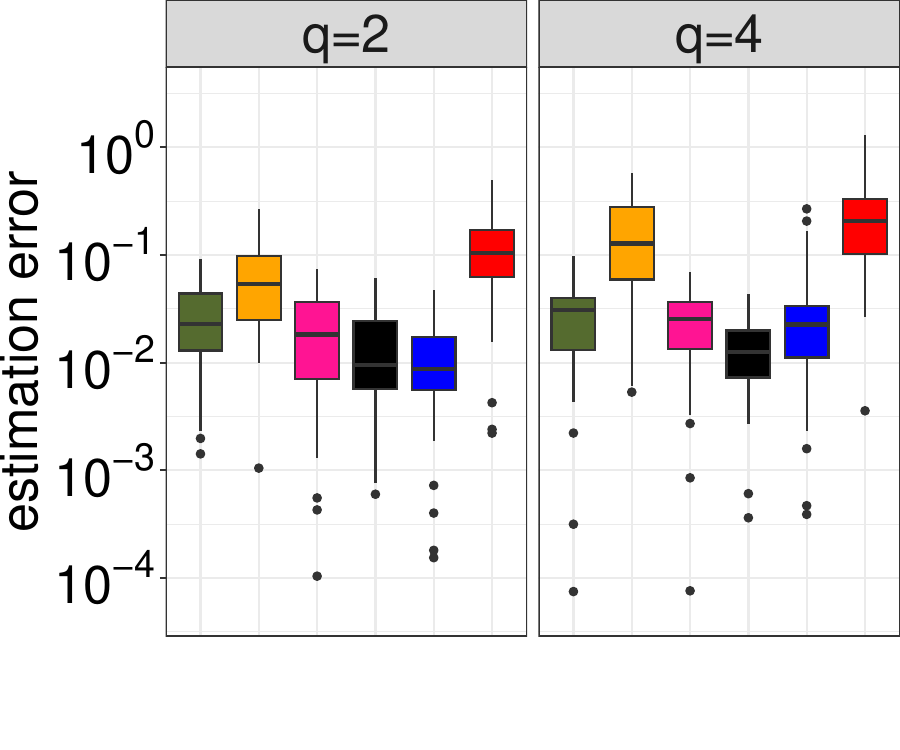}\hspace*{0.35cm}\includegraphics[scale=0.44]{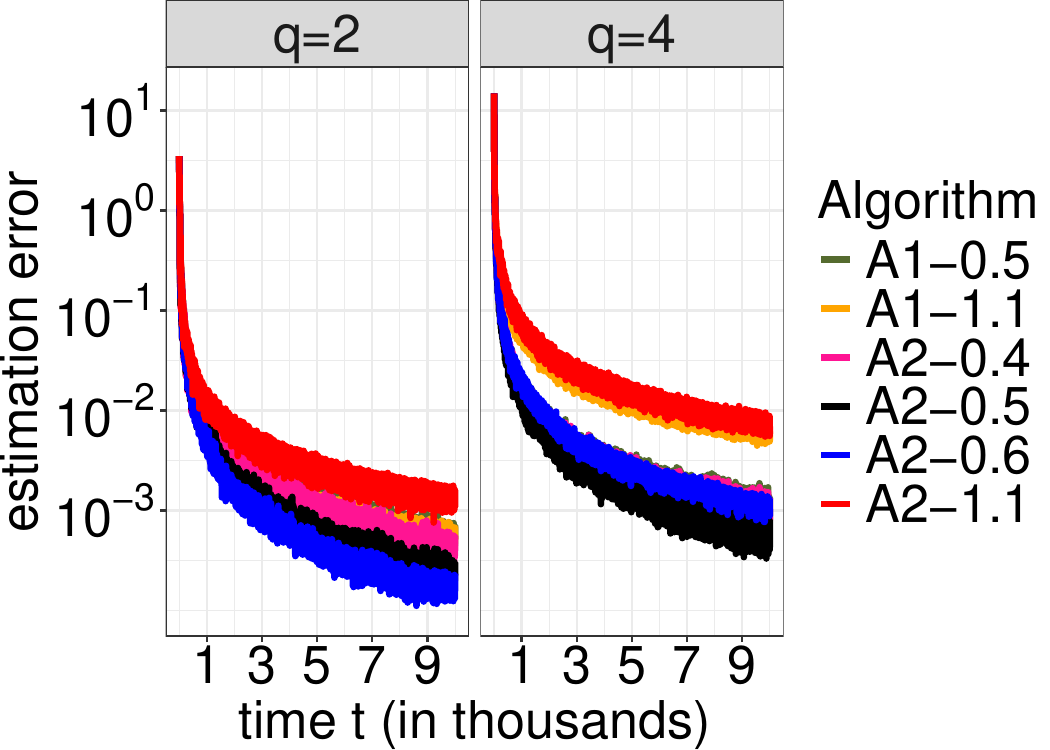}

\caption{Results for the experiments in Section \ref{sub:Num_LG_sim}. The boxplots show, for $T=10^4$, the value of $d^{-1}\|\hat{\theta}_T^N-\theta_\star\|$ obtained in 50 runs of the algorithms, and the other  plots show, as a function of $t$, the value of  $(qt)^{-1}\sum_{s=1}^t\|\int_{\R^q}x_s p^N_{s,\setX}(\dd x_s|y_{1:s})-\int_{\R^q}x_s \bar{p}_{s,\theta_\star}(\dd x_s|y_{1:s})\|$, averaged over the 50 runs of the algorithms.\label{fig:LG_sim}}
\end{figure}

\subsubsection{Application: online temperature forecasting\label{sub:LG_real}}

In this subsection, we consider the hourly temperature in the city of Bristol (UK),  $(W_t)_{t\geq 0}$ where $W_0$ is the temperature on the 1st of January  2020 at 00.00am. Our objective is to estimate future values of this process in an online fashion using SSM \eqref{eq:SSM_LG}, where online   state and parameter inference in this model is performed using the default approach proposed in Section \ref{sub:pratical}, that is using  Algorithm \ref{algo:online_2} with $\alpha=0.5$.

To account for the seasonality patterns inherent to this type of data (e.g.~the temperatures in Summers tend to be higher than in Winters) we define the observation process $(Y_t)_{t\geq 1}$
 by letting $Y_t=W_t-W_{24\lfloor(t-1)/24\rfloor}$ for all  $t\geq 1$. In words, $Y_t$ is the difference between the $t$-th hourly temperature and the temperature recorded on the same day at 00.00am, and we use SSM \eqref{eq:SSM_LG}  to model these intra-day temperature variations.  For all $t\geq 1$  we let $Z_t=(Y_{24(t-1)+1},\dots,Y_{24t})$ and note that model \eqref{eq:SSM_LG}  assumes that $(Z_t)_{t\geq 1}$ is a stationary process. 
 
This latter assumption, however, is not reasonable as intra-day temperature variations tend to be larger in Spring-Summer than in Autumn-Winter. To account for this non-stationarity,  we divide each year into two approximately six-month seasonal periods (namely, March-August and September-February) and assume $(Z_t)_{t\geq 1}$ is stationary within each period. Let $T_i$ denotes  the end of the $i$-th so-defined interval, with the convention that $T_0=0$. While the underlying model structure \eqref{eq:SSM_LG} remains consistent for these alternating seasonal intervals, its parameters are assumed different for every interval, allowing the model to capture seasonal and yearly variations. Our aim is to estimate these parameters for each interval, in an online fashion. 
It is reasonable to expect the distributions of adjacent intervals to share similarities, suggesting that rather than fully restarting   Algorithm \ref{algo:online_2} at the end of each six-month period, we can instead use the last estimate from the previous period as initialisation in the current period. Further, period-wise estimation also requires reinitialising  the sequence $(t^{-1}\Sigma)_{t\geq 2}$ of scale matrices to ensure continued adaptation. For the present  application  we implement Algorithm \ref{algo:online_2} with, for $i\geq 0$ and $t\in\{T_i+1,\dots,T_{i+1}\}$, the scale matrix $t^{-1}\Sigma$ replaced with $(t-T_i)^{-1}\Sigma$.

Let  $\bar{\E}_{t,\theta}[\cdot]$ denotes    expectations  under model \eqref{eq:SSM_LG} conditional upon  $Y_{1:t}$  and    let $\theta_\star=(\beta_\star,\rho_\star,\sigma_\star)$ be the target parameter value (assuming such a target parameter value exists). Then, for $k\in\mathbb{N}$,  we let $\bar{\E}_{t,\theta_\star}[W_{t+k}]$ be the    ideal model-based $k$ step-ahead prediction for $W_{t+k}$. To explain how Algorithm \ref{algo:online_2} is used to estimate this quantity, assume for simplicity that $W_{t+k}$ and $W_t$ are temperatures at two different hours on the same day.
Then, from the model,
\begin{align}\label{eq:model_pred}
\bar{\E}_{t,\theta_\star}\big[W_{t+k}\big]&=W_{24\lfloor\frac{t-1}{24}\rfloor }+\sum_{j=1}^q\big(\beta_{\star,j}+\rho_{\star,j}^k \bar\E_{t,\theta_\star}[X_{t,j}] \big)\tilde{b}_j(t+k).
\end{align}
We estimate $\bar{\E}_{t,\theta_\star}[W_{t+k}]$ using this expression  with the filtering expectation    $\bar{\E}_{t,\theta_\star}[X_{t}]$  estimated by $\int_{\R^q} x_{t} p^N_{t,\setX}(\dd x_t|Y_{1:t})$ and, writing $\hat{\theta}_t^N=(\hat{\beta}^N_t,\hat{\rho}_t^N,\hat{\sigma}_t^N)$, with the parameter $(\beta_\star,\rho_\star)$  estimated by $(\hat{\beta}_t^N,\hat{\rho}^N_t)$.

Starting from the 1st of January   2020 at 1.00am,  we use the above procedure to compute hourly predictions for the temperature in Bristol for $k\in\{1,2,\dots,8\}$ hours ahead. The final temperature prediction is for the 31st of December 2022 at 11pm, resulting  in a total of 26\,296 observations processed \footnote{Using observations bought from the \url{www.visualcrossing.com} website.}. For this experiment, model \eqref{eq:SSM_LG} is used with   $q=3$ and with $q=4$, and  we let $\Theta=[-10,10]^q\times[0,1]^q\times[0,4]^{q+1}$.  As a benchmark  we also consider two simple forecasting methods for predicting $W_{t+k}$, namely using the current temperature  $W_t$ as the prediction, and using the temperature from the same time the day before as the prediction, i.e.~letting $W_{t+k-24}$ be the predicted value of $W_{t+k}$.

 \begin{figure}[!t]
\centering
\hspace*{0.1cm}\includegraphics[scale=0.39]{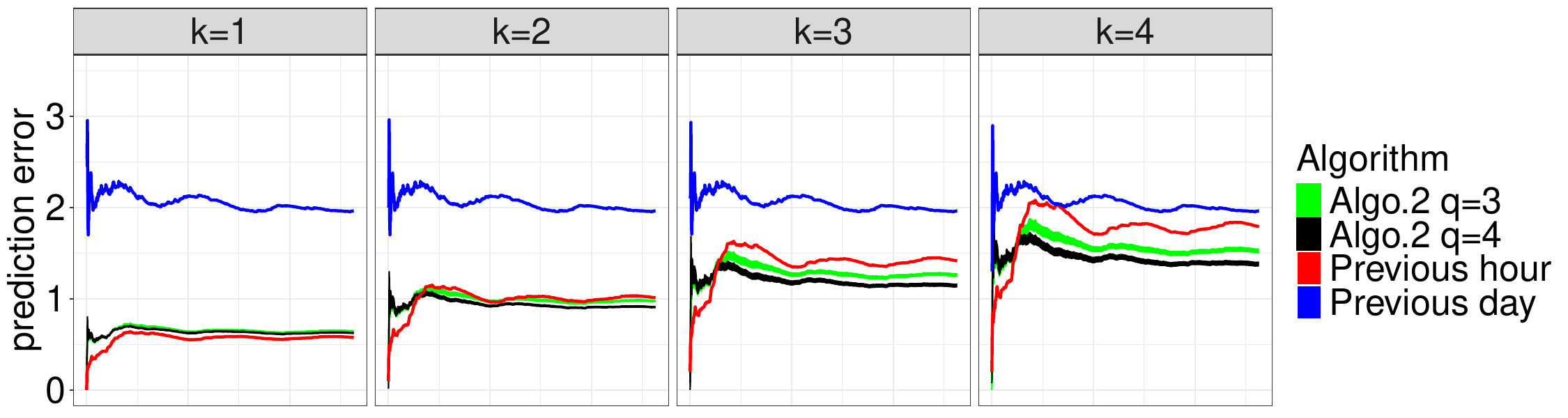}
\vspace{-0.1cm}

\hspace*{0cm}\includegraphics[scale=0.39]{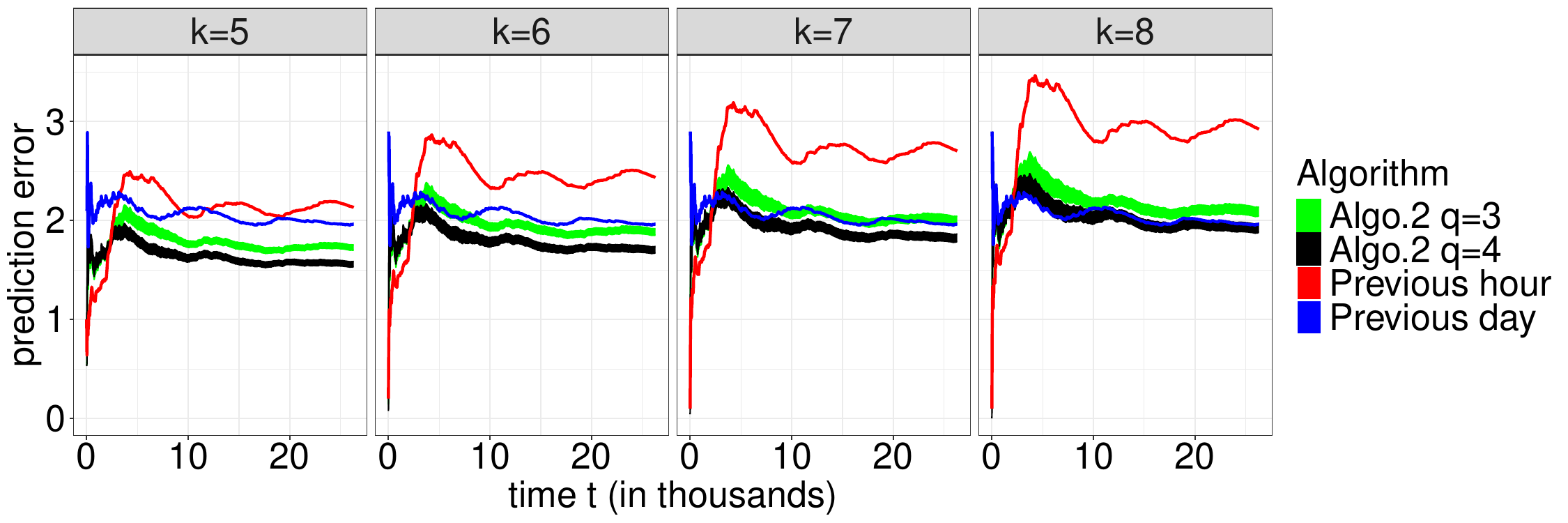}
\caption{Results for the experiments in Section \ref{sub:LG_real}. The plots show the evolution of the average prediction error $t^{-1}\sum_{s=1}^t|\widehat{W}_{s+k}-W_{s+k}|$ as a function of $t$, for all $k\in\{1,\dots,8\}$, and for each of the four forecasting algorithms considered. For Algorithm \ref{algo:online_2} (with $q=3$ and $q=4$),    the shaded areas represent the range of average prediction errors obtained from  50 runs of the algorithm.
\label{fig:LG_real}}

\end{figure}

Figure \ref{fig:LG_real}  shows, for all $k\in\{1,\dots,8\}$, the evolution  of the  average prediction error as $t$ increases, obtained with the four forecasting methods considered. Four comments are in order.
 Firstly,  we unsurprisingly observe  that for $k=1$ the current temperature $W_t$ is the best predictor of $W_{t+k}$, although  the predictive performance of the  model (for both $q=3$ and $q=4$) is  almost identical to that of this simple forecasting strategy.  Secondly, for both $q=3$ and $q=4$, the performance of the model-based predictions remains stable across the 50 runs of Algorithm \ref{algo:online_2}. Thirdly,   for any $k\in\{1,\dots,8\}$  the model-based predictions obtained with  $q=4$ outperform those computed with $q=3$, indicating that  Algorithm \ref{algo:online_2} successfully  exploits the greater flexibility of the model as  $q$ increases to improve predictions. In particular,  for all $k\in\{2,\dots,8\}$ the model-based predictions obtained with  $q=4$ tend to be the most accurate. Fourthly, as $k$ increases from two to eight and for $q=4$,  the model-based average prediction errors become more and more similar to those obtained when $W_{t+k}$ is predicted using $W_{t+k-24}$. This  suggests that model \eqref{eq:SSM_LG} with $q=4$ captures the data well.

To understand this latter assertion, let $\zeta_s=W_s-\E[W_s]$ for all $s\geq 1$ and assume that   model   \eqref{eq:SSM_LG} is well-specified, so that $\theta_\star$ in \eqref{eq:model_pred} is the true parameter value. Let $t\geq 1$ and $k\geq 1$ be such that $W_{t}$ and $W_{t+k}$ are temperatures within the same day. In this scenario, it can be shown from \eqref{eq:model_pred} that when $k$ is sufficiently large, so that $\rho_{\star,j}^k\approx 0$ for all $j\in\{1,\dots,q\}$, 
the model's prediction error is approximately $|\bar{\E}_{t,\theta_\star}[W_{t+k}]-W_{t+k}|\approx |\zeta_{t+k}-\zeta_{24\lfloor\frac{t-1}{24}\rfloor}|$. On the other hand, it is reasonable to expect that $\E[W_{t+k-24}]\approx \E[W_{t+k}]$, so the Previous-day benchmark's prediction error is $|W_{t+k}-W_{t+k-24}|\approx|\zeta_{t+k}-\zeta_{t+k-24}|$. Using these   approximations, and making the sensible assumption that the joint distribution of $(\zeta_{t+k},\zeta_{t+k-24})$ is similar to  that of $(\zeta_{t+k},\zeta_{24\lfloor\frac{t-1}{24}\rfloor})$,  for large values of $T$ we have
\begin{align*}
\frac{1}{T}\sum_{t=1}^T\big|W_{t+k}-\bar{\E}_{t,\theta_\star}[W_{t+k}]\big|&\approx \frac{1}{T}\sum_{t=1}^T \big|\zeta_{t+k}-\zeta_{24\lfloor\frac{t-1}{24}\rfloor}\big|
\approx \frac{1}{T}\sum_{t=1}^T  \big|W_{t+k}-W_{t+k-24}\big|. 
\end{align*}
To sum up, the observation that for $q=4$ and as $k$ increases, the model-based prediction error becomes similar to that of the Previous-day benchmark is consistent with model \eqref{eq:SSM_LG} being well-specified and Algorithm \ref{algo:online_2} learning the true parameter value. By contrast, when $k=8$  the model-based predictions with $q=3$ yield a larger average error than this Previous-day benchmark, indicating that this simpler $q=3$ model is insufficient to capture the underlying process well.

 \begin{figure}[!t]
\centering
\hspace*{-0.8cm}\includegraphics[scale=0.4]{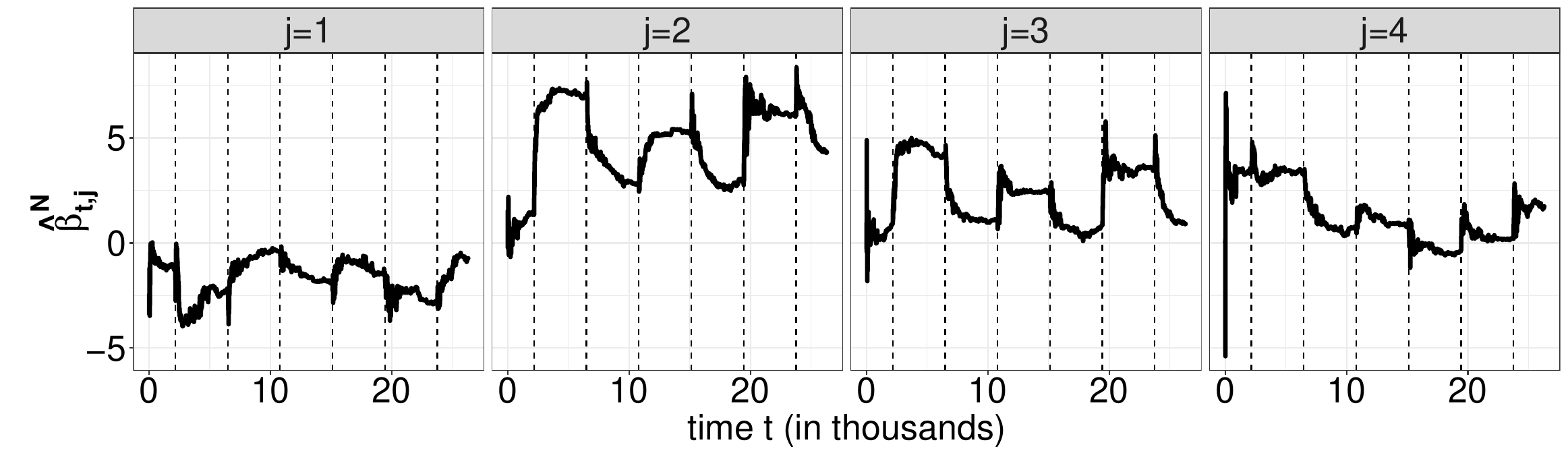}

\caption{Results for experiments in Section \ref{sub:LG_real}. The plots show, for $q=4$ and for the first run of Algorithm \ref{algo:online_2},  the evolution of the four components of $\hat{\beta}_t^N$ as a function of $t$, with the vertical lines representing the time instants $T_i$  after which the sequence   $(t^{-1}\Sigma)_{t\geq 2}$ is refreshed as explained in the text.
\label{fig:LG_real2}}
\end{figure}

Finally, Figure \ref{fig:LG_real2}   shows, for $q=4$ and for the first run of Algorithm \ref{algo:online_2},  the evolution of $\hat{\beta}_t^N$ as a function of $t$. The vertical lines in the plots represent the time instants $T_i$  after which the sequence of scale matrices $(t^{-1}\Sigma)_{t\geq 2}$  is refreshed as explained above. In this figure, we observe that once this sequence is refreshed, the estimate  $\hat{\beta}_t^N$  jumps to a new region of the parameter space and, in most cases, begins  converging to some specific value. This behaviour of $\hat{\beta}_t^N$ is consistent with the fact that the distribution of the data is non-stationary, but approximately stationary within each interval of the form $(T_i,T_{i+1}]$.

\subsection{COVID-19 pandemic analysis with a Beta-Dirichlet SSM\label{sub:covid}}

\subsubsection{Data and model}

We now aim to model the COVID-19 pandemic in the UK in order to estimate the effective reproduction (R)  number  in an offline fashion. For this experiment, we use the data from the \textit{Our World in Data} website\footnote{Available at \url{docs.owid.io/projects/etl/api/covid/\#publications}} consisting of daily new registered COVID-19 cases and daily new deaths for the period  4th March 2020--1st July 2020. We recall that the first UK lockdown was announced on the 23rd March 2020 and was slowly relaxed starting in May 2020.

To address weekly seasonality, where  fewer cases are registered on weekends and holidays, we use as observations the  7-day moving averages of daily new registered cases, denoted by $\{n_s\}_{s=1}^T$,  and  7-day moving averages of daily new deaths, denoted by $\{d_s\}_{s=1}^T$, for the $T=120$ consecutive days considered. The estimated population size in the UK in 2020 was $N_P:=67\,886\,004$ and,  assuming this number remained constant throughout the observation period, we let   $\tilde{y}_s=(n_s/N_P, d_s/N_P)$ for all $s\in\{1,\dots,T\}$.

To model the data $\{\tilde{y}_s\}_{s=1}^T$  we use an SSM that builds upon the  SEIRD model proposed by \citet{Calvetti_SEIRD}  to analyse the COVID-19 pandemic. While the traditional SEIRD model  segments the population  into five compartments, namely the proportion of persons who are Susceptible (S),   Exposed (E),   Infected (I),   Recovered (R) and Dead (D), to account for the specific characteristics  of the COVID-19 pandemic   this latter reference suggests to include infected but asymptomatic individuals in the Exposed group. Specifically, \citet{Calvetti_SEIRD} propose the following equations  to model the dynamics of the  five compartments:
\begin{equation}\label{eq:SEAIRD model}
\begin{split}
&\frac{\dd S_s}{\dd s}  = - \beta_s S_s \big(E_s + q I_s\big),\quad \frac{\dd E_s}{\dd s}  = \beta_s S_s \big(E_s + q I_s\big) - \eta E_s -\gamma E_s  \\
&\frac{\dd I_s}{\dd s}  = \eta E_s - \gamma I_s - \mu I_s,\quad \frac{\dd R_s}{\dd s}  = \gamma E_s + \gamma I_s,\quad \frac{\dd D_s}{\dd s}  =  \mu I_s \end{split}
\end{equation}
where   $\beta_s$ is the  transmission rate  (assumed time-varying due to public health policies implemented and subsequently relaxed during the observation period), $\gamma$ is the recovery rate, $\eta$ is the symptom rate, $\mu$ is the mortality rate and $q$ is the fraction of symptomatic individuals who are infectious. In the context of COVID-19  we expect $q \ll 1$ as individuals with symptoms typically reduce their contact (e.g.~via isolation or hospitalization) more significantly than asymptomatic infectious ones. Consequently, in \eqref{eq:SEAIRD model} the transmission contribution from symptomatic individuals $q I_t$ is expected to be small compared to that from $E_t$, and thus, for simplification and following \citet{menda2021explaining}, we let  $q=0$ in what follows. Under  model \eqref{eq:SEAIRD model}  the R number at time $s$ can be defined as $\mathrm{R}_{s}:=\beta_s S_s /(\gamma+\eta)$ \citep[see][]{Calvetti_SEIRD}.

We assume that $\{\tilde{y}_s\}_{s=1}^T$ is a realization of some random variables $\{\tilde{Y}_s\}_{s=1}^T$ whose joint distribution is, following  an approach proposed in \citet{Osthus_SIR}, modelled using a Dirichlet-Beta SSM derived from \eqref{eq:SEAIRD model}. Specifically,  the continuous time model \eqref{eq:SEAIRD model} is discretized using
Euler's method,  and we let  $F:\R^{9}\rightarrow\R^5$ be such that, for all $( s,e,i,r,d,\eta,\gamma,\mu,\beta)\in\R^{9}$ and with $w=(s,e,i,r,d)$,
\begin{align*}
F(w,\beta,\eta,\gamma,\mu)= \Big(
s(1-\beta  e ),
e(1-\eta-\gamma+ \beta  s),
i(1-\gamma-\mu) + \eta e ,
r + \gamma e  + \gamma i,
d+\mu i
\Big).
\end{align*}
Then, letting $\theta=(\eta,\gamma,\mu, \sigma_\beta, \kappa,\lambda_1,\lambda_2)\in(0,\infty)^7$ and, for all $s\geq 1$, letting  $\tilde{X}_s=(W_s,\beta_s)$ with $W_s=(S_s, E_s, I_s, R_s,D_s )$, we consider the following     generative model for $\{\tilde{Y}_s\}_{s=1}^T$:
\begin{equation}\label{eq:SEIRD_SSM}
\begin{split}
&\tilde{Y}_{s,1}\big |  \tilde{X}_{s} \sim \mathrm{Beta}\bigg(\frac{ \eta E_{s}}{\lambda_1/100},\frac{1-\eta E_{s}}{\lambda_1/100} \bigg),\quad \tilde{Y}_{s,2} \big | \tilde{X}_s   \sim \mathrm{Beta}\bigg(\frac{ \mu I_{s}}{\lambda_2/100},\frac{1- \mu I_{s}}{\lambda_2/100} \bigg)\\
&W_{s+1} \big | \tilde{X}_s   \sim \mathrm{Dirichlet}\bigg(\frac{f(W_s,\beta_s,\eta,\gamma,\mu)}{\kappa/100}\bigg),\quad  \log(\beta_{s+1})|\beta_s \sim\mathrm{TN}_{(-\infty,0]}\big(\log(\beta_s), \sigma_{\beta}^2\big)
\end{split}
\end{equation}
where $D_1=R_1=0$,  $S_1=1-I_1-E_1$ and where
\begin{align}\label{eq:SEIRD_SSM0}
I_1,E_1\iid\mathcal{U}([0,10^{-4}]),\quad\beta_1|S_1\sim\mathcal{U}\Big(0,\min\big\{1,4 (\eta+\gamma)/S_1\big\} \Big).
\end{align}

In \eqref{eq:SEIRD_SSM}, the parameters $\lambda_1$,  $\lambda_2$   and $\kappa$ control  the   variance of the Beta and Dirichlet  distributions  but not their expectations. Given that the observations we consider are exceedingly small numbers, these variances  are expected to be tiny, as discussed by \cite{Osthus_SIR}. The scaling factor $1/100$ applied to these three parameters is therefore used to prevent estimating values excessively close to zero. Recalling that under the model considered the R number is $\mathrm{R}_{s}=\beta_s S_s /(\gamma+\eta)$, the  distribution of $\beta_1$ ensures that  $\mathrm{R}_{1}\leq 4$, where 4 is a realistic upper bound for the R number in the UK on the 4th of March  2020. The conditional  distribution of  $\beta_{t+1}$ given $\beta_t$ is then chosen as in \citet{Calvetti_SEIRD}, with the difference that we constrain its support to be $[0,1]$ (recall that $\beta_t$ is the transmission rate).
To complete the specification of the model, we let $\Theta$ be the set of all $\theta\in [0,1]^7$ such that $\eta\in[0,1-\gamma]$, $\mu\in[0,1-\gamma]$ and $(\kappa,\lambda_1,\lambda_2)\in(0.01,1)^3$. The first two conditions on $\theta$ ensure  that   model \eqref{eq:SEIRD_SSM}-\eqref{eq:SEIRD_SSM0} is well-defined while the lower bound imposed on $\kappa$, $\lambda_1$ and $\lambda_2$  facilitates the deployment of a PF on  SSM \eqref{eq:SEIRD_SSM}-\eqref{eq:SEIRD_SSM0}, that is, it allows effective estimation of the filtering distributions of the SSM using a simple bootstrap PF with a reasonable number of particles.

\subsubsection{Maximum likelihood estimation}

In a first step, we set $N=5\,000$ and estimate the MLE of the model parameter using   Algorithm  \ref{algo:IF_2} (that is, using the estimate $\hat{\theta}_t^N$ computed by this algorithm). Results are presented below for  Algorithm  \ref{algo:IF_2} with $\alpha=0.5$ and with $\alpha=1.1$, and we will refer to former as Algorithm A3-0.5 and to the latter as  Algorithm A3-1.1. Remark that Algorithm A3-0.5 is the default   IF algorithm proposed in Section \ref{sub:IF} while Algorithm A3-1.1  imposes a dynamics on $\theta$ that vanishes quickly over time. For  Algorithm  A3-0.5 we also consider the averaging estimator $\hat{\theta}_{b,t}^N$ defined in  Section \ref{sub:IF} with $b=1\,500$ (i.e.~with a burn-in period of 1\,500 passes through the data), and to simplify the presentation, we will hereafter refer to Algorithm A3-0.5 outputting  this specific estimator as Algorithm AA3-0.5.

Figure \ref{fig:covid1} shows the evolution of the estimated MLE value during the first 2\,000 iterations (i.e.~passes through the data) of a single run for each of the three algorithms. As expected, we observe that   Algorithm A3-1.1 converges   more quickly than  Algorithm A3-0.5. In particular, for the $\sigma_\beta$  component  of $\theta$ we can see that 2\,000 iterations are insufficient for Algorithm A3-0.5 to  converge to a specific value. The difficulty in estimating  $\sigma_\beta$ is not surprising since this parameter governs the distribution of $\beta_{s+1}|\beta_{s}$  with $\beta_s$ being observed only indirectly through $\tilde{Y}_{s+1}$ via $E_{s+1}$. By contrast,  the averaging estimator (Algorithm AA3-0.5) converges  to an element of $\Theta$, illustrating the stabilizing property of averaging mentioned in Section \ref{sub:IF}. Finally, Figure \ref{fig:covid1} also shows that Algorithms A3-1.1 and   AA3-0.5 do not  converge to the same value; see in particular the results obtained for  $\gamma$ and $\sigma_\beta$. 

This latter observation is confirmed by  Figure \ref{fig:covid2}, which shows the estimated value of the MLE obtained after 2\,000 iterations  from 50 runs of the three algorithms considered. In particular,  similar to Figure \ref{fig:covid1}, we observe that the estimated value for $\gamma$, as well as that of $\kappa$, $\lambda_1$ and $\lambda_2$, tends to be smaller for Algorithm A3-1.1 than for Algorithm A3-0.5, while for $\sigma_\beta$ it tends to be larger. From Figure \ref{fig:covid2}  we also note the remarkable stability of Algorithm AA3-0.5 across these runs. The variability in the MLE estimates for Algorithm A3-0.5 observed in this figure (notably for $\gamma$ and $\sigma_\beta$) is therefore due to its slow convergence (see Figure \ref{fig:covid1}), rather than to it converging to different parameter values across runs. By contrast, as illustrated in Figure \ref{fig:covid1}, $2\,000$ iterations are generally sufficient for Algorithm A3-1.1 to converge, and therefore the variability in its MLE estimates shown in Figure \ref{fig:covid2} arises because it converges to different parameter values across runs. Notably, the estimated values for $\gamma$ and $\sigma_\beta$ with this algorithm are particularly unstable across runs. The contrasting behaviours observed in Figure \ref{fig:covid2}, the stability of  Algorithms A3-0.5 and AA3-0.5 versus the instability of Algorithm A3-1.1, confirm  that the filtering distributions of a SO-SSM are easier to estimate if the dynamics on $\theta$ vanishes slowly rather than  quickly.

 \begin{figure}[!t]
 \centering
 \includegraphics[scale=0.39]{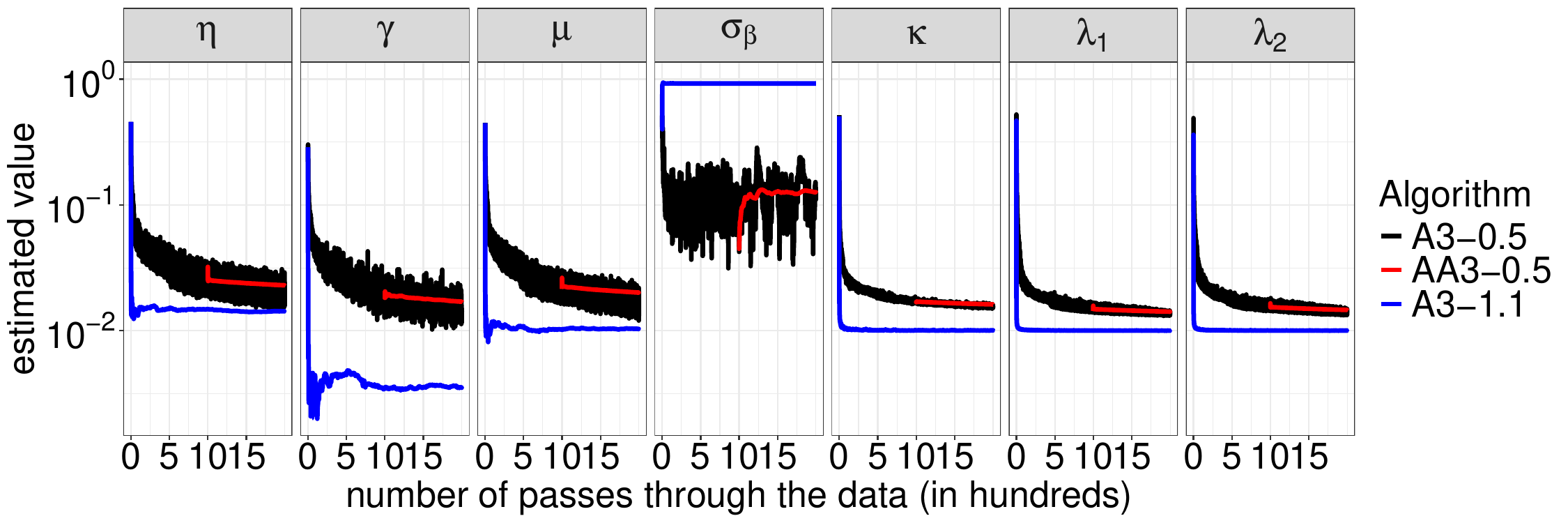}
 \caption{Results for the experiments in Section  \ref{sub:covid}. The plots show the evolution of the estimated  MLE  value during the first 2\,000 passes through the data for a single run of the different algorithms. For Algorithm AA3-0.5 the  MLE estimate is only defined when the number of passes through the data is larger than 1\,500.\label{fig:covid1}}
 \end{figure}

 \begin{figure}[!t]
 \centering
\includegraphics[scale=0.39]{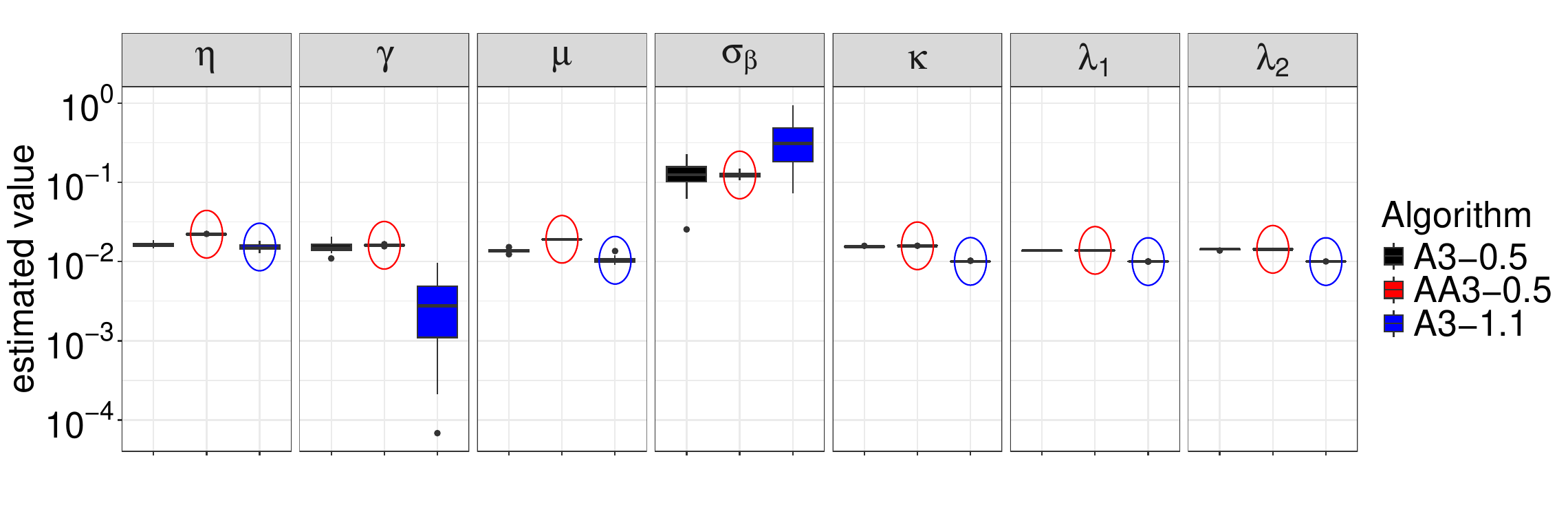}
 \caption{Results for the experiments in Section  \ref{sub:covid}. The boxplots show the estimated   MLE obtained after 2\,000 passes through the data obtained from 50  runs of the  different algorithms. Ellipses are for colour indication only. \label{fig:covid2}}
 \end{figure}
 
 \begin{figure}[!t]
 \centering
   \includegraphics[scale=0.39]{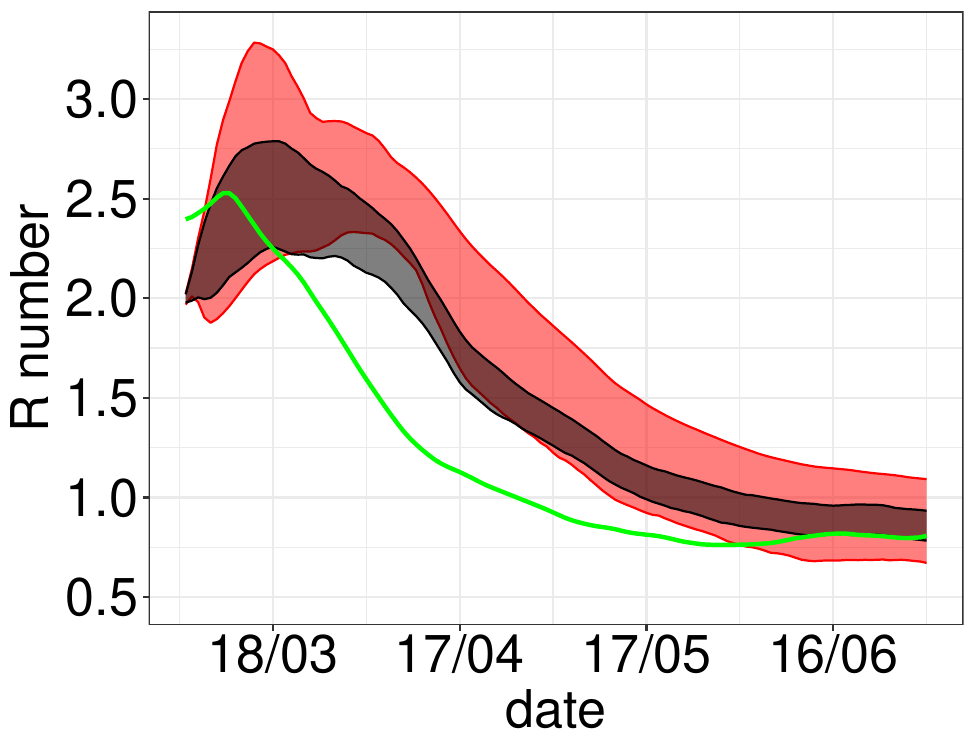}\includegraphics[scale=0.39]{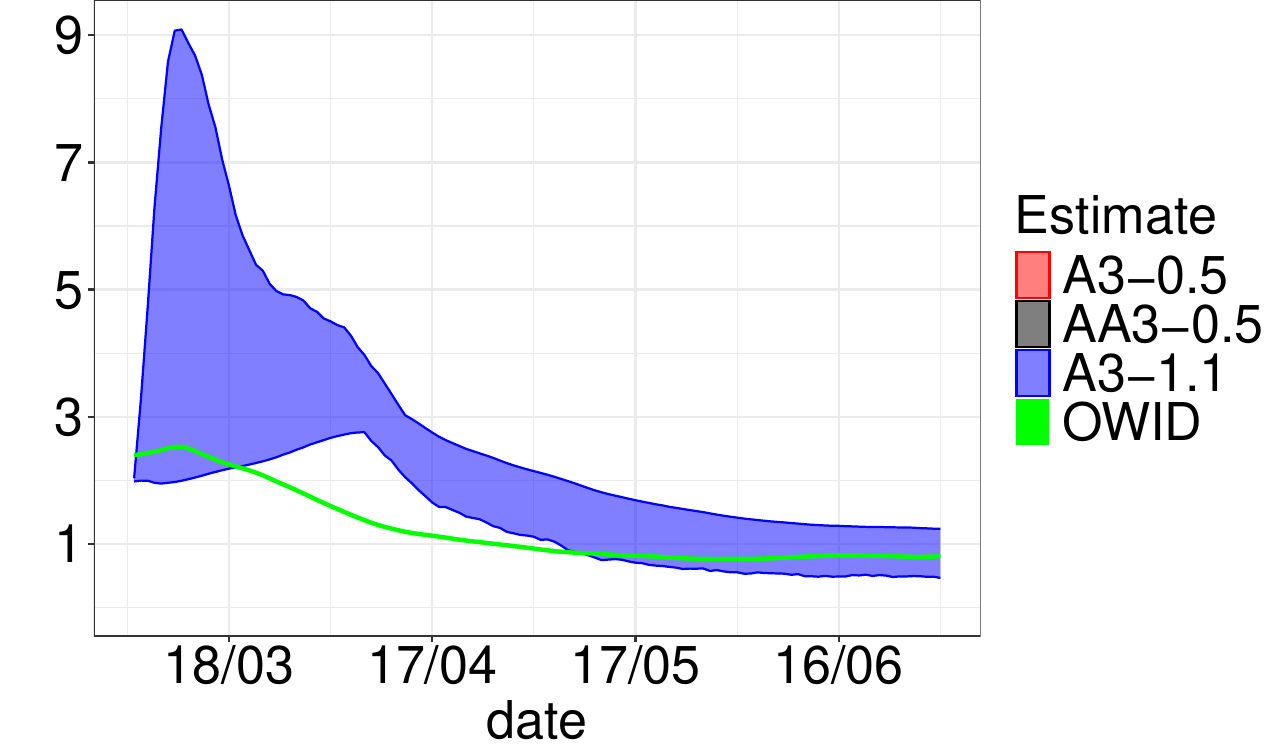}
 \caption{Results for the experiments of Section  \ref{sub:covid}. The plots show the evolution of the R number  during the observation period, as estimated by the different algorithms and as reported by the  \textit{OWID} website. For each algorithm, the shaded areas represents the range of estimated values obtained from 50 independent runs.  \label{fig:covid3}}
 \end{figure}

\subsubsection{Effective reproduction number estimation}

As mentioned above, under  SSM    \eqref{eq:SEIRD_SSM}-\eqref{eq:SEIRD_SSM0} the  R number at time $s$ can be defined as $\mathrm{R}_{s}=\beta_s S_s /(\gamma+\eta)$. For a given $\theta\in\Theta$, we define the filtered R number  as $\mathrm{R}_{s,\theta}=\bar{\E}_{s,\theta}\big[\beta_s S_s\big]/(\gamma+\eta)$, where  $\bar{\E}_{s,\theta}[\beta_s S_s]$ is the filtering expectation of $\beta_s S_s$ under model \eqref{eq:SEIRD_SSM}-\eqref{eq:SEIRD_SSM0}.  Then, our final model-based estimate of the R number at time $s$ is $\mathrm{R}_{s}:=\mathrm{R}_{s,\tilde{\theta}_T}$ where, as in Section \ref{sec:MLE}, $\tilde{\theta}_T$  denotes  the MLE of the model parameter. (Remark that, for simplicity, we use the filtering expectation of $\beta_s S_s$ to estimate the $R$ number at time $s$ although the smoothing expectation is ideally preferred.) For any given $\theta\in\Theta$, the filtered R numbers $\{\mathrm{R}_{s,\theta}\}_{s=1}^T$ can be estimated using a simple bootstrap PF with $M=20\,000$ particles, which provides the estimates $\{\mathrm{R}^M_{s, \theta}\}_{s=1}^T$. Thus, using an estimated MLE value $\tilde{\theta}^N_T$, we approximate the ideal model-based R numbers $\{\mathrm{R}_{s}\}_{s=1}^T$ with $\{\mathrm{R}^M_{s,\tilde{\theta}^N_T}\}_{s=1}^T$.

Figure \ref{fig:covid3} shows a  single realization of $\{\mathrm{R}^M_{s,\tilde{\theta}^N_T}\}_{s=1}^T$ for each MLE estimates $\tilde{\theta}^N_T$ reported in Figure \ref{fig:covid2}, and obtained from 50 runs of Algorithms A3-0.5, AA3-0.5 and A3-1.1, alongside the R numbers provided by the \textit{Our World in Data (OWID)} website. From this figure three main observations emerge. Firstly,  the R number estimates $\{\mathrm{R}^M_{s,\tilde{\theta}^N_T}\}_{s=1}^T$ derived from Algorithms A3-0.5 and  AA3-0.5 align relatively well with the \textit{OWID} benchmark.  Secondly, and in contrast, each MLE estimate $\tilde{\theta}^N_T$ computed with Algorithm A3-1.1 yields to a set of R numbers $\{\mathrm{R}^M_{s,\tilde{\theta}^N_T}\}_{s=1}^T$ that deviates radically from the benchmark, and some of these sets contain unrealistic values for the R number in the UK. Thirdly, the  ordering of the three algorithms in term of the variability of the R number estimates, represented by the shaded areas in Figure \ref{fig:covid3},  is identical to the ordering obtained in Figure \ref{fig:covid2} for parameter inference. This observation  suggests that, for each algorithm, the variability in the R number estimates  shown in Figure \ref{fig:covid3}  is mainly driven by the variability in the underpinning MLE estimates, and not by the randomness of the PF used to estimate  $\{\mathrm{R}^M_{s,\theta}\}_{s=1}^T$ for a given $\theta$.

\subsubsection{Summarizing comments}

This experiment illustrates that IF algorithms with a fast vanishing dynamics on $\theta$ are prone to being trapped in a local mode of the likelihood function. This occurs because estimating the filtering distributions of their underpinning SO-SSM is challenging with a PF.  Echoing the R number estimation results, the IF Algorithm A3-1.1 failed to provide a reasonable model parameter estimate across its 50 runs. Conversely, across 50 runs Algorithm AA3-0.5 consistently converged  to a nearly identical model parameter value, which is deemed sensible as it yields sensible R number estimates. This strong performance by Algorithm AA3-0.5 was expected from the discussion in Section \ref{sub:IF}. Indeed, the slowly vanishing dynamics on $\theta$  enables the particle system generated by this algorithm to effectively explore the parameter space  and escape local modes. Simultaneously, as argued in Section \ref{sub:IF}, the use of averaging allows to define an estimator of the MLE that converges quickly once a high (or the highest) mode of the likelihood function is found.

\section{Concluding remarks\label{sec:conclusion}}

In this work, we provided theoretical guarantees for a large class of SO-SSMs and advocated for deplying  PF algorithms on SO-SSMs where the dynamics on  $\theta$   vanishes slowly as $t\rightarrow\infty$. These algorithms are not only straightforward to implement but, as demonstrated by our numerical experiments in Section \ref{sec:num}, also require only minimal tuning to perform well in practice. We remind the reader, however, that these practical benefits 
are expected to come at the cost of reduced statistical efficiency,  a cost that has yet to be assessed from a theoretical perceptive. In our view, this positions SO-SSMs as particularly advantageous for two groups of researchers: those with limited  expertise in computational statistics who prefer simpler, more accessible methods, and those seeking to easily obtainable preliminary results, for instance while working on the development of a model   for a specific application.


The general theory we develop in Appendix \ref{sec:theory} can be used to justify various extensions of the algorithms presented in this work. First,  while our primary focus has been on SO-SSM \eqref{eq:SO-SSM_new2}, where $\theta_{t+1}$ is updated before $X_{t+1}$, we can alternatively consider a version of \eqref{eq:SO-SSM_new2}  in which $X_{t+1}$ is updated before $\theta_{t+1}$. This variant is formally defined in Section \ref{supp:X_first} of the \hyperlink{SM}{Supplementary Material}, where we show that the theoretical results established for SO-SSM \eqref{eq:SO-SSM_new2} also hold, under largely similar assumptions for SSM \eqref{eq:SSM}. Second, as explained in Section \ref{supp:adapt} of the \hyperlink{SM}{Supplementary Material}, our theory extends to cases where the scale matrix of the (truncated) Gaussian or Student-t Markov kernel $K_t$ is chosen adaptively. An example includes learning this matrix online during PF execution to ensure appropriate scaling for each components of $\theta$. Lastly, theoretically valid SO-SSMs can also be constructed for SSMs containing discrete parameters (i.e.~parameters whose values are restricted to a  finite set). For such models, Section \ref{supp:dis} of the \hyperlink{SM}{Supplementary Material} provides a suitable definition of the Markov kernels $(K_t)_{t\geq 2}$  along with theoretical guarantees and some numerical experiments for the resulting SO-SSM.




\bibliographystyle{apalike}
\bibliography{references}

 \appendix

\counterwithin{theorem}{section}
\counterwithin{definition}{section}
\counterwithin{lemma}{section}
\counterwithin{proposition}{section}
\counterwithin{corollary}{section}
\counterwithin{remark}{section}
\numberwithin{equation}{section}

\section{Self-organizing Feynman-Kac models in random environments\label{sec:theory}}

\subsection{Set-up and  additional notation\label{sub:gen_SetUp}}

All the random variables are assumed to be defined on a complete   probability space $(\Omega,\F,\P)$, and we let $(\F_t)_{t\geq 1}$ be a  filtration of $\F$. Unless   defined otherwise, we use the convention that the index $t\in\mathbb{N}$ of a random variable $Z_t$  indicates that $Z_t$ is $\F_t$-measurable. We let  $\setX$ be  a Polish space,  $\mathcal{X}$ be   the  Borel $\sigma$-algebra on $\setX$, $\mathsf{R}$ be  a separable Banach space and   $\mathcal{R}$ be the Borel $\sigma$-algebra on   $\mathsf{R}$. Next, we let $\Theta\in\mathcal{R}$ be a compact set and $\mathcal{T}$ be the Borel $\sigma$-algebra on $\Theta$.   Remark that these  assumptions on  $(\Omega,\F,\P)$,  $(\setX,\mathcal{X})$ and $(\Theta,\mathcal{T})$ ensure that, for any measurable function $\varphi:\Omega\times\Theta\times \setX\rightarrow\R$ and set $A\in\mathcal{T}\otimes \mathcal{X}$, the mappings $\omega\mapsto\sup_{(\theta,x)\in A}\varphi(\omega,\theta,x)$ and $\omega\mapsto\inf_{(\theta,x)\in A}\varphi(\omega,\theta,x)$ are measurable \citep[][Corollary 2.13, page 14]{Crauel2003}.

To introduce some notation let $(\mathsf{Z},\mathcal{Z})$ be an arbitrary measurable space. Then, we denote by $\mathcal{P}(\setZ)$ the set  of probability measures on $(\mathsf{Z},\mathcal{Z})$ and we say that  $\zeta:\Omega\times\mathcal{Z}\rightarrow[0,1]$ is a random (probability) measure on  $(\mathsf{Z},\mathcal{Z})$ if  for all $A\in\mathcal{Z}$   the mapping $\omega\mapsto \zeta(\omega,A)$ is measurable, and if  for all $\omega\in\Omega$ the mapping $B\rightarrow\zeta(\omega,B)$ is a (probability) measure on $(\mathsf{Z},\mathcal{Z})$. In what follows  we   will often use the shorthand $\zeta^\omega(\dd z)$ for the probability measure   $B\rightarrow\zeta(\omega,B)$. We say that $P$ is a random Markov kernel acting on $(\mathsf{Z},\mathcal{Z})$ if $P(z',\dd z)$ is a    random (probability) measure on  $(\mathsf{Z},\mathcal{Z})$ for all $z'\in\mathsf{Z}$.

Finally, for all $\epsilon\in (0,\infty)$ we let $B_\epsilon(r)=\{r'\in\setR:\, \|r-r'\|< \epsilon\}$ for all $r\in\setR$ and    $\mathcal{N}_\epsilon(A)=\{\theta\in\Theta:\, \inf_{\theta'\in A} \|\theta-\theta'\|<\epsilon\}$ for $A\in\mathcal{T}$, and introduce  the   notion   of regular compact sets:
\begin{definition}\label{def:sharp}
For $d\in\mathbb{N}$  we say that a set $A\subset\R^d$ is a regular compact set  if $A$ is compact and if there exits  a continuous function $f:[0,\infty)\rightarrow [0,\infty)$ such that $\lim_{x\downarrow 0}f(x)=0$ and such that for all $(\epsilon,x)\in (0,\infty)\times  A$ there exists an  $x'\in A$ for which the inclusion $\{u\in\R^d:\, \|u-x'\|< f(\epsilon)\}\subseteq \{u\in\R^d:\, \|u-x\|< \epsilon\}\cap A$ holds.
\end{definition}

 \subsection{The model\label{sub:Assumptions}\label{sub:FK_model} }
 
For all  $t\geq 1$ we let $\tilde{G}_t:\Omega\times\Theta\times \setX^2\rightarrow[0,\infty)$ and $\tilde{m}_t:\Omega\times\Theta\times \setX^2\rightarrow[0,\infty)$ be two measurable functions, and for all  $(\omega,\theta,x,x')\in \Omega\times\Theta\times \setX^2$ we let $G^\omega_{t,\theta}(x',x)=\tilde{G}_t(\omega,\theta,x',x)$ and  $m^\omega_{t,\theta}(x|x')=\tilde{m}_t(\omega,\theta,x',x)$. Next, we let $\chi\in \mathcal{P}(\setX)$ and $\lambda$ be a $\sigma$-finite measure on $(\setX,\mathcal{X})$, and for all $(\omega,\theta,x',t)\in \Omega\times \Theta\times\setX\times\mathbb{N}$, we assume that  $M^\omega_{t,\theta}(x',\dd x):=m^\omega_{t,\theta}(x|x')\lambda(\dd x)\in\mathcal{P}(\setX)$, that is that $\int_\setX M^\omega_{t,\theta}(x',\dd x)=1$. To simplify the notation, for all $(\omega,\theta,x,x',t)\in\Omega\times\Theta\times\setX^2\times\mathbb{N}$ we let $q^\omega_{t,\theta}(x|x')=G^\omega_{t,\theta}(x',x)m^\omega_{t,\theta}(x|x')$ and $Q^\omega_{t,\theta}(x',\dd x)=q^\omega_{t,\theta}(x|x')\lambda(\dd x)$. Then, for any integers $0\leq t_1<t_2$, $\theta\in\Theta$ and $\eta\in\mathcal{P}(\setX)$ we  let
\begin{align*}
L^{\eta}_{(t_1+1):t_2}(A|\theta)=\int_{\setX^{t_2-t_1+1}}\ind_A(x_{t_2}) \eta(\dd x_{t_1}) \prod_{s=t_1+1}^{t_2} Q_{s,\theta}(x_{s-1},\dd x_s),\quad\forall A\in\mathcal{X}
\end{align*}
and we denote by $\bar{p}_{(t_1+1):t_2}^{\eta}(\dd x|\theta)$   the random probability measure on $(\setX,\mathcal{X})$ such that, for all $\omega\in\Omega$, $\bar{p}_{(t_1+1):t_2}^{\omega,\eta}(A|\theta)=\eta(A)$ for all $A\in\mathcal{X}$ if $L^{\omega,\eta}_{(t_1+1):t_2}(\setX|\theta)\not\in (0,\infty)$ and such that
\begin{align*}
\bar{p}_{(t_1+1):t_2}^{\omega,\eta}(A|\theta)=\frac{L^{\omega,\eta}_{(t_1+1):t_2}(A|\theta)}{L^{\omega,\eta}_{(t_1+1):t_2}(\setX|\theta)},\quad\forall A\in\mathcal{X}
\end{align*}
otherwise. With this notation in place, for all $(\omega,\theta)\in \Omega\times\Theta$ such that $L^{\omega,\chi}_{1:t}(\setX|\theta)\in (0,\infty)$ for all $t\geq 1$,  the triplet  $\{ (G_{t,\theta}^\omega)_{t\geq 1}, (M^\omega_{t,\theta})_{t\geq 1},\chi\}$ defines a Feynman-Kac model on $(\setX,\mathcal{X})$. We refer  the reader to \citet[][Chapter 2]{del2004feynman} for a precise definition of Feynman-Kac models but, for the purpose of this work, it is only important to recall that the so-defined Feynman-Kac model specifies a distribution for a sequence $(X_t)_{t\geq 1}$ of $\setX$-valued random variables which is such that $X_t\sim \bar{p}_{1:t}^{\omega,\chi}(\dd x|\theta)$ for all $t\geq 1$.

\subsection{Assumptions on the model\label{sub:assumption}}

To state the first assumption we will impose on the model,  we need the following definition:
\begin{definition}\label{def:consistent}
Let  $\rho\in\mathbb{N}$, $\delta\in(0,\infty)$, $\mathcal{L}:\Theta\rightarrow\R$ and $(\phi_t)_{t\geq 1}$ be a sequence of  random  measurable functions defined on $\setX^2$. Then, we say that a  random probability measure $\eta_{\rho}$  on $(\setX,\mathcal{X})$     is $(\rho,\delta,\mathcal{L},(\phi_t)_{t\geq 1}\big)$-consistent if $\eta_{\rho}$ is $\F_\rho$-measurable and if, for all $\theta\in\Theta$,
\begin{equation*}
\Big|\frac{1}{\rho(t-1)}\log L_{(\rho+1):\rho t}^{\eta_{\rho}}(\setX|\theta)-\mathcal{L}(\theta)\Big|\PP 0,\quad\inf_{t\geq 2}\P\big(L_{(\rho+1):\rho t}^{\eta_{\rho}}(\setX|\theta)>0\big)=1
\end{equation*}
and if we the following two conditions hold:
\begin{align*}
&\frac{1}{\rho t}\log^+  \int_{\setX^{\rho(t-1)+1}}\exp\bigg( \delta  \sum_{s=\rho +1}^{\rho t}\phi_s(x_{s-1}, x_{s})\bigg) \eta_\rho(\dd x_{\rho})  \prod_{s=\rho+1}^{\rho t} Q_{s,\theta}(x_{s-1},\dd x_{s})=\bigO_\P(1)\\ 
&\inf_{t\geq 2}\P\bigg( \int_{\setX^{\rho(t-1)+1}}\exp\bigg( \delta  \sum_{s=\rho +1}^{\rho t}\phi_s(x_{s-1}, x_{s})\bigg) \eta_\rho(\dd x_{\rho})  \prod_{s=\rho+1}^{\rho t} Q_{s,\theta}(x_{s-1},\dd x_{s})<\infty\bigg)=1.
\end{align*}
\end{definition}

We then consider the following three assumptions:

\begin{assumption}\label{assume:Model}  There exist   an integer $r\in\mathbb{N}$, a constant $\delta_\star\in(0,1)$, a set $D\in\mathcal{X}$, a continuous function $\tilde{l}:\Theta\rightarrow\R$, a continuous and strictly  increasing function $g:[0,\infty)\rightarrow [0,\infty)$   such that $g(0)=0$  and a sequence $(\varphi_t)_{t\geq 1}$ of  random measurable functions defined on $\setX^2$, for which the following conditions hold:

\begin{enumerate}[label=A1(\alph*), align=left, leftmargin=*, labelindent=0em]

\item\label{assume:Model_stationary} For all $(\theta,x',x)\in\Theta\times\setX^2$  the sequence $(Z_t)_{t\geq 1}$, such that   $Z_t=\big(W_{(t-1)r+1},\dots, W_{rt}\big)$ with $W_t=\big(\tilde{G}_t(\cdot,\theta,x',x),\tilde{m}_t(\cdot,\theta,x',x) \big)$ for all $t\geq 1$, is  stationary.

\item\label{assume:Model_ThetaStar} The set $\tilde{\Theta}_\star:=\argmax_{\theta\in\Theta}\tilde{l}(\theta)$ is  such that   $\mathcal{N}_{\delta_\star}(\tilde{\Theta}_\star)\subseteq \Theta$.

\item\label{assume:Model_init}   The probability measure $\chi$ is $(r,\delta_\star, \tilde{l}, (\varphi_t)_{t\geq 1})$-consistent.
           
\item \label{assume:Model_smooth} 
 For all $t\geq 1$ and $\lambda$-almost every $x,x'\in\setX$ 
\begin{align*}
\big|\log  q_{t,\theta}(x|x')  -\log q_{t,\theta'}(x |x')\big|\leq  g\big(\|\theta-\theta'\|\big) \varphi_t(x',x),\quad \forall (\theta,\theta')\in \Theta^2,\quad\P-a.s.
\end{align*}

\item\label{assume:Model_G_lower} There exists an  $(r,\delta_\star, \tilde{l}, (\varphi_t)_{t\geq 1})$-consistent  random  probability measure  $\underline{\chi}_{r}(\dd x)$ on $(\setX,\mathcal{X})$ such that
\begin{equation*}
\underline{\chi}_{r}(A) =\int_A \frac{\inf_{(\theta,x')\in \Theta \times D} q_{r,\theta}(x|x') }{\int_{\setX}\big\{\inf_{(\theta,x')\in \Theta \times D} q_{r,\theta}(x|x') \big\}\lambda(\dd x)}\lambda(\dd x),\quad\forall A\in\mathcal{X},\quad \P-a.s.
\end{equation*}

\item\label{assume:Model_sup} 
There exists an  $(r,\delta_\star, \tilde{l}, (\varphi_t)_{t\geq 1})$-consistent random  probability measure  $\bar{\chi}_r(\dd x)$ on $(\setX,\mathcal{X})$ such that
\begin{equation*}
\bar{\chi}_r(A) =\int_A \frac{\sup_{(\theta,x')\in \Theta \times \setX} q_{r,\theta}(x|x') }{\int_{\setX}\big\{\sup_{(\theta,x')\in \Theta \times \setX}q_{r,\theta}(x|x') \big\}\lambda(\dd x)}\lambda(\dd x),\quad\forall A\in\mathcal{X},\quad \P-a.s.
\end{equation*}   

\end{enumerate}

\end{assumption}

\begin{assumption}\label{assume:G_bounded}
$\frac{1}{t}\sum_{s=1}^{t} \log^+ \sup_{(\theta,x,x')\in\Theta\times\setX^2} G_{s,\theta}(x',x)=\bigO_\P(1)$.
\end{assumption}

\begin{assumption}\label{assume:G_l_E}
Under \ref{assume:Model} and with $D\in\mathcal{X}$ as in \ref{assume:Model},  $\log^- \inf_{(\theta,x')\in\Theta\times D}Q_{t,\theta}(x',D)<\infty$ for all $t\geq 1$ and $\P$-a.s.
\end{assumption}

\subsection{Extending the model to include its parameters\label{sub:SO-FK_model}}

Let $\mu_0\in\mathcal{P}(\Theta)$   and $(\mu_t)_{t\geq 1}$ be a sequence of random probability measures on $(\setR,\mathcal{R})$. Then, for all $t\geq 1$,   we let $K_{\mu_t}$ be the Markov kernel on $(\setR,\mathcal{R})$  such that, for all $\theta'\in\setR$, we have $\theta\sim K_{\mu_t}(\theta',\dd\theta)$ if and only if $\theta\dist \theta'+\tilde{U}$ where  $\tilde{U}|\F_t\sim\mu_{t}$, and we let $K_{\mu_t|\Theta}$  be the Markov kernel on $(\Theta,\mathcal{T})$ defined by
\begin{equation*}
K_{\mu_t|\Theta}(\theta',\dd\theta)=\frac{\ind_\Theta(\theta)K_{\mu_t}(\theta',\dd\theta)}{\int_\Theta K_{\mu_t}(\theta',\dd\theta)},\quad\theta'\in\Theta.
\end{equation*}
Next, we let $\pi_0=\mu_0\otimes\chi$ and, for all $t\geq 1$, we   let
\begin{align}\label{eq:normalizing}
\mathfrak{Z}_t= \int_{(\Theta\times\setX)^{t+1}}\mu_0(\dd\theta_0)\chi(\dd x_0) \prod_{s=1}^{t} Q_{s,\theta_{s}}(x_{s-1},\dd x_s) K_{\mu_s|\Theta}(\theta_{s-1},\dd\theta_{s})
\end{align}
and  $\pi_t$ be the random probability measure on $(\Theta\times\setX,\mathcal{T}\times\mathcal{X})$ such that, for all $\omega\in\Omega$, 
\begin{align*}
\pi^\omega_t\big (A\times B)=
\frac{\int_{(\Theta\times\setX)^{t+1}}\ind_{A}(\theta_t)\ind_B(x_t)\mu_0(\dd\theta_0)\chi(\dd x_0) \prod_{s=1}^{t} Q^\omega_{s,\theta_{s}}(x_{s-1},\dd x_s) K^\omega_{\mu_s|\Theta}(\theta_{s-1},\dd\theta_{s})}{\mathfrak{Z}^\omega_t}
\end{align*}
for all $(A,B)\in\mathcal{T}\times\mathcal{X}$ if $\mathfrak{Z}^\omega_t\in(0,\infty)$ while $\pi^\omega_t=\pi_0$ otherwise. In what follows, for all $t\geq 0$  we le $\pi_{t,\Theta}$  and $\pi_{t,\setX}$  be the random probability measures on $(\Theta,\mathcal{T})$ and  $(\setX,\mathcal{X})$, respectively, such that $\pi_{t,\Theta}(A)=\pi_t(A\times\setX)$ and $\pi_{t,\setX}(B)=\pi_t(\Theta\times B)$ for all $(A,B)\in\mathcal{T}\times\mathcal{X}$. Remark that if for all $(\theta',x',\omega,t)\in\Theta\times\setX\times \Omega\times\mathbb{N}$  we let
\begin{align*}
\check{M}_{t}((\theta',x'),\dd (\theta,x))=M_{t,\theta'}(x',\dd x)K_{\mu_t|\Theta}(\theta',\dd\theta),\quad \check{G}_t( (\theta',x'),(\theta,x))=G_{t,\theta}(x',x) 
\end{align*}
then, for all $\omega\in\Omega$ such that $\mathfrak{Z}^\omega_t\in(0,\infty)$ for all $t\geq 1$, the triplet  $\{ (\check{G}_{t}^\omega)_{t\geq 1}, (\check{M}_{t }^\omega)_{t\geq 1},\pi_0\}$ defines a Feynman-Kac model on the extended space $(\Theta\times\setX,\mathcal{T}\times\mathcal{X})$. In particular, this model  specifies a distribution for a sequence of $(\setX\times\Theta)$-valued random variables $((\theta_t,X_t))_{t\geq 1}$  which is such that $(\theta_t,X_t)\sim \pi_t(\dd(\theta,x))$ for all $t\geq 1$.

\subsection{Conditions on the artificial dynamics\label{sub:noise}}


To facilitate the statement of some conditions   we let   $(U_t)_{t\geq 1}$ be  a sequence of $\Theta$-valued random variables such that
\begin{align*}
\P\Big( U_s\in A_s,\,\forall s\in\{1,\dots,t\}\,\big| \F_t\Big)=\prod_{s=1}^t\mu_s(A_s),\quad\forall A_1,\dots, A_t\in\mathcal{T},\quad\forall t\geq 1.
\end{align*}

We will consider the following conditions on $(\mu_t)_{t\geq 0}$:

\begin{condition}\label{condition:mu1}
$\mu_0\big(B_{\epsilon}(\theta)\big)>0$ for all $\theta\in\Theta$ and $\epsilon\in(0,\infty)$.
\end{condition}

\begin{condition}\label{condition:Inf_mu}
There exist a family  of $(0,1]$-valued random variables  $(\Gamma^\mu_{\delta})_{\delta\in(0,\infty)}$ such that, for all $\delta\in(0,\infty)$, we have $\P\big(\inf_{t\geq 1}\mu_t\big(B_\delta(0)\big)\geq \Gamma^\mu_{\delta}\big)=1$.
\end{condition}

\begin{condition}\label{condition:Inf_K}
 $\P\big(\inf_{t\geq 1} \inf_{\theta'\in\Theta}\int_\Theta K_{\mu_t}(\theta',\dd\theta)\geq \Gamma^\mu\big)=1$ for some  $(0,1]$-valued random variable $\Gamma^\mu$.
\end{condition}

\begin{condition}\label{condition:mu_seq} 
There exist (i) two  sequences     $(t_p)_{p\geq 1}$ and $(v_p)_{p\geq 1}$ in $\mathbb{N}$,  such that 
\begin{align*}
\inf_{p\geq 1}(t_{p+1}/t_p)>1,\quad \lim_{p\rightarrow\infty}v_p/(t_{p+1}-t_p)=0,\quad \lim_{p\rightarrow\infty}v_p=\infty,\quad  \limsup_{p\rightarrow\infty}(v_{p+1} /v_p)< 3/2,
\end{align*}
and (ii) a sequence $(f_p)_{p\geq 1}$ of  $(0,\infty]$-valued functions defined on $(0,\infty)$, such that  $\lim_{p\rightarrow\infty}f_p(\epsilon)=0$ for all $\epsilon\in (0,\infty)$, 
for which, with $\P$-probability one, the following conditions hold for all $\epsilon\in(0,\infty)$ and all   $p\geq p':=\inf\{p\in\mathbb{N}: t_p\geq 1+3v_p\}$:
\begin{enumerate}[label=C4(\alph*), align=left, leftmargin=*, labelindent=0em]

\item\label{C41} For all   non-empty sets $A\in\mathcal{T}$ we have $ \frac{1}{t_{p+1}-t_{p}}\inf_{\theta\in \Theta}\log   K_{\mu_{t_{p}}}\big(\theta,\mathcal{N}_{\epsilon}(A)\big)\geq -f_{p}(\epsilon)$.

\item\label{C42} For all   sequences $(l_i)_{i\geq 1}$ and $(u_i)_{i\geq 1}$ in $\mathbb{N}$ such that $t_i\leq l_i< u_i\leq t_{i+1}-1$ for all $i\geq p'$,
\begin{align*}
\frac{1}{u_p-l_p}\log\P\Big(\exists s\in\big\{l_p+1,\dots,u_p\big\}:\, \sum_{i=l_p+1}^s  U_i\not\in B_{\epsilon}(0)\big|\,\F_{u_p}\Big)\leq-\frac{1}{f_p(\epsilon)}.
\end{align*}

\item\label{C43}  For all sequences $(l_i)_{i\geq 1}$ and $(u_i)_{i\geq 1}$ in $\mathbb{N}$ such that $t_{i}-3v_i\leq l_i<  u_i\leq t_{i}+v_i$ for all $i\geq p'$, 
\begin{align*}
\frac{1}{u_p-l_p}\log\P\Big(\exists s\in\big\{l_p+1,\dots,u_p\big\}:\, \sum_{i=l_p+1}^s  U_i\not\in B_{\epsilon}(0)\big|\,\F_{u_p}\Big)\leq -\frac{1}{f_p(\epsilon)}.
\end{align*}

\end{enumerate}
\end{condition}

Instead of \ref{condition:mu_seq}, we will also consider the following alternative condition:

\setcounter{conditionB}{3} 
\begin{conditionB}\label{condition:mu_seq2} 
There exist (i) a sequence $(k_t)_{t\geq 1}$ in $\mathbb{N}$, such that $\lim_{t\rightarrow\infty} k_t/t=\lim_{t\rightarrow\infty}(1/k_t)=0$, and (ii) for all integer $m\in\mathbb{N}_0$ there exists a sequence  $(f_{t,m})_{t\geq 1}$ of  $(0,\infty]$-valued functions defined on $(0,\infty)$, such that $\lim_{t\rightarrow\infty}f_{t,m}(\epsilon)=0$ for all $\epsilon\in (0,\infty)$, for which, with $\P$-probability, the following conditions hold for all $\epsilon\in(0,\infty)$:
\begin{enumerate}[label=C4'(\alph*), align=left, leftmargin=*, labelindent=0em]

\item\label{C4'2}  For all $t\geq 1$ such that $t>2k_t$ and all $l_t\in\{k_t,\dots,2k_t\}$, 
\begin{align*}
\frac{1}{t-l_t}\log\P\Big(\exists s\in\big\{l_t+1,\dots,  t\big\}:\, \sum_{i=l_t+1}^s  U_i\not\in B_{\epsilon}(0)\big|\F_{t}\Big)\leq -\frac{1}{f_{t,m}(\epsilon)}.
\end{align*}  
\item\label{C4'3}  $\frac{1}{t-m}\log\P\big(\sum_{i=m+1}^s  U_i \in B_{\epsilon}(0),\,\forall s\in\big\{m+1,\dots,t\big\}\big|\F_{t}\big)\geq -f_{t,m}(\epsilon)$ for all $t> m$.
\end{enumerate}
\end{conditionB}

Lastly, some of our results will require that $(\mu_t)_{t\geq 1}$  satisfies the following additional condition:
\begin{condition}\label{condition:mu_extra}
There exists a sequence    $(f'_t)_{t\geq 1}$ of $(0,\infty]$-valued functions defined on $[0,\infty)^2$, such that  the function  $f'_t(\cdot,\epsilon)$ is non-decreasing for all $t\geq 1$ and all $\epsilon\in(0,\infty)$  and such that  $\lim_{t\rightarrow\infty}f'_{t}\big(1/\delta'_t, \delta'_t\big)=0$ for some   $(\delta'_t)_{t\geq 1}$ in $(0,\infty)$ verifying $\lim_{t\rightarrow\infty}\delta'_t=0$, for which, with $\P$-probability one, for all integers $1\leq a<b$ we have
\begin{align*}
\frac{1}{b-a}\log\P\Big(\exists s\in\big\{a+1,\dots,  b\big\}:\, \sum_{i=a+1}^s  U_i\not\in B_{\epsilon}(0)\big|\F_{b}\Big)\leq -\frac{1}{f'_{b}(b-a,\epsilon)},\quad\forall\epsilon\in(0,\infty).
\end{align*}
\end{condition}

Examples of random probability measures  $(\mu_t)_{t\geq 1}$ verifying the above conditions are given in Section \ref{sub:continuous} of the  \hyperlink{SM}{Supplementary Material} (see also Section \ref{sec:SSM}).

\subsection{Main results}

\begin{theorem}\label{thm:main}
Assume that Assumptions \ref{assume:Model}-\ref{assume:G_bounded} hold and that $(\mu_t)_{t\geq 0}$ is such that Conditions \ref{condition:mu1}-\ref{condition:mu_seq}  are satisfied, and let $r\in\mathbb{N}$, $\tilde{\Theta}_\star\subset\Theta$ and $D\in\mathcal{X}$ be as in  \ref{assume:Model}. In addition, assume that either   \ref{condition:mu_seq} holds for a sequence $(t_p)_{p\geq 1}$ in $\{sr,\,s\in\mathbb{N}\}$ or that Assumption  \ref{assume:G_l_E} holds. Finally, assume  that the following condition holds:
\begin{equation}\label{eq:W_star}
\exists\text{   $(W^*_{t})_{t\geq 0}$} 
\text{ s.t. $\inf_{t\geq 1}\P\big(\pi_{t,\setX}(D)\geq W^*_{t}\big)=1$  and s.t. $\log W^*_{t}=\bigO_\P(1)$}.
\end{equation} 
Then,  for all $\epsilon\in(0,\infty)$, there exists a sequence $(\xi_{\epsilon,t})_{t\geq 1}$ of $[0,1]$-valued random variables, depending on $(\mu_t)_{t\geq 1}$  only through  (i)  $(W^*_{t})_{t\geq 0}$, (ii) the sequence $(\Gamma^\mu_{\delta})_{\delta\in(0,\infty)}$ defined in \ref{condition:Inf_mu}, (iii) the random variable $\Gamma^\mu$ defined in \ref{condition:Inf_K} and (iv) the sequence $( (t_p,v_p,f_p))_{p\geq 1}$ defined in \ref{condition:mu_seq}, such that $\xi_{\epsilon,t}\PP 0$ and such that $\pi_{t,\Theta}\big(\mathcal{N}_\epsilon(\tilde{\Theta}_\star)\big)\geq 1-\xi_{\epsilon,t}$ for all $t\geq 1$ and $\P$-a.s.
\end{theorem}

The following proposition gives a sufficient condition for \eqref{eq:W_star} to hold:
\begin{proposition}\label{prop:stability}
Assume that Assumptions \ref{assume:Model}-\ref{assume:G_bounded} hold  and that $(\mu_t)_{t\geq 0}$ is such that Conditions \ref{condition:mu1}-\ref{condition:Inf_mu} hold, and let $D\in\mathcal{X}$ be as in \ref{assume:Model}. In addition, for all $t\geq 1$  let
\begin{align*}
W_t=\inf_{(\theta,x')\in\Theta\times\setX}\frac{Q_{t,\theta}(x',D)}{Q_{t,\theta}(x',\setX)} 
\end{align*}
and assume that $\chi$ is such that $\chi(D)>0$ and that $\log W_t=\bigO_\P(1)$. Then, \eqref{eq:W_star} holds with $W^*_0=\chi(D)$ and $W^*_t=W_t$ for all $t\geq 1$.
\end{proposition}

The next theorem is the counterpart  of Theorem \ref{thm:main}   when we use \ref{condition:mu_seq2}   instead of  \ref{condition:mu_seq}:

\begin{theorem}\label{thm:main2}
Assume that Assumptions \ref{assume:Model}-\ref{assume:G_bounded} hold and that $(\mu_t)_{t\geq 0}$ is such that Conditions \ref{condition:mu1}-\ref{condition:Inf_K} and \ref{condition:mu_seq2} are satisfied. In addition, assume that \eqref{eq:W_star} holds and let $r\in\mathbb{N}$ and $\tilde{\Theta}_\star\in\Theta$ be as in \ref{assume:Model}. Then,  for all $\epsilon\in(0,\infty)$, there exists a sequence $(\xi_{\epsilon,t})_{t\geq 1}$ of $[0,1]$-valued random variables, depending on $(\mu_t)_{t\geq 1}$  only through (i) the sequence $(W^*_{t})_{t\geq 0}$ defined in \eqref{eq:W_star}, (ii)  the sequence $(\Gamma^\mu_{\delta})_{\delta\in(0,\infty)}$ defined in \ref{condition:Inf_mu}, (iii) the random variable $\Gamma^\mu$ defined in \ref{condition:Inf_K} and (iv) the sequence $( (k_t, f_{t,r+1}))_{t\geq 1}$ defined in \ref{condition:mu_seq2}, such that $\xi_{\epsilon,t}\PP 0$ and such that $\pi_{t,\Theta}\big(\mathcal{N}_\epsilon(\tilde{\Theta}_\star)\big)\geq 1-\xi_{\epsilon,t}$ for all $t\geq 1$ and $\P$-a.s.

\end{theorem}

Finally, the following theorem studies the behaviour of $\pi_{t,\setX}$ as $t\rightarrow\infty$:
\begin{theorem}\label{thm:pred}
Consider the set-up of Theorem \ref{thm:main} (resp.~the set-up of Theorem \ref{thm:main2}) and let $r\in\mathbb{N}$ be as in \ref{assume:Model}. In addition, assume  that $(\mu_t)_{t\geq 1}$ is such that Condition \ref{condition:mu_extra} holds and  that \ref{assume:Model} holds with $\tilde{\Theta}_\star=\{\tilde{\theta}_\star\}$ for some $\tilde{\theta}_\star\in \Theta$. Then, there exists a sequence $(s_t)_{t\geq 1}$ in $\{sr,\,s\in\mathbb{N}_0\}$, depending on $(\mu_t)_{t\geq 1}$ only through the sequence $(f'_t)_{t\geq 1}$ defined in \ref{condition:mu_extra} and such that $\sup_{t\geq 1}(s_t/t)< 1$ and   $\lim_{t\rightarrow\infty}(t-s_t)=\infty$, and there exists a sequence $(\xi_t)_{t\geq 1}$ of $[0,1]$-valued random variables, depending on $(\mu_t)_{t\geq 1}$ only through (i) the  sequence $(\Gamma^\mu_{\delta})_{\delta\in(0,\infty)}$ defined in \ref{condition:Inf_mu}, (ii) the random variable $\Gamma^\mu$ defined in \ref{condition:Inf_K}, (iii) the sequence $((t_p,v_p,f_p))_{p\geq 1}$ defined in \ref{condition:mu_seq} (resp.~the sequence $((k_t,f_{t,r+1}))_{t\geq 1}$ defined in \ref{condition:mu_seq2}) and (iv) the sequence $(f'_t)_{t\geq 1}$ defined in \ref{condition:mu_extra}, such that $\xi_t\PP 0$ and such that
\begin{align*}
\|\pi_{t,\setX}(\dd x)-\bar{p}_{(s_t+1):t}^{\pi_{s_t,\setX}}(\dd x|\tilde{\theta}_\star)\|_{\mathrm{TV}}\leq \xi_{t},\quad\forall t\geq 1,\quad\P-a.s.
\end{align*}
\end{theorem}

Remark that Theorem \ref{thm:pred} is only a partial result, in the sense that we ideally want to show  that $
\|\pi_{t,\setX}(\dd x)-\bar{p}_{1:t}^{\chi}(\dd x|\tilde{\theta}_\star)\|_{\mathrm{TV}}\PP 0$. One way to obtain such a result using Theorem \ref{thm:pred} is to show that, as $t\rightarrow\infty$,
\begin{align*}
\|\bar{p}_{(s_t+1):t}^{\pi_{s_t,\setX}}(\dd x|\tilde{\theta}_\star)-\bar{p}_{1:t}^{\chi}(\dd x|\tilde{\theta}_\star)\|_{\mathrm{TV}}=\|\bar{p}_{(s_t+1):t}^{\pi_{s_t,\setX}}(\dd x|\tilde{\theta}_\star)-\bar{p}_{(s_t+1):t}^{\bar{p}_{1:s_t}^{\chi}(\dd x|\tilde{\theta}_\star)}(\dd x|\tilde{\theta}_\star)\|_{\mathrm{TV}}\PP0
\end{align*}
which, informally speaking, requires to study the speed at  which the distribution of $X_t$ under the model $\{ (G_{t,\theta_\star}^\omega)_{t\geq 1}, (M^\omega_{t,\theta_\star})_{t\geq 1},\chi\}$  forgets its initial distribution as $t$ increases. Such a result  is given in Proposition \ref{prop:forget} of the \hyperlink{SM}{Supplementary Material}  in the context of SSMs.

\section{Technical assumptions}

\subsection{Technical assumptions  for the results in Section \ref{sec:SSM}\label{app:assume_online}}

For all $t\geq 1$ we let $i_t=t-\tau\lfloor (t-1)/\tau\rfloor$ and, for all   $s\in\{1,\dots,\tau\}$ and $(\theta,x')\in\Theta\times \setX$,  we let $m_{s,\theta}(x|x')$ denotes  the  Radon-Nikodym derivative of $M_{s,\theta}(x',\dd x)$ w.r.t.~$\lambda(\dd x)$. Next, for all $\theta\in\Theta$ and $z=(y_1,\dots,y_\tau)\in\setY^{\tau}$,   we let ${L}_\theta[z]$ be the unnormalized kernel on $(\setX,\mathcal{X})$ defined by
\begin{align*}
{L}_\theta[z](x_1,A)=\int_{\setX^{\tau}} \ind_A(x_{\tau+1} ) \prod_{s=1}^{\tau}f_{s,\theta}(y_s|x_s)M_{s+1,\theta}(x_{s},\dd x_{s+1}),\quad (x_1,A)\in\setX\times \mathcal{X}.
\end{align*}
Finally, to state our assumptions on the SSM we   need the following definition:

\begin{definition}\label{def:LD_sets}
A set $C\in\mathcal{X}$  is a  $\tau$-local Doeblin set if there exist (i) two positive functions $\epsilon^-_C:\setY^{\tau}\rightarrow(0,\infty)$ and  $\epsilon^+_C:\setY^{\tau}\rightarrow(0,\infty)$, (ii) a family $\{\lambda^\theta_C[z],\,(\theta,z)\in \Theta\times \setY^{\tau}\}$ of probability measures on $(\setX,\mathcal{X})$ such that $\lambda^\theta_C[z](C)=1$ for all $(\theta,z)\in\Theta\times \setY^{\tau}$ and (iii) positive functions  $\{\psi^\theta_C[z],\,(\theta,z)\in \Theta\times \setY^{\tau}\}$, such that
\begin{align*}
\epsilon^-_C(z)\psi^\theta_C[z](x)\lambda_C^\theta[z](A)\leq L_\theta[z](x,A\cap C)\leq \epsilon^+_C(z)\psi^\theta_C[z](x)\lambda_C^\theta[z](A),\quad\forall (x,A)\in C\times\mathcal{X}.
\end{align*} 
\end{definition}

The following three assumptions are from \citet{Douc_MLE}:

\setcounter{assumptionSSM}{0} 

\begin{assumptionSSM}\label{assumeSSM:K_set}
There exists a set $K\in\mathcal{Y}^{\otimes \tau}$ such that the following holds:
\begin{enumerate}
\item $\P(Z_1\in K)>2/3$.
\item For all $\eta\in(0,\infty)$ there exists an $\tau$-local Doeblin set $C\in\mathcal{X}$ such that,    with $\epsilon_C^+$ and $\epsilon_C^-$ as in Definition \ref{def:LD_sets},
\begin{align*}
\sup_{x\in C^c}{L}_\theta[z](x,\setX)\leq \eta \sup_{x\in\setX}{L}_\theta[z](x,\setX),\quad  \inf_{z\in K}\frac{\epsilon_C^-(z)}{\epsilon_C^+(z)}>0,\quad\forall(\theta,z)\in\Theta\times K.
\end{align*}
\end{enumerate}
\end{assumptionSSM}

\begin{assumptionSSM}\label{assumeSSM:D_set}
For all compact set $E\in\mathcal{X}$ such that $\lambda(E)>0$, we have
\begin{align*}
\E\Big[\log^-\inf_{(\theta,x)\in\Theta\times E} f_{s,\theta}(Y_s|x)\Big]<\infty,\quad \inf_{(\theta,x,x')\in\Theta\times E\times E} m_{s+1,\theta}(x|x')>0,\quad\forall s\in\{1,\dots,\tau\}.
\end{align*}
\end{assumptionSSM}

\begin{assumptionSSM}\label{assumeSSM:G}
For all $s\in\{1,\dots,\tau\}$ we have both $f_{s,\theta}(y|x)>0$ for all $(\theta,y,x)\in\Theta\times\setY\times\setX$  and
\begin{align*}
\E\Big[\log^+\sup_{(\theta,x)\in\Theta\times\setX}f_{s,\theta}(Y_s|x)\Big]<\infty.
\end{align*}
\end{assumptionSSM}

\begin{assumptionSSM}\label{assumeSSM:smooth}
Let $m_{1,\theta}(x|x')=m_{\tau+1,\theta}(x|x')$  and, for all $t\in\mathbb{N}$, let  $q_{t,\theta}(x|x')=f_{t,\theta}(Y_t|x )m_{t,\theta}(x|x')$ for all    $(\theta,x,x')\in\Theta\times\setX^2$.  Then, we $\P$-a.s.~have  $ \int_{\setX}\big\{\sup_{(\theta,x')\in \Theta \times \setX} q_{\tau,\theta}(x|x')\big\}\lambda(\dd x)<\infty$. 
In addition, there exist a  $\delta\in(0,\infty)$, a continuous and strictly  increasing function $g':[0,\infty)\rightarrow [0,\infty)$   such that $g'(0)=0$  and a family  $\{\varphi'_s\}_{s=1}^\tau$ of   measurable functions defined on $\setY\times \setX^2$, for which the following conditions hold:
 \begin{enumerate}
 \item  For all  $(x,x')\in\setX^2$, all $(\theta,\theta')\in \Theta\times\Theta$ and all $s\in\{1,\dots,\tau\}$, we have
\begin{align*}
\big|\log q_{s,\theta}(x|x')   -\log q_{s,\theta'}(x|x') \big|\leq  g'\big(\|\theta-\theta'\|\big) \varphi'_s(Y_s,x',x).
\end{align*}
\item There exists  a measurable function $F:\setY^\tau\rightarrow [1,\infty)$ such that $\E[\log F(Z_1)]<\infty$ and such that, letting $Z_t=(Y_{(t-1)\tau+1},\dots,Y_{t\tau})$ for all $t\geq 1$, for all $\theta\in\Theta$ and $t\geq 2$  we have, $\P$-a.s.,
\begin{align*}
\int_{\setX^{\tau t+1}} \exp\Big(\delta \sum_{s=1}^{\tau t}\varphi'_{i_s}(Y_s,x_{s-1},x_s)\Big)\chi_\theta(\dd x_0) &\prod_{s=1}^{\tau t} q_{i_s,\theta}(x_s| x_{s-1})\lambda(\dd x_s)\leq \prod_{s=1}^t F(Z_s). 
\end{align*}
\item For all $\theta\in\Theta$ we have $\frac{1}{t}\log^+ W_{t,\theta}=\bigO_\P(1)$ and $\inf_{t\geq 2}\P(W_{t,\theta}<\infty)=1$  where
\begin{align*}
 W_{t,\theta}&=\int_{\setX^{\tau (t-1)+1}}\hspace{-0.6cm}  \exp\Big(\delta \hspace{-0.2cm}\sum_{s=\tau+1}^{\tau t}  \varphi'_{i_s}(Y_s,x_{s-1},x_s)\Big) \bar{q}_{\tau}(x_\tau) \lambda(\dd x_\tau)\prod_{s=\tau+1}^{\tau t}   q_{i_s,\theta}(x_s| x_{s-1})\lambda(\dd x_s).
\end{align*}
for all $t\geq 2$ and with $\bar{q}_{\tau}(x)=\sup_{(\theta',x')\in \Theta \times \setX}q_{\tau,\theta'}(x| x')$ for all $x\in\setX$.

\end{enumerate}

\end{assumptionSSM}

\subsection{Technical assumption   for the results in Section \ref{sec:MLE}\label{app:assume_mle}}

\begin{assumptionMLEB}\label{assumeSSMB:smooth_MLE}
We have $\tilde{L}_T(\theta)>0$ for all $\theta\in\Theta$ and  $\sup_{(\theta,x)\in\Theta\times\setX} \tilde{f}_{s,\theta}(\tilde{y}_s|x)<\infty$ for all $s\in\{1,\dots,T\}$. In addition,  there exist   a continuous and strictly  increasing function $g':[0,\infty)\rightarrow [0,\infty)$   such that $g'(0)=0$  and a family  $\{\varphi'_s\}_{s=1}^T$ of   measurable functions defined on $\setX^2$ and such that $\varphi'_1(x',x)=\varphi'_1(x,x)$ for all $(x,x')\in\setX^2$, for which the following conditions hold:
 \begin{enumerate}
 \item   For all  $(x,x',\theta,\theta')\in \setX^2 \times\Theta^2$ and all $s\in\{2,\dots,T\}$, we have
\begin{align*}
\big|\log\big(\tilde{f}_{s,\theta}(\tilde{y}_s|x)\tilde{m}_{s,\theta}(x|x')\big)  -\log\big(\tilde{f}_{s,\theta'}(\tilde{y}_s|x)\tilde{m}_{s,\theta'}(x|x')\big) \big|\leq  g'\big(\|\theta-\theta'\|\big) \varphi'_s(x',x).
\end{align*}
\item We have, for all $(x, \theta,\theta')\in \setX \times \Theta^2$,
\begin{align*}
\big|\log\big(\tilde{f}_{1,\theta}(\tilde{y}_1|x)\tilde{p}_{\theta}(x)\big)  -\log\big(\tilde{f}_{1,\theta'}(\tilde{y}_1|x)\tilde{p}_{\theta'}(x)\big) \big|\leq  g'\big(\|\theta-\theta'\|\big) \varphi'_1(x,x)|.
\end{align*}

\item  There exists a $\delta\in(0,\infty)$ such that, with the convention that $x_{s-1}=x_s$ when $s=1$.  
\begin{align*}
  \int_{\setX^{T}}\hspace{-0.3cm} \exp\Big(  \delta\sum_{s=1}^{T}\varphi'_s(x_{s-1},x_{s})\Big)\tilde{\chi}_\theta(\dd x_1)\tilde{f}_{1,\theta}(\tilde{y}_{1}|x_{1})  \prod_{s=2}^{T}\tilde{f}_{s,\theta }(\tilde{y}_{s}|x_{s})\tilde{M}_{s,\theta}(x_{s-1},\dd x_{s})<\infty.
\end{align*}

 \end{enumerate}

\end{assumptionMLEB}

\clearpage

\makeatletter
\@removefromreset{theorem}{section}
\@removefromreset{lemma}{section}
\@removefromreset{figure}{section}
\@removefromreset{proposition}{section}
\@removefromreset{definition}{section}
\@removefromreset{corollary}{section}
\@removefromreset{remark}{section}
\counterwithout{equation}{section}
\makeatother

\renewcommand*{\thelemma}{S\arabic{lemma}} 
\renewcommand*{\thefigure}{S\arabic{figure}} 
\renewcommand*{\thetheorem}{S\arabic{theorem}}
\renewcommand*{\theproposition}{S\arabic{proposition}}
\renewcommand*{\thedefinition}{S\arabic{definition}}
\renewcommand*{\thecorollary}{S\arabic{corollary}}
\renewcommand*{\theremark}{S\arabic{remark}}
\newcommand{\snum}{S}
\renewcommand{\theequation}{\snum.\arabic{equation}}
 
\renewcommand{\thepage}{S\arabic{page}}

\renewcommand{\thesection}{S\the\numexpr\value{section}-3\relax}

\begin{center}

\Large{\hypertarget{SM}{Supplementary Material}  for ``Self-Organizing State-Space Models with\\
\vspace{0.1cm}

 Artificial Dynamics'' }
\vspace{0.7cm}

\large{Yuan Chen$^{(1)}$, Mathieu Gerber$^{(1)}$, Christophe Andrieu$^{(1)}$, Randal Douc$^{(2)}$}\\
\vspace{0.4cm}

\small{(1) School of Mathematics, University of Bristol, UK }\\
\vspace{0.1cm}

\small{(2) SAMOVAR, Telecom SudParis, Institut Polytechnique de Paris, France}

\vspace{0.2cm}

\end{center}

\tableofcontents

\addtocontents{toc}{\protect\setcounter{tocdepth}{2}}

\setcounter{lemma}{0}
\setcounter{theorem}{0}
\setcounter{proposition}{0}
\setcounter{definition}{0}
\setcounter{corollary}{0}
\setcounter{equation}{0}
\setcounter{page}{1}

\section{Set-up, notation and references to the main article}

Throughout this document, unless  stated otherwise, we consider the general set-up introduced in Section \ref{sub:gen_SetUp}  and  the notation   is as introduced in the main article. All references to objects (such as equations, sections, results or assumptions) not defined in this document refer to objects defined in the main article.

\section{Extensions}

\subsection{SO-SSMs where    \texorpdfstring{$X$}{Lg} moves before \texorpdfstring{$\theta$}{Lg}}\label{supp:X_first}

In this subsection we consider the set-up of Section \ref{sec:SSM}, and thus $\Theta\subset\R^d$ is assumed to be a regular compact set (see Definition \ref{def:sharp}) and, for some $\tau\in\mathbb{N}$, the observation process $(Y_t)_{t\geq 1}$ is assumed to be a $\tau$-stationary process and SSM \eqref{eq:SSM} is assumed to be a $\tau$-periodic SSM (see Definition \ref{def:periodic}).

\subsubsection{The model}
The self-organizing SSM \eqref{eq:SO-SSM_new2} moves $\theta_t$ before $X_t$, but we can also consider an SO-SSM where $X_t$ moves before $\theta_t$. In this case, the SO-SSM assumes the following generative model  for the observation process $(Y_t)_{t\geq 1}$:  
\begin{align}\label{eq:SO-SSM_new}
\theta_1\sim\pi_1(\dd\theta_1),\quad X_1\sim \chi_{\theta_1}(\dd x_1),\quad \begin{cases}
Y_t|(\theta_t,X_t)\sim f_{t,\theta_t}(y|X_t)\dd y,\\
X_{t+1}|(\theta_t, X_t)\sim M_{t+1,\theta_{t}}(X_{t},\dd x_{t+1}),\\
 \theta_{t+1}|\theta_{t}\sim K_{t+1}(\theta_{t},\dd \theta_{t+1}) 
\end{cases} \quad t\geq 1.
\end{align}

\subsubsection{Technical assumptions}

In order to obtain theoretical guarantees for this SO-SSM we need to replace Assumption \ref{assumeSSM:smooth} by the following Assumption \ref{assumeSSMB:smooth}:

\setcounter{assumptionSSMB}{3} 

\begin{assumptionSSMB}\label{assumeSSMB:smooth}
Let $q_{t,\theta}(x|x')=f_{t,\theta}(Y_t|x')m_{t+1,\theta}(x|x')$ for all $(t,\theta,x,x')\in\mathbb{N}\times\Theta\times\setX^2$. Then, we $\P$-a.s.~have  $ \int_{\setX}\big\{\sup_{(\theta,x')\in \Theta \times \setX} q_{\tau,\theta}(x|x')\big\}\lambda(\dd x)<\infty$. 
In addition, there exist a  $\delta\in(0,\infty)$, a continuous and strictly  increasing function $g':[0,\infty)\rightarrow [0,\infty)$   such that $g'(0)=0$  and a family  $\{\varphi'_s\}_{s=1}^\tau$ of   measurable functions defined on $\setY\times \setX^2$, for which the following conditions hold:
 \begin{enumerate}
 \item  For all  $(x,x')\in\setX^2$, all $(\theta,\theta')\in \Theta\times\Theta$ and all $s\in\{1,\dots,\tau\}$, we have
\begin{align*}
\big|\log q_{s,\theta}(x|x')   -\log q_{s,\theta'}(x|x') \big|\leq  g'\big(\|\theta-\theta'\|\big) \varphi'_s(Y_s,x',x).
\end{align*}
\item There exists  a measurable function $F:\setY^\tau\rightarrow [1,\infty)$ such that $\E[\log F(Z_1)]<\infty$ and such that, letting $Z_t=(Y_{(t-1)\tau+1},\dots,Y_{t\tau})$ for all $t\geq 1$, for all $\theta\in\Theta$ and $t\geq 2$  we have, $\P$-a.s.,
\begin{align*}
\int_{\setX^{\tau t+1}} \exp\Big(\delta \sum_{s=1}^{\tau t}\varphi'_{i_s}(Y_s,x_{s-1},x_s)\Big)\chi_\theta(\dd x_0) &\prod_{s=1}^{\tau t} q_{i_s,\theta}(x_s| x_{s-1})\lambda(\dd x_s)\leq \prod_{s=1}^t F(Z_s). 
\end{align*}
\item For all $\theta\in\Theta$ we have $\frac{1}{t}\log^+ W_{t,\theta}=\bigO_\P(1)$ and $\inf_{t\geq 2}\P(W_{t,\theta}<\infty)=1$  where, for all $t\geq 2$,
\begin{align*}
W_{t,\theta}=\int_{\setX^{\tau (t-1)+1}}\hspace{-0.6cm}  \exp\Big(\delta \hspace{-0.2cm}\sum_{s=\tau+1}^{\tau t}  \varphi'_{i_s}(Y_s,x_{s-1},x_s)\Big)\hspace{-0.1cm}\sup_{(\theta',x')\in \Theta \times \setX}\hspace{-0.3cm}q_{\tau,\theta'}(x_\tau| x')  \lambda(\dd x_\tau)\prod_{s=\tau+1}^{\tau t}   q_{i_s,\theta}(x_s| x_{s-1})\lambda(\dd x_s).
\end{align*}

\end{enumerate}

\end{assumptionSSMB}

The theoretical analysis of SO-SSM \eqref{eq:SO-SSM_new} is  then  conducted under Assumptions  \ref{assumeSSM:K_set}-\ref{assumeSSM:G}  and \ref{assumeSSMB:smooth}. Examples of SSMs for which these assumptions are satisfied are provided in Section \ref{sup:examples}.

\subsubsection{Theoretical guarantees}

We   have the following result for SO-SSM \eqref{eq:SO-SSM_new}:
\begin{theorem}\label{thm:SSMB}
Consider the set-up of Theorem \ref{thm:SSM} where  \ref{assumeSSMB:smooth} is assumed instead of \ref{assumeSSM:smooth} and  where  $\chi_\theta(\dd x)=\chi(\dd x)$ for all $\theta\in\Theta$. For all $t\geq 1$ let $p_{t,\Theta}(\dd \theta_t |Y_{1:t})$ and $p_{t,\setX}(\dd x_t |Y_{1:t})$ denote the filtering distribution of $\theta_t$ and the filtering distribution $X_t$ under  \eqref{eq:SO-SSM_new}. Then, 
\begin{align*}
\int_\Theta \theta_t\, p_{t,\Theta}(\dd \theta_t |Y_{1:t})\PP \theta_\star,\quad  \big\|p_{t,\setX}(\dd x_t|Y_{1:t})-\bar{p}_{t,\theta_\star}(\dd x_t|Y_{1:t})\big\|_{\mathrm{TV}}\PP 0\quad\text{(as $t\rightarrow\infty$)}.
\end{align*}
\end{theorem}
\begin{proof}
See Section \ref{app:proof_TSSM}.
\end{proof}

Following the discussion of Section \ref{sub:ssm_adapt}, we can easily define the counterpart of Algorithm \ref{algo:online_2} and of Theorem \ref{thm:SSM2}  when we consider  SO-SSM \eqref{eq:SO-SSM_new} instead of SO-SSM \eqref{eq:SO-SSM_new2}. The details are however omitted to save space.

\subsection{SO-SSMs with adaptive scale matrices\label{supp:adapt}}

All the definitions of Markov kernels $(K_t)_{t\geq 2}$ discussed in the main article are such that $K_t(\theta_{t-1},\dd \theta_t)$ is a (truncated) Gaussian or Student-t distribution with scale matrix $h_t^2\Sigma$ (for some $h_t>0$). In the numerical experiments of Section \ref{sec:num}  it is observed that taking for $\Sigma$ the identity matrix  works well in practice, but our theoretical results allow to use an adaptive strategy in order to learn a ``good'' scale matrix. More specifically, our theory allows to replace, in the definition of $K_t(\theta_{t-1},\dd \theta_t)$, the fixed scale matrix $\Sigma$ by a scale matrix   $\Sigma_t$  depending on the observations $\{Y_s\}_{s=1}^{t-1}$ and on all the random variables generated    by the particle filter algorithm up to the time instant  $t-1$ (included), provided that for some constant $c\in (0,1]$ it holds that $\P( \|\Sigma_t\|\in [c,1/c])=1$ for all $t\geq 1$. For instance, denoting by $\tilde{s}_{t,i}^{N}$ the variance of the $i$-th component of $\theta$ under $p^N_{t-1,\Theta}(\dd\theta_t|Y_{1:(t-1)})$, for some constant $\epsilon\in(0,1]$ we can let 
\begin{align*}
\Sigma_t=  \mathrm{diag}\big(s_{t,1}^N,\dots,s_{t,d}^N\big),\quad s_{t,i}^N=
\begin{cases}
\tilde{s}_{t,i}^{N}, &\tilde{s}_{t,i}^{N}\in[\epsilon,1/\epsilon]\\
\epsilon, &\tilde{s}_{t,i}^{N}< \epsilon\\
(1/\epsilon), &\tilde{s}_{t,i}^{N}>1/\epsilon
\end{cases} \quad i\in\{1,\dots,d\}.
\end{align*}
This approach may be useful in applications where   different components on $\theta_\star$ are on different scales, and the scales of each component of $\theta_\star$ is unknown/uncertain.

\subsection{SO-SSMs for models containing discrete parameters\label{supp:dis}}

In this subsection we consider the case where SSM \eqref{eq:SSM} contains both continuous and discrete parameters. More specifically, we consider the case where $\Theta=\Theta_1\times \Theta_2$ with $\Theta_1\subset\R^{d_1}$ and $ \Theta_2\subset\mathbb{Z}^{d_2}$ for some $(d_1,d_2)\in\mathbb{N}^2$, and assuming that $\Theta_1$ is a regular compact set (see Definition \ref{def:sharp}) and  that $\Theta_2$ is a finite set.   In this subsection we write $\theta\in\Theta$ as $\theta=(\vartheta,\psi)$ with $\vartheta\in\Theta_1$ a continuous parameter and $\psi\in\Theta_2$ a discrete parameter.  If SSM \eqref{eq:SSM} contains only continuous parameters, respectively only discrete parameters, then we can let $\pi_1$ and $(K_t)_{t\geq 2}$ be as defined below, omitting their $\psi$, respectively $\vartheta$, component.

\subsubsection{Definition of the SO-SSM\label{sub:def_disc}}
For   SSMs with both continuous and discrete parameters we let $\pi_1(\dd\theta_1)=(\pi_{1,1}\otimes \pi_{2,1})(\dd (\vartheta_1,\psi_1))$ be the initial distribution of the Markov chain $(\theta_t)_{t\geq 1}$, with $\pi_{1,1}(\dd \vartheta_1)$   a probability distribution on $\Theta_1$  having a   strictly positive density w.r.t.~the Lebesgue measure on $\R^{d_1}$ and with $\pi_{2,1}(\dd\psi_1)$     a probability  distribution on $\Theta_2$ such that $\pi_{2,1}(\{\psi\})>0$ for all $\psi\in\Theta_2$ (e.g.~$\pi_1(\dd\theta_1)$ is the uniform distribution on $\Theta$). 

Next, to introduce    Markov kernels $(K_t)_{t\geq 2}$ suitable for  such SSMs, we let $a$ and $b$ be two integers such that $\Theta_2\subseteq\{a,a+1,\dots,b\}^{d_2}$  and we let $(\alpha_1,\alpha_2,\alpha_3, c)\in(0,\infty)^4$, $\beta\in(0,1/2)$  and $(t_1,\Delta)\in\mathbb{N}^2$ be some constants. In addition, for all $t\geq 2$ we let $\beta_t=  \log (t)^{-\beta}$, for all $p\geq 2$ we let  $t_p=t_{p-1}+\Delta\lceil  (\log t_{p-1})^{1-\beta/2}\rceil$, and for all $t\geq 1$ we let $p_t=c\, t^{-\alpha_1}$ if $t\not\in(t_p)_{p\geq 1}$ and $p_t=c  t^{-\alpha_2\beta_t}$ otherwise. Then, for all $t\geq 2$ and $\psi'=(\psi'_1,\dots,\psi'_{d_2})\in\mathbb{Z}^{d_2}$, we let $\tilde{K}_t(\psi',\dd \psi)$ denote the probability distribution on $\mathbb{Z}^{d_2}$ such that $(\psi_1,\dots,\psi_{d_2}) \sim \tilde{K}_{t}(\psi',\dd \psi)$ if and only if
\begin{align*}
\psi_i=\psi_i'+(2I_{i,t}-1)B_{i,t},\quad B_{i,t}\sim \mathrm{Binomial}(b-a, p_t),\quad I_{i,t}\sim\mathrm{Bernoulli}(0.5),\quad\forall i\in\{1,\dots,d_2\}
\end{align*}
where $\{ (I_{i,t},B_{i,t})\}_{i=1}^{d_2}$ is a set of $2d_2$ independent random variables. Remark that   sampling from the distribution $\tilde{K}_{t}(\psi',\dd \psi)$ amounts to, independently for all $i\in\{1,\dots,d_2\}$, adding  or subtracting  the Binomial random variable $B_{i,t}$ to the $i$-th component of $\psi'$  with probability  1/2. Finally, for all $t\geq 2$ and $\psi'\in\Theta_2$ we let $\tilde{K}_{t,\Theta_2}(\psi',\dd\psi)$ denote the restriction of $\tilde{K}_{t}(\psi',\dd\psi)$ to the set $\Theta_2$,  and we let $\Sigma$ be a $d_1\times d_1$ symmetric and positive definite matrix and $\nu\in(0,\infty)$.

With this notation in place, we let $(K_t)_{t\geq 2}$  be such that, for all $\theta=(\vartheta,\psi)\in\Theta$,
\begin{equation}\label{eq:SSM3}
\begin{aligned}
\theta_t\sim K_t(\theta ,\dd \theta_t)\Leftrightarrow \psi_t \sim \tilde{K}_{t,\Theta_2}(\psi, \dd \psi_t) \text{ and }\vartheta_t \sim
\begin{cases}
 \mathrm{TS}_{\Theta_1}\big(\vartheta, t^{-2\alpha_3\beta_t} \Sigma, \nu\big), & t \in (t_p)_{p \geq 1} \\
  \mathrm{TN}_{\Theta_1}\big(\vartheta, t^{-2\alpha_3} \Sigma\big), & t \not\in (t_p)_{p \geq 1}
  \end{cases} \quad \ \forall t \geq 2.
\end{aligned}
\end{equation}

\subsubsection{Some comments}\label{sub:explain}


In the definition \eqref{eq:SSM3} of $(K_t)_{t\geq 2}$  the sequence $(t_p)_{p\geq 1}$ is depends on $\beta$ and, in this sense, is   determined by the discrete parameter $\psi$. This is the case for  the following reason. Informally speaking, it is  harder to move in the finite space $\Theta_2$, where each move must be at least of length one,   than in the continuous space $\Theta_1$, where the size of the moves can be  arbitrarily small. Therefore, to prevent the distribution $p_{t,\Theta}(\dd \theta_t|Y_{1:t})$ from concentrating on the wrong parameter value,  the presence of a discrete parameter   often necessitates using a Markov kernel $K_t$  under which $\|\mathrm{Var}(\theta_{t}|\theta_{t-1})\|$ is not too small, that is to use a sequence $(t_p)_{p\geq 1}$ such that $(t_{p+1}-t_p)\rightarrow\infty$ slowly. In particular, remark that the speed at which $(t_{p+1}-t_p)\rightarrow\infty$ is  slower  in \eqref{eq:SSM3} than for the   sequence $(t_p)_{p\geq 1}$ used in \eqref{eq:SSM2} to define $(K_t)_{t\geq 2}$ when the model contains only continuous parameters.

Focusing now on  the continuous parameter, we observe that when $t\in   (t_p)_{p\geq 1}$ not only a distribution with heavier tails is used, as in Section \ref{sub:ssm_cont_slow}, but also that another, larger,  scale  factor is applied to the matrix $\Sigma$. In particular, in \eqref{eq:SSM3} the scale factor $t^{-2\alpha_3\beta_t}$ applied to $\Sigma$  when   $t\in   (t_p)_{p\geq 1}$ converges to zero at a slower rate than that of the scale factor $t^{-2\alpha_3}$ used when $t\not\in (t_p)_{p\geq 1}$. The scale factors  applied to $\Sigma$  are defined in that way and not as in Section \ref{sub:ssm_cont_slow}   because the  dynamics  imposed on $\theta$ when $t\in(t_p)_{p\geq 1}$ depends, to some extend, on the  sequence $(t_p)_{p\geq 1}$. Indeed,  following an idea introduced in  \citet{gerber2022global}, the   definition \eqref{eq:SSM3}    of $K_t$ for $t\in (t_p)_{p\geq 1}$ is used to ensure that $p_{t_{p},\Theta}(\dd\theta_{t_p}|Y_{1:t_{p}})$ has enough mass around $\theta_\star$ while that of $K_t$ for $t\not\in (t_p)_{p\geq 1}$ allows to guarantee that the distribution $p_{t,\Theta}(\dd\theta_{t}|Y_{1:t})$  concentrates on $\theta_\star$ between time $t=t_p+1$ and time $t=t_{p+1}-1$. Then, letting $U_\star$ be a small neighbourhood of $\theta_\star$,  for this mechanism to ensure that  $p_{t,\Theta}(\dd\theta_{t}|Y_{1:t})$   concentrates on $\theta_\star$ as $t\rightarrow\infty$ it is needed that $p_{t_{p+1}-1,\Theta}(U_\star|Y_{1:(t_{p+1}-1)})$ is sufficiently large. Informally speaking, this latter quantity is increasing with $(t_{p+1}-t_p)$ (i.e.~the time  the distribution has to concentrate on $\theta_\star$ between time $t=t_p+1$ and time $t=t_{p+1}-1$) and with $p_{t_{p},\Theta}(U_\star|Y_{1:t_{p}})$ (i.e.~the mass of the `initial' distribution   on $U_\star$). For this reason,  as explained above, since the presence of a discrete parameter  imposes to use in \eqref{eq:SSM3} a sequence $(t_p)_{p\geq 1}$  such that  $(t_{p+1}-t_p)\rightarrow\infty$ slowly, we need to define $(K_{t_p})_{p\geq 1}$  in such a way   that $\|\mathrm{Var}(\theta_{t_p}|\theta_{t_p-1})\|$  converges to zero slowly to  guarantee that $p_{t_{p},\Theta}(U_\star|Y_{1:t_{p}})$ is large enough for   $p_{t,\Theta}(\dd\theta_{t}|Y_{1:t})$  to  concentrate  on $\theta_\star$ as $t\rightarrow\infty$.

\subsubsection{Practical recommendations\label{sub:U_recom}}

As argued and illustrated in the main article, SO-SSMs with adaptive artificial dynamics should be preferred to those that impose  a dynamics  on $\theta$ at every time instant $t\geq 2$. On a model having only discrete parameters we find that, in the above definition of $(K_t)_{t\geq 2}$,  letting $c=1$, $\beta=0.01$ and $\alpha_1=\alpha_2=0.5$ works well in practice (see Section \ref{supp:num} below).  If the SSM also contains continuous parameters,  following the discussion in Section \ref{sub:pratical} we propose as default approach to let $\alpha_3=0.5$,  $\nu=t_1=100$, $\Delta=1$, and $\Sigma$ be the identity matrix. Remark that for the proposed small value of $\beta$, we have $t^{-2\alpha_3\beta_t}\approx t^{-2\alpha_3}$ for any realistic values of $t$ so that, in  \eqref{eq:SSM3}, essentially the same scaling factor is applied to $\Sigma$ at any realistic  values of $t$.

\subsubsection{Theoretical guarantees}

We  have the following result for SO-SSM  \eqref{eq:SO-SSM_new2} when SSM \eqref{eq:SSM} contains   discrete parameters.
\begin{theorem}\label{thm:dis}
Consider the set-up of Theorem \ref{thm:SSM} with $\Theta$, $\pi_1$ and $(K_t)_{t\geq 2}$ as in Section \ref{sub:def_disc}.  Then,  
\begin{align*}
\int_\Theta \theta_t\, p_{t,\Theta}(\dd \theta_t |Y_{1:t})\PP \theta_\star,\quad  \big\|p_{t,\setX}(\dd x_t|Y_{1:t})-\bar{p}_{t,\theta_\star}(\dd x_t|Y_{1:t})\big\|_{\mathrm{TV}}\PP 0\quad\text{(as $t\rightarrow\infty$)}.
\end{align*}
\end{theorem}
\begin{proof}
See Section \ref{app:proof_TSSM}.
\end{proof}

Algorithm \ref{algo:online_2} and   Theorem \ref{thm:SSM2} can then be easily adapted to the setting of this subsection, and we can easily obtain a counterpart of Theorem \ref{thm:dis} when we consider SO-SSM  \eqref{eq:SO-SSM_new} instead of  SO-SSM  \eqref{eq:SO-SSM_new2} (as well as a version of Algorithm \ref{algo:online_2} and  of Theorem \ref{thm:SSM2} adapted to the setting of this subsection). The details are however omitted to save space.

\subsubsection{Application: Maximum likelihood estimation in the Bernoulli-Laplace urn model\label{supp:num}}

We consider the following Bernoulli-Laplace urn model. There are two urns,  Urn 1   and Urn 2, containing  $j_\star\in\mathbb{N}$ and $k_\star\in\mathbb{N}$ balls  respectively, with a total of $r_\star<j_\star+k_\star$ red balls. Initially, all the red balls are placed in Urn 1, and at each time $s\geq 2$   one ball is chosen uniformly at random from each urn  and swapped into the other urn. Denoting by  $W_s$ the number of red balls in Urn 2 at time $s\geq 1$, below we aim at estimating $\theta_\star:=(j_\star, k_\star, r_\star)$ in an offline fashion when only a realization of $(W_s)_{s\geq 1}$ is observed.

For all  $\theta\in\{(j,k,r)\in\mathbb{N}^3\text{ s.t. }j+k\leq r\}$   we let $S_\theta=\{\max\{0,r-j\},\dots,\min\{k,r\}\}$ and
\begin{align*}
p_\theta(w_2|w_1)=
\begin{cases}
\frac{(j-r+w_1)w_1}{jk}, &\text{if $w_2=w_1-1$ and $w_1\in S_\theta$}\\
 \frac{(r-w_1)w_1+(j-r+w_1)(k-w_1)}{jk}, &\text{if $w_2=w_1$ and $w_1\in S_\theta$}\\
 \frac{(r-w_1)(k-w_1)}{jk}, &\text{if $w_2=w_1+1$ and $w_1\in S_\theta$}\\
 0, &\mathrm{otherwise}
\end{cases},\quad\forall (w_1,w_2)\in\mathbb{Z}^2.
\end{align*}
With this notation in place, it can be shown that  the Markov chain $(W_s)_{s\geq 1}$ takes its values in  the set $S_{\theta_\star}$ and is such that we have $\P(W_{s+1}=w_2|W_s=w_1)= p_{\theta_\star}(w_2|w_1)$ for all $(w_1,w_2)\in S_{\theta_\star}^2$ and all $t\geq 1$. Then, for some $T\geq 1$   we simulate a realization  $\{w_s\}_{s=1}^{T+1}$ of $\{W_t\}_{t=1}^{T+1}$   and consider the problem of computing  the estimate $\tilde{\theta}_T:=\argmax_{\theta\in\Theta} \sum_{s=1}^{T} \log p_\theta(w_{s+1}|w_s)$ of $\theta_\star$ using an IF algorithm. 

To do so we need to define the observations $\{\tilde{y}_s\}_{s=1}^T$ and SSM \eqref{eq:SSM_MLE} in such a way that the likelihood function $\tilde{L}_T(\cdot)$  defined in \eqref{eq:lik_IF}  satisfies  $\tilde{\theta}_T=\argmax_{\theta\in\Theta}\tilde{L}_T(\theta)$. It is straightforward to see that this is achieved by letting, for all $s\in\{1,\dots,T\}$,  $\tilde{y}_s=(w_{s}, w_{s+1})$ and $\tilde{f}_{s,\theta}(y|x)=  p_\theta(y_2|y_1)\varphi(y_1)$ for all $(x,y)\in\setX\times \mathbb{Z}^2$, where   $\varphi:\mathbb{Z}:\rightarrow[0,\infty)$ is    some arbitrary function such that $\sum_{y_1\in\mathbb{Z}}\varphi(y_1)=1$.    Remark that the resulting SSM \eqref{eq:SSM_MLE} is peculiar, in the sense that it does not assume that the distribution of $\{\tilde{Y}_s\}_{s=1}^T$ depends on some unobservable  random variables $\{\tilde{X}_s\}_{s=1}^T$. We also stress that $\varphi(\cdot)$ is only used to  define SSM \eqref{eq:SSM_MLE} and  that the IF algorithms discussed in this work  for maximizing the  corresponding function $\tilde{L}_T(\cdot)$ do not depend on this function. 

Following the discussion in Section \ref{sub:U_recom}, we compute the MLE $\tilde{\theta}_T$ using a specific version of  the IF  Algorithm \ref{algo:IF_2} obtained by defining $(K_t)_{t\geq 2}$ as specified in Section \ref{sub:def_disc}, with  $t_1=100$, $\Delta=c=1$, $\beta=0.01$ and $\alpha_1=\alpha_2=\alpha$. Different values of $\alpha$ will also be considered in the experiments. All the results presented below are obtained from 50 runs of the  algorithm  using $N=10^3$ particles.  In addition, letting $v_T=\max\{w_s,\,s=1,\dots,T+1\}$ and noting that we have both  $r_\star\geq v_T$ and $k_\star\geq v_T$,  we let 
\begin{align*}
\Theta\subseteq A_T:=\big\{ (j,k,r)\in\{1,\dots,200\}^3:\,\,  j+k\geq r,\,\,r\geq v_T,\,k\geq v_T\big\}.
\end{align*}
We remark that if $\Theta$ is the largest subset of $A_T$ such that  $p_\theta(w_{s+1}|w_s)>0$ for all $\theta\in\Theta$ and all $s\in\{1,\dots,T\}$ then Assumption \ref{assumeSSMB:smooth_MLE} holds, but to simply the definition of the parameter space  we simply let $\Theta=A_T$ in what follows. Finally,  we  let $\pi_1(\dd\theta_1)$ be the uniform distribution on $\Theta$ and $\theta_\star=(c_\star,c_\star,c_\star)$  where two different values of $c_\star\in\mathbb{N}$ are considered. Recall that as $c_\star$ increases the mixing time of the Markov chain $(W_s)_{s\geq 1}$  becomes longer \citepsup[see][and references therein]{nestoridi2019shuffling}, making the estimation of $\theta_\star$ more challenging.

\begin{figure}[!t]
\centering
\includegraphics[scale=0.4]{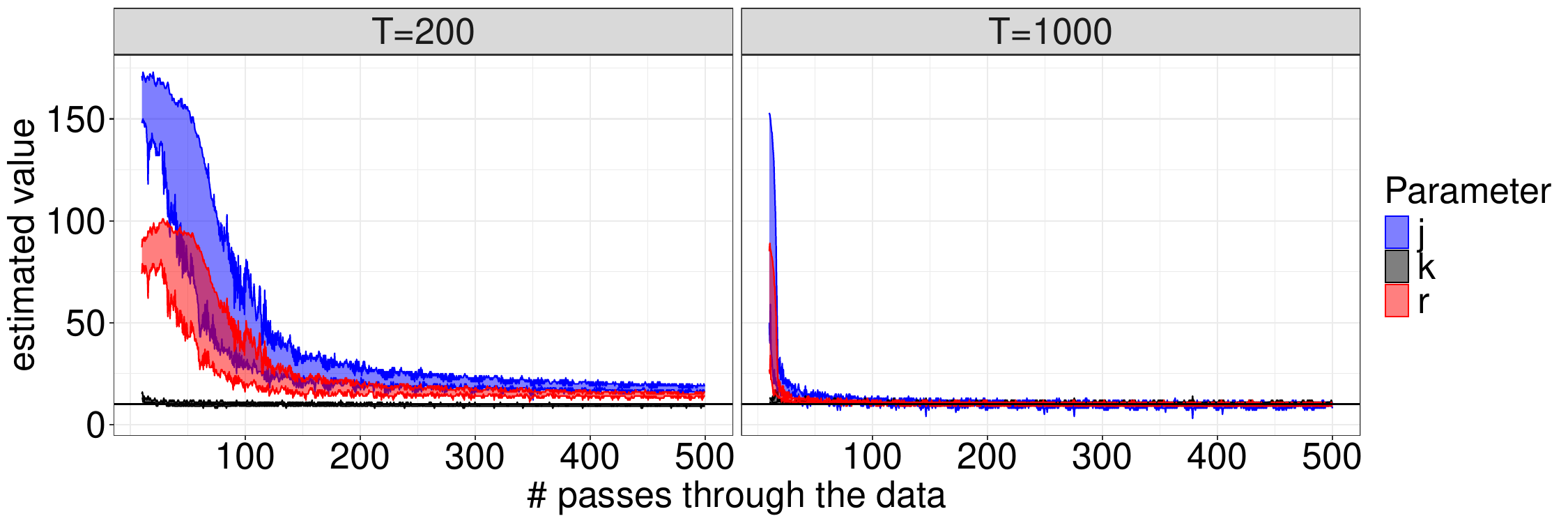}
\caption{Results for the experiments in Section \ref{p-sub:noise}. The plots show the estimated value of the MLE as a function of the number of passes through the data. The results are for $c_\star=10$ and $\alpha=0.5$, the horizontal lines show the value of $c_\star$, and the shaded areas represent the range of estimated values of the model parameters   obtained in 50 independent runs of the IF algorithm. \label{fig:urn}}
\end{figure}

\begin{figure}[!t]
\centering
\includegraphics[scale=0.4]{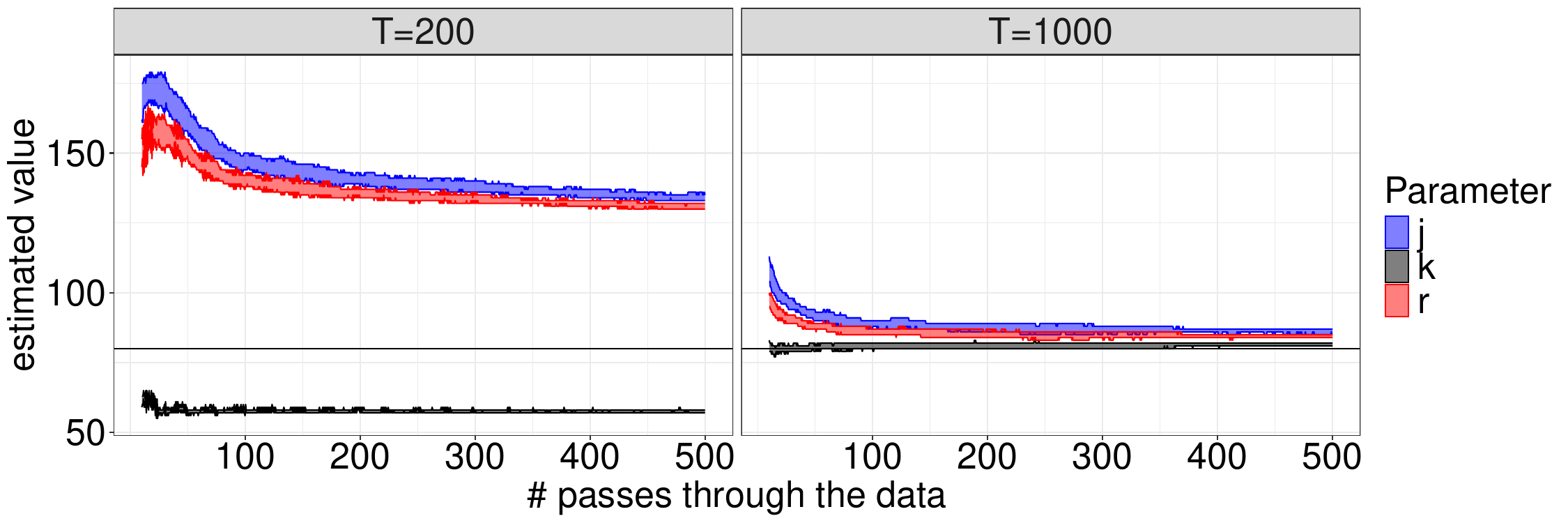}
\caption{Results for the experiments in Section \ref{p-sub:noise}. The plots show the estimated value of the MLE as a function of the number of passes through the data. The results are for $c_\star=80$ and $\alpha=0.5$, the horizontal lines show the value of $c_\star$, and the shaded areas represent the range of estimated values of the model parameters   obtained in 50  independent runs of the IF algorithm. \label{fig:urn2}}
\end{figure}

\begin{figure}[!t]
\centering
\includegraphics[scale=0.4]{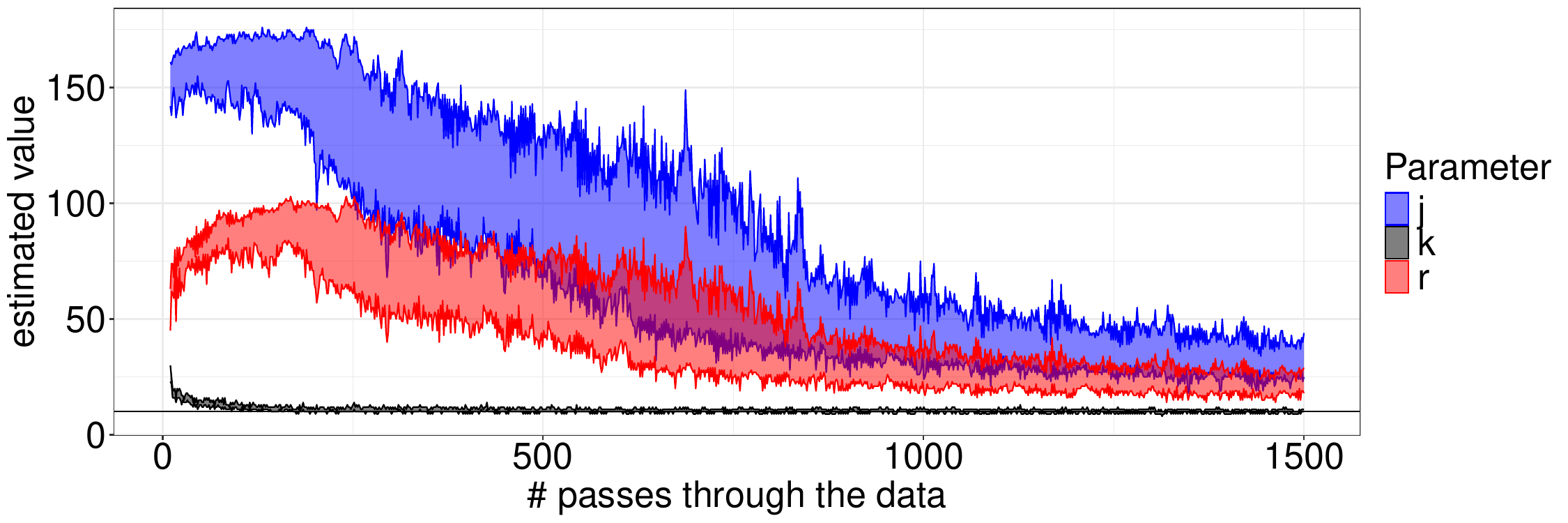}
\caption{Results for the experiments in Section \ref{p-sub:noise}. The plots show the estimated value of the MLE as a function of the number of passes through the data. The results are for $T=200$, $c_\star=10$ and $\alpha=0.4$, the horizontal lines show the value of $c_\star$, and the shaded areas represent the range of estimated values of the model parameters   obtained in 50 independent runs of the IF algorithm.  \label{fig:urn3}}
\end{figure}

Figures \ref{fig:urn}-\ref{fig:urn2} show,  both for $T=200$ and for $T=1\,000$, the estimated values of the three model parameters as a function of the number of passes through the data when $\alpha=0.5$, with $c_\star=10$ for  Figure \ref{fig:urn}  and  $c_\star=80$ for Figure \ref{fig:urn2}. From these two figures, several comments are in order. Firstly, we observe that  the convergence behaviour of the IF algorithm is stable across its 50 runs.  Secondly, for a given sample size $T$ the difference between the estimated value of the MLE and $\theta_\star$ is larger when $c_\star=80$ than when $c_\star=10$. Remark that this result was expected, since estimating $\theta_\star$ is statistically harder for the former value of $c_\star$  than for the latter. Thirdly, it appears in Figures \ref{fig:urn}-\ref{fig:urn2} that the number of iterations needed for the IF algorithm to converge is larger when $c_\star=80$ than when $c_\star=10$, a phenomenon which was expected for the following reason. When $c_\star$ is small, the Markov chain $(W_s)_{s\geq 1}$ converges quickly to its stationary distribution (recall that $\P(W_1=0)=1$) meaning that, informally speaking,  the observations  $\{w_s\}_{s=1}^{T+1}$ is approximately a sample from a stationary process. In this case, based on the results of Section \ref{sec:SSM} for online parameter inference, we expect  online learning to occur within each pass of the data, and thus the estimate of the model parameter to improve not only from one pass to the next  but also within a single pass. This phenomenon is particularly clear when $T=1\,000$, for instance, the IF algorithm converges in less than 50 passes through the data for $c_\star = 10$ (as shown in Figure \ref{fig:urn}); however, when $c_\star=80$ more iterations are needed for the IF  algorithm to converge (as shown in Figure \ref{fig:urn2}) because the underlying  Markov chain $(W_s)_{s\geq 1}$ converges slowly to its stationary distribution. 


Finally, Figure \ref{fig:urn3}  shows for $T=200$, the estimated values of the three model parameters as a function of the number of passes through the data when $\alpha=0.4$ and $c_\star=10$. Comparing the results shown in this figure with those in Figure \ref{fig:urn} (obtained with $\alpha=0.5$),  we observe that, as expected, the IF algorithm requires more iterations to converge when we reduce the speed at which the dynamics on $\theta$ decreases (i.e., when using a smaller $\alpha$).

\section{Examples of SSMs satisfying  Assumptions \ref{assumeSSM:K_set}-\ref{assumeSSM:smooth} and  \ref{assumeSSMB:smooth}\label{sup:examples}}
 
\subsection{Linear Gaussian models\label{app:sup:LG}}

As shown in the following proposition, Assumptions \ref{assumeSSM:K_set}-\ref{assumeSSM:smooth} and \ref{assumeSSMB:smooth} are satisfied for  a large class of multivariate linear Gaussian SSMs.

\begin{proposition}\label{prop:Assume_LG}
Let $(\tau,d,d_x,d_y) \in \mathbb{N}^4$  and $\Theta\subset\R^d$ be a (non-empty) compact set. In addition, for all $t\in\{1,\dots,\tau\}$ let $m_t:\Theta\rightarrow\R^{d_y}$, $A_t:\Theta\rightarrow\R^{d_y\times d_x}$,  $B_t:\Theta\rightarrow\R^{d_y\times d_y}$, $C_t:\Theta\rightarrow\R^{d_x\times d_x}$ and $D_t:\Theta\rightarrow\R^{d_x\times d_x}$ be four    continuously differentiable mappings such that, for all $\theta\in\Theta$, the matrices $B_t(\theta)$ and $D_t(\theta)$ are symmetric and positive definite and such that the matrix  $A_t(\theta)$ is full rank column.  Let \eqref{eq:SSM} be a $\tau$-periodic SSM such that, for all $(\theta,x)\in\Theta\times\R^{d_x}$,
 \begin{align*}
&f_{t,\theta}(y|x)\dd y=\mathcal{N}_{d_y}\big(m_t(\theta)+A_t(\theta)x, B_t(\theta)\big),\quad M_{t+1,\theta}\big(x,\dd x_{t+1}\big)=\mathcal{N}_{d_x}\big(C_{t}(\theta) x, D_{t}(\theta)\big),\,\,\,\forall t\in\{1,\dots,\tau\}
\end{align*}
and assume that $(Y_t)_{t\geq 1}$ is a $\tau$-stationary process such that $\E[\|Y_t\|^2]<\infty$ for all $t\in\{1,\dots,\tau\}$. Then, Assumptions  \ref{assumeSSM:K_set}-\ref{assumeSSM:G} and \ref{assumeSSMB:smooth} hold. In addition, Assumption \ref{assumeSSM:smooth} also holds if  $\chi_\theta(\dd x_1)=\mathcal{N}_{d_x}(\mu(\theta),\Sigma(\theta))$ for all $\theta\in\Theta$ and  some continuously differentiable mappings $\mu:\Theta\rightarrow\R^{d_x}$ and $\Sigma:\Theta\rightarrow\R^{d_x\times d_x}$ such that $\Sigma(\theta)$ is a symmetric and positive  definite matrix for all $\theta\in\Theta$.
\end{proposition}

\subsection{Stochastic volatility models}

As illustrated by the next proposition, the validity of Assumptions \ref{assumeSSM:K_set}-\ref{assumeSSM:smooth} and \ref{assumeSSMB:smooth} is of course  not limited to linear Gaussian SSMs.
\begin{proposition}\label{prop:Online_SV_model}
Let $\Theta\subset\R^3$ be a compact set such that for all $\theta=(\alpha,\beta,\sigma)\in\Theta$ we have both $\beta>0$ and $\sigma>0$. Let   \eqref{eq:SSM} be a $1$-periodic SSM such that, for all $(\theta,x)\in\Theta\times\R$ and some $\nu\in(0,\infty)$,
\begin{align*}
f_{1,\theta}(y|x)\dd y=\mathcal{N}_1\big(0,\beta^2\exp(x)\big),\quad M_{2,\theta}(x,\dd x_2)=t_{1,\nu}\big(\alpha x,\sigma^2\big)
\end{align*}
and assume that $(Y_t)_{t\geq 1}$ is a $1$-stationary process such that $\E\big[\big|\log(|Y_1|)\big|\big]<\infty$.
Then, Assumptions  \ref{assumeSSM:K_set}-\ref{assumeSSM:G} and \ref{assumeSSMB:smooth} hold. In addition, Assumption \ref{assumeSSM:smooth} also holds if  $\int_\R e^{c|x|}\chi_\theta(\dd x)<\infty$ for all $\theta\in\Theta$ and for some constant $c\in(0,\infty)$.
\end{proposition}

\begin{remark}
If, in Proposition  \ref{prop:Online_SV_model},  the $t_{1,\nu}(\alpha x,\sigma^2)$ distribution is replaced by the $\mathcal{N}_1(\alpha x,\sigma^2)$ distribution then the resulting SSM is a standard stochastic volatility model considered e.g.~in \citetsup{kim1998stochastic}. However, for technical reasons, for this model we did not manage to establish that Assumptions \ref{assumeSSM:smooth} and \ref{assumeSSMB:smooth} hold.
\end{remark}

\subsection{Proofs}

Throughout this subsection $\lambda(\dd x)$ denotes the Lebesgue measure on $\R^{p_x}$ and   $\dd y$ denote the Lebesgue measure on $\R^{p_y}$, where the value  $(p_x,p_y)\in\mathbb{N}^2$ will be clear from the context. 

\subsubsection{Proof of Proposition \ref{prop:Assume_LG}}

\paragraph{Preliminaries}

Let $p\in\mathbb{N}$. Then, in what follows for any $x\in\R^p$ we denote by $\|x\|$ the Euclidean norm while for any $p\times p$ matrix $M$ the notation $\|M\|$ is used for the spectral norm of $M$. Finally, for $\mu\in\R^p$ and a $p\times p$ symmetric and positive definite matrix $\Sigma$ we use the notation $\phi_p(\cdot; \mu,\Sigma)$ for the density function of the $\mathcal{N}_p(\mu,\Sigma)$ distribution (w.r.t.~$\lambda(\dd x)$).

With this notation in place the following two lemmas will be used to prove the propositions.

\begin{lemma}\label{lemma::lipchitz}

Let $d\in\mathbb{N}$, $\Theta\subset\R^d$ be a non-empty compact set and   $(d_z,d_x)\in\mathbb{N}^2$. Next, let $Q:\Theta\rightarrow \R^{d_z\times d_z}$ and $F:\Theta\rightarrow\R^{d_z\times d_x}$ be two continuously differentiable functions such that, for all $\theta\in\Theta$, the matrix $Q(\theta)$ is symmetric and positive definite. Then, there exists a constant $\kappa\in (0,\infty)$ such that, for all $(z,x)\in \mathbb{R}^{d_z}\times\mathbb{R}^{d_x}$ and all $(\theta,\theta')\in\Theta^2$, we have
\begin{align*}
\big|\log \phi_{d_z}\big(z; F(\theta) x,Q(\theta)\big) - \log \phi_{d_z}\big(z; F(\theta') x,Q(\theta')\big)\big| \leq \kappa\,\|\theta - \theta'\| \big(1+\|z\|^2 + \|x\|^2\big).
\end{align*}

\begin{proof}

Let $(\theta,z,x)\in\Theta\times\R^{d_z}\times\R^{d_x}$ and $i\in\{1,\dots,d\}$. Then,
\begin{equation}\label{eq:LG_assume1}
\begin{split}
\frac{\partial}{\partial \theta_i} \log \phi_{d_z}\big(z; F(\theta) x,Q(\theta)\big)& = - \frac{1}{2} \frac{\partial}{\partial \theta_i}  \log \mathrm{det}\big(Q(\theta)\big) \\
&- \frac{1}{2} \frac{\partial}{\partial \theta_i} \bigg\{ \big(z-F(\theta)x\big)^\top Q(\theta)^{-1}\big(z-F(\theta)x\big)\bigg\}
\end{split}
\end{equation}
where 
\begin{equation}\label{eq:LG_assume2}
\begin{split}
\Big|\frac{\partial}{\partial \theta_i} \log \mathrm{det}\big(Q(\theta)\big)\Big|=\Big| \mathrm{tr}\Big(Q(\theta)^{-1} \frac{\partial }{\partial \theta_i}Q(\theta) \Big)\Big|&\leq  \mathrm{tr}\Big(Q(\theta)^{-1}\Big)\Big|\mathrm{tr}\Big(\frac{\partial}{\partial \theta_i}Q(\theta)\Big)\Big|
\end{split}    
\end{equation}
and where
\begin{equation}\label{eq:LG_assume3}
\begin{split}
\frac{\partial}{\partial \theta_i} &\big(z-F(\theta)x\big)^\top  Q(\theta)^{-1}\big(z-F(\theta)x\big)\\
&= -2\big(z-F(\theta)x\big)^\top Q(\theta)^{-1} \frac{\partial F(\theta)}{\partial \theta_i}x -  \big(z-F(\theta)x\big)^\top Q(\theta)^{-1} \frac{\partial Q(\theta)}{\partial \theta_i} Q(\theta)^{-1} \big(z-F(\theta)x\big). 
\end{split}    
\end{equation}
To proceed further note that
\begin{equation}\label{eq:LG_assume4}
\begin{split}
\Big| \big(z-F(\theta)x\big)^\top Q(\theta)^{-1} \frac{\partial F(\theta)}{\partial \theta_i}x\Big|&\leq \|Q(\theta)^{-1}\| \Big\| \frac{\partial F(\theta)}{\partial \theta_i}\Big\|\Big(\|z\| \|x\|+\|F(\theta)\|\|x\|^2 \Big)\\
&\leq \|Q(\theta)^{-1}\| \Big\| \frac{\partial F(\theta)}{\partial \theta_i}\Big\|\Big(\|z\|^2+ \|x\|^2+\|F(\theta)\|\|x\|^2 \Big)
\end{split}    
\end{equation}
while
\begin{equation}\label{eq:LG_assume5}
\begin{split}
\Big|\big(z -F(\theta)x\big)^\top  Q(\theta)^{-1} \frac{\partial Q(\theta)}{\partial \theta_i} &Q(\theta)^{-1} \big(z-F(\theta)x\big)\Big|\\
&\leq \big\|Q(\theta)^{-1}\|^2\Big\|\frac{\partial Q(\theta)}{\partial \theta_i}\Big\|\,\|z-F(\theta)x\|^2\\
&\leq 2\big\|Q(\theta)^{-1}\|^2\Big\|\frac{\partial Q(\theta)}{\partial \theta_i}\Big\|\big(\|z\|^2+  \|F(\theta)\|^2 \|x\|^2\big).
\end{split}    
\end{equation}
Since $\Theta$ is compact and the mappings $Q(\cdot)$ and $F(\cdot)$ are continuous,  we have
\begin{align*}
\sup_{\theta\in\Theta}\|Q(\theta)^{-1}\|<\infty,\quad   \sup_{\theta\in\Theta}\mathrm{tr}\Big(Q(\theta)^{-1}\Big)<\infty,\quad\sup_{\theta\in\Theta}\Big|\mathrm{tr}\Big(\frac{\partial}{\partial \theta_i}Q(\theta)\Big)\Big|<\infty,\quad \sup_{\theta\in\Theta}\frac{\partial F(\theta)}{\partial \theta_i}<\infty
\end{align*}
which, together with \eqref{eq:LG_assume1}-\eqref{eq:LG_assume5}, implies that there exists a constant $C\in(0,\infty)$ such that 
\begin{align*}
\Big| \frac{\partial}{\partial \theta_i} \log \phi_{d_z}\big(z; F(\theta) x,Q(\theta)\big)\Big|\leq C\big(1+\|z\|^2+\|x\|^2\big),\quad\forall (\theta,z,x)\in\Theta\times\R^{d_z}\times\R^{d_x}.
\end{align*}
The result of the lemma follows.
\end{proof}
\end{lemma}

\begin{lemma}\label{lemma:tech_Gaussian}
Let $d_x\in\mathbb{N}$, $\mu\in\R^{d_x}$ and $V$ be a $d_x\times d_x$ symmetric and positive definite matrix. Then,  
\begin{align*}
\int_{\R^{d_x}} e^{ c \|x\|^2}\phi_{d_x}(x;\mu,V)\lambda(\dd x)\leq  \frac{(2\|V\|)^{d_x/2}}{\mathrm{det}(V)^{1/2}}\exp\big( 2c\|\mu\|^2\big),\quad\forall c\in \Big(0,\frac{1}{4\|V\|}\Big].
\end{align*}
\end{lemma}
\begin{proof}

Let $c$ be as in the statement of the lemma and note that $\|V^{-1}\|=\|V\|^{-1}$. Then,
\begin{equation*}
\begin{split}
\int_{\R^{d_x}} e^{c\|x\|^2}&\phi_{d_x}(x;\mu,V) \lambda(\dd x)\\
&=\big(2\pi\, \mathrm{det}(V)^{\frac{1}{d_x}}\big)^{-\frac{d_x}{2}}\int_{\R^{d_x}} \exp\Big(c\|x\|^2 - \frac{1}{2} (x-\mu)^\top V^{-1}(x-\mu) \Big) \lambda(\dd x) \\
&\leq \big(2\pi \mathrm{det}(V)^{\frac{1}{d_x}}\big)^{-\frac{d_x}{2}}
\exp\bigg(-\frac{\|\mu\|^2}{2 \|V\|}\bigg) \int_{\R^{d_x}} \exp\Big(c\|x\|^2   - \frac{\|x\|^2}{2 \|V\|}   +\frac{\|x\|\, \|\mu\|}{\|V\|}\Big) \lambda(\dd x) \\
&=
\bigg(\frac{\|V\|}{\mathrm{det}(V)^{\frac{1}{d_x}}(1-2c\|V\|)} \bigg)^{\frac{d_x}{2}}\exp\bigg(\frac{ c \|\mu\|^2  ) }{ (1-2c\|V\|)} \bigg)\\
&\leq \bigg(\frac{2\|V\|}{\mathrm{det}(V)^{\frac{1}{d_x}}}\bigg)^{\frac{d_x}{2}}\exp(2 c \|\mu\|^2 )
\end{split}    
\end{equation*}
and the proof of the lemma is complete.
\end{proof}

\paragraph{Proof of the  proposition}

\begin{proof}

Below we prove the result of the proposition for $\tau=1$ and when $m_1(\theta)=0$ for all $\theta\in\Theta$, the generalization of the proof to an arbitrary $\tau\in\mathbb{N}$ and to non-zero functions $\{m_s\}_{s=1}^\tau$ being trivial. In addition, to simplify the notation in what follows, we let $\setY=\R^{d_y}$, $\setX=\R^{d_x}$, and for all $\theta\in\Theta$ we let $A_\theta=A_1(\theta)$, $B_\theta=B_1(\theta)$, $C_\theta=C_1(\theta)$ and $D_\theta=D_1(\theta)$. 

Under the assumptions of the proposition, Assumptions \ref{assumeSSM:D_set}-\ref{assumeSSM:G}  trivially hold while   it can be established that Assumption \ref{assumeSSM:K_set} holds by following the computations in \citetsup[][Section 4.1]{Douc_MLE2}.

To show that \ref{assumeSSMB:smooth} holds,  for all $\theta\in\Theta$ we let $E_\theta=\big(A_\theta^\top A_\theta+C_\theta^\top C_\theta\big)^{-1}$ and note that this matrix is invertible under the assumptions of the proposition. Therefore, for all $(y,x,x')\in \setY\times\setX^2$ we have
\begin{equation}\label{eq:LG_last}
\begin{split}
&\exp\Big(-\frac{(y-A_\theta x')^\top B_\theta^{-1}(y-A_\theta x')}{2}-\frac{(x-C_\theta x')^\top D_\theta^{-1}(x-C_\theta x')}{2}\Big)\\
&\leq \exp\Big(-\frac{\|y-A_\theta x'\|^2+\|x-C_\theta x'\|^2}{2 (\|B_\theta\|+  \|D_\theta\|)}\Big)\\
&\leq  \exp\Big(-\frac{y^\top (I_{d_y}-A_\theta E_\theta  A_\theta^\top) y+x^\top(I_{d_x}-C_\theta E_\theta C_\theta^\top)x-2|x^\top C_\theta E_\theta A_\theta^\top y|}{2 (\|B_\theta\|+  \|D_\theta\|)}\Big).
\end{split}
\end{equation}
Using Woodbury identity, we have
\begin{align*}
I_{d_x}-C_\theta E_\theta C_\theta^\top=\Big(I_{d_x}+C_\theta (A_\theta^\top A_\theta)^{-1}C_\theta^\top\Big)^{-1}
\end{align*}
from which we conclude that  the matrix $I_{d_x}-C_\theta E_\theta C_\theta^\top$  is positive definite. This matrix being positive definite, by using \eqref{eq:LG_last} and  the assumptions of proposition  it is easily verified that there exist    constants $(C_1,C_2)\in(0,\infty)^2 $ such that 
\begin{align*}
f_{1,\theta}(y|x')m_{2,\theta}(x|x')\leq C_1 e^{C_1 \|y\|^2}\exp\Big(-\frac{(\|x\|-C_2\|y\|)^2}{C_1}\Big),\quad\forall (\theta,x,x',y)\in\Theta\times\setX^2\times\setY.
\end{align*}
Using this  result and the fact that  $\P(\|Y_1\|<\infty)=1$ since  $\E[\|Y_1\|^2]<\infty$ by assumption, it follows that $\P\big(\int_{\setX}\big\{\sup_{(\theta,x')\in \Theta \times \setX} f_{1,\theta}(Y_1|x')m_{2,\theta}(x|x')\big\}\lambda(\dd x)<\infty\big)=1$   showing that the first part of  \ref{assumeSSMB:smooth} holds.

To show that   the second   part of  \ref{assumeSSMB:smooth} holds  as well remark that, by Lemma \ref{lemma::lipchitz}, there exists a constant $\kappa \in (1,\infty)$ such that, for all $(x',x,y)\in\setX^2\times\setY$ and all $(\theta,\theta')\in\Theta^2$, we have 
\begin{align}\label{eq:kappa}
\big|\log\big(f_{1,\theta}(y|x')m_{2,\theta}(x|x')\big)  -\log\big(f_{1,\theta'}(y|x')m_{2,\theta'}(x|x')\big) \big|\leq \kappa\|\theta - \theta'\| \big(1+\|x\|^2+\|x'\|^2+\|y\|^2\big)
\end{align}
and thus in the following we show that the second   part of  \ref{assumeSSMB:smooth} holds with $g':[0,\infty)\rightarrow[0,\infty)$ defined by $g'(x)=x$ for all  $x\in[0,\infty)$ and with the function $\varphi'_1$ such that $\varphi'_1(y,x',x)=\kappa \big(1+\|x\|^2+\|x'\|^2+\|y\|^2\big)$ for all $(y,x,x')\in\setY\times\setX^2$.

To do so note that, under the assumptions of the proposition, and using the fact that the mapping $\mathrm{det}:\R^{d_x}\times\R^{d_x}\rightarrow\R$ is continuous, there exists a constant  $C\in(1,\infty)$  such that 
\begin{align*}
\sup_{\theta\in\Theta} \|D_\theta\|\leq C,\quad \sup_{\theta\in\Theta}\| C_\theta\| \leq C,\quad \big(\inf_{\theta\in\Theta}\mathrm{det}(D_\theta)\big)^{-1}\leq C
\end{align*}
and let  $c\in (0,1/(4C))$. Then, by Lemma \ref{lemma:tech_Gaussian}, for all    $(\theta,x,x')\in\Theta\times\setX^2$ we have
\begin{align*}
\int_{\setX}  e^{c\|x\|^2} M_{2,\theta}( x',\dd x)\leq (2 C^2)^{\frac{d_x}{2}}\exp(2 c \|C_\theta x'\|^2)\leq  (2 C)^{d_x}\exp(2 c C^2\|x'\|^2).
\end{align*}
Therefore, with $\varphi'_1$ as defined above, for all $\delta\in(0,c/\kappa)$ and all $(\theta,x,x',y)\in\Theta\times\setX^2\times\setY$ we have 
\begin{align}\label{eq:use_LG00}
\int_{\setX}\exp\big(\delta\varphi'_1(y,x',x)\big) M_{2,\theta}( x',\dd x)&\leq  (2 C)^{d_x} \exp\Big( \delta\kappa(1+\|x'\|^2(1+2C^2)+\|y\|^2\Big).
\end{align}

To proceed further remark that, since $\Theta$ is a compact set and   the mapping $\mathrm{det}:\R^{d_y}\times\R^{d_y}\rightarrow\R$ is continuous, under the assumptions of the proposition, and notably using the fact that the matrix $ A_\theta^\top A_\theta$ is assumed to be invertible,  we readily obtain that for some constants $(C_1,C_2)\in(0,\infty)^2$ we have
\begin{equation}\label{eq:use_LG}
\begin{split}
f_{1,\theta}(y|x')\leq C_1 e^{C_1\|y\|^2}\exp\Big(-\frac{(\|x'\|-C_2\|y\|)^2}{C_1}\Big),\quad\forall (\theta,x',y)\in\Theta\times\setX\times\setY.
\end{split}    
\end{equation}
Then, by using  \eqref{eq:use_LG00}-\eqref{eq:use_LG}, and since $\E[\|Y_1\|^2]<\infty$ by assumption,   it follows that  for $\delta'\in(0,c/\kappa)$  sufficiently small   we have
\begin{align}\label{eq:F_def0}
\E\Big[\log^+ \sup_{x' \in \setX} f_{1,\theta}(Y_1|x')\int_\setX  \exp\big(\delta'  \varphi_1'(Y_1,x',x)\big)M_{2,\theta}(x',\dd x)\Big]<\infty.
\end{align}
Therefore,  letting $F:\setY\rightarrow [1,\infty)$ be defined by
\begin{align}\label{eq:F_def}
F(y)=1\vee \sup_{x' \in \setX} f_{1,\theta}(y|x')\int_\setX \exp\big(\delta' \varphi_1'(y,x',x)\big) M_{2,\theta}(x',\dd x),\quad y\in\setY,
\end{align}
for all $\theta\in\Theta$ and  random measure $\eta$ on $(\setX,\mathcal{X})$ we have
\begin{align*}
\int_{t+1} \exp\Big(\delta \sum_{s=1}^{t}\varphi'_{1}(Y_s,x_{s-1},x_s)\Big)\eta(\dd x_0) &\prod_{s=1}^{ t}f_{1,\theta}(Y_s|x_{s-1})M_{2,\theta}(x_{s-1},\dd x_s)\leq \eta(\setX)\prod_{s=1}^t F(Y_s),\quad\forall t\geq 1. 
\end{align*}
Using this   result and the fact that $\P\big(\int_{\setX}\big\{\sup_{(\theta,x')\in \Theta \times \setX} f_{1,\theta}(Y_1|x')m_{2,\theta}(x|x')\big\}\lambda(\dd x)<\infty\big)=1$, as  shown above, it readily follows that under the assumptions of the proposition  the second part of \ref{assumeSSMB:smooth} holds, which concludes to show that Assumption \ref{assumeSSMB:smooth} is satisfied.

To conclude the proof it remains to show that Assumption \ref{assumeSSM:smooth} holds. To this aim  we remark that, under the assumptions of the  proposition  and by \eqref{eq:use_LG}, we readily obtain that we $\P$-a.s.~have $ \int_{\setX}\big\{ f_{1,\theta}(Y_1|x)\sup_{(\theta,x')\in \Theta \times \setX}m_{2,\theta}(x|x')\big\}\lambda(\dd x)<\infty$.

To show that the second part of \ref{assumeSSM:smooth} holds note that, by Lemma \ref{lemma::lipchitz}, there exists a constant $\kappa' \in (1,\infty)$ such that for all   $(x',x,y)\in\setX^2\times\setY$ and all $(\theta,\theta')\in\Theta^2$  we have 
\begin{align*}
\big|\log\big(f_{1,\theta}(y|x)m_{2,\theta}(x|x')\big)&  -\log\big(f_{1,\theta'}(y|x)m_{2,\theta'}(x|x')\big) \big|\leq \kappa'\|\theta - \theta'\| \big(1+\|x\|^2+\|x'\|^2+\|y\|^2\big)
\end{align*}
and thus in the following we show that the second   part of  \ref{assumeSSM:smooth} holds with $g':[0,\infty)\rightarrow[0,\infty)$ defined by $g'(x)=x$ for all  $x\in[0,\infty)$ and with the function $\varphi'_1$ such that $\varphi'_1(y,x',x)=\kappa' \big(1+\|x\|^2+\|x'\|^2+\|y\|^2\big)$ for all $(y,x,x')\in\setY\times\setX^2$. Remark that we can without loss of generality assume that   $\kappa'=\kappa$, with $\kappa$ as in \eqref{eq:kappa}, and   let $\delta'\in(0,\infty)$ be such that \eqref{eq:F_def0} holds. Note  that such  a $\delta'\in(0,\infty)$ exists from the above calculations.

Let $t\geq 2$, $\eta$ be an arbitrary  random measure on $(\setX,\mathcal{X})$ and $\theta\in\Theta$. Then,
\begin{equation}\label{eq:LG_p1}
\begin{split}
 & \int_{\setX^{ t+1}}\exp\bigg( \delta'  \sum_{s= 1}^{  t}\varphi'_1(Y_s,x_{s-1}, x_{s})\bigg) \eta (\dd x_{0})  \prod_{s=1}^{ t} f_{1,\theta}(Y_s|x_s)M_{2,\theta}(x_{s-1},\dd x_{s})\\
 &\leq Z_{t}\Big(\prod_{s=2}^{t}F(Y_{s})e^{\delta'\kappa \|Y_{s}\|^2}\Big)\int_{\setX^2}\exp\big( \delta'  \varphi'_1(Y_{1}, x_{0}, x_{1})\big) \eta (\dd x_{0})M_{2,\theta}(x_0,\dd x_{1})
 \end{split}
\end{equation}
with $Z_{t}=\sup_{(\theta,x)\in\Theta\times\setX}f_{1,\theta}(Y_{ t}|x_{t})$ and with the function $F:\setY\rightarrow [1,\infty)$ as defined in  \eqref{eq:F_def}.

Under the assumptions on $\{\chi_\theta,\,\theta\in\Theta\}$ imposed in the second part of the proposition, it is easily checked that for $\delta'$ sufficiently small there exists a constant $C\in(0,\infty)$ such that
\begin{align}\label{eq:LG_p2}
\int_{\setX^2}\exp\big( \delta'  \varphi'_1(Y_{1}, x_{0}, x_{1})\big)\chi_{\theta'} (\dd x_{0})M_{2,\theta'}(x_0,\dd x_{1})\leq e^{C\|Y_1\|^2},\quad\forall \theta'\in\Theta
\end{align}
which, together with \eqref{eq:LG_p1}, shows that for all $\theta\in\Theta$ and we have
\begin{align*}
\int_{\tau t+1} \exp\Big(\delta'\sum_{s=1}^{\tau t}\varphi'_{i_s}(Y_s,x_{s-1},x_s)\Big)\chi_\theta(\dd x_0) &\prod_{s=1}^{\tau t}f_{1,\theta}(Y_s|x_{s})M_{2,\theta}(x_{s-1},\dd x_s) \leq \prod_{s=1}^t \tilde{F}(Y_s),\quad\forall t\geq 1 
\end{align*}
with the function $\tilde{F}:\setY\rightarrow [1,\infty)$ defined by
\begin{align*}
\tilde{F}(y)=F(y)e^{(C+ \delta'\kappa)\|y\|^2},\quad y\in\setY.
\end{align*}

Finally,  using  \eqref{eq:use_LG} and under the assumptions of the proposition, there exists a constant  $C'\in(0,\infty)$ such that, for all $\theta'\in\Theta$,
\begin{equation}\label{eq:LG_p3}
\begin{split}
\int_{\setX^2}\exp\big( \delta'  \varphi'_1(Y_{1}, x_{0}, x_{1})\big)  f_{1,\theta'}(Y_1|x_0) \Big(\sup_{(\theta,x')\in \Theta \times \setX}m_{2,\theta}(x_0|x')\Big)&\lambda(\dd x_0) M_{2,\theta'}(x_0,\dd x_{0})\\
&\leq e^{C'\|Y_1\|^2}.
\end{split}
\end{equation}
By combining \eqref{eq:LG_p1}-\eqref{eq:LG_p3}, it readily follows that under the assumption of the proposition   the     second   part of  \ref{assumeSSM:smooth} holds, and the proof of the proposition is complete.
\end{proof}

\subsubsection{Proof of Proposition \ref{prop:Online_SV_model}}

\paragraph{Preliminary results}

We start with three technical lemmas.
\begin{lemma}\label{lemma:techSV}
Let $(a,b,c)\in (0,\infty)^3$ be such that   $c>b$   and let
\begin{align*}
h(x,y)=-a y^2 e^{-x}+b|x|-c x,\quad\forall (x,y)\in\R^2.
\end{align*}
Then,   $h(x,y)\leq D_{a,b,c}+4c|\log |y||$ for all $(x,y)\in\R^2$, with
\begin{align*}
D_{a,b,c}=-(c-b)\log\Big(\frac{a}{c-b}\Big)-(b+c)\log\Big(\frac{a}{b+c}\Big).
\end{align*}

\end{lemma}
\begin{proof}
The proof of the lemma is trivial and is therefore omitted to save space.
\end{proof}

\begin{lemma}\label{lemma:techSV2}
Let $(\nu,\sigma,z)\in (0,\infty)^3$,  $(x,\alpha)\in\R^2$ be such that $x/\alpha>0$ and   $c\in (0,\infty)$. Then, 
\begin{equation*}
\begin{split}
\sup_{a\in [0,1]}\exp\Big(- &c\frac{(1-a)x}{2\alpha} \bigg) \Big(\nu \sigma^2+a^2z^2\Big)^{-\frac{\nu+1}{2}}\\
&\leq \max\bigg( \exp\Big(-\frac{c x}{2\alpha}+\frac{(\nu+1)}{2}\Big)(\nu\sigma^2)^{-\frac{\nu+1}{2}}, \big(\nu \sigma^2+z^2\big)^{-\frac{\nu+1}{2}}\bigg).
\end{split}
\end{equation*}
\end{lemma}

\begin{proof}
 
Remark that the result of the lemma trivially holds if $z=0$ and thus below we assume that $z\neq 0$.

To prove the lemma let
\begin{align*}
g(a)=\exp\Big(- c\frac{(1-a)x}{2\alpha} \bigg) \big(\nu \sigma^2+a^2z^2\big)^{-\frac{\nu+1}{2}},\quad\forall a\in \R
\end{align*}
and note that 
\begin{align*}
g'(a):=\frac{\dd g(a)}{\dd a}=\frac{c x}{2\alpha}g(a)-a\frac{(\nu+1)z^2 g(a)}{\nu \sigma^2+a^2 z^2},\quad\forall a\in\R.
\end{align*}
Remark that $g'(0)>0$ and that for $a\in\R$ we have
\begin{equation}\label{eq:SV_key1}
\begin{split}
g'(a)\geq 0&\Leftrightarrow \frac{c x}{2\alpha}\big(\nu\sigma^2+a^2z^2\big)-a(\nu+1)z^2\geq 0\\
&\Leftrightarrow a^2  \frac{cxz^2}{2\alpha}-a(\nu+1)z^2+\frac{cx\nu\sigma^2}{2\alpha}\geq 0 
\end{split}
\end{equation}
where 
\begin{equation}\label{eq:SV_key2}
\begin{split}
\frac{\partial}{\partial a}\Big( a^2  \frac{  cx z^2}{2\alpha} -a(\nu+1) z^2\Big)< 0\Leftrightarrow   a  \frac{c x z^2}{ \alpha}-(\nu+1)z^2< 0 \Leftrightarrow a\leq \frac{(\nu+1)\alpha}{x c}.
\end{split}
\end{equation}
We now let $a^*\in \argmax_{a\in[0,1]}g(a)$. Then, by \eqref{eq:SV_key1}-\eqref{eq:SV_key2}, we have either  $a^*\in [0, (\nu+1)\alpha/(x c)]$ or $a^*=1$, and thus, noting that $\sup_{a\in[0,1]}g(a)=g(a^*)$ since   $g(\cdot)$ is continuous and $[0,1]$ is a compact set,
\begin{align*}
\sup_{a\in[0,1]}g(a) \leq \max\bigg( \exp\Big(-\frac{c x}{2\alpha}+\frac{(\nu+1)}{2}\Big)(\nu\sigma^2)^{-\frac{\nu+1}{2}}, \big(\nu \sigma^2+z^2\big)^{-\frac{\nu+1}{2}}\bigg)
\end{align*}
showing the result of the lemma.
\end{proof}

Using Lemma \ref{lemma:techSV2} we obtain the following result.
\begin{lemma}\label{lemma:techSV3}
Let $(\nu,\sigma,\mu)\in (0,\infty)^3$, $a\in [0,1]$ and $(x,\alpha,w)\in\R^3$ be such that one of following three conditions holds:
\begin{enumerate}
\item $x/\alpha>0$,
\item $x/\alpha<0$ and $\mu\geq 1$,
\item $x/\alpha<0$,  $\mu<1$ and, for some $\epsilon\in (0,1)$,    $|x|\geq   2 |\alpha|/(\epsilon \mu )$ and $a\in[0,1-\epsilon]$.
\end{enumerate}
Then,
\begin{equation*}
\begin{split}
\exp\Big(-\frac{(1-a)x}{2\alpha}&-\frac{1}{2}\mu^{1-a}   e^{-(1-a)\frac{x}{\alpha}} \Big) \Big(\nu\sigma^2+a^2 (x-\alpha w)^2 \Big)^{-\frac{\nu+1}{2}}\\
&\leq \max\bigg( \exp\Big(-\frac{ |x|}{4|\alpha|}+\frac{(\nu+1)}{2}\Big)(\nu\sigma^2)^{-\frac{\nu+1}{2}}, \big(\nu \sigma^2+(x-\alpha w)^2\big)^{-\frac{\nu+1}{2}}\bigg).
\end{split}
\end{equation*}
\end{lemma}

\begin{proof}

We first consider the case where $x/\alpha>0$. In this case, for all $a\in[0,1]$ we have
\begin{align*}
\exp\Big(-\frac{(1-a)x}{2\alpha}-&\frac{1}{2}\mu^{1-a}   e^{-(1-a)\frac{x}{\alpha}} \Big) \Big(\nu\sigma^2+a^2(x-\alpha w)^2 \Big)^{\frac{\nu+1}{2}}\\
&\leq  \exp\Big(-\frac{(1-a)x}{2\alpha}\Big)\Big(\nu\sigma^2+a^2(x-\alpha w)^2 \Big)^{\frac{\nu+1}{2}}\\
&\leq  \max\bigg( \exp\Big(-\frac{ x}{2\alpha}+\frac{(\nu+1)}{2}\Big)(\nu\sigma^2)^{-\frac{\nu+1}{2}}, \Big(\nu \sigma^2+(x-\alpha w)^2\Big)^{-\frac{\nu+1}{2}}\bigg)
\end{align*}
where the second inequality holds by applying Lemma \ref{lemma:techSV2} with $c=1$ and with $z=(x-\alpha w)$. 

We now consider the case where $x/\alpha<0$ and $\mu\geq 1$. In this situation,   that for all $a\in[0,1]$ we have
\begin{align*}
\exp\Big(-\frac{(1-a)x}{2\alpha}-\frac{1}{2}\mu^{1-a}   e^{-(1-a)\frac{x}{\alpha}} \Big)&=\exp\Big( \frac{(1-a)|x|}{2|\alpha|}-\frac{1}{2}\mu^{1-a}   e^{ (1-a)\frac{|x|}{|\alpha|}} \Big)\\
&\leq \exp\Big( \frac{(1-a)|x|}{2|\alpha|}-\frac{1}{2}    e^{ (1-a)\frac{|x|}{|\alpha|}} \Big)\\
&\leq \exp\Big(-(e^1-1) \frac{(1-a)|x|}{2|\alpha|}\Big)\\
&\leq \exp\Big(-  \frac{(1-a)|x|}{4|\alpha|}\Big)
\end{align*}
where the penultimate inequality uses the fact that $u e^{-u}\leq e^{-1}$ for all $u\in\R$. Then, the result of the lemma for the case where $x/\alpha<0$ and where $\mu\geq 1$ follows by   applying Lemma \ref{lemma:techSV2} with $c=1/2$ and with $z=(x-\alpha w)$. 

We finally let $\epsilon\in (0,1)$ and consider the case where $x/\alpha<0$, $\mu<1 $, $|x|\geq    2|\alpha|/(\epsilon\mu)$ and $a\in[0,1-\epsilon]$. To this aim, remark that for all $u\in \R$ we have
\begin{align*}
\frac{1}{2}u-\frac{1}{2}\mu e^u\leq -\frac{1}{2}u\Leftrightarrow ue^{-u}\leq\frac{\mu}{2} 
\end{align*}
and thus, using the fact that $u e^{-u}\leq 1/u$ for all $u>0$, it follows that
\begin{align}\label{eq:u_in}
\frac{1}{2}u-\frac{1}{2}\mu e^u\leq -\frac{1}{2}u,\quad\forall u\geq\frac{2}{\mu}.
\end{align}
Since the condition on $|x|$ and $a$ ensures that
\begin{align*}
\frac{(1-a)|x|}{|\alpha|}\geq   \frac{2}{\mu} 
\end{align*}
it follows from  \eqref{eq:u_in} that
\begin{align*}
 \frac{(1-a)|x|}{2|\alpha|}-\frac{1}{2}\mu    e^{ (1-a)\frac{|x|}{|\alpha|}}\leq  - \frac{(1-a)|x|}{2|\alpha|}.
\end{align*}
Therefore,  noting that $\mu\leq \mu^{1-a}$ for all $a\in[0,1]$ since $\mu\leq 1$, we have
\begin{align*}
\exp\Big(-\frac{(1-a)x}{2\alpha}-\frac{1}{2}\mu^{1-a}   e^{-(1-a)\frac{x}{\alpha}} \Big)&\leq \exp\Big(-\frac{(1-a)x}{2\alpha}-\mu\frac{1}{2}   e^{-(1-a)\frac{x}{\alpha}} \Big)\\
&= \exp\Big( \frac{(1-a)|x|}{2|\alpha|}-\frac{1}{2}\mu    e^{ (1-a)\frac{|x|}{|\alpha|}} \Big)\\
&\leq \exp\Big(- \frac{(1-a)|x|}{2|\alpha|} \Big).
\end{align*}
The result of the last part of the lemma then follows by applying Lemma \ref{lemma:techSV2} with $c=1$ and with $z=(x-\alpha w)$. The proof of the lemma is complete.
\end{proof}

\paragraph{Proof of the  proposition}

\begin{proof}
Under the assumptions of the proposition, Assumptions \ref{assumeSSM:D_set}-\ref{assumeSSM:G}  trivially hold while   Assumption \ref{assumeSSM:K_set} holds as shown in \citetsup[][Section 4.3]{Douc_MLE2}.

We now show that Assumption \ref{assumeSSMB:smooth} holds. To this aim let
\begin{align*}
\tilde{f}_\theta(y|x') = \exp\Big(-\frac{x'}{2}-\frac{y^2}{2\beta^2} e^{-x'} \Big),\quad \tilde{m}_\theta(x|x') = \Big(1+\frac{(x-\alpha x')^2}{\nu\sigma^2}\Big)^{-\frac{\nu+1}{2}},\quad\forall (\theta, x,x',y)\in\Theta\times\R^3 
\end{align*}
and note that,  under the assumptions of the proposition, there exists a constant $C_1\in(0,\infty)$ such that
\begin{equation}\label{eq:B_f0}
\begin{split}
\int_\R \Big\{\sup_{(\theta,x')\in\Theta\times\R } f_{1,\theta}(Y_1|x')m_{2,\theta}(x|x')\Big\}\lambda(\dd x) \leq  C_1 \int_\R \Big\{\sup_{(\theta,x')\in\Theta\times\R } \tilde{f}_{\theta}(Y_1|x')\tilde{m}_{\theta}(x|x')\Big\}\lambda(\dd x) 
\end{split}    
\end{equation}
and such that
\begin{align}\label{eq:B_f}
 \tilde{f}_{\theta}(y|x')\leq \frac{C_1}{|y|},\quad\forall (\theta,x',y)\in\Theta\times\R^2.
\end{align}

We now let  $\Theta_0=\{(\alpha,\beta,\sigma)\in\Theta:\alpha\neq 0\}$ and note that 
\begin{equation}\label{eq:SV1}
\begin{split}
\int_\R \Big\{\sup_{(\theta,x')\in\Theta\times\R } \tilde{f}_{\theta}(Y_1|x')\tilde{m}_{\theta}(x|x')\Big\}\lambda(\dd x)&\leq \int_\R \Big\{\sup_{(\theta,x')\in\Theta_0\times\R } \tilde{f}_{\theta}(Y_1|x')\tilde{m}_{\theta}(x|x')\Big\}\lambda(\dd x)\\
&+\int_\R \Big\{\sup_{(\theta,x')\in(\Theta\setminus\Theta_0)\times\R } \tilde{f}_{\theta}(Y_1|x')\tilde{m}_{\theta}(x|x')\Big\}\lambda(\dd x)\\
&\leq \int_\R \Big\{\sup_{(\theta,x')\in\Theta_0\times\R } \tilde{f}_{\theta}(Y_1|x')\tilde{m}_{\theta}(x|x')\Big\}\lambda(\dd x)+\frac{C_1}{|Y_1|} 
\end{split}    
\end{equation}
where the last inequality holds by \eqref{eq:B_f}.

To study the last integral in \eqref{eq:SV1}  let $y\in\R$ be such that $y\neq 0$. Then, for a given $ \theta \in\Theta$  the function $\tilde{x}\mapsto \tilde{f}_\theta(y|\tilde{x})$  has a unique maximum at $\tilde{x}_{\theta,y}:=2\log(|y|/\beta)$, and is increasing on the interval $(-\infty,\tilde{x}_{\theta,y}]$ and decreasing on the interval $[\tilde{x}_{\theta,y},\infty)$. In addition, for a given $(\theta,x)\in\Theta_0\times\R$ the function $\tilde{x}\mapsto  \tilde{m}_\theta(x|\tilde{x})$ as a unique maximum at  $\tilde{x}'_{\theta,x}:=x/\alpha$, and is    increasing on the interval $(-\infty,\tilde{x}'_{\theta,x}]$ and decreasing on the interval $[\tilde{x}'_{\theta,x},\infty)$. Therefore, for all $(\theta,x)\in\Theta_0\times\R^2$ and letting 
\begin{align*}
K_\theta(x,y)  = \big\{x'\in\R: x' = a \tilde{x}_{\theta,y}+(1-a) \tilde{x}'_{\theta,x},\quad a\in[0,1]\big\},
\end{align*}
 we have
\begin{align*}
\sup_{ x'\in \R\setminus  K_\theta(x,y)} \tilde{f}_{\theta}(y|x')\tilde{m}_{\theta}(x|x')&=  \sup_{ x'\in   K_\theta(x,y)} \tilde{f}_{\theta}(y|x')\tilde{m}_{\theta}(x|x')\\
&=\sup_{a\in [0,1]}\tilde{f}_{\theta}\Big(y|a \tilde{x}_{\theta,y}+(1-a) \tilde{x}'_{\theta,x}\Big)\tilde{m}_{\theta}\Big(x|a \tilde{x}_{\theta,y}+(1-a) \tilde{x}'_{\theta,x}\Big)
\end{align*}
and thus
\begin{equation}\label{eq:SV22}
\begin{split}
\int_\R \Big\{ &\sup_{(\theta,x')\in\Theta_0\times\R }  \tilde{f}_{\theta}(y|x')\tilde{m}_{\theta}(x|x')\Big\}\lambda(\dd x)\\
&=\int_\R \Big\{\sup_{(\theta,a)\in\Theta_0\times[0,1] } \tilde{f}_{\theta}\Big(y|a \tilde{x}_{\theta,y}+ (1-a) \tilde{x}'_{\theta,x}\Big)\tilde{m}_{\theta}\Big(x|a \tilde{x}_{\theta,y}+(1-a) \tilde{x}'_{\theta,x}\Big)\Big\}\lambda(\dd x).
\end{split}    
\end{equation}

We low let $\epsilon\in(0,1)$ and note that
\begin{equation}\label{eq:SV_eps0}
\begin{split}
\int_\R &\Big\{\sup_{(\theta,a)\in\Theta_0\times[0,1] } \tilde{f}_{\theta}\Big(y|a \tilde{x}_{\theta,y}+ (1-a) \tilde{x}'_{\theta,x}\Big)\tilde{m}_{\theta}\Big(x|a \tilde{x}_{\theta,y}+(1-a) \tilde{x}'_{\theta,x}\Big)\Big\}\lambda(\dd x)\\
&\leq \int_\R \Big\{\sup_{(\theta,a)\in\Theta_0\times[0,1-\epsilon] } \tilde{f}_{\theta}\Big(y|a \tilde{x}_{\theta,y}+ (1-a) \tilde{x}'_{\theta,x}\Big)\tilde{m}_{\theta}\Big(x|a \tilde{x}_{\theta,y}+(1-a) \tilde{x}'_{\theta,x}\Big)\Big\}\lambda(\dd x)\\
&+\int_\R \Big\{\sup_{(\theta,a)\in\Theta_0\times[1-\epsilon,1] } \tilde{f}_{\theta}\Big(y|a \tilde{x}_{\theta,y}+ (1-a) \tilde{x}'_{\theta,x}\Big)\tilde{m}_{\theta}\Big(x|a \tilde{x}_{\theta,y}+(1-a) \tilde{x}'_{\theta,x}\Big)\Big\}\lambda(\dd x)
\end{split}
\end{equation}
 where, using \eqref{eq:B_f}, for all $(\theta,x)\in\Theta_0\times\R$ we have
\begin{equation}\label{eq:SV_eps}
\begin{split}
\sup_{a\in[1-\epsilon,1]}\tilde{f}_{\theta}\Big(y|a \tilde{x}_{\theta,y}+&(1-a) \tilde{x}'_{\theta,x}\Big)\tilde{m}_{\theta}\Big(x|a \tilde{x}_{\theta,y}+(1-a) \tilde{x}'_{\theta,y}\Big)\\
&=\frac{C_1}{|y|}\sup_{a\in[1-\epsilon,1]} \Big(1+a^2\frac{( x- \alpha   \log(y^2/\beta^2))^2}{\nu\sigma^2}\Big)^{-\frac{\nu+1}{2}}\\
&\leq \frac{C_1}{|y|} \Big(1+(1-\epsilon)^2\frac{( x- \alpha   \log(y^2/\beta^2))^2}{\nu \sigma^2 }\Big)^{-\frac{\nu+1}{2}}.
\end{split}
\end{equation}
To proceed further let $\bar{C}\in (1,\infty)$ be such that
\begin{align*}
|\alpha|\leq \bar{C},\quad \frac{1}{\bar{C}}\leq \beta\leq \bar{C},\quad \frac{1}{\bar{C}}\leq \sigma\leq \bar{C},\quad \forall (\alpha,\beta,\sigma)\in\Theta
\end{align*}
and note that
\begin{align*}
 \alpha   \log(y^2/\beta^2)\leq \kappa_y:=2 \bar{C}|\log(|y|)|+ 2 \bar{C} \log(\bar{C}),\quad\forall\theta\in\Theta.
\end{align*}
In addition, note that for any constant  $D\in(0,\infty)$ and  any $\theta\in\Theta$ we have 
\begin{align}\label{eq:S_b1}
 \Big(1+ \frac{( x- \alpha   \log(y^2/\beta^2))^2}{D}\Big)^{-\frac{\nu+1}{2}}\leq  \Big(1+ \frac{( x- \kappa_y)^2}{D}\Big)^{-\frac{\nu+1}{2}},\,\, \forall x\in[\kappa_y,\infty)
\end{align}
and
\begin{align}\label{eq:S_b2}
\Big(1+ \frac{( x- \alpha   \log(y^2/\beta^2))^2}{D}\Big)^{-\frac{\nu+1}{2}}\leq  \Big(1+ \frac{( x+ \kappa_y)^2}{D}\Big)^{-\frac{\nu+1}{2}},\,\, \forall x\in(-\infty,-\kappa_y].
\end{align}
Therefore, by using \eqref{eq:SV_eps}-\eqref{eq:S_b2}, we obtain that
\begin{equation}\label{eq:SV_eps2}
\begin{split}
\int_\R \sup_{a\in[1-\epsilon,1]} \tilde{f}_{\theta}\Big(y|a \tilde{x}_{\theta,y}+ (1-a) \tilde{x}'_{\theta,x}\Big)\tilde{m}_{\theta}&\Big(x|a \tilde{x}_{\theta,y}+(1-a) \tilde{x}'_{\theta,y}\Big)\lambda(\dd x)\\
&\leq  \frac{2C_1}{|y|}\Big(2\bar{C}|\log(|y|)+2 \bar{C} \log(\bar{C})+C'\Big)
\end{split}
\end{equation}
with
\begin{equation}\label{eq:tildeC}
\begin{split}
C'=\frac{\Gamma(\nu/2)\big(\pi\nu \bar{C}^2/(1-\epsilon)^2\big)^{1/2}}{\Gamma\big((\nu+1)/2\big)}&=\int_\R \Big(1+(1-\epsilon)^2\frac{( x- \kappa_y)^2}{\nu \bar{C}^2}\Big)^{-\frac{\nu+1}{2}}\lambda(\dd x)\\
&=\int_\R \Big(1+(1-\epsilon)^2\frac{( x+\kappa_y)^2}{\nu \bar{C}^2}\Big)^{-\frac{\nu+1}{2}}\lambda(\dd x).
\end{split}
\end{equation}

Next, to bound the first integral in \eqref{eq:SV_eps0}, we note first that  for all $a\in [0,1]$ and $(\theta,x)\in\Theta_0\times\R$  we have
\begin{align*}
\tilde{f}_{\theta}\Big(y|a \tilde{x}_{\theta,y}+(1-a) \tilde{x}'_{\theta,x}\Big)&=\exp\Big(-\frac{(1-a)x}{2\alpha}-a\log(|y|/\beta)-\frac{y^2}{2\beta^2}\big( e^{-(1-a)\frac{x}{\alpha}}e^{-2a\log(|y|/\beta)} \Big)\\
&=\Big(\frac{\beta}{|y|}\Big)^a\exp\Big(-\frac{(1-a)x}{2\alpha}-\frac{1}{2}\Big(\frac{y^2}{\beta^2}\Big)^{1-a}  e^{-(1-a)\frac{x}{\alpha}} \Big)
\end{align*}
and we let $\xi_{y}=(2 \bar{C})/\big(\epsilon\min(1,|y|/\bar{C})^2\big)$ and $\kappa_y$ be as above 

Then, by applying Lemma \ref{lemma:techSV3} with $\mu=(y/\beta)^2$ and with $w=\kappa_y/\alpha$,  and by using \eqref{eq:S_b1}-\eqref{eq:S_b2}, it follows that for all $(\theta,x)\in\Theta_0\times\R$ such that $x\geq  \max(\xi_{y},\kappa_y)$ we have
\begin{align*}
\sup_{a\in[0,1-\epsilon]}\tilde{f}_{\theta}\Big(y|a \tilde{x}_{\theta,y}+&(1-a) \tilde{x}'_{\theta,x}\Big)\tilde{m}_{\theta}\Big(x|a \tilde{x}_{\theta,y}+(1-a) \tilde{x}'_{\theta,y}\Big)\\
&\leq  \frac{\bar{C}}{\min(1,|y|)}\max\bigg( \exp\Big(-\frac{ |x|}{4\bar{C}}+\frac{(\nu+1)}{2}\Big) , \Big(1+ \frac{(x-\kappa_y)^2}{\nu \bar{C}^2}\big)^{-\frac{\nu+1}{2}}\bigg)
\end{align*}
while,   for all $(\theta,x)\in\Theta_0\times\R$ such that $x\leq -  \max(\xi_{y},\kappa_y)$, we have
\begin{align*}
\sup_{a\in[0,1-\epsilon]}\tilde{f}_{\theta}\Big(y|a \tilde{x}_{\theta,y}+&(1-a) \tilde{x}'_{\theta,x}\Big)\tilde{m}_{\theta}\Big(x|a \tilde{x}_{\theta,y}+(1-a) \tilde{x}'_{\theta,y}\Big)\\
&\leq \frac{\bar{C}}{\min(1,|y|)}\max\bigg( \exp\Big(-\frac{ |x|}{4\bar{C}}+\frac{(\nu+1)}{2}\Big) , \Big(1+ \frac{(x+\kappa_y)^2}{\nu \bar{C}^2}\big)^{-\frac{\nu+1}{2}}\bigg).
\end{align*}
Using the latter two results, we readily obtain that
\begin{equation}\label{eq:SV_eps222}
\begin{split}
 \int_\R \Big\{&\sup_{(\theta,a)\in\Theta_0\times[0,1-\epsilon] } \tilde{f}_{\theta}\Big(y|a \tilde{x}_{\theta,y}+ (1-a) \tilde{x}'_{\theta,x}\Big)\tilde{m}_{\theta}\Big(x|a \tilde{x}_{\theta,y}+(1-a) \tilde{x}'_{\theta,x}\Big)\Big\}\lambda(\dd x)\\
 &\leq \frac{2C_1}{|y|}  \max(\xi_{y},\kappa_y)+ \frac{\bar{C}}{\min(1,|y|)}\bigg(C'+\int_\R \exp\Big(-\frac{ |x|}{4\bar{C}}+\frac{(\nu+1)}{2}\Big)\bigg)
\end{split}
\end{equation}
 with $C_1$ as in \eqref{eq:B_f} and with $C'$ as defined in \eqref{eq:tildeC}.

Under the assumptions of the proposition   we have $\P(|Y_1|\in(0,\infty))=1$, and thus by combining \eqref{eq:B_f0}, \eqref{eq:SV1}, \eqref{eq:SV22}, \eqref{eq:SV_eps0}, \eqref{eq:SV_eps2} and \eqref{eq:SV_eps222},   we   obtain that
\begin{align*}
\P\Big(\int_\R \Big\{\sup_{(\theta,x')\in\Theta\times\setX } f_{1,\theta}(Y_1|x')m_{2,\theta}(x|x')\Big\}\lambda(\dd x)<\infty\Big)=1 
\end{align*}
showing that the first part of  \ref{assumeSSMB:smooth} holds.
 
 To show that the second part of  \ref{assumeSSMB:smooth} holds as well remark first that under the assumptions of the proposition there exists a constant $C\in(0,\infty)$ such that  
\begin{align*}
\Big|\frac{\partial}{\partial \beta} \log f_{1,\theta}(y|x')\Big|\leq C(1+y^2e^{-x'}),\quad\forall (\theta,x',y)\in\Theta\times \R^2
\end{align*}
and such that
\begin{equation*}
\begin{split}
\bigg|\frac{\partial}{\partial \alpha} \log m_{2,\theta}(x|x')  + \frac{\partial}{\partial \sigma} \log m_{2,\theta}(x|x') \bigg| &\leq C |1+ x' |,\quad\forall (\theta,x,x',\theta)\in\Theta\times\R^2.
\end{split}    
\end{equation*}      
Consequently, there exists a constant $\kappa\in (1,\infty)$ such that, for all $(\theta,\theta')\in\Theta^2$, we have
\begin{equation*}
\big|\log\big(f_{1,\theta}(y|x')m_{2,\theta}(x|x')\big)  -\log\big(f_{1,\theta'}(y|x')m_{2,\theta'}(x|x')\big) \big|\leq \|\theta - \theta'\|  \kappa(1+|x'| +y^2e^{-x'}),\quad \forall (\tilde{x},x,y)\in\R^3 
\end{equation*}
and thus in the following we show that the second   part of  \ref{assumeSSMB:smooth} holds with $g':[0,\infty)\rightarrow[0,\infty)$ defined by $g'(x)=x$ for all  $x\in[0,\infty)$ and with the function $\varphi'_1$ such that $\varphi'_1(y,x',x)=\kappa  (1+ |x'|+  y^2e^{-x'})$ for all $(y,x,x')\in\R^3$.

To this aim let $C_\beta\in (1,\infty)$ be such that $(\beta,\beta^2)\in(1/C_\beta, C_\beta)^2$ for all $\theta=(\alpha,\beta,\sigma)\in\Theta$, and let $\delta =1/(4C_\beta\kappa)$ and   $D_{a,b,c}$ be as defined in Lemma \ref{lemma:techSV} for all $(a,b,c)\in (0,\infty)^3$.

Then, for all $(\theta,x,x',y)\in\Theta\times\R^3$ we have
\begin{equation*}
\begin{split}
f_{1,\theta}(y|x')  \int_\R  \exp\Big(\delta \kappa  (1+ |x'|+  y^2e^{-x'})\Big)  M_{2,\theta}(x',\dd x) &=f_{1,\theta}(y|x')\exp\Big(\delta \kappa  (1+ |x'|+  y^2e^{-x'})\Big)\\
&\leq \frac{C_\beta e^{\delta\kappa}}{ \sqrt{2\pi}}
\exp\Big(-y^2 e^{-x'}\frac{1}{4C_\beta} + \frac{1}{4} |x'| - \frac{x'}{2}\Big)\\
&\leq\frac{C_\beta e^{\delta\kappa}}{ \sqrt{2\pi}}
\exp\Big(D_{\frac{1}{4C_\beta},\frac{1}{4},\frac{1}{2}}+2|\log (|y|)|\Big)
\end{split}    
\end{equation*}
where the last inequality holds by Lemma \ref{lemma:techSV}.

Therefore, since by assumption $\E[|\log (|Y_1|)|]<\infty$, this shows that
\begin{align*}
\E\Big[\log^+ \sup_{x' \in \R} f_{1,\theta}(Y_1|x')\int_\R \exp\big(\delta  \varphi_1'(Y_1,x',x)\big) M_{2,\theta}(x',\dd x)\Big]<\infty 
\end{align*}
and, using a similar argument as in the proof of Proposition \ref{prop:Assume_LG}, we conclude 
that under the assumptions of the proposition    the second part of \ref{assumeSSMB:smooth} holds. This concludes to show  that Assumption   \ref{assumeSSMB:smooth} is satisfied.

To conclude the proof it remains to show that Assumption \ref{assumeSSM:smooth} holds. To do so   remark first that, under the assumptions of the proposition,   we readily have that \begin{align*}
\P\big(\int_{\R}\big\{ f_{1,\theta}(Y_1|x)\sup_{(\theta,x')\in \Theta \times \R}m_{2,\theta}(x|x')\big\}\lambda(\dd x)<\infty\big)=1.
\end{align*}

To show the second part of  \ref{assumeSSM:smooth} note that, by using the above computations, there exists a constant $\kappa\in(0,\infty)$ such that for all $(\theta,\theta')\in\Theta^2$ we have
\begin{equation*}
\big|\log\big(f_{1,\theta}(y|x)m_{2,\theta}(x|x')\big)  -\log\big(f_{1,\theta'}(y|x)m_{2,\theta'}(x|x')\big) \big|\leq \|\theta - \theta'\|  \kappa(1+|x'| +y^2e^{-x}),\quad \forall (\tilde{x},x,y)\in\R^3 
\end{equation*}
and thus in the following we show that the second   part of  \ref{assumeSSM:smooth} holds with $g':[0,\infty)\rightarrow[0,\infty)$ defined by $g'(x)=x$ for all  $x\in[0,\infty)$ and with the function $\varphi'_1$ such that $\varphi'_1(y,x',x)=\kappa  (1+ |x'|+  y^2e^{-x})$ for all $(y,x,x')\in\R^3$.

To this aim, let $\delta\in(0,1/(4C_\beta\kappa))$, with $C_\beta$ as above, $\bar{C}= C^2_\beta  e^{ \delta\kappa}$ and let $D_{a,b,c}$ be as defined in Lemma \ref{lemma:techSV} for all $(a,b,c)\in (0,\infty)^3$. Without loss of generality, we assume in what follows that  $D_{\frac{1}{4C_\beta},\delta\kappa,\frac{1}{2}}\leq \log C_\beta$. Then, for all $(\theta, x',y)\in\Theta\times\R^2$ such that $y\neq 0$, we have,
\begin{equation}\label{eq:SV_int2}
\begin{split}
\int_{\R} e^{\delta\kappa |x|} &\exp\Big(\delta \varphi'_1 (y,x',x )\Big) f_{1,\theta}(y|x) M_{2,\theta}(x',\dd x)\\
&=\frac{1}{\beta\sqrt{2\pi}}e^{\delta\kappa(1+|x'|)}\int_{\R}\exp\Big(- y^2e^{-x }\big(\frac{1}{2\beta^2}-\delta\kappa\big)-\frac{x }{2}+\delta\kappa|x| \Big) M_{2,\theta}(x',\dd x)\\
&\leq   C_\beta e^{\delta\kappa(1+|x'|)}\int_{\R}\exp\Big(- y^2e^{-x } \frac{1}{4C_\beta}-\frac{x }{2} +\delta\kappa|x|\Big) M_{2,\theta}(x',\dd x)\\
&\leq   C_\beta e^{ \delta\kappa(1+ |x'|)}\exp\Big(D_{\frac{1}{4C_\beta},\delta\kappa,\frac{1}{2}}+2|\log (|y|)|\Big)\\
&\leq   C^2_\beta e^{ \delta\kappa(1+ |x'|)}\exp\Big( 2|\log (|y|)|\Big)\\
&=\bar{C}\exp(2|\log (|y|)|) e^{\delta\kappa |x'|}.
\end{split}
\end{equation}
where the second inequality holds by Lemma \ref{lemma:techSV}.

Then, by using \eqref{eq:SV_int2}, we obtain that for all $\theta\in\Theta$,   integers $t_2> t_1\geq 0$ and $x_{t_1}\in\R$, we have
\begin{equation}\label{eq:last_SV}
\begin{split}
\int_{\R^{t_2-t_2+1}}\exp\Big(\delta\sum_{s=t_1+1}^{t_2}&\varphi'_1(Y_s,x_{s-1},x_s)\Big)\prod_{s=t_1+1}^{t_2} f_{1,\theta}(Y_s|x_s)M_{2,\theta}(x_{s-1},\dd x_s)\\
&\leq e^{\delta\kappa |x_{t_1}|}\bar{C}^{\,t_2-t_1}\exp\Big(2\sum_{s=t_1+1}^{t_2}|\log(|Y_s|)|\Big).
\end{split}
\end{equation}

Under the assumptions om $\{\chi_\theta,\,\theta\in\Theta\}$ imposed in the second part of the proposition  there exists a constant $c\in(0,\infty)$ such that $\int_{\R} e^{c|x|}\chi_\theta(\dd x)<\infty$ for all $\theta\in\Theta$ while it is readily checked that, if $c$ is sufficiently small,
\begin{align*}
    \int_{\R}f_{1,\theta}(Y_1|x)e^{c|x|}\lambda(\dd x)<\infty,\quad\P-a.s.
\end{align*}
Therefore, since $\E[|\log(|Y_1|)]<\infty$ by assumption, it follows from \eqref{eq:last_SV}  that the second part of  assumption \ref{assumeSSM:smooth} holds for any $\delta\in(0,c/\kappa)$ sufficiently small. The proof of the proposition is complete.
\end{proof}

\section{Examples of SSMs satisfying   Assumption \ref{assumeSSMB:smooth_MLE}\label{supp_MLE}}

\subsection{Linear Gaussian models}

\begin{proposition}\label{prop:Assume_LG_MLE}
Let $(d,d_x,d_y) \in \mathbb{N}^3$,    $\Theta\subset\R^d$ be a non empty compact set and, with $\tau=T$, let $\mu$, $\Sigma$, $\{m_t\}_{t=1}^\tau$, $\{A_t\}_{t=1}^\tau$,  $\{B_t\}_{t=1}^\tau$, $\{C_t\}_{t=1}^\tau$  and $\{D_t\}_{t=1}^\tau$  be as in Proposition \ref{prop:Assume_LG}. Then, SSM \eqref{eq:SSM_MLE} where, for all $(\theta,x)\in\Theta\times\R^{d_x}$, 
 \begin{align*}
&\tilde{f}_{t,\theta}(y|x)\dd y=\mathcal{N}_{d_y}\big(m_t(\theta)+A_t(\theta)x, B_t(\theta)\big),\quad \forall t\in\{1,\dots,T\}\\
&\tilde{M}_{t,\theta}\big(x,\dd x_t\big)=\mathcal{N}_{d_x}\big(C_{t}(\theta) x, D_{t}(\theta)\big),\quad\hspace{1.1cm}\forall t\in\{2,\dots,T\}
\end{align*}
and where $\tilde{\chi}_\theta(\dd x)=\mathcal{N}_{d_x}(\mu(\theta), \Sigma(\theta))$ is such that Assumptions  \ref{assumeSSMB:smooth_MLE} holds.

\end{proposition} 

\subsection{Stochastic volatility  models}

The validity of Assumption  \ref{assumeSSMB:smooth_MLE}  is of course not limited to linear Gaussian SSMs. For instance, as shown in the next proposition, it holds for the stochastic volatility model considered e.g.~in \citetsup{kim1998stochastic}.

\begin{proposition}\label{prop:SV_model}
Let $\Theta\subset\R^3$ be a compact set such that for all $\theta=(\alpha,\beta,\sigma)\in\Theta$ we have both $\beta>0$ and $\sigma>0$, and let $m:\Theta\rightarrow\R$ and $\sigma_0:\Theta\rightarrow (0,\infty)$ be two continuously differentiable functions. Consider SSM \eqref{eq:SSM_MLE} where, for all $\theta\in\Theta$, we have $\tilde{\chi}_\theta(\dd x)=\mathcal{N}_1\big(m(\theta),\sigma^2_0(\theta)\big)$, $\tilde{f}_{t,\theta}=\tilde{f}_{1,\theta}$ for all $t\in\{1,\dots,T\}$, and   $\tilde{M}_{t,\theta}=\tilde{M}_{2,\theta}$ for all $t\in\{2,\dots,T\}$, with $\tilde{f}_{1,\theta}(y|x)\dd y=\mathcal{N}_1\big(0,\beta^2 e^x\big)$ and $\tilde{M}_{2,\theta}(x,\dd x_2)=\mathcal{N}_1(\alpha x,\sigma^2)$ for all $x\in\R$. Assume that $\tilde{y}_t\neq 0$ for all $t\in\{1,\dots,T\}$. Then, Assumption  \ref{assumeSSMB:smooth_MLE} holds. In addition,  Assumption  \ref{assumeSSMB:smooth_MLE} also holds if $\tilde{M}_{2,\theta}(x,\dd x_2)=t_{1,\nu}(\alpha x,\sigma^2)$ for all $x\in\R$ and some $\nu\in(0,\infty)$.
\end{proposition}

\subsection{Proofs}

Proposition \ref{prop:Assume_LG_MLE} can be readily established using the calculations done in the proof of Proposition \ref{prop:Assume_LG}, and its proof is therefore omitted to save space. 

\subsubsection{Proof of   Proposition \ref{prop:SV_model}}

\begin{proof}

The second part of the proposition can be readily established using the calculations done in the proof of Proposition \ref{prop:Online_SV_model}, and thus below we only prove its first part. More precisely, we only show that  \ref{assumeSSMB:smooth_MLE} holds

First, we note that under the assumptions of the proposition  we   have  $\sup_{(\theta,x)\in\Theta\times\R} \tilde{f}_{s,\theta}(\tilde{y}_s|x)<\infty$  for all    $s\in\{1,\dots,T\}$ and $\tilde{L}_T(\theta)>0$ for all $\theta\in\Theta$, showing that the first part of \ref{assumeSSMB:smooth_MLE} holds.

 Next,  simple computations show that  the next two parts  of \ref{assumeSSMB:smooth_MLE} hold with $g':[0,\infty)\rightarrow [0,\infty)$ such that $g'(x)=x$ for all $x\in[0,\infty)$ and with $\{\varphi'_s\}_{s=2}^{T+1}$ defined, for some constant $\kappa\in(2,\infty)$, by
\begin{align*}
\varphi_{1}'(\tilde{x},x)=\kappa(1+x^2+\tilde{y}_1^2e^{-x}),\quad \varphi_{s}'(\tilde{x},x)=\kappa\big(1+x^2+\tilde{x}^2+\tilde{y}^2_{s} e^{-x}\big),\quad (x,\tilde{x})\in\R^2,\quad s\in\{2,\dots,T\}.
\end{align*}
 
To show that the last part of \ref{assumeSSMB:smooth_MLE} holds as well,  we let $c\in(0,1)$ be such that  for all $\theta=(\alpha,\beta,\sigma)\in\Theta$  we have 
\begin{align*}
 \alpha^2 \leq \frac{1}{c},\quad c\leq \beta\leq  \frac{1}{c},\quad c\leq \sigma\leq  \frac{1}{c}, \quad \int_\R e^{c^2x^2}\tilde{\chi}_\theta(\dd x) <\infty
\end{align*}
and we let $\bar{\delta}= c^2/(4\kappa)$.

As preliminary computations, we note that for all $\delta\in(0,\bar{\delta})$, all   $\theta=(\alpha,\beta,\sigma)\in\Theta$ and all $\tilde{x}\in\R$, we have
\begin{equation}\label{eq:SV_0}
\begin{split}
\int_{\R}\exp\Big(-\frac{(x-\alpha \tilde{x})^2}{2\sigma^2}+\delta \kappa x^2\Big)\dd x&=\bigg(\frac{2\pi\sigma^2}{1-2\delta \kappa\sigma^2}\bigg)^{\frac{1}{2}}\exp\Big(\frac{\alpha^2\delta \kappa\tilde{x}^2}{1-2\delta \kappa\sigma^2 }\Big)\\
&\leq \bigg(\frac{2\pi }{c^2(1-2\delta \kappa\sigma^2)}\bigg)^{\frac{1}{2}}\exp\Big(\frac{\delta \kappa\tilde{x}^2}{ c (1-2\delta \kappa\sigma^2) }\Big)\\
&\leq  (4\pi/c^2)^{\frac{1}{2}}\exp\big((2/c)\delta \kappa\tilde{x}^2\big).
\end{split}
\end{equation}
Using this latter result   it follows that for all $\delta\in(0,\bar{\delta})$ we have, for all $\theta\in\Theta$ and all $s\in\{2,\dots,T\}$,
\begin{equation}\label{eq:SV_1}
\begin{split}
\int_{\R}\tilde{f}_{s,\theta}(\tilde{y}_s|x_s)&\exp\Big(\delta \kappa(1+x_{s-1}^2+x_s^2+\tilde{y}_s^2 e^{-x_s})\Big)\tilde{M}_{s,\theta}(x_{s-1},\dd x_{s})\\
&=\frac{e^{\delta\kappa(1+x_{s-1}^2)}}{\beta(2\pi)^{1/2}}\int_{\R} \exp\Big(- \tilde{y}_s^2 e^{-x_s}\big(\frac{1}{2\beta^2}-\delta\kappa\big)-\frac{x_s}{2}+ \delta \kappa x_s^2 \Big)\tilde{M}_{s,\theta}(x_{s-1},\dd x_{s})\\
&\leq  \frac{e^{\delta\kappa(1+x_{s-1}^2)}}{c(2\pi)^{1/2}}\int_{\R} \exp\Big(- \tilde{y}_s^2 e^{-x_s} \frac{c^2}{4}-\frac{x_s}{2}+ \delta \kappa x_s^2 \Big)\tilde{M}_{s,\theta}(x_{s-1},\dd x_{s})\\
&\leq   \frac{e^{\delta\kappa(1+x_{s-1}^2)}}{c^2|\tilde{y}_s| \pi ^{1/2}}\int_{\R} e^{\delta \kappa x_s^2}\tilde{M}_{s,\theta}(x_{s-1},\dd x_{s})\\
&\leq \frac{e^{\delta\kappa(1+x_{s-1}^2)}}{c^4|\tilde{y}_s| (\pi/2)^{1/2})}  \exp\big((2/c)\delta \kappa x_{s-1}^2\big)  \\
&\leq \exp\Big((3/c)\delta \kappa (1+x_s^2)\Big) \frac{\bar{C}}{|\tilde{y}_s|}
\end{split}
\end{equation}
where $\bar{C}=(c^4 (\pi/2)^{1/2})^{-1}$ and where the fourth inequality uses \eqref{eq:SV_0}.

We now let $\delta\in(0,1)$ be such that $(4/c)^T\delta< \bar{\delta}$. Then,  for all $\theta\in\Theta$   we have,   using \eqref{eq:SV_1},
\begin{align*}
 \int_{\R^{T}}&\exp\Big(\delta \varphi'_1(x_1,x_1)+\delta \sum_{s=2}^{T}\varphi'_s(x_{s-1},x_{s})\Big)\tilde{\chi}_\theta(\dd x_1)\tilde{f}_{1,\theta}(\tilde{y}_1|x_1)\prod_{s=2}^{T}\tilde{f}_{s,\theta}(\tilde{y}_s|x_s)\tilde{M}_{s,\theta}(x_{s-1},\dd x_{s})\\
&\leq \frac{\bar{C}^{T-1}}{\prod_{s=2}^T|\tilde{y}_s|}\int_{\R} \exp\Big((4/c)^T\delta \kappa(1+x_1^2) \Big)\tilde{f}_{1,\theta}(\tilde{y}_1|x_1)\tilde{\chi}_\theta(\dd x_1) \\
&\leq \frac{\bar{C}^{T-1}}{c^2 \prod_{s=1}^T|\tilde{y}_s|}\int_{\R} \exp\Big((4/c)^T\delta \kappa(1+x_1^2) \Big) \tilde{\chi}_\theta(\dd x_1) \\
&\leq \frac{\bar{C}^{T-1}}{c^2 \prod_{s=1}^T|\tilde{y}_s|}\int_{\R} \exp\big(c^2(1+x_1^2)\big)\tilde{\chi}_\theta(\dd x_1) \\
&<\infty.
\end{align*}
This completes to show that   \ref{assumeSSMB:smooth_MLE} holds and the proof of the proposition is complete.

\end{proof}

\section{Examples of artificial dynamics verifying the conditions of Section \ref{sub:noise}\label{p-sub:noise}}

\subsection{Artificial dynamics for models with continuous parameters \label{sub:continuous}}

In this subsection we consider the case where $\setR=\R^d$ for some $d\in\mathbb{N}$. 

\begin{proposition}\label{prop:Gauss}

Let $\setR=\R^d$ for some $d\in\mathbb{N}$ and assume that $\Theta$ is a regular compact set. Let    $(\alpha,\nu) \in (0,\infty)^2$, $c\in (0,1)$ and $(\tilde{t}_1,m)\in\mathbb{N}^2$ be some constants, and let  $(\alpha_t)_{t\geq 1}$ and $(\delta_t)_{t\geq 1}$ be two sequences in $(0,\infty)$ such that  
\begin{align*}
\lim_{t\rightarrow\infty} \delta_t=0,\quad \lim_{t\rightarrow\infty} \frac{\alpha_t \log t}{t^{\alpha} \delta_t}=0,\quad\lim_{t\rightarrow\infty} t^{\alpha-1}\delta_t=0,\quad\lim_{t\rightarrow\infty}\alpha_t\log t=\infty.
\end{align*}
In addition, let  $\tilde{t}_{p}=\tilde{t}_{p-1}+ m\lceil \tilde{t}^\alpha_{p-1} \delta_{\tilde{t}_{p-1}}\rceil$ for all $p\geq 2$,  $(\Sigma_t)_{t\geq 1}$ be a sequence of random $d\times d$ symmetric and positive definite matrices such that $\P(\|\Sigma_t\|\in [c,1/c])=1$ for all $t\geq 1$,  and let  $(\beta_t)_{t\geq 1}$ and $(\gamma_t)_{t\geq 1}$ be two sequences of $[0,1]$-valued random variables, with $(\gamma_t)_{t\geq 1}$ such that $\P(\gamma_{\tilde{t}_p}\leq 1-c)=1$ for all $p\geq 1$. Finally, let $(h_t)_{t\geq 1}$ be a sequence in $[0,\infty)$ such that $h_t=\bigO(t^{-\alpha})$.  Then, if $(\mu_t)_{t\geq 1}$ is defined by
\begin{align*}
\mu_t(\dd\theta)=
\begin{cases}
\beta_t\delta_{\{0\}}(\dd\theta)+(1-\beta_t)\mathcal{N}_d(0,h_t^2\Sigma_t), &t\not\in (\tilde{t}_p)_{p\geq 1}\\
\gamma_t\delta_{\{0\}}(\dd\theta)+(1-\gamma_t)t_{d,\nu}(0, t^{-2\alpha_t}\Sigma_t), & t\in (\tilde{t}_p)_{p\geq 1}
\end{cases},\quad  t\geq 1 
\end{align*}
Conditions \ref{condition:Inf_mu}-\ref{condition:mu_extra} hold with $(\Gamma^\mu_\delta)_{\delta\in(0,\infty)}$,  $\Gamma^\mu$, $\big((t_p, v_p, f_p)\big)_{p\geq 1}$  and $(f'_t)_{t\geq 1}$ independent of the sequences $(\Sigma_t)_{t\geq 1}$, $(\beta_t)_{t\geq 1}$ and $(\gamma_t)_{t\geq 1}$.  In particular, \ref{condition:mu_seq} holds with $(t_p)_{p\geq 1}=(\tilde{t}_p)_{p\geq 1}$.
\end{proposition}
\begin{proof}
See Section \ref{p-prop:Gauss}.
\end{proof}

The following proposition shows that if the artificial dynamics vanish sufficiently quickly as $t\rightarrow\infty$ then the use of Student-t distributions is not needed.

\begin{proposition}\label{prop:Gauss2}

Let $\setR=\R^d$ for some $d\in\mathbb{N}$ and assume that $\Theta$ is a regular compact set.  Let $(\Sigma_t)_{t\geq 1}$ be a sequence of random $d\times d$ symmetric and positive definite matrices such that $\P(\|\Sigma_t\|\in [c,1/c])=1$ for all $t\geq 1$ and some constant $c\in (0,1)$, and let   $(\beta_t)_{t\geq 1}$ be a sequence  of $[0,1]$-valued random variables and $(h_t)_{t\geq 1}$ be a sequence in $[0,\infty)$ such that $h_t=\smallo(t^{-1})$. Then, if $(\mu_t)_{t\geq 1}$ is defined by
\begin{align*}
\mu_t(\dd\theta)=
\beta_t\delta_{\{0\}}(\dd\theta)+(1-\beta_t)\mathcal{N}_d(0,h_t^2\Sigma_t), \quad  t\geq 1 
\end{align*}
 Conditions \ref{condition:Inf_mu}-\ref{condition:Inf_K}, \ref{condition:mu_seq2} and \ref{condition:mu_extra} hold with $(\Gamma^\mu_{\delta})_{\delta\in(0,\infty)}$,  $\Gamma^\mu$, $\big((k_t, f_{t,m})\big)_{t\geq 1}$ and $(f'_t)_{t\geq 1}$ independent of $(\beta_t)_{t\geq 1}$ (for all $m\in\mathbb{N}_0$).

\end{proposition}
\begin{proof}
See Section \ref{p-prop:Gauss2}.
\end{proof}

\subsection{Artificial dynamics for models with discrete parameters\label{sub:discrete}}

 In this subsection we consider the case where $\setR=\mathbb{Z}^d$ for some $d\in\mathbb{N}$, so that $\Theta$ contains finitely many elements.

\begin{proposition}\label{prop:discrete}
Let $\setR=\mathbb{Z}^d$ for some $d\in\mathbb{N}$ and    $(a,b)\in\mathbb{Z}^2$ be such that $\Theta\subseteq\{a,a+1,\dots,b\}^d$. Let $(\alpha_1,\alpha_2)\in(0,\infty)^2$, $c\in (0,1)$ and $(\tilde{t}_1,m)\in\mathbb{N}^2$ be some constants,
and let $(\delta_t)_{t\geq 1}$ be  a sequence in $(0,\infty)$ such that $\lim_{t\rightarrow\infty}\delta_t=0$, $\lim_{t\rightarrow}(\delta_{t+1}/\delta_t)=1$, and  such that $\lim_{t\rightarrow\infty}\delta^{2}_t\log t=\infty$. In addition, for all $p\geq 2$ let $\tilde{t}_p=\tilde{t}_{p-1}+m\lceil \delta^{1/2}_{\tilde{t}_{p-1}}\log \tilde{t}_{p-1}\rceil$ and for all $t\geq 1$ let  $\{p_t(i)\}_{i=a-b}^{b-a}\subset[0,1]$ be such that $\sum_{i=a-b}^{b-a}p_t(i)=1$ and such that  $p_t(0)\geq 1-ct^{-\alpha_1}$ if  $t\not\in (\tilde{t}_p)_{p\geq 1}$ and such that we have both $p_{t}(0)\geq 1-ct^{-\alpha_2\delta_{t}}$ and  $p_{t}(i)\geq c t^{-(\alpha_2\delta_{\tilde{t}_p})^{3/4}}$ if $t\in (\tilde{t}_p)_{p\geq 1}$. Finally, let $(\beta_t)_{t\geq 1}$ and $(\gamma_t)_{t\geq 1}$ be two sequences of $[0,1]$-valued random variables, with $(\gamma_t)_{t\geq 1}$ such that $\P(\gamma_{\tilde{t}_p}\leq 1-c)=1$ for all $p\geq 1$. Then, if for all $t\geq 1$ we let $\mu_t=\otimes_{i=1}^d\tilde{\mu}_t$ with     $(\tilde{\mu}_t)_{t\geq 1}$   defined by
\begin{align*}
\tilde{\mu}_t(\{i\})=
\begin{cases}
\beta_t\ind_{\{0\}}(i)+(1-\beta_t)p_t(i), &t\not\in(\tilde{t}_p)_{p\geq 1}\\
\gamma_t\ind_{\{0\}}(i)+(1-\gamma_t)p_t(i), &t \in(\tilde{t}_p)_{p\geq 1}
\end{cases},\quad  i\in\{a-b,a-b+1,\dots,b-a\},
\end{align*}
Conditions \ref{condition:Inf_mu}-\ref{condition:mu_extra} hold with $(\Gamma^\mu_\delta)_{\delta\in(0,\infty)}$,  $\Gamma^\mu$, $\big((t_p, v_p, f_p)\big)_{p\geq 1}$  and $(f'_t)_{t\geq 1}$ independent of the sequences   $(\beta_t)_{t\geq 1}$ and $(\gamma_t)_{t\geq 1}$. In particular, \ref{condition:mu_seq} holds with $(t_p)_{p\geq 1}=(\tilde{t}_p)_{p\geq 1}$.
\end{proposition}
\begin{proof}
See Section \ref{p-prop:discrete}.
\end{proof}

The following lemma can be used to define, in practice, the sequence $(\tilde{t}_p)_{p\geq 1}$ and the   probabilities  $(\{p_t(i)\}_{i=a-b}^{b-a})_{t\geq 1}$ required by Proposition \ref{prop:discrete} to defined $(\mu_t)_{t\geq 1}$.
\begin{lemma}\label{lemma:discrete}
Let  $(\alpha_1,\alpha_2)\in(0,\infty)^2$, $\beta\in(0,1/2)$, $c \in(0,1)$, $(\tilde{t}_1,m)\in\mathbb{N}^2$ and for all $t\geq 1$ let $\alpha_t= (\log t)^{-\beta}$. Next, let  $\tilde{t}_p=\tilde{t}_{p-1}+m\lceil   (\log \tilde{t}_{p-1})^{1-\beta/2}\rceil$ for all $p\geq 2$ and, for all $t\geq 1$ let $h_t=c t^{-\alpha_1}$ if $t\not\in(\tilde{t}_p)_{p\geq 1}$ and $h_t=c t^{-\alpha_2\alpha_t}$ otherwise, and let $\eta_t$ be the probability measure on $\mathbb{Z}$ such that
\begin{align*}
V_t\sim \eta_t\Leftrightarrow V_t\dist (2I_t-1) B_t,\quad B_t\sim \mathrm{Binomial}(b-a, h_t),\quad I_t\sim\mathrm{Bernoulli}(0.5),\quad\forall t\geq 1.
\end{align*}
Finally, let $p_t(i)=\eta_t(\{i\})$ for all   $i\in\{a-b,a-b+2,\dots,b-a\}$ and all $t\geq 1$. Then, the sequence $(\tilde{t}_p)_{p\geq 1}$ and the   probabilities  $(\{p_t(i)\}_{i=a-b}^{b-a})_{t\geq 1}$  satisfy the assumptions of Proposition \ref{prop:discrete}.

\end{lemma}
 
 \subsection{Artificial dynamics for models with both continuous and discrete   parameters\label{sub:hybrid}}

We   finally consider the situation where, for some $d_1,d_2\in\mathbb{N}$, we have $\Theta=\Theta_1\times \Theta_2$  with $\Theta_1\subset\R^{d_1}$ a regular compact set and $ \Theta_2\subset\mathbb{Z}^{d_2}$, in which case $\setR=\R^{d_1}\times \mathbb{Z}^{d_2}$. 

By combining Proposition  \ref{prop:Gauss} and Proposition \ref{prop:discrete} we readily obtain the following result:
\begin{corollary}\label{cor:mix}
Let $\setR=\R^{d_1}\times \mathbb{Z}^{d_2}$ for some $(d_1,d_2)\in\mathbb{N}^2$ and let $\Theta=\Theta_1\times \Theta_2$  with $\Theta_1\subset\R^{d_1}$ a regular compact set and $ \Theta_2\subset\mathbb{Z}^{d_2}$. Let $(\mu_t^{(1)})_{t\geq 1}$ be as constructed  in Proposition \ref{prop:Gauss} (for $\Theta=\Theta_{d_1}$ and $d=d_1$) and $(\mu_t^{(2)})_{t\geq 1}$  be as constructed  in  Proposition \ref{prop:discrete} (for $\Theta=\Theta_2$ and $d=d_2$). Assume that these two sequences of random probability measures are defined using the same sequence $(\tilde{t}_p)_{p\geq 1}$ and  the same two sequences $(\beta_t)_{t\geq 1}$ and $(\gamma_t)_{t\geq 1}$ of $[0,1]$-valued random variables. Let  $\mu_t=\mu_t^{(1)}\otimes \mu_t^{(2)}$ for all $t\geq 1$. Then,  Conditions \ref{condition:Inf_mu}-\ref{condition:mu_extra} hold with $(\Gamma^\mu_\delta)_{\delta\in(0,\infty)}$,  $\Gamma^\mu$, $\big((t_p, v_p, f_p)\big)_{p\geq 1}$  and $(f'_t)_{t\geq 1}$ independent of the sequences   $(\beta_t)_{t\geq 1}$ and $(\gamma_t)_{t\geq 1}$. In particular, \ref{condition:mu_seq} holds with $(t_p)_{p\geq 1}=(\tilde{t}_p)_{p\geq 1}$.
\end{corollary}

To construct  $(\mu_t^{(1)})_{t\geq 1}$ and $(\mu_t^{(2)})_{t\geq 1}$ such that the assumptions of Corollary \ref{cor:mix} hold we can for instance proceed as follows. First, let  $(\mu_t^{(2)})_{t\geq 1}$ be as in Proposition \ref{prop:discrete}, with  $(\tilde{t}_p)_{p\geq 1}$ and the   probabilities  $(\{p_t(i)\}_{i=a-b}^{b-a})_{t\geq 1}$ as in  Lemma \ref{lemma:discrete} for some $\beta\in(0,1/2)$.  Then, with the sequences $(\beta_t)_{t\geq 1}$, $(\gamma_t)_{t\geq 1}$  and $(\tilde{t}_p)_{p\geq 1}$ used to define $(\mu_t^{(2)})_{t\geq 1}$, and with
$(h_t)_{t\geq 1}$ and $(\Sigma_t)_{t\geq 1}$ as in  Proposition \ref{prop:Gauss} (for $d=d_1$),  let $(\mu_t^{(1)})_{t\geq 1}$ be defined by
\begin{align*}
\mu^{(1)}_t(\dd\theta)=
\begin{cases}
\beta_t\delta_{\{0\}}(\dd\theta)+(1-\beta_t)\mathcal{N}_{d_1}(0,h_t^2\Sigma_t), &t\not\in (\tilde{t}_p)_{p\geq 1}\\
\gamma_t\delta_{\{0\}}(\dd\theta)+(1-\gamma_t)t_{d_1,\nu}(0, t^{-2\alpha_t}\Sigma_t), & t\in (\tilde{t}_p)_{p\geq 1}
\end{cases},\quad  t\geq 1
\end{align*}
with $\alpha_t=c(\log t)^{-\beta}$ for all $t\geq 1$ and some constant $c\in(0,\infty)$. It is readily checked that this sequence  $(\mu_t^{(1)})_{t\geq 1}$  satisfies the assumption of Proposition \ref{prop:Gauss}, and therefore the conclusion of Corollary holds for the so-defined sequences $(\mu_t^{(1)})_{t\geq 1}$ and $(\mu_t^{(2)})_{t\geq 1}$.

\subsection{Proofs of the results presented in this section}

\subsubsection{Proof of Proposition \ref{prop:Gauss}\label{p-prop:Gauss}}

\begin{proof} 

Remark first that since $\lim_{t\rightarrow\infty}h_t=\lim_{t\rightarrow\infty}t^{-\alpha_t}=0$ and $\Theta$ is assumed to be a regular compact set it is trivial to check that Conditions \ref{condition:Inf_mu} and  \ref{condition:Inf_K}   hold. Below we show that \ref{condition:mu_seq} holds with $(t_p)_{p\geq 1}=(\tilde{t}_p)_{p\geq 1}$ and that condition \ref{condition:mu_extra} is satisfied as well

To do so we start with some preliminary computations. First, noting that under the assumptions of the proposition we have $\lim_{p\rightarrow\infty}(t_{p+1}-t_p)/t_p=0$, we remark that  there exists a constant $C_1\in(0,\infty)$ such that
\begin{equation}\label{eq:var_limit}
\begin{split}
\limsup_{p\rightarrow\infty}(t_{p+2}-t_p)\sum_{i=t_p}^{t_{p+2}} h_s^2\leq C_1 \limsup_{p\rightarrow\infty}\delta_{t_{p}}^2=0
\end{split}
\end{equation}
and let $p'\in\mathbb{N}\setminus\{1\}$ be such that $t_{p+1}\geq t_p+2$ for all $p\geq p'$.

Next, using the definition of $(\mu_t)_{t\not\in(t_p)_{p\geq 1}}$, the assumptions on $(\Sigma_t)_{t\geq 1}$ and   Doob's martingale inequality, it is easily checked that there exists a constant $C_2\in(0,\infty)$  such that, $\P$-a.s., for all $\epsilon\in(0,\infty)$, $p\geq p'$ and integers $t_p\leq a<b<t_{p+1}$, we have
\begin{align*}
\P\Big(\exists s\in\{a+1,\dots,b\}:\, \sum_{i=a+1}^s  U_i\not\in B_{\epsilon}(0)\big|\F_{b}\Big)\leq  C_2\exp\bigg(-\frac{\epsilon^2}{C_2\sum_{i=t_p}^{t_{p+2}} h_s^2}\bigg).
\end{align*}
Together with \eqref{eq:var_limit}, this implies that, $\P$-a.s.,   for all $\epsilon\in(0,\infty)$, all $p\geq p'$ and all integers $t_p\leq a<b<t_{p+1}$, we have
\begin{equation}\label{eq:fp_pp1}
\begin{split}
\frac{\log \P\Big(\exists s\in\{a+1,\dots,b\}:\, \sum_{i=a+1}^s  U_i\not\in B_{\epsilon}(0)\big|\F_{b}\Big)}{b-a}   \leq g^{(1)}_{t_p}(\epsilon)
\end{split}
\end{equation}
where, for some constant   $C_3\in(0,\infty)$,
\begin{align}\label{eq:g_1}
g^{(1)}_{t_p}(\epsilon)=\min\bigg(0,-\frac{\epsilon^2}{C_3 \delta_{t_p}^2}+\log C_3\bigg),\quad\forall p\geq 1,\quad\forall\epsilon>0.
\end{align}

To proceed further note that we $\P$-a.s.~have, for all $\epsilon\in(0,\infty)$, all $p\geq p'$ and all integers $a$ and $b$ such that  $t_{p-1}\leq a<b<t_{p+1}$,
\begin{align*}
\P\Big(\exists  s\in\{a+1,\dots,&b\}  :\, \sum_{i=a+1}^s   U_i\not\in B_{\epsilon}(0)\big| \F_{b}\Big)\\
&=\P\Big(\exists s\in\{a+1,\dots,b\} :\, \Big\|\sum_{i=a+1}^s   U_i\Big\|\geq \epsilon\big| \F_{b}\Big)\\
&\leq \P\Big(\exists s\in\{a+1,\dots,b\} :\,  \Big\|\sum_{i=a+1}^s   U_i-\ind_{t_p}(s)U_{t_p}\Big\|+\|U_{t_p}\| \geq \epsilon\big| \F_{b}\Big)\\
&\leq \P\Big(\exists s\in\{a+1,\dots,b\} :\,  \Big\|\sum_{i=a+1}^s   U_i-\ind_{t_p}(s)U_{t_p}\Big\|\geq\epsilon/2\text{ or }\|U_{t_p}\| \geq \epsilon/2\big| \F_{b}\Big)\\
&\leq \P\Big(\exists s\in\{a+1,\dots,b\} :\,  \Big\|\sum_{i=a+1}^s   U_i-\ind_{t_p}(s)U_{t_p}\Big\|\geq\epsilon/2\Big| \F_b\Big)+\P\big(\|U_{t_p}\| \geq \epsilon/2\big).
\end{align*}
Using this result,  the definition of $(\mu_{t})_{t\geq 1}$,  the assumptions on $(\Sigma_t)_{t\geq 1}$ and  Doob's martingale inequality, as well as the fact that
\begin{align*}
\int_a^\infty \Big(1+\frac{x^2}{\nu}\Big)^{-\frac{\nu+1}{2}}\dd x\leq \nu^{\frac{\nu+1}{2}}\int_a^\infty x^{-(\nu+1)}\dd x=\nu^{\frac{\nu-1}{2}} a^{-\nu},\quad\forall a>0,
\end{align*}
it is readily checked that  there exists a constant $C_4\in(0,\infty)$  such that we $\P$-a.s.~have, for all $\epsilon\in(0,\infty)$, all $p\geq p'$ and all integers $a$ and $b$ such that  $t_{p-1}\leq a<b<t_{p+1}$,  
\begin{equation}\label{eq:Doob}
\begin{split}
\P\Big(\exists s\in\{a+1,\dots,b\} :\, \sum_{i=a+1}^s&  U_i\not\in B_{\epsilon}(0)\big| \F_{b}\Big)\leq C_4\,\exp\Big(-\frac{\epsilon^2}{C_4\sum_{i=t_{p-1}}^{t_{p+1}} h_s^2}\Big)+C_4\Big(\frac{t_p^{-\alpha_{t_p}}}{\epsilon}\Big)^\nu.
\end{split}
\end{equation}
Together with \eqref{eq:var_limit},  \eqref{eq:Doob}  implies that, $\P$-a.s., for all $\epsilon\in(0,\infty)$, all $p\geq p'$ and all integers $a$ and $b$ such that  $t_{p-1}<a<b<t_{p+1}$, we have
\begin{equation}\label{eq:fp_pp2}
\begin{split}
\frac{\log \P\Big(\exists s\in\{a+1,\dots,b\} :\, \sum_{i=a+1}^s  U_i\not\in B_{\epsilon}(0)\big| \F_{b}\Big)}{b-a}\leq g_{t_p}^{(2)}(b-a,\epsilon),\,\,\forall p\geq 1,\,\,\forall\epsilon>0
\end{split}
\end{equation}
where, for some constant $C_6\in(0,\infty)$,
\begin{align}\label{eq:g_2}
g_{t_p}^{(2)}(\Delta,\epsilon)=\min\bigg(0, \log C_6 - \min\bigg(\frac{\epsilon^2}{C_6\delta_{t_{p-1}}^2},\frac{\alpha_{t_p}\log t_p+\log\epsilon}{C_6\Delta}\bigg)\bigg),\quad\forall p\geq 1,\quad\forall(\Delta,\epsilon)\in(0,\infty)^2.
\end{align}

We are now in position to establish that Conditions \ref{condition:mu_seq}-\ref{condition:mu_extra} hold.

First, Condition \ref{condition:mu_extra} directly follows from \eqref{eq:fp_pp2}-\eqref{eq:g_2}.

To show that \ref{condition:mu_seq} holds   remark first that, as required by \ref{condition:mu_seq},  $(t_p)_{p\geq 1}$ is such that $\lim_{p\rightarrow\infty}(t_{p+1}-t_p)=\infty$ while $t_{p+1}> t_p$ for all $p\geq 1$. Next, let $(v_p)_{p\geq 1}$ be a sequence in $\mathbb{N}$ such that $v_p\rightarrow\infty$ sufficiently slowly so that
\begin{align*}
\limsup_{p\rightarrow\infty}\frac{v_{p+1}}{v_p}<\frac{3}{2},\quad \lim_{p\rightarrow\infty}\frac{\alpha_{t_p}\log t_p}{v_p}=\infty,\quad=\lim_{p\rightarrow\infty}\frac{v_p}{t_{p}-t_{p-1}}=0.
\end{align*}
Remark that such a sequence $(v_p)_{p\geq 1}$ exists under the assumptions of the proposition and is, as requested by \ref{condition:mu_seq}, such that $\lim_{p\rightarrow\infty} v_p/(t_{p+1}-t_p)=0$.

Next, let $p''\in\mathbb{N}$ be such that $t_{p-1}\leq t_p-3v_p<t_p+v_p<t_{p+1}$ for all $p\geq p''$. Then, for any sequences $(l_p)_{p\geq 1}$ and $(u_p)_{p\geq 1}$   in $\mathbb{N}$ such that $t_p-3v_p\leq l_p<u_p\leq t_{p}+v_p$ for all $p\geq p'''$, it follows from  \eqref{eq:fp_pp2}-\eqref{eq:g_2} that we $\P$-a.s.~have, for all $p\geq p''$ and all $\epsilon\in(0,\infty)$,
\begin{equation}\label{eq:fp_pp22}
\begin{split}
\frac{\log \P\Big(\exists s\in\{l_p+1,\dots,u_p\} :\, \sum_{i=l_p+1}^s  U_i\not\in B_{\epsilon}(0)\big| \F_{u_p}\Big)}{u_p-l_p}\leq g_p^{(2)}(u_p-l_p,\epsilon)\leq g_p^{(2)}(4v_p+1,\epsilon).
\end{split}
\end{equation}

Next, for all $t\geq 1$ we let $p_{t}(\cdot)$ denote  the density of the $t_{d,\nu}(0, t^{-2\alpha_t}\Sigma_t)$ distribution w.r.t.~$\dd\theta$, the Lebesgue measure on $\R^d$, and note that, for all $\theta'\in\Theta$, $A\in\mathcal{T}$ and $\epsilon\in (0,\infty)$,
\begin{align*}
K_{\mu_{t_{p}}}\big(\theta',\mathcal{N}_\epsilon(A)\big)\geq (1-\gamma_{t_p})\int_{\mathcal{N}_\epsilon(A)} p_{t_p}(\theta'-\theta)\dd \theta.
\end{align*}
Then, using the definition of $p_{t_p}(\cdot)$, the assumptions on $(\Sigma_t)_{t\geq 1}$ and the fact that $\Theta$ is assumed to be a  regular compact set, it is easily verified that there exist a constant  $c_1 \in(0,1)$ such that we $\P$-a.s.~have, for all $p\geq 1$, $A\in\mathcal{T}$ and $\epsilon\in(0,\infty)$, 
\begin{align*}
\inf_{\theta'\in\Theta} K_{\mu_{t_{p}}}(\theta',\mathcal{N}_\epsilon(A))    \geq (1-\gamma_{t_p})c_1   f(\epsilon) t_p^{-\nu\alpha_{t_p}}.
\end{align*}
for some  continuous function $f:[0,\infty)\rightarrow[0,\infty)$ such that $\lim_{x\downarrow 0}f(x)=0$.

Therefore, recalling that by assumption we have   $\P(1-\gamma_{t_p}\geq c)=1$ for all $p\geq 1$ and   some constant $c\in(0,1)$, it follows that there exist a constant $c_2\in (-\infty,0)$ such that we $\P$-a.s.~have,  for all $p\geq 1$ and $\epsilon\in(0,\infty)$,
\begin{align}\label{eq:fp_pp3}
\frac{1}{t_{p+1}-t_{p}}\,\inf_{(\theta,A)\in\Theta\times\mathcal{T}}\log   K_{\mu_{t_{p}}}(\theta,\mathcal{N}_\epsilon(A))&\geq \frac{c_2+ \log f(\epsilon)-\nu\,\alpha_{t_p} \log t_p}{t_{p+1}-t_{p}}=:\psi_p(\epsilon).
\end{align}

To conclude the proof note that 
\begin{align*}
\lim_{p\rightarrow\infty}g^{(1)}_p(\epsilon)=\lim_{p\rightarrow\infty}g_p^{(2)}(4v_p+1,\epsilon)=-\infty,\quad \lim_{p\rightarrow\infty}\psi_p(\epsilon)=0,\quad\forall \epsilon\in(0,\infty)
\end{align*}
and therefore, it follows from \eqref{eq:fp_pp1}, \eqref{eq:fp_pp22} and  \eqref{eq:fp_pp3} that conditions \ref{C41}-\ref{C43} hold, for instance,  for the sequence of functions $(f_p)_{p\geq 1}$ such that 
\begin{align*}
f_p(\epsilon)=\max\bigg(-\frac{1}{g_p^{(1)}(\epsilon)}, -\frac{1}{g_p^{(2)}(4v_p+1,\epsilon)}, -\psi_p(\epsilon)\bigg),\quad\forall\epsilon\in(0,\infty),\quad\forall p> \max\{p',p'',p'''\}
\end{align*}
and such that $f_p(\epsilon)=\infty$ for  all $\epsilon\in(0,\infty)$ and all $p\in\{1,\dots,\max\{p',p'',p'''\}\}$.
The proof of the lemma is complete.

\end{proof}

\subsubsection{Proof of Proposition \ref{prop:Gauss2}\label{p-prop:Gauss2}}

In what follows  we will need the following result, due to \citetsup[][Lemma 6.2]{smallball1994}:  
 \begin{lemma}\label{lemmaB:smallball}  
Let $(W_s)_{s\in[0,\infty)}$ be a centred Gaussian process. Assume that there exists a non-negative and non-decreasing function $\sigma:(0,\infty)\rightarrow\R$  such that
\begin{align*}
\E[W_s-W_u]^2 \leq \sigma(|s-u|)^2,\quad \forall s,u\in [0,\infty)
\end{align*}
and such that, for some constant $\alpha\in(0,\infty)$ and $\gamma\in(1,\infty)$, we have
\begin{align*}
\sigma(as)&\leq \gamma a^\alpha \sigma(s),\quad \forall (a,s)\in(0,1)\times (0,\infty)  \\
\sigma(ks)&\leq k^2 \sigma(s),\quad \hspace{0.23cm} \forall (k,s)\in\mathbb{N}\times (0,\infty).
\end{align*}
Then, there exists a constant $d_\alpha\in(0,\infty)$, depending only on $\alpha$,   such that
\begin{align*}
\P\Big(\sup_{s\in[0,t]} |W_s - W_0| \leq \theta \sigma(\delta) \Big) \geq \exp\Big(-\frac{d_\alpha t}{\delta}\Big),\quad \forall \delta\in (0,t),\quad\forall t\in (0,\infty).
\end{align*}
\end{lemma}

From Lemma \ref{lemmaB:smallball} we obtain the following corollary that will be used to prove Proposition \ref{prop:Gauss2}:

\begin{corollary}\label{corB:small_balls}
Let $t\in\mathbb{N}$, $\{Z_s\}_{s=1}^t$ be a collection of i.i.d.~$\mathcal{N}_1(0,1)$ random variables  and $\{h_s\}_{s=1}^t\subset(0,\bar{h})$ for some constant $\bar{h}\in (0,\infty)$. Then, for all $a\in\{1,\dots,t\}$,
\begin{align*}
\P\Big(\max_{s\in\{a,\dots,t\}} \big|\sum_{i=a}^s h_i Z_i\big|\leq \epsilon\Big)\geq \exp\Big(-\frac{d_{1/2} \bar{h}^2\,t}{\epsilon^2}\Big),\quad\forall\epsilon\in \big(0,\bar{h} t^{1/2}\big)
\end{align*}
with $d_{1/2}\in(0,\infty)$ as in Lemma \ref{lemmaB:smallball}.
\end{corollary}
\begin{proof}
Let $\tilde{W}_s=\sum_{i=a}^s h_i Z_i$ for all $s\in\{a,\dots,t\}$  and note that, for all $a\leq s<u\leq t$ (and using the convention that empty sums equal zero)
\begin{align}\label{eqB:sigmaW}
\E\big[(\tilde{W}_u-\tilde{W}_s)^2\big]=\sum_{i=s+1}^{u}h_s^2\leq \bar{h}^2(u-s).
\end{align}
Therefore, letting $\tilde{\sigma}:(0,\infty)\rightarrow\R$ be defined by $\tilde{\sigma}(s)=\bar{h} s^{1/2}$, $s\geq 0$, it follows from \eqref{eqB:sigmaW} that
\begin{align*}
\E\big[(\tilde{W}_u-\tilde{W}_s)^2\big]\leq \tilde{\sigma}(|u-s|)^2,\quad\forall s,u\in\{a,\dots,t\}.
\end{align*}
Next, let $(W_\tau)_{\tau\in[0,\infty)}$ be a centred Gaussian process such that 
\begin{align*}
 (W_0,\dots, W_{t-a})\dist (\tilde{W}_{a},\dots,\tilde{W}_t) 
\end{align*}
and such that 
\begin{align*}
\E\big[(W_s-W_u)^2\big]\leq \tilde{\sigma}(|u-s|)^2,\quad\forall (s,u)\in (0,\infty).
\end{align*}
In addition, for all  $(k,a,s)\in\mathbb{N}\times(0,1)\times (0,\infty)$ we have
\begin{align*}
\tilde{\sigma}(a s)=a^{1/2} \bar{h} s^{1/2}=a^{1/2}\tilde{\sigma}(s),\quad \tilde{\sigma}(ks)=k^{1/2} \bar{h} s^{1/2}=k^{1/2}\tilde{\sigma}(s)\leq k^2\sigma(s).
\end{align*}
Therefore, for all $\epsilon\in \big(0,\bar{h} t^{1/2}\big)$, by applying Lemma \ref{lemmaB:smallball} with $\gamma=1$, $\alpha=1/2$, $\sigma=\tilde{\sigma}$ and $\delta=\epsilon^2/\bar{h}^2\in(0,t)$, we have
\begin{align*}
\P\Big(\max_{s\in\{a,\dots,t\}} \big|\sum_{i=a}^s h_i Z_i\big|\leq \epsilon\Big)&=\P\big(\max_{s\in\{0,\dots,t-a\}} \big|W_s\big|\leq \epsilon\big)\geq \P\big(\sup_{s\in[0,t-a]} \big|W_s\big|\leq \epsilon\big)\geq \exp\Big(-\frac{d_{1/2} \bar{h}^2 t}{\epsilon^2}\Big)
\end{align*}
and the proof of the lemma is complete.
\end{proof}

We are now in position to prove Proposition \ref{prop:Gauss2}:

\begin{proof}[Proof of Proposition \ref{prop:Gauss2}]

Since $\lim_{t\rightarrow\infty} h_t=0$ and $\Theta$ is assumed to be a regular compact set Conditions \ref{condition:Inf_mu} and \ref{condition:Inf_K} trivially  holds. In addition, using Doob's martingale inequality and the assumptions on $(h_t)_{t\geq 1}$ it is direct to see that \ref{condition:mu_extra} holds,  and thus below we only show that \ref{condition:mu_seq2} is satisfied. To do so we assume that $h_t>0$ for infinitely many  $t\in\mathbb{N}$, since  otherwise \ref{condition:mu_seq2} trivially holds.

For all $t\geq 1$ let $\delta_t=\max\{h_s,\,s\geq t \}$ and note  that the resulting sequence $(\delta_t)_{t\geq 1}$ is strictly positive, non-increasing, and such that we have both $\lim_{t\rightarrow\infty} t\delta_t=0$ and   $h_t\leq \delta_t$ for all $t\geq 1$. Next, let  $(k_t)_{t\geq 1}$ be a sequence in $\mathbb{N}$ such that $\lim_{t\rightarrow\infty} k_t/t=\lim_{t\rightarrow\infty}(1/k_t)=0$ and such that    $\lim_{t\rightarrow\infty} t \delta_{k_t}=0$.  Remark that such a sequence $(k_t)_{t\geq 1}$ exists since for all $t\geq 1$ we have $t\delta_{k_t}= (t/k_t)(k_t\delta_{k_t})$ where $\lim_{t\rightarrow\infty}k_t\delta_{k_t}=0$. In what follows we let $t'\in\mathbb{N}$ be such that $t\geq 2k_t+1$ for all $t\geq t'$.

Then, using the definition of $(\mu_t)_{t\geq 1}$, the assumptions on $(\Sigma_t)_{t\geq 1}$ and  Doob's Martingale inequality, it is easily checked that there exists a constant $C_1\in(0,\infty)$ such that we $\P$-a.s.~have, for all $t\geq t'$, $\epsilon\in(0,\infty)$ and integer $l_t\in\{k_t,\dots,2k_t\}$,
\begin{equation}\label{eq:prop2_p1}
\begin{split}
\frac{1}{t-l_t}\log\P\Big(\exists s\in\{l_t+1,\dots,t\}:\, &\sum_{i=l_t+1}^s  U_i\not\in B_{\epsilon}(0)\big|\,\F_{t}\Big)\\
&\leq \frac{t}{t-l_t}\min\bigg(0, \frac{\log C_1}{t}-\frac{\epsilon^2}{C_1\,t\sum_{s=k_t+1}^t h_s^2}\bigg)\\
&\leq g^{(1)}_t(\epsilon).
\end{split}
\end{equation}
where
\begin{align*}
g^{(1)}_t(\epsilon)= \frac{t}{t-k_t}\min\bigg(0,\frac{\log C_1}{t}-\frac{\epsilon^2}{C_1\,t^2 \delta_{k_t}^2}\bigg),\quad\forall t\geq 1,\quad\forall \epsilon>0.
\end{align*}

To proceed further we assume without loss of generality that $\|\cdot\|$ is the supremum norm, and we let $(\tilde{Z}_t)_{t\geq 1}$ be a sequence of independent $\mathcal{N}_d(0, I_d)$ random variables and  $(Z_t)_{t\geq 1}$ be a sequence of independent $\mathcal{N}_1(0, 1)$ random variables.

Next, let $a\in\mathbb{N}$ and note that, using the definition of $(\mu_t)_{t\geq 1}$ and the assumptions on $(\Sigma_t)_{t\geq 1}$, there exists a   constant $c\in(0,1]$ such that we $\P$-a.s.~have,  for all $t\geq a$ and $\epsilon\in(0,\infty)$,
\begin{equation}\label{eq:lower_p00}
\begin{split}
\P\Big( \sum_{i=a}^s U_i\in B_\epsilon(0),\,\forall s\in\{a,\dots,t\}\big|\F_t\Big)&=\P\Big(\max_{s\in\{a,\dots,t\}}\big\|\sum_{i=a}^s U_i\big\|_\infty< \epsilon\big|\F_t\Big)\\
&\geq \P\Big(\max_{s\in\{a,\dots,t\}}\big\|\Sigma_t^{1/2}\sum_{i=a}^s h_i Z_i\big\|_\infty< \epsilon\big|\F_t\Big)\\
&\geq \P\Big(\max_{s\in\{a,\dots,t\}}\big\|\sum_{i=a}^s h_i Z_i\big\|_\infty< c\,\epsilon\Big)\\
&=\P\Big(\max_{s\in\{a,\dots,t\}}\big|\sum_{i=a}^s h_i Z_i\big|< c\,\epsilon\Big)^d.
\end{split}
\end{equation}

We now let $\bar{h}=\max\{h_t,\,t\geq 1\}$, $\epsilon\in(0,\infty)$ and  $k_{a,\epsilon}\in\mathbb{N}$ be such that $k_{a,\epsilon}\geq a \vee (\epsilon/\bar{h})^2$. Then, for all $t>k_{a,\epsilon}$ we have
\begin{equation}\label{eq:lower_p0}
\begin{split}
&\P\Big(\max_{s\in\{a,\dots,t\}}|\sum_{i=a}^s h_i Z_i|\leq c\, \epsilon\Big)\\
&\geq \P\Big(\max_{s\in\{a,\dots,t\}}|\sum_{i=a}^s h_i Z_i|\leq c\, \epsilon\,\big| \max_{s\in\{a,\dots,k_{a,\epsilon}\}}|\sum_{i=a}^s h_i Z_i|\leq c\,\frac{\epsilon}{2}\Big)\P\Big(\max_{s\in\{a,\dots,k_{a,\epsilon}\}}|\sum_{i=a}^s h_i Z_i|\leq c\,\frac{\epsilon}{2}\Big)
\end{split}
\end{equation}
where, by Corollary \ref{corB:small_balls},
\begin{align}\label{eq:lower_p1}
\P\Big(\max_{s\in\{a,\dots,k_{a,\epsilon}\}}|\sum_{i=a}^s h_i Z_i|\leq c\,\frac{\epsilon}{2}\Big)\geq\exp\Big(-\frac{4 d_{1/2} \bar{h}^2 k_{a,\epsilon}}{(c\,\epsilon)^2}\Big)
\end{align}
with the constant $d_{1/2}\in(0,\infty)$ as in the statement of Corollary \ref{corB:small_balls}.

On the other hand, for all $t>k_{a,\epsilon}$ we have
\begin{equation}\label{eq:lower_p2}
\begin{split}
\P\Big(\max_{s\in\{a,\dots,t\}}|\sum_{i=a}^s h_i Z_i|\leq  & c\, \epsilon\Big| \max_{s\in\{a,\dots,k_{a,\epsilon}\}}|\sum_{i=a}^s h_i Z_i|\leq c\,\frac{\epsilon}{2}\Big)\\
&=\P\Big(\max_{s\in\{k_{a,\epsilon}+1,\dots,t\}}|\sum_{i=a}^s h_i Z_i|\leq  c\, \epsilon\Big| \max_{s\in\{a,\dots,k_{a,\epsilon}\}}|\sum_{i=a}^s h_i Z_i|\leq c\,\frac{\epsilon}{2}\Big)\\
&\geq \P\Big(\max_{s\in\{k_{a,\epsilon}+1,\dots,t\}} |\, \sum_{i=k_{a,\epsilon}+1}^s h_i Z_i|\leq c\,\frac{\epsilon}{2}\Big)\\
&\geq 0\vee\bigg(1-C\exp\Big(-\frac{\epsilon^2}{C\sum_{s=k_{a,\epsilon}+1}^t h_s^2}\Big)\bigg)
\end{split}
\end{equation}
where the last inequality holds by Doob's martingale  inequality and for some  constant $C\in(0,\infty)$ independent of $\epsilon$.

Assuming that $k_{a,\epsilon}$ is sufficiently large, for all $t>k_{a,\epsilon}$ we have
\begin{align*}
C\exp\Big(-\frac{\epsilon^2}{C\sum_{s=k_{a,\epsilon}+1}^t h_s^2}\Big)\leq \frac{1}{2}
\end{align*}
which, together with   \eqref{eq:lower_p00}-\eqref{eq:lower_p2}, shows that, $\P$-a.s.,
\begin{equation}\label{eq:prop2_p2}
\frac{\log \P\Big( \sum_{i=a}^s U_i\in B_\epsilon(0),\,\forall s\in\{a,\dots,t\}\big|\F_t\Big)}{t-a+1}\geq g_{a,t}^{(2)}(\epsilon),\quad\forall t>k_{a,\epsilon},\quad\forall\epsilon\in(0,\infty)
\end{equation}
where 
\begin{align*}
 g_{a,t}^{(2)}(\epsilon)=\frac{d}{t-a+1}\bigg(-\frac{4 d_{1/2} \bar{h}^2 k_{a, \epsilon}}{(c\,\epsilon)^2}-\log 2\bigg), \quad\forall \epsilon\in(0,\infty). 
\end{align*}

To conclude the proof note that for all $m\in\mathbb{N}_0$ we have
\begin{align*}
\lim_{t\rightarrow\infty}g^{(1)}_t(\epsilon)=-\infty,\quad \lim_{t\rightarrow\infty}g^{(2)}_{m+1,t}(\epsilon)=0,\quad\forall \epsilon\in(0,\infty)
\end{align*}
and therefore, using \eqref{eq:prop2_p1} and  \eqref{eq:prop2_p2} with $a=m+1$, it follows that for all $m\in\mathbb{N}_0$ condition \ref{condition:mu_seq2}  holds  for the sequence of functions $(f_{t,m})_{t\geq 1}$ such that, for all $\epsilon\in(0,\infty)$,
\begin{align*}
f_{t,m}(\epsilon)=\max\bigg(-\frac{1}{g_t^{(1)}(\epsilon)},-g^{(2)}_{m+1,t}(\epsilon)\bigg),\quad\forall t> \max(t',k_{m+1,\epsilon})
\end{align*}
and such that $f_{t,m}(\epsilon)=\infty$ for  all $t\in\{1,\dots,\max\{t', k_{m+1,\epsilon}\}\}$.
The proof of the proposition is complete.
\end{proof}

\subsubsection{Proof of Proposition \ref{prop:discrete}\label{p-prop:discrete}}

\begin{proof}

Since $\lim_{t\rightarrow\infty}p_t(\{0\})=1$ Conditions \ref{condition:Inf_mu} and \ref{condition:Inf_K} trivially  hold, and thus below we only show that \ref{condition:mu_seq} holds with $(t_p)_{p\geq 1}=(\tilde{t}_p)_{p\geq 1}$ and that \ref{condition:mu_extra} holds. It is also trivial to see that $(\mu_t)_{t\geq 1}$ is such that to prove the result of the proposition it is enough to consider the case where $d=1$, and thus below we assume that $d=1$.

To prove the proposition for $d=1$ we let
\begin{align*}
h_t=
\begin{cases}
c t^{-\alpha_1}, &t\not\in(t_p)_{p\geq 1}\\
c t^{-\alpha_2\delta_t}, &t\in(t_p)_{p\geq 1}
\end{cases},\quad p_t=\mu_t(\{0\}),\quad \Delta_t=\delta^{1/4}_t \log (1/h_t),\quad \kappa_t=(h_t\Delta_t)^{1/\Delta_t},\quad\forall t\geq 1 
\end{align*}
and start with some preliminary computations.

First, we remark that $\lim_{t\rightarrow\infty}\kappa_t=\lim_{t\rightarrow\infty}(1/\Delta_t)=\lim_{t\rightarrow\infty}h_t \Delta_t=0$ and we let  $t'\in\mathbb{N}$ be such that $h_t\Delta_t<1$, $\alpha_2\delta_t\leq \alpha_1$, and such that $\Delta_t\geq 1$ for all $t\geq t'$. Then, using Bernoulli's inequality, we have
\begin{align*}
(1-h_t)^{\Delta_t}\geq 1-h_t\Delta_t=1-\kappa_t^{\Delta_t},\quad\forall t\geq t'
\end{align*} 
and thus
\begin{align}\label{eq:use_htbound2}
\Big(1-(1-h_t)^{\Delta_t}\Big)^{\frac{1}{\Delta_t}}\leq \kappa_t,\quad\forall t\geq t'.
\end{align}

Next, remark that that under the assumptions of the proposition there exists a $p'\in\mathbb{N}$   such that $(t_{p+1}-t_p)\leq \Delta_{t_p+1}$ for all $p\geq p'$. We assume henceforth that $p'$ is sufficiently large so that  $t\geq t'$ for all $t\geq t_{p'}$.

Then, $\P$-a.s., for all $p>p'$, integers $a$ and $b$ such that $t_p\leq a<b<t_{p+1}$ and $\epsilon\in(0,1)$, we have
\begin{equation}\label{eq:discrete}
\begin{split}
\frac{1}{b-a}\log \P\Big(\exists s\in\{a+1,\dots,b\} :\, \sum_{i=a+1}^s & U_i\not\in B_{\epsilon}(0)\,\big|\F_{b}\Big)\\
&=\frac{1}{b-a}\log\Big(1-\prod_{s=a+1}^{b} p_s\Big)\\
&\leq \log\Big( \big(1-(1-h_{a})^{b-a}\big)^{\frac{1}{b-a}}\Big)\\
&\leq \log\Big( \big(1-(1-h_{t_p+1})^{t_{p+1}-t_p}\big)^{\frac{1}{t_{p+1}-t_p}}\Big)\\
&\leq \log\Big( \big(1-(1-h_{t_p+1})^{\Delta_{t_p+1}}\big)^{\frac{1}{\Delta_{t_p+1}}}\Big)
\end{split}
\end{equation}
where the last  two inequalities use  the fact that the function $x\mapsto (1-a^x)^{1/x}$ is increasing on $(0,\infty)$ for all $a\in(0,1)$. Together with \eqref{eq:use_htbound2}, this shows that, $\P$-a.s., for all $p>p'$, integers $a$ and $b$ such that $t_p\leq a<b<t_{p+1}$ and $\epsilon\in(0,1)$, we have
\begin{align}\label{eq:prop3_p1}
\frac{\log \P\Big(\exists s\in\{a+1,\dots,b\}:\, \sum_{i=a+1}^s  U_i\not\in B_{\epsilon}(0)\big|\,\F_{b}\Big)}{b-a}\leq \log( \kappa_{t_p+1}).
\end{align}

In addition, using similar computations, it follows that we $\P$-a.s.~have, for all $p>p'$, integers $a$ and $b$ such that $t_{p-1}\leq a<b<t_{p+1}$ and $\epsilon\in(0,\infty)$, 
\begin{equation}\label{eq:discrete2}
\begin{split}
\frac{1}{b-a}\log\P\Big(\exists s\in\{a+1,\dots,b\}:\,  \sum_{i=a+1}^s  &U_i\not\in B_{\epsilon}(0)\,\big|\F_{b}\Big)\leq\log\Big( (1-(1-h_{t_{p-1}}^{b-a})^{\frac{1}{b-a}}\Big).
\end{split}
\end{equation}

We are now in position to prove the proposition.

First, using \eqref{eq:discrete2} and the fact that the function $x\mapsto (1-a^x)^{1/x}$ is increasing on $(0,\infty)$ for all $a\in(0,1)$, we readily obtain that \ref{condition:mu_extra} holds.

To show that \ref{condition:mu_seq} holds   remark first that, as required by \ref{condition:mu_seq},  $(t_p)_{p\geq 1}$ is such that $\lim_{p\rightarrow\infty}(t_{p+1}-t_p)=\infty$ while $t_{p+1}> t_p$ for all $p\geq 1$. Next, let $(v_p)_{p\geq 1}$ be a sequence in $\mathbb{N}$ such that $v_p\rightarrow\infty$ sufficiently slowly so that
\begin{align*}
\limsup_{p\rightarrow\infty}\frac{v_{p+1}}{v_p}<\frac{3}{2},\quad \lim_{p\rightarrow\infty}\frac{\Delta_{t_{p-1}}}{v_p}=\infty,\quad=\lim_{p\rightarrow\infty}\frac{v_p}{t_{p}-t_{p-1}}=0.
\end{align*}
Remark that such a sequence $(v_p)_{p\geq 1}$ exists under the assumptions of the proposition and is, as requested by \ref{condition:mu_seq}, such that $\lim_{p\rightarrow\infty} v_p/(t_{p+1}-t_p)=0$. 

Let $p''\in\mathbb{N}$ be such that
\begin{align*}
4 v_p\leq \Delta_{t_{p-1}},\quad t_p-3v_p\geq t_{p-1},\quad t_p+v_p<t_{p+1},\quad h_{t_{p-1}}\leq \min\{h_t\}_{t=t_p-3v_p}^{t_p+v_p},\quad\forall p\geq p''.
\end{align*}
 Then,  $\P$-a.s., for all $p\geq p''$, all integers $l_p$ and $u_p$ such that $t_p-3v_p\leq l_p<u_p\leq t_{p}+v_p$ and all $\epsilon\in(0,1)$, we have
\begin{equation}\label{eq:discrete22}
\begin{split}
\frac{1}{u_p-l_p}\log\P\Big(\exists s\in\{l_p+1,\dots,u_p\}:\,  \sum_{i=l_p+1}^s  &U_i\not\in B_{\epsilon}(0)\,\big|\F_{u_p}\Big)\\
&\leq\log\Big( (1-(1-h_{t_{p-1}}^{u_p-l_p})^{\frac{1}{u_p-l_p}}\Big)\\
&\leq\log\Big( \big(1-(1-h_{t_{p-1}}^{\Delta_{t_{p-1}}}\big)^{\frac{1}{\Delta_{t_{p-1}}}}\Big)\\
&\leq \log\big(\kappa_{t_{p-1}}\big)
\end{split}
\end{equation}
where the first two inequalities use the fact that the function $x\mapsto (1-a^x)^{1/x}$ is increasing on $(0,\infty)$ while the last inequality holds by \eqref{eq:use_htbound2}. 

Next, remark that under the assumptions of the proposition there exists a $p'''\in\mathbb{N}$  such that, with $c\in(0,1)$ and $m\in\mathbb{N}$ as in the statement of the proposition,
\begin{align}\label{eq:discrete3}
\frac{1}{t_{p+1}-t_p}\log  \min_{i\in\{a-b,\dots,b-a\}}\mu_{t_p}(\{i\})&\geq \frac{\log(1-c)-  (\alpha_2\delta_{t_p})^{3/4}  \log  t_{p}}{2m\delta^{1/2}_{t_p} \log  t_{p}+2m},\quad\forall p\geq p''',\quad\P-a.s.
\end{align}

Therefore, recalling that $\lim_{t\rightarrow\infty}\delta_t=\lim_{t\rightarrow\infty}\kappa_t=0$, it follows from \eqref{eq:prop3_p1}, \eqref{eq:discrete22} and \eqref{eq:discrete3} that  \ref{condition:mu_seq} holds for  the sequence of functions $(f_p)_{p\geq 1}$ such that, for all $\epsilon\in(0,\infty)$, 
\begin{align*}
f_p(\epsilon)=
\begin{cases}
\max\bigg(-\frac{1}{\min\big(0,\log (\gamma_{t_p+1})\big)}, -\frac{1}{\min\big(0,\log (\gamma_{t_{p-1}})\big)},\frac{ (\alpha_2\delta_{t_p})^{3/4} \log  t_{p}-\log(1-c)}{2m\delta^{1/2}_{t_p}\log t_p+2m}\bigg), &p> \max(p',p'', p''')\\
\infty, &p\leq \max(p',p'',p'').
\end{cases}
\end{align*}
The proof is complete.
\end{proof}

\section{Proofs of the results in Appendix \ref{sec:theory}\label{p-sec:theory}}

\subsection{Additional notation}

Let $t_1$ and $t_2$ be two strictly positive integers. Then, we let
\begin{align}\label{eq:Dst}
D_{t_1:t_2}=\prod_{s=t_1}^{t_2}\sup_{(\theta,x,x')\in\Theta\times\setX^2} G_{s,\theta}(x,x'). 
\end{align}
with the convention $D_{t_1:t_2}=1$  if $t_1>t_2$. For a    random probability measure  $\zeta$ on $(\Theta\times\setX, \mathcal{T}\otimes\mathcal{X}$), we denote by  $\E_{t_1:t_2}^{\zeta}[\cdot]$ the expectation operator under the random probability measure
\begin{align}\label{eq:eta_distribution}
\zeta_{t_1:t_2}:=\zeta\big(\dd (\theta_{t_1-1},x_{t_1-1})\big)\prod_{s=t_1}^{t_2}m_{s,\theta_{s}}(x_s|x_{s-1})\lambda(\dd x_s)K_{\mu_s|\Theta}(\theta_{s-1},\dd \theta_s)
\end{align}
using the convention that $\zeta_{t_1:t_2}=\zeta$ if $t_1>t_2$. Next, letting $\big( (\vartheta^{\eta}_{t_1-1}, X^{\eta}_{t_1-1}),\dots,(\vartheta^{\eta}_{t_2},   X^{\eta}_{t_2})\big)$ be an $(\setR\times\setX)^{t_2-t_1+2}$-valued random variable such that
\begin{align*}
\Big( (\vartheta^{\eta}_{t_1-1}, X^{\eta}_{t_1-1}),\dots,(\vartheta^{\eta}_{t_2},   X^{\eta}_{t_2})\Big)\big|\F_{t_2}\sim \zeta_{t_1:t_2}
\end{align*}
 we let
  \begin{align*}
\Lambda^{\zeta}_{t_1:t_2}(A\times B)=\E_{t_1:t_2}^{\zeta}\Big[\ind_A(\vartheta^\eta_{t_2})\ind_B(X^{\eta}_{t_2})\prod_{s=t_1}^{t_2} G_{s,\vartheta^\eta_s}(X_{s-1}^{\eta}, X^{\eta}_s)\Big],\quad\forall (A,B)\in\mathcal{T}\times\mathcal{X}
\end{align*}
and 
  \begin{align*}
P^{\zeta}_{t_1:t_2}(A)=\Lambda^{\zeta}_{t_1:t_2}(A\times \setX),\quad\forall A\in\mathcal{T}
\end{align*}
using the convention that $\Lambda^{\zeta}_{t_1:t_2}(A\times B)=\zeta(A\times B)$ and $P^{\zeta}_{t_1:t_2}(A)=\zeta(A\times\setX)$ for all $(A,B)\in\mathcal{T}\times\mathcal{X}$ if $t_1>t_2$. In addition, for $t_1\leq t_2$ and  all $\eta\in\mathcal{P}(\setX)$ we let
\begin{align}\label{eq:Pbamodelef}
\bar{P}^{\eta}_{t_1: t_2}(\theta) =\int_{\setX^{t_2-t_1+2}}\eta(\dd x_{t_1-1})  \prod_{s=t_1}^{t_2} Q_{s,\theta}(x_{s-1},\dd x_{s}),\quad\forall \theta\in\Theta, 
\end{align}
and, whenever  \ref{assume:Model} is assumed to hold, we let
\begin{align}\label{eq:W_def} 
W^{\eta}_{t_1:t_2}=\sup_{\theta\in\Theta}\int_{\setX^{t_2-t_1+2}} \exp\Big( \frac{\delta_\star}{2}\sum_{s=t_1}^{t_2}\varphi_s(x_{s-1},x_{s})\Big) \frac{\eta(\dd x_{t_1-1})  \prod_{s=t_1}^{t_2} Q_{s,\theta}(x_{s-1},\dd x_{s})}{\bar{P}^{\eta}_{t_1:t_2}(\theta)}
\end{align}
with $g(\cdot)$, $\delta_\star\in(0,1)$ and $(\varphi_t)_{t\geq 1}$  as in \ref{assume:Model}.

For all integer $t\geq 1$ and $\epsilon\in(0,\infty)$ we let $f_{t,\epsilon}:\setR^{t+1}\rightarrow\{0,1\}$ be  the function defined by
\begin{align}\label{eq:ft_def}
f_{t,\epsilon}(u_1,\dots,u_{t+1})=
\begin{cases}
1, &\text{if }\max_{1\leq i\leq t+1}\|u_i-u_1\|< \epsilon\\
0, &\text{otherwise}
\end{cases},\quad  (u_1,\dots,u_{t+1})\in\setR^{t+1}
\end{align}
and $\bar{\pi}_t$ be the random probability measure on $(\Theta\times\mathcal{T})$ such that
\begin{align}\label{eq:pi_bar}
\bar{\pi}_t(A)=\int_{\Theta}  K_{\mu_{t}|\Theta}(\theta_{t-1},A)\pi_{t-1,\Theta}(\dd\theta_{t-1}),\quad\forall A\in\mathcal{T}.
\end{align}

\subsection{Preliminary results}

We start with a useful technical lemma:
\begin{lemma}\label{lemma:tech}
Let $1\leq t_1\leq t_2<t_3$ be three integers and $\zeta$ be random probability measure on $(\Theta\times\setX, \mathcal{T}\otimes\mathcal{X}$). Then,
\begin{align*}
P_{t_1: t_2}^\zeta(\Theta)\geq D_{(t_2+1):t_3}^{-1}  P_{t_1: t_3}^\zeta(\Theta).
\end{align*}
\end{lemma}
\begin{proof}
See Section \ref{p-lemma:tech}.
\end{proof}

The following two lemmas play a key role in establishing the subsequent results:

\begin{lemma}\label{lemma::W}
Assume that Assumption \ref{assume:Model} holds    and let $\eta_r$ be an  $(r,\delta_\star, \tilde{l},(\varphi_t)_{t\geq 1})$-consistent random probability measures on $(\setX,\mathcal{X})$, with   $(r,\delta_\star, \tilde{l},(\varphi_t)_{t\geq 1})$ as in  \ref{assume:Model}. Then,  
\begin{align*}
\frac{1}{rt}\log W^{\eta_r}_{(r+1):rt}=\bigO_\P(1),\quad \inf_{t\geq 2}\P\Big(\log W^{\eta_r}_{(r+1):rt}\in[0,\infty)\Big)=\inf_{t\geq 2}\P\Big(\inf_{\theta\in\Theta}\log \bar{P}^{\eta_r}_{(r+1):rt}(\theta) >-\infty\Big)=1.
\end{align*}
\end{lemma}
\begin{proof}
See Section \ref{p-lemma::W}.
\end{proof}

\begin{lemma}\label{lemma:Unif_convergence}
Assume that Assumption \ref{assume:Model} holds    and let $\eta_r$ be an  $(r,\delta_\star, \tilde{l},(\varphi_t)_{t\geq 1})$-consistent random probability measure on $(\setX,\mathcal{X})$, with  $(r,\delta_\star, \tilde{l},(\varphi_t)_{t\geq 1})$ as in  \ref{assume:Model}. Then,    
\begin{align*}
 \sup_{\theta\in \Theta}\Big|  \frac{1}{rt}\log \bar{P}^{\eta_r}_{(r+1):rt}(\theta)-\tilde{l}(\theta)\Big|=\smallo_\P(1).
\end{align*}
\end{lemma}
\begin{proof}
See Section \ref{p-lemma:Unif_convergence}.
\end{proof}

The next lemma shows that $\P\big(\mathfrak{Z}_t\in(0,\infty)\big)=1$ for all $t\geq 1$, where $(\mathfrak{Z}_t)_{t\geq 1}$ is as defined in \eqref{eq:normalizing} and noting that $\mathfrak{Z}_t=\Lambda^{\mu_0\otimes\chi}_{1:t}(\Theta\times\setX)$ for all $t\geq 1$.

\begin{lemma}\label{lemma:Zt}
Assume that  Assumptions \ref{assume:Model}-\ref{assume:G_bounded} hold and that $(\mu_t)_{t\geq 0}$ is such that   Conditions \ref{condition:mu1}-\ref{condition:Inf_mu} are satisfied. Then,   for all $t\geq 1$, we have
\begin{align*}
\pi_{t}(A\times B)=\frac{  \Lambda^{\pi_0}_{1:t}(A\times B)}{P^{\pi_0}_{1:t}(\Theta)},\quad\forall (A,B)\in\mathcal{T}\times\mathcal{X},\quad \P-a.s.
\end{align*}
\end{lemma}

\begin{proof}
See Section \ref{p-lemma:Zt}.
\end{proof}

Together with    Lemma \ref{lemma:Zt}, the following result will be used to obtain an upper bound for $\pi_t(A\times B)$ for any  $(A,B)\in\mathcal{T}\times\mathcal{X}$. 
\begin{lemma}\label{lemma:gamma}
Assume that Assumptions  \ref{assume:Model}-\ref{assume:G_bounded} and \eqref{eq:W_star}  hold and  that $(\mu_t)_{t\geq 0}$ is such that   Conditions \ref{condition:mu1}-\ref{condition:Inf_mu} are satisfied, and let $(r,\underline{\chi}_{r},\bar{\chi}_{r})$ be as in \ref{assume:Model} and $(W^*_{t})_{t\geq 0}$ be as in \eqref{eq:W_star}. Then, there exist a  sequence $(\tilde{\xi}_{rt})_{t\geq 1}$ of  $[1,\infty)$-valued random variables and two sequences $(\underline{\chi}_{rt})_{t\geq 1}$ and $(\bar{\chi}_{rt})_{t\geq 1}$ of  random probability measures on $(\setX,\mathcal{X})$, all three sequences  depending of $(\mu_t)_{t\geq 1}$ only through the sequence $(W^*_{t})_{t\geq 0}$, such that $\tilde{\xi}_{rt}=\bigO_\P(1)$, such that
\begin{align*}
\underline{\chi}_{rt}\dist\underline{\chi}_{r},\quad \bar{\chi}_{rt}\dist \bar{\chi}_{r},\quad\forall t\geq 1
\end{align*}
and such that, for all $t\geq r$ and $\tau\in\{r,\dots,t\}\cap\{sr,\,s\in\mathbb{N}\}$, we have
\begin{align*}
\frac{1}{\tilde{\xi}_{\tau}}  \Lambda^{\bar{\pi}_\tau\otimes\underline{\chi}_\tau}_{(\tau+1):t}(A\times B)\leq  \frac{\Lambda^{\pi_0}_{1:t}(A\times B)}{P^{\pi_0}_{1:(\tau-1)}(\Theta)}\leq  \tilde{\xi}_\tau  \Lambda^{\bar{\pi}_\tau\otimes\bar{\chi}_\tau}_{(\tau+1):t}(A\times B),\quad\forall (A,B)\in\mathcal{T}\times\mathcal{X},\quad\P-a.s.
\end{align*}
\end{lemma}
\begin{proof}
See Section \ref{p-lemma:gamma}.
\end{proof}

\subsection{Proof of Theorem \ref{thm:main}}

\subsubsection{Additional preliminary results}

By combining Lemma \ref{lemma:Zt} and Lemma \ref{lemma:gamma} we obtain that
\begin{align}\label{eq:bound_piA}
\pi_{t,\Theta}(A)\leq \tilde{\xi}^2_{\tau}\,\frac{P^{\bar{\pi}_\tau\otimes\bar{\chi}_\tau}_{(\tau+1):t}(A)}{P^{\bar{\pi}_\tau\otimes\underline{\chi}_\tau}_{(\tau+1):t}(\Theta)},\quad\forall A\in\mathcal{T}\quad\forall \tau\in\{r,\dots,t\}\cap\{sr,\,s\in\mathbb{N}\},\quad t\geq r,\quad\P-a.s.
\end{align}
with $r$ as in \ref{assume:Model} and with $(\tilde{\xi}_{rt})_{t\geq 1}$,  $(\underline{\chi}_{rt})_{t\geq 1}$ and $(\bar{\chi}_{rt})_{t\geq 1}$ as in Lemma \ref{lemma:gamma}. Noting that under \ref{assume:Model}  both   $\underline{\chi}_{\tau}$ and  $\bar{\chi}_{\tau}$ are   $(r,\delta_\star, \tilde{l},(\varphi_t)_{t\geq 1})$-consistent random probability measures on $(\setX,\mathcal{X})$, the  following two lemmas will be used to  study the upper   bound for $\pi_{t,\Theta}(A)$ given in \eqref{eq:bound_piA}.

\begin{lemma}\label{lemma::upper_general}
Assume that Assumptions \ref{assume:Model}-\ref{assume:G_bounded} hold and  that $(\mu_t)_{t\geq 1}$ is such that Condition \ref{condition:Inf_K} is satisfied. Let  $(\kappa_{\tau})_{\tau\geq 1}$ be  a  sequence of random probability measures on $(\Theta,\mathcal{T})$,   $(a_\tau)_{\tau\geq 1}$  be a   sequence  in $\{sr,\,s\in\mathbb{N}\}$, with $r$ as in \ref{assume:Model},  and $(b_\tau)_{\tau\geq 1}$ be a sequence in $\mathbb{N}$ such that  $a_\tau< b_\tau$ for all $\tau\geq 1$ and such that $\lim_{\tau\rightarrow\infty} (b_\tau-a_\tau)=\infty$. Assume that there exist  two sequences   $(c_{\tau})_{\tau\geq 1}$ and  $(\epsilon_{\tau})_{\tau\geq 1}$ in $(0,\infty)$ such that $ \lim_{\tau\rightarrow\infty}\epsilon_{\tau}=\lim_{\tau\rightarrow\infty}(1/c_{\tau})=0$ and such that
\begin{align}\label{eq:conv_prob}
\frac{1}{b_\tau-a_\tau}\log\P\Big(\exists s\in\{a_\tau+1,\dots,  b_\tau\}:  \sum_{i=a_\tau+1}^s  U_i\not\in B_{\epsilon_{\tau}}(0)\,\big|\, \F_{b_\tau}\Big)\leq -c_{\tau},\quad\forall \tau\geq 1,\quad\P-a.s.
\end{align}
Finally, with  $(\delta_\star, \tilde{l},(\varphi_t)_{t\geq 1})$ as in \ref{assume:Model}, let  $\eta_r$  be an $(r,\delta_\star, \tilde{l},(\varphi_t)_{t\geq 1})$-consistent  random probability measure   on $(\setX,\mathcal{X})$ and $(\eta_{\tau})_{\tau\geq 1}$ be a  sequence  of random probability measures on $(\setX,\mathcal{X})$ such that $\eta_{\tau}\dist \eta_r$ for all $\tau\geq 1$. 
Then, there exists a  sequence $(\xi^+_{b_\tau})_{\tau\geq 1}$ of $(0,\infty]$-valued random variables, independent of $(\kappa_{\tau})_{\tau\geq 1}$ and depending on $(\mu_t)_{t\geq 1}$ only through the sequences $(c_{\tau})_{\tau\geq 1}$ and $(\epsilon_{\tau})_{\tau\geq 1}$ and through the random variable $\Gamma^\mu$  defined in \ref{condition:Inf_K}, such  that $\xi_{b_\tau}^+=\smallo_\P(1)$ and such that
\begin{align*}
\P\bigg(P^{\kappa_{a_\tau}\otimes\eta_{a_\tau}}_{(a_\tau+1):b_\tau}(A)\leq (1+\xi^+_{b_\tau})\exp\Big\{ (b_\tau-a_\tau)\big(\sup_{\theta\in A} \tilde{l}(\theta)+\xi^+_{b_\tau}\big)\Big\}\bigg)=1,\quad\forall A\in\mathcal{T},\quad\forall \tau\geq 1.
\end{align*}
\end{lemma}

\begin{proof}
See Section \ref{p-lemma::upper_general}.
\end{proof}

\begin{lemma}\label{lemma::lower_general}
Assume that  Assumptions \ref{assume:Model}-\ref{assume:G_bounded} hold. Let  $(a_\tau)_{\tau\geq 1}$, $(b_\tau)_{\tau\geq 1}$,  $(\kappa_\tau)_{\tau\geq 1}$ and $(\eta_\tau)_{\tau\geq 1}$    be as in Lemma \ref{lemma::upper_general}, $(r,\delta_\star,\tilde{l},\tilde{\Theta}_\star)$ be as in \ref{assume:Model} and for all $\tau\geq 1$ let $e_\tau\in\{sr,\,s\in\mathbb{N}\}$ be such that $e_\tau-r<b_\tau\leq e_\tau$. Assume that there exist  two sequences   $(c_{\tau})_{\tau\geq 1}$ and  $(\epsilon_{\tau})_{\tau\geq 1}$ in $(0,\infty)$ such that $ \lim_{\tau\rightarrow\infty}\epsilon_{\tau}=\lim_{\tau\rightarrow\infty}(1/c_{\tau})=0$ and such that
\begin{align*}
\frac{1}{e_\tau-a_\tau}\log \P\Big( \sum_{i=a_\tau+1}^s  U_i\in B_{\epsilon_\tau}(0),\,\forall   s\in\{a_\tau+1,\dots,  e_\tau\}\big| \,\F_{e_\tau}\Big)\geq -\frac{1}{c_\tau},\quad\forall \tau \geq 1.
\end{align*}
Then, there exist  an $\bar{\epsilon}\in (0,\delta_\star/2)$, independent of $(\mu_t)_{t\geq 1}$,  and a sequence $(\xi^-_{b_\tau})_{\tau\geq 1}$ of $(0,\infty]$-valued  random variables, independent of $(\kappa_{\tau})_{\tau\geq 1}$ and depending on  $(\mu_t)_{t\geq 1}$ only through the sequences $(c_\tau)_{\tau\geq 1}$ and $(\epsilon_\tau)_{\tau\geq 1}$, such  that $\xi_{b_\tau}^-=\smallo_\P(1)$ and such that
\begin{align*}
\P\bigg( P^{\kappa_{a_\tau}\otimes\eta_{a_\tau}}_{(a_\tau+1):b_\tau}(\Theta)  &\geq  \kappa_{a_\tau}\big(\mathcal{N}_\epsilon(\tilde{\Theta}_\star)\big)  \exp\Big\{ (b_\tau-a_\tau)\big(\inf_{\theta\in \mathcal{N}_\epsilon(\tilde{\Theta}_\star)}\tilde{l}(\theta)-\xi_{b_\tau}^-\big)\Big\}\bigg)=1,\quad\forall \epsilon\in (0,\bar{\epsilon}].
\end{align*}
\end{lemma}

\begin{proof}
See Section \ref{p-lemma::lower_general}.
\end{proof}

By combining \eqref{eq:bound_piA} and Lemmas \ref{lemma::upper_general}-\ref{lemma::lower_general} we obtain the following key result for proving Theorem \ref{thm:main}.

\begin{theorem}\label{thm:key}
Assume that Assumptions \ref{assume:Model}-\ref{assume:G_bounded} and \eqref{eq:W_star} hold  and that $(\mu_t)_{t\geq 0}$ is such that Conditions \ref{condition:mu1}-\ref{condition:Inf_K} are satisfied. Let  $(r,\delta_\star,\tilde{\Theta}_\star)$ be as in \ref{assume:Model} and let $(a_\tau)_{\tau\geq 1}$  be a    sequence  in $\{sr,\,s\in\mathbb{N}\}$ and   $(b_\tau)_{\tau\geq 1}$ be a sequence in $\mathbb{N}$  such that $ a_\tau< b_\tau$ for all $\tau\geq 1$ and such that $\lim_{\tau\rightarrow\infty} (b_\tau-a_\tau)=\infty$.  In addition, for all $\tau\geq 1$ let $e_\tau\in\{sr,\,s\in\mathbb{N}\}$ be such that $e_\tau-r<b_\tau\leq e_\tau$ and assume that there exist two sequences  $(\tilde{c}_{\tau})_{\tau\geq 1}$  and  $(\tilde{\epsilon}_{\tau})_{\tau\geq 1}$  in $(0,\infty)$ such that $\lim_{\tau\rightarrow\infty}\tilde{\epsilon}_{\tau}=\lim_{\tau\rightarrow\infty}(1/\tilde{c}_\tau)=0$ and such that 
\begin{align}\label{eq:Thm_C1}
\frac{1}{e_\tau-a_\tau}\log\P\Big(\exists s\in\{a_\tau+1,\dots,  e_\tau\}:\,  \sum_{i=a_\tau+1}^s  U_i\not\in B_{\tilde{\epsilon}_{\tau}}(0)\big|\,\F_{e_\tau}\Big)\leq - \tilde{c}_\tau,\quad\forall \tau\geq 1,\quad\P-a.s.
\end{align}
Moreover, assume that for all $\delta\in (0,\delta_\star/2)$ there exists a sequence $(\xi_{\delta,a_\tau})_{\tau\geq 1}$ of $[0,\infty]$-valued random variables such that $\xi_{\delta,a_\tau}=\smallo_\P(1)$ and such that
\begin{align}\label{eq:Thm_C2}
 \P\bigg(\frac{1}{b_\tau-a_\tau}\log \bar{\pi}_{a_\tau}\big(\mathcal{N}_{\delta}(\tilde{\Theta}_\star)\big)\geq -\xi_{\delta,a_\tau}\bigg)=1,\quad\forall \tau\geq 1.
\end{align}
Then, for all $\epsilon\in(0,\infty)$ there exists a sequence $(\xi^*_{\epsilon,b_\tau})_{\tau \geq 1}$ of $[0,1]$-valued random variables, depending on $(\mu_t)_{t\geq 1}$  only through the sequences  $(\tilde{c}_{\tau})_{\tau\geq 1}$, $(\epsilon_{\tau})_{\tau\geq 1}$ and $((\xi_{\delta,a_\tau})_{\delta\in(0,\delta_\star/2)})_{\tau\geq 1}$, through the sequence $(W^*_{t})_{t\geq 0}$ defined in \eqref{eq:W_star}  and  through the random variable   $\Gamma^\mu$  defined in \ref{condition:Inf_K}, such that $\xi^*_{\epsilon,b_\tau}=\smallo_\P(1)$ and such that 
\begin{align*}
\P\Big(\pi_{b_\tau,\Theta}\big(\mathcal{N}_\epsilon(\tilde{\Theta}_\star)\big)\geq 1-\xi^*_{\epsilon,b_\tau}\Big)=1,\quad\forall \tau\geq 1.
\end{align*}
\end{theorem}
\begin{proof}
See Section \ref{p-thm:key}.
\end{proof}

The following   result  will be used  to establish that condition \eqref{eq:Thm_C2}  of Theorem \ref{thm:key} holds.
\begin{lemma}\label{lemma:Extend_initit}
Assume that Assumptions \ref{assume:Model}-\ref{assume:G_l_E} and \eqref{eq:W_star} hold and that $(\mu_t)_{t\geq 0}$ is such that   Conditions \ref{condition:mu1}-\ref{condition:Inf_mu} are satisfied, and let $(\delta_\star,\tilde{\Theta}_\star)$ be as in \ref{assume:Model}. Then, for all $\bar{m}\in\mathbb{N}$ there exists a sequence $(\xi'_{t,\bar{m}})_{t\geq 1}$ of $(0,1)$-valued random variables, depending on $(\mu_t)_{t\geq 1}$ only through the sequence $(W^*_t)_{t\geq 1}$ defined  in \eqref{eq:W_star}, such that $\log \xi'_{t,\bar{m}}=\bigO_\P(1)$ and such that, $\P$-a.s.,
\begin{align*}
\bar{\pi}_{t}\big(\mathcal{N}_\delta(\tilde{\Theta}_\star)\big)\geq \xi'_{t,\bar{m}} \inf_{\theta\in\Theta} K_{\mu_{t-m}}\big(\theta, \mathcal{N}_{\delta/\bar{m}}(\tilde{\Theta}_\star)\big),\quad\forall \delta\in (0,\delta_\star),\quad \forall t\geq \bar{m}+1,\quad\forall m\in\{1,\dots,\bar{m}\}.
\end{align*}
If $\bar{m}=1$ then Assumption  \ref{assume:G_l_E} can be omitted.
\end{lemma}
\begin{proof}
See Section \ref{p-lemma:Extend_initit}. 
\end{proof}

\subsubsection{Proof of the theorem }

The result of Theorem \ref{thm:main}   is proved by three successive applications of Theorem \ref{thm:key}, given in Lemmas \ref{lemma:piece1}-\ref{lemma:piece3} below.

Firstly, under \ref{condition:mu_seq}, a first application of Theorem \ref{thm:key} yields the following result:

\begin{lemma}\label{lemma:piece1}
Assume that Assumptions \ref{assume:Model}-\ref{assume:G_l_E} and \eqref{eq:W_star} hold  and that $(\mu_t)_{t\geq 0}$ is such that Conditions \ref{condition:mu1}-\ref{condition:mu_seq} are satisfied. Let $r$ be as in   \ref{assume:Model}, $(W^*_{t})_{t\geq 0}$ be as in \eqref{eq:W_star}, $\Gamma^\mu$ be as in \ref{condition:Inf_K} and  $((t_p,v_p,f_{p}))_{p\geq 1}$ be as in \ref{condition:mu_seq}.  In addition, let $p^*\in\mathbb{N}$ be such that
\begin{align}\label{eq:P_star}
t_{p+1}> t_p+4v_p,\quad t_p\geq 2v_p+r,\quad 3v_p\geq 2 (v_{p+1}+r)+1,\quad v_{p+1}>3r,\quad\forall p\geq p^*
\end{align}
 and
\begin{align*}
\mathbb{N}^{(1)}=\bigcup_{p\geq p^*}\big\{t_{p+1}-3v_p,\dots,t_{p+1}-r\big\}.
\end{align*}
 Then, for  all  strictly increasing sequence $(t'_k)_{k\geq 1}$ in $\mathbb{N}^{(1)}$ and all  $\epsilon\in(0,\infty)$,    there exists a sequence $(\xi^{(1)}_{\epsilon,t'_k})_{k \geq 1}$ of $[0,1]$-valued random variables, depending on $(\mu_t)_{t\geq 1}$  only through $(W^*_{t})_{t\geq 0}$, $\Gamma^\mu$ and through the sequence $\big((t_p, v_p, f_p)\big)_{p\geq 1}$, such that $\xi^{(1)}_{\epsilon,t'_k}=\smallo_\P(1)$ and such that
\begin{align*}
\P\Big(\pi_{t'_k,\Theta}\big(\mathcal{N}_\epsilon(\tilde{\Theta}_\star)\big) \geq 1-\xi^{(1)}_{\epsilon,t'_k}\Big)=1,\quad\forall k\geq 1.
\end{align*}
If $(t_p)_{p\geq 1}$ is a sequence in $\{sr,\,s\in\mathbb{N}\}$ Assumption  \ref{assume:G_l_E} can be omitted.
\end{lemma}
\begin{remark}
To show that there exists a $p^*\in\mathbb{N}$ such that \eqref{eq:P_star} holds note that we trivially have $t_{p+1}\geq t_p+4v_p$, $t_p\geq 2v_p+r$ and $v_{p+1}> 3r$ for $p$ large enough since $\lim_{p\rightarrow\infty}(t_{p+1}-t_p)/v_p=\lim_{p\rightarrow\infty}v_p=\infty$. In addition, since $\limsup_{p\rightarrow\infty} v_{p+1}/v_p<3/2$ there exist a constant $c\in(0,1/2)$ and a $p'\in\mathbb{N}$ such that
\begin{align*}
(1+c)v_p\geq v_{p+1},\quad v_p\geq \frac{2r+1}{1-2c},\quad\forall p\geq p'
\end{align*}
implying that 
\begin{align*}
3v_p\geq 2 (1+c)v_p+2r+1\geq 2v_{p+1}+2r+1,\quad\forall p\geq p'.
\end{align*}
\end{remark}

\begin{proof}
Let $\delta_\star$ be as in \ref{assume:Model} and $(\gamma_p)_{p\geq 1}$ be a sequence in $(0,\infty)$ such that $\lim_{p\rightarrow\infty}\gamma_p=\lim_{p\rightarrow\infty} f_p(\gamma_p)=0$. Remark that such  a sequence  $(\gamma_p)_{p\geq 1}$ exists since, by assumption, $\lim_{p\rightarrow\infty} f_p(\gamma)=0$ for all $\gamma\in(0,\infty)$. Next, let $(t'_k)_{k\geq 1}$ be as in the statement of the lemma and for all $k\geq 1$ let $p_k$ be such that $t_{p_k}< t'_k< t_{p_k+1}$ and $m_k\in\{0,\dots,r-1\}$ be such that $t_{p_k}+m_k\in\{sr,\,s\in\mathbb{N}\}$. Remark that  
\begin{align}\label{eq:Tk_in}
t_{p_k}+m_k\leq t_{p_k}+v_{p_k}\leq t'_k<t_{p_k+1},\quad\forall k\geq 1  
\end{align}
and let  $\bar{m}=1+\max\{m_k,\,k\in\mathbb{N}\}$. Finally, for all $t\geq r+1$   let $(\xi'_{t,\bar{m}})_{t\geq 1}$ be as defined in Lemma \ref{lemma:Extend_initit}.

Then, the result of the lemma follows by applying Theorem \ref{thm:key}  with $(a_\tau)_{\tau\geq 1}$, $(b_\tau)_{\tau}$, $(\tilde{c}_\tau)_{\tau\geq 1}$, $(\tilde{\epsilon}_\tau)_{\tau\geq 1}$ and $( (\xi_{\delta,a_\tau})_{\delta\in(0,\delta_\star/2)})_{\tau\geq 1}$  defined by
\begin{equation}\label{eq:def_seq}
\begin{split}
&a_\tau=t_{p_\tau}+m_{\tau},\quad b_\tau=t'_{\tau},\quad \tilde{c}_\tau=\frac{1}{f_{p_\tau}(\gamma_{p_\tau})},\quad\tilde{\epsilon}_\tau=\gamma_{p_\tau},\quad\hspace{1.8cm}\tau\geq 1\\
&\xi_{\delta,a_\tau}= \bigg(\frac{t_{p_\tau+1}-t_{p_\tau}}{t_{p_\tau+1}-t_{p_\tau}-4v_{p_\tau}} f_{p_\tau}(\delta)-\frac{1}{t_{p_\tau+1}-t_{p_\tau}-4v_{p_\tau} }\log \xi'_{a_\tau,\bar{m}}\bigg),\quad\tau\geq 1,\quad\delta\in(0,\delta_\star/2).
\end{split}
\end{equation}

To see this  remark  first that,  as required by Theorem \ref{thm:key},  the  sequence  $(a_\tau)_{\tau\geq 1}$ defined in \eqref{eq:def_seq} belongs to $\{sr,\,s\in\mathbb{N}\}$   and the sequence $(b_\tau)_{\tau\geq 1}$  defined in \eqref{eq:def_seq} is  such that $ b_\tau> a_\tau$ for all $\tau\geq 1$. Next, by \eqref{eq:Tk_in}, $\liminf_{\tau\rightarrow\infty}(b_\tau-a_\tau)\geq \liminf_{\tau\rightarrow\infty}(v_{p_\tau}-m_\tau)\geq \liminf_{\tau\rightarrow\infty}(v_{p_\tau}-r)$ and thus,  as required by Theorem \ref{thm:key}, $\lim_{\tau\rightarrow\infty} (b_\tau-a_\tau)=\infty$ since $\lim_{p\rightarrow\infty}v_p=\infty$  by assumption. In addition, as required by Theorem \ref{thm:key}, the  sequences   $(\tilde{c}_{\tau})_{\tau\geq 1}$ and $(\tilde{\epsilon}_{\tau})_{\tau\geq 1}$ defined in \eqref{eq:def_seq} are such that $ \lim_{\tau\rightarrow\infty}\tilde{\epsilon}_\tau=\lim_{\tau\rightarrow\infty}(1/\tilde{c}_\tau)=0$. To proceed further for all $\tau\geq 1$ we let $e_\tau\in\{sr,\,s\in\mathbb{N}\}$ be such that $e_\tau-r<t'_\tau\leq e_\tau$. Then, noting that $t_{p_\tau}<e_\tau<t_{p_\tau+1}$ for all $\tau\geq 1$ and using \ref{C42} with $(l_p)_{p\geq 1}=(a_p)_{p\geq 1}$ and $(u_p)_{p\geq 1}=(e_{p})_{p\geq 1}$, it follows that  the sequences  $(a_\tau)_{\tau\geq 1}$, $(b_\tau)_{\tau}$, $(\tilde{c}_\tau)_{\tau\geq 1}$  and $(\tilde{\epsilon}_\tau)_{\tau\geq 1}$ defined in \eqref{eq:def_seq} are such that \eqref{eq:Thm_C1} holds. Finally, using Lemma \ref{lemma:Extend_initit} and   under \ref{C41}, for all $\delta\in(0,\delta_\star)$ and $\tau\geq 1$  we $\P$-a.s.~have 
\begin{equation}\label{eq:Show_lower}
\begin{split}
0&\geq \frac{1}{b_\tau-a_\tau}\log\bar{\pi}_{a_\tau}\big(\mathcal{N}_\delta(\tilde{\Theta}_\star)\big)
\\
&\geq   \bigg(\frac{t_{p_\tau+1}-t_{p_\tau}}{t'_\tau-t_{p_\tau}-m_\tau}\bigg)\frac{1}{t_{p_\tau+1}-t_{p_\tau}} \inf_{\theta\in\Theta}\log K_{\mu_{t_{p_\tau}}}(\theta, \mathcal{N}_{\delta/\bar{m}}(\tilde{\Theta}_\star))+\frac{1}{t'_\tau-t_{p_\tau}-m_\tau}\log \xi'_{a_\tau,\bar{m}}\\
&\geq -\frac{t_{p_\tau+1}-t_{p_\tau}}{t'_{\tau}-t_{p_\tau}-m_\tau} f_{p_\tau}(\delta/\bar{m})+\frac{1}{t'_\tau-t_{p_\tau}-m_\tau}\log \xi'_{a_\tau,\bar{m}}\\
&\geq -\bigg(\frac{t_{p_\tau+1}-t_{p_\tau}}{t_{p_\tau+1}-t_{p_\tau}-4v_{p_\tau}} f_{p_\tau}(\delta/\bar{m})-\frac{1}{t_{p_\tau+1}-t_{p_\tau}-4v_{p_\tau} }\log \xi'_{a_\tau,\bar{m}}\bigg).
\end{split}
\end{equation}
where the last inequality uses the fact that $t'_{\tau}-t_{p_\tau}-m_\tau\geq t_{p_\tau+1}-t_{p_\tau}-4v_{p_\tau}\geq 1$ for all $\tau\geq 1$. By assumption we have $\lim_{p\rightarrow\infty}v_p/(t_{p+1}-t_p)=0$ and $\lim_{p\rightarrow\infty}f_p(\delta)=0$ for all $\delta\in(0,\infty)$ and thus, using the fact that $\log \xi'_{a_\tau,r}=\bigO_\P(1)$,
\begin{align*}
\bigg(\frac{t_{p_\tau+1}-t_{p_\tau}}{t_{p_\tau+1}-t_{p_\tau}-4v_{p_\tau}} f_{p_\tau}(\delta)-\frac{1}{t_{p_\tau+1}-t_{p_\tau}-4v_{p_\tau} }\log \xi'_{a_\tau,r}\bigg)=\smallo_\P(1),\quad\forall\delta\in(0,\infty)
\end{align*}
and thus it follows from \eqref{eq:Show_lower} that  \eqref{eq:Thm_C2} holds for the sequences  $( (\xi_{\delta,a_\tau})_{\delta\in(0,\delta_\star/2)})_{\tau\geq 1}$, $(a_\tau)_{\tau\geq 1}$ and $(b_\tau)_{\tau\geq 1}$ defined in \eqref{eq:def_seq}. The proof of the lemma is complete upon noting that \ref{assume:G_l_E} is only required by Lemma \ref{lemma:Extend_initit} and that $\bar{m}=1$ if $(t_p)_{p\geq 1}$ is a sequence in $\{sr,\,s\in\mathbb{N}\}$.
\end{proof}

Secondly, under \ref{condition:mu_seq} and  using Lemma \ref{lemma:piece1},   a second application of  Theorem \ref{thm:key} yields the following result:

\begin{lemma}\label{lemma:piece2}
Assume that Assumptions \ref{assume:Model}-\ref{assume:G_l_E} and \eqref{eq:W_star} hold and that $(\mu_t)_{t\geq 0}$ is such that Conditions \ref{condition:mu1}-\ref{condition:mu_seq} are satisfied.  Let  $r$ be as in \ref{assume:Model},  $(W^*_{t})_{t\geq 0}$ be as in \eqref{eq:W_star},  $(\Gamma^\mu_\delta)_{\delta\in(0,\infty)}$ be as in \ref{condition:Inf_mu}, $\Gamma^\mu$ be as in \ref{condition:Inf_K} and $\big((t_p, v_p, f_p)\big)_{p\geq 1}$ be as in \ref{condition:mu_seq}. In addition let  $p^*\in\mathbb{N}$ be as in Lemma \ref{lemma:piece1}   and
\begin{align*}
\mathbb{N}^{(2)}= \bigcup_{p\geq p^*}\big\{t_{p+1}-r,\dots,t_{p+1}+v_{p+1}\big\}.
\end{align*}
 Then, for all strictly increasing sequence $(t''_k)_{k\geq 1}$ in $\mathbb{N}^{(2)}$  and all  $\epsilon\in(0,\infty)$, there exists a sequence $(\xi^{(2)}_{\epsilon,t''_k})_{k \geq 1}$ of $[0,1]$-valued random variables, depending on $(\mu_t)_{t\geq 1}$  only through $(W^*_{t})_{t\geq 0}$, $(\Gamma^\mu_\delta)_{\delta\in(0,\infty)}$,  $\Gamma^\mu$ and trough the $\big((t_p, v_p, f_p)\big)_{p\geq 1}$, such that $\xi^{(2)}_{\epsilon,t''_k}=\smallo_\P(1)$ and such that
\begin{align*}
\P\Big(\pi_{t''_k,\Theta}\big( \mathcal{N}_\epsilon(\tilde{\Theta}_\star)\big)\geq 1-\xi^{(2)}_{\epsilon,t''_k}\Big)=1,\quad\forall k\geq 1.
\end{align*}
If $(t_p)_{p\geq 1}$ is a sequence in $\{sr,\,s\in\mathbb{N}\}$ then Assumption  \ref{assume:G_l_E} can be omitted.
\end{lemma}

\begin{proof}
Let $\delta_\star$ be as in \ref{assume:Model}, $(\gamma_p)_{p\geq 1}$ be as defined in the proof of Lemma \ref{lemma:piece1}, $(t''_k)_{k\geq 1}$ be as in the statement of the lemma, and  for all $k\geq 1$, let $p_k\in\mathbb{N}$ be such that $t''_k\in\{t_{p_k}-r,\dots,t_{p_k}+v_{p_k}\}$ and $m_k\in\{0,\dots, r-1\}$ be such that $t''_k-2v_{p_k}-m_k\in\{sr,\,s\in\mathbb{N}\}$. Since we have both $v_{p+1}>r$ and $3v_p\geq 2(v_{p+1}+r)+1$ for all $p\geq p^*$ it follows that $t''_{k}-2v_{p_k}-1-m_k\in\mathbb{N}^{(1)}$ for all $k\geq 1$, with $\mathbb{N}^{(1)}$ as defined in Lemma \ref{lemma:piece1}. In what follows, for all $\epsilon\in(0,\infty)$ we let $(\xi^{(1)}_{\epsilon, t'_k})_{\tau\geq 1}$ be such that the conclusion of Lemma \ref{lemma:piece1} holds when $t'_k=t''_{k}-2v_{p_k}-1-m_k$ for all $k\geq 1$. Then,  the result of the lemma follows by applying Theorem \ref{thm:key} with    $(a_\tau)_{\tau\geq 1}$, $(b_\tau)_{\tau}$, $(\tilde{c}_\tau)_{\tau\geq 1}$,  $(\tilde{\epsilon}_\tau)_{\tau\geq 1}$ and $((\xi_{\delta,a_\tau})_{\delta\in(0,\delta_\star/2)})_{\tau\geq 1}$ defined by
\begin{equation}\label{eq:def_seq2}
\begin{split}
&a_\tau=t''_{\tau}-2v_{p_\tau}-m_\tau,\quad b_\tau=t''_{\tau},\quad \tilde{c}_\tau=\frac{1}{f_{p_\tau}(\gamma_{p_\tau})},\quad\tilde{\epsilon}_\tau=\gamma_{p_\tau},\quad\tau\geq 1\\
&\xi_{\delta,a_\tau}= -\frac{\log  \Gamma^\mu_{\delta/2}+\log \big(1-\xi^{(1)}_{\delta/2,a_\tau-1}\big)}{3v_{p_\tau}},\hspace{3.7cm} \tau\geq 1,\quad \delta\in(0,\delta_\star/2).
\end{split}
\end{equation}

To see this  remark  first that,   as required by Theorem \ref{thm:key},  the   sequence $(a_\tau)_{\tau\geq 1}$ defined in \eqref{eq:def_seq2} belongs to $\{sr,\,s\in\mathbb{N}\}$ and   the sequence $(b_\tau)_{\tau\geq 1}$ defined in \eqref{eq:def_seq2} is  such that $b_\tau> a_\tau$ for all $\tau\geq 1$. Next,   as required by Theorem \ref{thm:key},  the conditions on  $((t_p,v_p))_{p\geq 1}$ ensure that $\lim_{\tau\rightarrow\infty} (b_\tau-a_\tau)=\infty$. In addition, as required by Theorem \ref{thm:key}, the  sequences   $(\tilde{c}_{\tau})_{\tau\geq 1}$ and $(\tilde{\epsilon}_{\tau})_{\tau\geq 1}$ defined in \eqref{eq:def_seq2} are such that $ \lim_{\tau\rightarrow\infty}\tilde{\epsilon}_\tau=\lim_{\tau\rightarrow\infty}(1/ \tilde{c}_\tau)=0$. To proceed further for all $\tau\geq 1$ we let $e_\tau\in\{sr,\,s\in\mathbb{N}\}$ be such that $e_\tau-r<t''_\tau\leq e_\tau$. Then, noting that $t_{p_\tau}-3v_{p_\tau}\leq a_\tau<e_\tau\leq t_{p_{\tau}}+v_{p_\tau}$ for all $\tau\geq 1$, by  using \ref{C43} with $(l_p)_{p\geq 1}=(a_p)_{p\geq 1}$ and  $(u_p)_{p\geq 1}=(e_p)_{p\geq 1}$  it follows   that the sequences   $(a_\tau)_{\tau\geq 1}$, $(b_\tau)_{\tau}$, $(\tilde{c}_\tau)_{\tau\geq 1}$ and $(\tilde{\epsilon}_{\tau})_{\tau\geq 1}$ defined in \eqref{eq:def_seq2} are such that \eqref{eq:Thm_C1} holds.

 Finally, to prove that \eqref{eq:Thm_C2} holds let $\delta\in(0,\delta_\star/2)$. Then, under \ref{assume:Model_ThetaStar}  we have $\theta+u\in \mathcal{N}_\delta(\tilde{\Theta}_\star)\subseteq\Theta$ for all $\theta\in \mathcal{N}_{\delta/2}(\tilde{\Theta}_\star)$ and all $u\in B_{\delta/2}(0)$. Therefore,  for all $\tau\geq 1$ we have
\begin{equation}\label{eq:split_ll1}
\begin{split}
\bar{\pi}_{a_\tau}\big(\mathcal{N}_\delta(\tilde{\Theta}_\star)\big)=\int_\Theta K_{\mu_{a_\tau}|\Theta}\big(\theta_{a_\tau-1}, \mathcal{N}_\delta(\tilde{\Theta}_\star)\big) \pi_{a_\tau-1,\Theta}(\dd \theta_{a_\tau-1})
&\geq\mu_{a_\tau}\big( B_{\delta/2}(0)\big)\pi_{a_\tau-1,\Theta}\big(\mathcal{N}_{\delta/2}(\tilde{\Theta}_\star)\big) 
\end{split}
\end{equation}
and thus, under \ref{condition:Inf_mu} and  using Lemma \ref{lemma:piece1},  
\begin{align}\label{eq:split_ll2}
\frac{1}{b_\tau-a_\tau}\log \bar{\pi}_{a_\tau}\big(\mathcal{N}_\delta(\tilde{\Theta}_\star)\big) 
&\geq \frac{1}{3v_{p_\tau}}\Big( \log \Gamma^\mu_{\delta/2}+\log\big(1- \xi^{(1)}_{\delta/2,a_\tau-1}\big)\Big),\quad\forall\tau\geq 1,\quad \P-a.s.
\end{align}
where  $( \xi^{(1)}_{\delta/2,a_\tau-1})_{\tau\geq 1}$ depends on  $(\mu_t)_{t\geq 1}$  only through $(W^*_{t})_{t\geq 0}$, $\Gamma^\mu$ and  $\big((t_p,v_p,f_p)\big)_{p\geq 1}$. Since under \ref{condition:Inf_mu}-\ref{condition:Inf_K}  we have $\P(\Gamma^\mu_{\delta/2}>0)=\P(\Gamma^\mu>0)=1$ while, by  Lemma \ref{lemma:piece1}, we have $\xi^{(1)}_{\delta/2,a_\tau-1}=\smallo_\P(1)$, it follows that
\begin{align*}
\frac{1}{3v_{p_\tau}}\Big( \log \Gamma^\mu_{\delta/2}+\log\big(1- \xi^{(1)}_{\delta/2,a_\tau-1}\big)\Big)\PP 0.
\end{align*}
This concludes to show that \eqref{eq:Thm_C2} holds for the sequences $(a_\tau)_{\tau\geq 1}$, $(b_\tau)_{\tau\geq 1}$ and $( (\xi_{\delta,a_\tau})_{\delta\in(0,\delta_\star/2)})_{\tau\geq 1}$  defined in \eqref{eq:def_seq2} and the proof of the lemma is complete.
\end{proof}

\begin{lemma}\label{lemma:piece3}
Assume that Assumptions \ref{assume:Model}-\ref{assume:G_l_E} and \eqref{eq:W_star} hold  and that $(\mu_t)_{t\geq 0}$ is such that Conditions \ref{condition:mu1}-\ref{condition:mu_seq} are satisfied. Let $r$ be as in \ref{assume:Model}, $(W^*_{t})_{t\geq 0}$ be as in \eqref{eq:W_star}, $(\Gamma^\mu_\delta)_{\delta\in(0,\infty)}$ be as in \ref{condition:Inf_mu}, $\Gamma^\mu$ be as in \ref{condition:Inf_K} and $\big((t_p, v_p, f_p)\big)_{p\geq 1}$ be as in \ref{condition:mu_seq}. In addition let $p^*$ be as in Lemma \ref{lemma:piece1}  and
\begin{align*}
\mathbb{N}^{(3)}= \bigcup_{p\geq p^*}\big\{t_{p+1}+v_{p+1},\dots,t_{p+2}-r\big\}.
\end{align*}
 Then, for all strictly increasing sequence $(t'''_k)_{k\geq 1}$ in $\mathbb{N}^{(3)}$  and  all $\epsilon\in(0,\infty)$,  there exists a sequence $(\xi^{(3)}_{\epsilon,t'''_k})_{k \geq 1}$ of $[0,1]$-valued random variables, depending on $(\mu_t)_{t\geq 1}$  only through $(W^*_{t})_{t\geq 0}$,  $(\Gamma^\mu_\delta)_{\delta\in(0,\infty)}$,  $\Gamma^\mu$ and through the sequence $\big((t_p, v_p, f_p)\big)_{p\geq 1}$, such that $\xi^{(3)}_{\epsilon,t'''_k}=\smallo_\P(1)$ and such that
\begin{align*}
\P\Big(\pi_{t'''_k,\Theta}\big( \mathcal{N}_\epsilon(\tilde{\Theta}_\star)\big)\geq 1-\xi^{(3)}_{\epsilon,t'''_k}\Big)=1,\quad\forall k\geq 1.
\end{align*}
If $(t_p)_{p\geq 1}$ is a sequence in $\{sr,\,s\in\mathbb{N}\}$ then Assumption  \ref{assume:G_l_E} can be omitted.
\end{lemma}

\begin{proof}
Let $\delta_\star$ be as in \ref{assume:Model}, $(\gamma_p)_{p\geq 1}$ be as defined in the proof of Lemma \ref{lemma:piece1},  $(t'''_k)_{k\geq 1}$ be as in the statement of the lemma, and  for all $k\geq 1$  let $p_k=\sup\big\{p\in\mathbb{N}: t_p<t'''_k\}$ and $m_k\in\{0,\dots,r-1\}$ be such that $t_{p_k}+m_k\in\{sr,\,s\in\mathbb{N}\}$. In addition, for all $\epsilon\in(0,\infty)$ we let $(\xi^{(2)}_{\epsilon, t''_k})_{\tau\geq 1}$ be such that the conclusion of Lemma \ref{lemma:piece2} holds when $t''_k=t_{p_k}+m_k-1$ for all $k\geq 1$. Then,  the result of the lemma follows by applying Theorem \ref{thm:key} with    $(a_\tau)_{\tau\geq 1}$, $(b_\tau)_{\tau}$, $(\tilde{c}_\tau)_{\tau\geq 1}$,  $(\tilde{\epsilon}_\tau)_{\tau\geq 1}$ and $((\xi_{\delta,a_\tau})_{\delta\in(0,\delta_\star/2)})_{\tau\geq 1}$ defined by
\begin{equation}\label{eq:def_seq3}
\begin{split}
&a_\tau=t_{p_\tau}+m_\tau,\quad b_\tau=t'''_{\tau},\quad \tilde{c}_\tau=\frac{1}{f_{p_\tau}(\gamma_{p_\tau})},\quad\tilde{\epsilon}_\tau=\gamma_{p_\tau},\quad\tau\geq 1\\
&\xi_{\delta,a_\tau}= -\frac{\log  \Gamma^\mu_{\delta/2}+\log \big(1-\xi^{(2)}_{\delta/2,a_\tau-1}\big)}{v_{p_\tau}-r},\hspace{2.8cm} \tau\geq 1,\quad \delta\in(0,\delta_\star/2).
\end{split}
\end{equation}

To see this  remark  first that,  as required by Theorem \ref{thm:key}, the conditions  the   sequence $(a_\tau)_{\tau\geq 1}$ defined in \eqref{eq:def_seq2} belongs to $\{sr,\,s\in\mathbb{N}\}$ and that the sequence $(b_\tau)_{\tau\geq 1}$ defined in \eqref{eq:def_seq2} is  such that $b_\tau> a_\tau$ for all $\tau\geq 1$. Next, since $\lim_{p\rightarrow\infty}v_p=\infty$ it follows  that $\lim_{\tau\rightarrow\infty} (b_\tau-a_\tau)=\infty$,  as required by Theorem \ref{thm:key}. In addition, as required by Theorem \ref{thm:key}, the  sequences   $(\tilde{c}_{\tau})_{\tau\geq 1}$ and $(\tilde{\epsilon}_{\tau})_{\tau\geq 1}$ defined in \eqref{eq:def_seq2} are such that $ \lim_{\tau\rightarrow\infty}\tilde{\epsilon}_\tau=\lim_{\tau\rightarrow\infty}(1/\tilde{c}_\tau)=0$. To proceed further for all $\tau\geq 1$ we let $e_\tau\in\{sr,\,s\in\mathbb{N}\}$ be such that $e_\tau-r<t'''_\tau\leq e_\tau$. Then, noting that $e_\tau\leq t_{p_{\tau}+1}-1$ for all $\tau\geq 1$ and using \ref{C42} with $(l_p)_{p\geq 1}=(a_p)_{p\geq 1}$ and  $(u_p)_{p\geq 1}=(e_p)_{p\geq 1}$, it follows  that the sequences   $(a_\tau)_{\tau\geq 1}$, $(b_\tau)_{\tau}$, $(\tilde{c}_\tau)_{\tau\geq 1}$ and $(\tilde{\epsilon}_{\tau})_{\tau\geq 1}$ defined in \eqref{eq:def_seq2} are such that \eqref{eq:Thm_C1} holds.

 Finally, to prove that \eqref{eq:Thm_C2} holds let $\delta\in(0,\delta_\star/2)$. Then, under \ref{assume:Model_ThetaStar} and \ref{C41}, and noting that $b_\tau-a_\tau\geq v_{p_\tau}-m_\tau\geq v_\tau-r>0$ for all $\tau\geq 1$, by using similar   computations  as in \eqref{eq:split_ll1}-\eqref{eq:split_ll2}  it follows from Lemma \ref{lemma:piece2}  
\begin{align}\label{eq:split_ll3}
\frac{1}{b_\tau-a_\tau}\log \bar{\pi}_{a_\tau}\big(\mathcal{N}_\delta(\tilde{\Theta}_\star)\big) 
&\geq \frac{1}{v_{p_\tau}-r}\Big( \log \Gamma^\mu_{\delta/2}+\log\big(1- \xi^{(2)}_{\delta/2,a_\tau-1}\big)\Big),\quad\forall \tau\geq 1,\quad\P-a.s.
\end{align}
where the sequence $( \xi^{(2)}_{\delta/2,a_\tau-1})_{\tau\geq 1}$ depends on  $(\mu_t)_{t\geq 1}$  only through $(W^*_{t})_{t\geq 0}$, $(\Gamma^\mu_{\epsilon})_{\epsilon\in (0,\infty)}$, $\Gamma^\mu$ and through the sequence $\big((t_p,v_p,f_p)\big)_{p\geq 1}$. Since under \ref{condition:Inf_mu}-\ref{condition:Inf_K}  we have $\P(\Gamma^\mu_{\delta/2}>0)=\P(\Gamma^\mu>0)=1$ while $\xi^{(2)}_{\delta/2,a_\tau-1}=\smallo_\P(1)$  by  Lemma \ref{lemma:piece2}, it follows that
\begin{align*}
\frac{1}{v_{p_\tau}-r}\Big( \log \Gamma^\mu_{\delta/2}+\log\big(1- \xi^{(2)}_{\delta/2,a_\tau-1}\big)\Big)\PP 0.
\end{align*}
Thus concludes to show that \eqref{eq:Thm_C2} holds for the sequences $(a_\tau)_{\tau\geq 1}$, $(b_\tau)_{\tau\geq 1}$ and $( (\xi_{\delta,a_\tau})_{\delta\in(0,\delta_\star/2)})_{\tau\geq 1}$  defined in \eqref{eq:def_seq2}. The proof of the lemma is complete.
\end{proof}

We are now in position to prove   Theorem \ref{thm:main}:

\begin{proof}[Proof of Theorem \ref{thm:main}]
Let $p^*$ be as in Lemma \ref{lemma:piece1} and let $\mathbb{N}^{(2)}$ and $\mathbb{N}^{(3)}$ be as defined in Lemma \ref{lemma:piece2} and in Lemma \ref{lemma:piece3}, respectively. Let  $\epsilon\in(0,\infty)$  and, for all $t\geq t_{p^*+1}$, let
\begin{align*}
W_{\epsilon,t}^{(1)}=
\begin{cases}
\pi_{t,\Theta}\big(\mathcal{N}_\epsilon(\tilde{\Theta}_\star)^c\big), &t\in\mathbb{N}^{(2)}\\
0, &\text{otherwise}
\end{cases},\quad W_{\epsilon,t}^{(2)}=
\begin{cases}
 \pi_{t,\Theta}\big(\mathcal{N}_\epsilon(\tilde{\Theta}_\star)^c\big), &t\in\mathbb{N}^{(3)}\\
0, &\text{otherwise}
\end{cases}.
\end{align*}

Then, by Lemmas \ref{lemma:piece2}-\ref{lemma:piece3}, there exists    a sequence $(\xi^\star_{\epsilon,t})_{t \geq 1}$ of $[0,1]$-valued random variables, depending on $(\mu_t)_{t\geq 1}$  only   through $(W^*_{t})_{t\geq 0}$, $(\Gamma^\mu_\delta)_{\delta\in(0,\infty)}$,  $\Gamma^\mu$ and $\big((t_p, v_p, f_p)\big)_{p\geq 1}$, such that $\xi^\star_{\epsilon,t}=\smallo_\P(1)$ and such that
\begin{align*}
W_{\epsilon,t}^{(1)}\leq \xi^\star_{\epsilon,t},\quad W_{\epsilon,t}^{(2)}\leq \xi^\star_{\epsilon,t},\quad\forall t\geq t_{p_\star+1},\quad\P-a.s.
\end{align*}
Therefore, since $\mathbb{N}^{(2)}\cup \mathbb{N}^{(3)}=\{t\in\mathbb{N}:\, t\geq t_{p_\star+1}-r\}$, it follows that  
\begin{align*}
0\leq \pi_{t,\Theta}\big( \mathcal{N}_\epsilon(\tilde{\Theta}_\star)^c\big)=W_{\epsilon,t}^{(1)}+W_{\epsilon,t}^{(2)}\leq 2 \xi^\star_{\epsilon,t},\quad\forall t\geq t_{p_\star+1}
\end{align*} 
and the result of the theorem follows.
\end{proof}

\subsection{Proof of Proposition \ref{prop:stability}\label{p-prop:stability}}
\begin{proof}
By Lemma \ref{lemma:Zt}, for all $t\geq 1$ we $\P$-a.s.~have
\begin{align*}
&\pi_{t,\setX}(D)\\
&=\frac{\int_{(\Theta\times\setX)^{t+1}} \ind_D(x_t)\mu_0(\dd\theta_0)\chi(\dd x_0)\prod_{s=1}^t Q_{s,\theta_s}(x_{s-1},\dd x_s)K_{\mu_s|\Theta}(\theta_{s-1},\dd\theta_s)}{\int_{(\Theta\times\setX)^{t+1}} \mu_0(\dd\theta_0)\chi(\dd x_0)\prod_{s=1}^t Q_{s,\theta_s}(x_{s-1},\dd x_s)K_{\mu_s|\Theta}(\theta_{s-1},\dd\theta_s)}\\
&=\frac{\int_{ \Theta^{t+1}\times\setX^{t}} Q_{t,\theta_t}(x_{t-1},D)  \mu_0(\dd\theta_0)\chi(\dd x_0)\Big(\prod_{s=1}^{t-1} Q_{s,\theta_s}(x_{s-1},\dd x_s)K_{\mu_s|\Theta}(\theta_{s-1},\dd\theta_s)\Big)K_{\mu_t|\Theta}(\theta_{t-1},\dd\theta_t) }{\int_{ \Theta^{t+1}\times\setX^{t}} Q_{t,\theta_t}(x_{t-1},\setX) \mu_0(\dd\theta_0)\chi(\dd x_0)\Big(\prod_{s=1}^{t-1} Q_{s,\theta_s}(x_{s-1},\dd x_s)K_{\mu_s|\Theta}(\theta_{s-1},\dd\theta_s)\Big)K_{\mu_t|\Theta}(\theta_{t-1},\dd\theta_t)}\\
&\geq W_t
\end{align*}
and the proof of the proposition is complete.
\end{proof}

\subsection{Proof the Theorem \ref{thm:main2}}

\begin{proof}

Let $A\in\mathcal{T}$, $(\tilde{\xi}_{rt})_{t\geq 1}$ and $(\underline{\chi}_{rt})_{t\geq 1}$ be as in the statement of Lemma \ref{lemma:gamma}  and note that, by Lemma \ref{lemma:Zt} and Lemma \ref{lemma:gamma}, 
\begin{align}\label{eq:bound_piAAA}
\pi_{t,\Theta}(A)\leq \tilde{\xi}^2_{r}\,\frac{P^{\pi_r}_{(r+1):t}(A)}{P^{\bar{\pi}_{r}\otimes\underline{\chi}_{r}}_{(r+1):t}(\Theta)},\quad \forall t>r,\quad \P-a.s.
\end{align}

To study the r.h.s.~of \eqref{eq:bound_piAAA} let $(k_t)_{t\geq 1}$ be as in \ref{condition:mu_seq2} and, for all $t\geq 1$, let $l_t\in\{sr,\,s\in\mathbb{N}\}$ be such that $l_t-r< k_t\leq l_t$. Next,  let $t'\in\mathbb{N}$ be such that  $t>2k_t>l_t>r+1$  for all $t\geq t'$, for all $t\geq 1$ let  $\bar{\pi}_{t}$ be as defined in \eqref{eq:pi_bar}, and let
\begin{align*}
\pi'_{t}(\dd\theta_{t})=\int_{\Theta^{t-r}} \bar{\pi}_{r,\Theta}(\dd\theta_r)\prod_{s=r+1}^{t}K_{\mu_{s}|\Theta}(\theta_{s-1},\dd\theta_{s}),\quad\forall t>r.
\end{align*}
Finally, for all integers $1\leq t_1\leq t_2$ we let $D_{t_1:t_2}$ be as defined in \eqref{eq:Dst}.

Then,  for all $t\geq t'$ we have
\begin{align}\label{eq:bound1}
 P_{(r+1):t}^{\pi_r}(A)\leq Z_{l_t}  P_{(l_t+1):t}^{\pi'_{l_t}\otimes\bar{\chi}'_{l_t}}(A),\quad \P-a.s.
\end{align}
where $\bar{\chi}'_{l_t}$ is a random probability measure on $(\setX,\mathcal{X})$ such that
\begin{align*}
\bar{\chi}'_{l_t}(A)=\frac{\int_A \Big(\sup_{(\theta,x')\in \Theta \times \setX}  q_{l_t,\theta}(x|x')\Big)\lambda(\dd x) }{\int_\setX\Big( \sup_{(\theta,x')\in \Theta \times \setX}  q_{l_t,\theta}(x|x')\Big)\lambda(\dd x) },\quad\forall A\in\mathcal{X},\quad\P-a.s.
\end{align*}
and where
\begin{align*}
Z_{l_t}=D_{(r+1):(l_t-1)}\bigg(\int_\setX \sup_{(\theta,x')\in \Theta \times \setX}  q_{l_t,\theta}(x|x') \lambda(\dd x)\bigg).
\end{align*}

To proceed further let $\tilde{f}_t=f_{t,r+1}$ for all $t\geq 1$ and   $(\gamma_t)_{t\geq 1}$ be a sequence in $(0,\infty)$ such that $\lim_{t\rightarrow\infty}\gamma_t=\lim_{t\rightarrow\infty} \tilde{f}_t(\gamma_t)=0$. Remark that such  a sequence  $(\gamma_t)_{t\geq 1}$ exists since, by assumption, $\lim_{t\rightarrow\infty} \tilde{f}_t(\gamma)=0$ for all $\gamma\in(0,\infty)$. In addition, remark that  under \ref{assume:Model_stationary} and \ref{assume:Model_sup}  $(\bar{\chi}'_{l_t})_{t\geq 1}$ is a sequence of $(r,\delta_\star,\tilde{l},(\varphi_t)_{t\geq 1})$-consistent random probability measures on  $(\setX,\mathcal{X})$, with $(\delta_\star,\tilde{l},(\varphi_t)_{t\geq 1})$ as in \ref{assume:Model}. Therefore,  under \ref{condition:mu_seq2}, we can apply  Lemma \ref{lemma::upper_general} with the sequences $(\kappa_\tau)_{\tau\geq 1}$, $(\eta_{\tau})_{\tau\geq 1}$ and $(a_\tau)_{\tau\geq 1}$, $(b_\tau)_{\tau\geq 1}$, $(c_\tau)_{\tau\geq 1}$ and $(\epsilon_\tau)_{\tau\geq 1}$ such that
\begin{align*}
\kappa_{a_\tau}=\pi'_{l_{t'+\tau}},\quad \eta_{a_\tau}=\bar{\chi}_{l_{t'+\tau}},\quad a_\tau=l_{t'+\tau},\quad b_\tau=t'+\tau,\quad c_\tau=\frac{1}{\tilde{f}_{t'+\tau}(\gamma_{t'+\tau})},\quad\epsilon_\tau=\gamma_{t'+\tau},\quad   \forall\tau\geq 1
\end{align*}
and thus there exists a sequence $(\xi^+_{t})_{t\geq 1}$ of $[0,\infty]$-valued random variables, depending on $(\mu_t)_{t\geq 1}$ only through    $\Gamma^\mu$ and $((k_t,f_t))_{t\geq 1}$,  such that $\xi^+_{t}=\smallo_\P(1)$ and such that
\begin{align}\label{eq:bound22}
\P\bigg(  P_{(l_t+1):t}^{\pi'_{l_t}\otimes\bar{\chi}_{l_t}}(A)\leq (1+\xi^+_{t})\exp\Big\{ (t-l_t)\big(\sup_{\theta\in A} \tilde{l}(\theta)+\xi^+_{t}\big)\Big\}\bigg)=1,\quad\forall t\geq 1. 
\end{align}
Remark now that, under \ref{assume:Model_stationary}, \ref{assume:G_bounded} and \ref{assume:Model_sup}, and noting  that $\lim_{t\rightarrow\infty} l_t/t=0$,  
\begin{align}\label{eq:bound3}
\frac{1}{t-l_t}\log^+ Z_{l_t}=\smallo_\P(1).
\end{align}
Therefore, by combining \eqref{eq:bound1}-\eqref{eq:bound3}, it follows that   there exists a sequence $(\tilde{\xi}^+_{t})_{t\geq 1}$ of $[0,\infty]$-valued random variables, depending on $(\mu_t)_{t\geq 1}$ only through  $\Gamma^\mu$ and $((k_t,f_t))_{t\geq 1}$,  such that $\tilde{\xi}^+_{t}=\smallo_\P(1)$ and such that
\begin{align}\label{eq:bound2}
\P\bigg(  P_{(r+1):t}^{\pi_{r}}(A)\leq (1+\tilde{\xi}^+_{t})\exp\Big\{ (t-l_t)\big(\sup_{\theta\in A} \tilde{l}(\theta)+\tilde{\xi}^+_{t}\big)\Big\}\bigg)=1,\quad\forall t\geq 1. 
\end{align}

To proceed further, for all $t\geq 1$ we let $e_t\in\{sr,\,s\in\mathbb{N}\}$ be such that $ e_t-r <t\leq e_t$ and recall that, by Lemma \ref{lemma:gamma}, $(\underline{\chi}_{rt})_{t\geq 1}$ is a sequence of   $(r,\delta_\star,\tilde{l},(\varphi_t)_{t\geq 1})$-consistent random probability measures on  $(\setX,\mathcal{X})$. Therefore, under \ref{condition:mu_seq2}, we can apply Lemma \ref{lemma::lower_general}   with the sequences $(\kappa_\tau)_{\tau\geq 1}$, $(\eta_{\tau})_{\tau\geq 1}$ and $(a_\tau)_{\tau\geq 1}$, $(b_\tau)_{\tau\geq 1}$, $(c_\tau)_{\tau\geq 1}$ and $(\epsilon_\tau)_{\tau\geq 1}$ such that
\begin{align*}
\kappa_{a_\tau}=\bar{\pi}_{r},\quad\eta_{a_\tau}=\underline{\chi}_{r},\quad a_\tau=r,\quad b_\tau=t'+\tau,\quad c_\tau=\frac{1}{\tilde{f}_{e_{t'+\tau}}(\gamma_{e_{t'+\tau}})},\quad\epsilon_\tau=\gamma_{e_{t'+\tau}},\quad \forall\tau\geq 1,
\end{align*}
and thus there exist an  $\bar{\epsilon}\in (0,\delta_\star/2)$, independent of $(\mu_t)_{t\geq 1}$,   and a sequence $(\xi^-_{t})_{t\geq 1}$ of $[0,\infty]$-valued random variables,  depending on $(\mu_t)_{t\geq 1}$ only through through  $((k_t,f_t))_{t\geq 1}$, such that $\xi^-_{t}=\smallo_\P(1)$ and such that 
\begin{align}\label{eq:bound222}
\P\bigg( P^{\bar{\pi}_{r}\otimes\underline{\chi}_{r}}_{(r+1):t}(\Theta)  &\geq  \bar{\pi}_{r}\big(\mathcal{N}_\epsilon(\tilde{\Theta}_\star)\big) \exp\Big\{ (t-r)\big(\inf_{\theta\in \mathcal{N}_\epsilon(\tilde{\Theta}_\star)}\tilde{l}(\theta)-\xi_{t}^-\big)\Big\}\bigg)=1,\quad\forall \epsilon\in (0,\bar{\epsilon}], \quad\forall t\geq 1. 
\end{align}

In order to complete the proof of the theorem we need to show that
\begin{align}\label{eq:lower_pir}
\P\big(\bar{\pi}_{r}\big(\mathcal{N}_\epsilon(\tilde{\Theta}_\star)\big)>0\big)=1,\quad\forall\epsilon\in(0,\infty).
\end{align}
 To this aim let $\epsilon\in(0,\infty)$, $(\mathfrak{Z}_t)_{t\geq 1}$ be as defined in \eqref{eq:normalizing} and note that, by Lemma \ref{lemma:Zt}, we $\P$-a.s.~have
 \begin{align*}
\bar{\pi}_{r}\big(\mathcal{N}_\epsilon(\tilde{\Theta}_\star)\big)&=\int_{\Theta}  K_{\mu_{r}|\Theta}\big(\theta_{r-1},\mathcal{N}_{\epsilon}(\tilde{\Theta}_\star)\big)\pi_{r-1,\Theta}(\dd\theta_{r-1})\\
&=\frac{\int_{(\Theta\times\setX)^{r}}K_{\mu_{r}|\Theta}\big(\theta_{r-1},\mathcal{N}_{\epsilon}(\tilde{\Theta}_\star)\big) \mu_0(\dd\theta_0)\chi(\dd x_0)\prod_{s=1}^{r-1} Q_{s,\theta_s}(x_{s-1},\dd x_s)K_{\mu_s|\Theta}(\theta_{s-1},\dd\theta_s)}{\mathfrak{Z}_{r-1}}\\
&\geq \frac{\int_{(\Theta\times\setX)^{r+1}}\ind_{\mathcal{N}_{\epsilon}(\tilde{\Theta}_\star)}(\theta_{r}) \mu_0(\dd\theta_0)\chi(\dd x_0)\prod_{s=1}^{r} Q_{s,\theta_s}(x_{s-1},\dd x_s)K_{\mu_s|\Theta}(\theta_{s-1},\dd\theta_s)}{\mathfrak{Z}_{r-1}D_{r:r}}
\end{align*}
where $\P(\mathfrak{Z}_{r-1}<\infty)=1$   by \eqref{eq:ZT_show} while  $\P(D_{r,r}<\infty)=1$  under  \ref{assume:G_bounded}.

For all $t\geq 1$ we now let $L_t$  be as defined in \eqref{eq:Lt_def} and $f_{t,\epsilon}$ be as defined in \eqref{eq:ft_def}  for all $\epsilon\in(0,\infty)$. Then, following the computations   in \eqref{eq:lower_z1}-\eqref{eq:lower_z2}, for all $\epsilon\in (0,\delta_\star)$ we have
\begin{align*}
\int_{(\Theta\times\setX)^{r+1}}\ind_{\mathcal{N}_{\epsilon}(\tilde{\Theta}_\star)}&(\theta_r) \mu_0(\dd\theta_0)\chi(x_0)\prod_{s=1}^{r} Q_{s,\theta,s}(x_{s-1},\dd x_s)K_{\mu_s|\Theta}(\theta_{s-1},\dd\theta_s)\\
&\geq L_{r} \int_{ \setR^{r+1}} \ind_{\mathcal{N}_{\epsilon}(\tilde{\Theta}_\star)}(\theta_{r})f_{r,\epsilon}(\theta_0,\dots,\theta_r)\,\mu_0(\dd\theta_0) \prod_{s=1}^{r}   \ind_{\Theta}(\theta_s)K_{\mu_{s}}(\theta_{s-1},\dd\theta_{s})\\
&\geq L_{r}\mu_0\big(\mathcal{N}_{\epsilon/2}(\tilde{\Theta}_\star)\big)\prod_{s=1}^{r}\mu_s\big(B_{\epsilon/(2r)}(0)\big)\\
&\geq   L_{r}\mu_0\big(\mathcal{N}_{\epsilon/2}(\tilde{\Theta}_\star)\big)\big(\Gamma^\mu_{\epsilon/(2r)}\big)^{r}
\end{align*}
where the last inequality holds under \ref{condition:Inf_mu}. Since   $\mu_0\big(\mathcal{N}_{\epsilon/2}(\tilde{\Theta}_\star)\big)>0$  under \ref{condition:mu1},    $\P(\Gamma^\mu_{\epsilon/(2r)}>0)=1$ under \ref{condition:Inf_mu} while $\P(L_{r}>0)=1$ as shown in the proof of Lemma \ref{lemma:Zt}, the proof of \eqref{eq:lower_pir} is complete.

The result of the theorem then follows directly from   \eqref{eq:bound_piAAA},  \eqref{eq:bound2}-\eqref{eq:lower_pir}  and by using Assumption \ref{assume:Model_ThetaStar}.

\end{proof}

\subsection{Proof of Theorem \ref{thm:pred}}

\subsubsection{Additional preliminary results}

The strategy of the proof of Theorem \ref{thm:pred} is as follows. By Theorem \ref{thm:main},   there exists a sequence $(\epsilon_t)_{t\geq 1}$ in $(0,\infty)$ such that $\lim_{t\rightarrow\infty}\epsilon_t=0$  and such that $\pi_{t,\Theta}\big(\mathcal{N}_{\epsilon_t}(\{\tilde{\Theta}_\star\})\big)\PP 1$, and thus to establish the result of the theorem it is enough to study $\pi_{t}(\mathcal{N}_{\epsilon_t}(\{\tilde{\Theta}_\star\})\times A)$ for all $A\in\mathcal{X}$. 

Together with Lemma \ref{lemma:Zt}, the following three lemmas will be used to obtain, for an arbitrary set $A\in\mathcal{X}$, an upper bound for $\pi_{t}(\mathcal{N}_{\epsilon_t}(\{\tilde{\Theta}_\star\})\times A)$.

\begin{lemma}\label{lemma:pred_def}
Assume that Assumptions \ref{assume:Model}-\ref{assume:G_bounded} and \eqref{eq:W_star} hold and that $(\mu_t)_{t\geq 0}$ is such that   Conditions \ref{condition:mu1}-\ref{condition:Inf_mu} are satisfied, and let $r$ be as in \ref{assume:Model}. Then, for all $t\geq 2r$, all $\tau\in\{2r,\dots,t\}\cap\{sr,\,s\in\mathbb{N}\}$ and all $\theta\in\Theta$,
\begin{align*}
\bar{p}_{(\tau+1):t}^{\pi_{\tau,\setX}}(A|\theta)=\frac{\int_{\setX^{t-\tau+1}}\ind_A(x_{t}) \pi_{\tau,\setX}(\dd x_{\tau}) \prod_{s=\tau+1}^{t} Q_{s,\theta}(x_{s-1},\dd x_s)}{\int_{\setX^{t-\tau+1}}  \pi_{\tau,\setX}(\dd x_{\tau}) \prod_{s=\tau+1}^{t} Q_{s,\theta}(x_{s-1},\dd x_s)},\quad\forall A\in\mathcal{X},\quad\P-a.s.
\end{align*}
\end{lemma}
\begin{proof}
See Section \ref{p-lemma:pred_def}.
\end{proof}

\begin{lemma}\label{lemma:upper_X}
Assume that  \ref{assume:Model}-\ref{assume:G_bounded} and \eqref{eq:W_star} holds and  that $(\mu_t)_{t\geq 1}$ is such that Condition \ref{condition:Inf_K} is satisfied, and let $r$ and $g(\cdot)$ be as in \ref{assume:Model}. Let  $(a_\tau)_{\tau\geq 1}$  be a sequence in  $\{sr,\,s\in\mathbb{N}\}$ and $(b_\tau)_{\tau\geq 1}$  be a   sequence  in $\mathbb{N}_0$   such that  $a_\tau< b_\tau$ for all $\tau\geq 1$ and such that $\lim_{\tau\rightarrow\infty} (b_\tau-a_\tau)=\infty$, and assume that there exist  two sequences   $(c_{\tau})_{\tau\geq 1}$ and  $(\epsilon_{\tau})_{\tau\geq 1}$ in $(0,\infty)$ such that $ \lim_{\tau\rightarrow\infty}\epsilon_{\tau}=\lim_{\tau\rightarrow\infty}(1/c_{\tau})=0$, such that
\begin{align*}
\frac{1}{b_\tau-a_\tau}\log\P\Big(\exists s\in\{a_\tau+1,\dots,  b_\tau\}:  \sum_{i=a_\tau+1}^s  U_i\not\in B_{\epsilon_{\tau}}(0)\,\big|\, \F_{b_\tau}\Big)\leq -c_{\tau},\quad\forall \tau\geq 1,\quad\P-a.s.
\end{align*}
and  such that $\lim_{\tau\rightarrow\infty}(b_\tau-a_\tau)^{-1}\log g(3\epsilon_\tau)=-\infty$. Then, for all $\theta\in\Theta$ there exist  a $\tau'\in\mathbb{N}$ and a  sequence $(\xi^+_{b_\tau})_{\tau\geq 1}$ of $(0,\infty]$-valued random variables, both depending on $(\mu_t)_{t\geq 1}$ only through the sequences $(c_{\tau})_{\tau\geq 1}$ and $(\epsilon_{\tau})_{\tau\geq 1}$, through the sequence $(W^*_{t})_{t\geq 0}$ defined in \eqref{eq:W_star} and through the random variable $\Gamma^\mu$  defined in \ref{condition:Inf_K}, such  that $\xi_{b_\tau}^+=\smallo_\P(1)$ and such that, for all $\tau\geq \tau'$,
\begin{align*}
\Lambda^{\pi_{a_\tau}}_{(a_\tau+1):b_\tau}\big(\mathcal{N}_{\epsilon_t}(\{\theta\})\times A\big)\leq \bar{P}^{\pi_{a_\tau,\setX}}_{(a_\tau+1):b_\tau}(\theta)\big( \bar{p}_{(a_\tau+1):b_\tau}^{\pi_{a_\tau,\setX}}(A|\theta)+\xi_{b_\tau}^+\big),\quad\forall A\in  \mathcal{X},\quad\P-a.s.
\end{align*}
\end{lemma}
\begin{proof}
See Section \ref{p-lemma:upper_X}.
\end{proof}

\begin{lemma}\label{lemma:lower_X}
Consider the set-up of Lemma \ref{lemma:upper_X}. Then, for all $\theta\in\Theta$ there exists a  sequence $(\xi^-_{b_\tau})_{\tau\geq 1}$ of $(0,1]$-valued random variables,  depending on $(\mu_t)_{t\geq 1}$ only through the sequences $(c_{\tau})_{\tau\geq 1}$ and $(\epsilon_{\tau})_{\tau\geq 1}$ and through the sequence $(W^*_{t})_{t\geq 0}$ defined in \eqref{eq:W_star}, such  that $\xi_{b_\tau}^-=\smallo_\P(1)$ and such that, $\P$-a.s.,
\begin{align*}
\Lambda^{\pi_{a_\tau}}_{(a_\tau+1):b_\tau}(\mathcal{N}_{2\epsilon_t}(\{\theta\})\times \setX)\geq \pi_{a_\tau,\Theta}\big(\mathcal{N}_{\epsilon_t}(\{\theta\})\big)e^{-\frac{b_\tau-a_\tau}{c_\tau}}\bar{P}^{\pi_{a_\tau,\setX}}_{(a_\tau+1):b_\tau}(\theta)( 1-\xi^-_{b_\tau}),\quad\forall \tau\geq 1.
\end{align*}
\end{lemma}
\begin{proof}
See Section \ref{p-lemma:lower_X}.
\end{proof}

\subsubsection{Proof of the theorem}

\begin{proof}

Let $(\delta'_t)_{t\geq 1}$  be as in \ref{condition:mu_extra} and, to simplify the notation, for all $\epsilon\in(0,\infty)$ let $\mathcal{N}^\star_\epsilon=\mathcal{N}_\epsilon(\{\tilde{\theta}_\star\})$. Then, by Theorem \ref{thm:main} (resp.~by  Theorem \ref{thm:main2} if \ref{condition:mu_seq2} holds) there exists a sequence $(\epsilon'_t)_{t\geq 1}$ in $(0,\infty)$ and a sequence $(\xi_t)_{t\geq 1}$ of $[0,1]$-valued random variables, both sequences depending on $(\mu_t)_{t\geq 1}$ only through   $(W^*_{t})_{t\geq 0}$, $(\Gamma^\mu_\delta)_{\delta\in (0,\infty)}$, $\Gamma^\mu$ and $((v_p, t_p,f_p))_{p\geq 1}$ (resp.~$(k_t,f_{t,r+1})_{t\geq 1}$), such that $\lim_{t\rightarrow\infty}\epsilon'_t=0$, such that $\epsilon'_t\geq \delta'_t$ for all $t\geq 1$,  such that  $\xi_t=\smallo_\P(1)$ and such that 
\begin{align}\label{eq:lower_PI}
\pi_{t,\Theta} ( \mathcal{N}^\star_{\epsilon'_t/2})\geq 1-\xi_t,\quad \forall t\geq 1,\quad\P-a.s.
\end{align}

Let $r$ and $g(\cdot)$ be as in \ref{assume:Model},     $c'_{t}=1/f_{t}'(1/\delta'_{t},\delta'_{t})$ for all $t\geq 1$  and note that $\lim_{t\rightarrow\infty}c_t'=\infty$ under \ref{condition:mu_extra}. Next, let $(s_t)_{t\geq 1}$ be  a sequence in $\{sr,\,s\in\mathbb{N}_0\}$ such that
\begin{align}\label{eq:g_lim0}
t-\frac{1}{\delta'_{t}}\leq s_t<t,\quad \epsilon'_t\geq\frac{\epsilon'_{s_t}}{2},\quad \forall t\geq 1
\end{align}
and such that 
\begin{align}\label{eq:g_lim}
\lim_{t\rightarrow\infty}\frac{\log g(3\epsilon'_{t})}{t-s_t}=-\infty,\quad \lim_{t\rightarrow\infty}\frac{t-s_t}{c'_t}=0,\quad \lim_{t\rightarrow\infty}(t-s_t)=\infty,\quad\liminf_{t\rightarrow\infty}\frac{\epsilon'_t}{\epsilon'_{s_t}}\geq \frac{3}{4}.
\end{align}
Remark that  \eqref{eq:g_lim} is satisfied if $(t-s_t)\rightarrow\infty$ sufficiently slowly. In addition, remark that
\begin{equation}\label{eq:cond_K}
\begin{split}
\frac{1}{t-s_t} &\log\P\Big(\exists    s\in\{s_t+1,\dots,  t\}:  \sum_{i=s_t+1}^s  U_i\not\in  B_{\epsilon'_{t}}(0)\,\big|\, \F_{t}\Big)\\
&\leq \frac{1}{t-s_t}\log\P\Big(\exists s\in\{s_t+1,\dots,  t\}:  \sum_{i=s_t+1}^s  U_i\not\in B_{\delta'_{t}}(0)\,\big|\, \F_{t}\Big)\\
&\leq -c'_{t},\hspace*{9cm}\quad\forall t\geq 1,\quad\P-a.s.
\end{split} 
\end{equation}
where the first inequality uses the fact that $\epsilon'_t\geq \delta'_t$ for all $t\geq 1$ while the second inequality holds under \ref{condition:mu_extra}. Finally, remark that, by \eqref{eq:lower_PI} and the last condition given in \eqref{eq:g_lim}, there exists a $t'\in\mathbb{N}$ such that
\begin{align}\label{eq:lower_PI2}
\pi_{s_t,\Theta}(\mathcal{N}^\star_{\epsilon'_t})\geq \pi_{s_t,\Theta}( \mathcal{N}^\star_{\epsilon'_{s_t}/2})\geq 1-\xi_{s_t},\quad \forall t\geq t',\quad\P-a.s.
\end{align}

To prove the result of the theorem remark first that, by Lemma \ref{lemma:Zt}, for all $t\geq 1$ we have
\begin{align}\label{eq:bound_AX}
\pi_{s_t}\big(\mathcal{N}^\star_{\epsilon'_t}\times A\big)=\frac{\Lambda^{\pi_{s_t}}_{(s_t+1):t}\big(\mathcal{N}^\star_{\epsilon'_t}\times A\big)}{\Lambda^{\pi_{s_t}}_{(s_t+1):t}(\Theta\times \setX)}\leq  \frac{\Lambda^{\pi_{s_t}}_{(s_t+1):t}\big(\mathcal{N}^\star_{\epsilon'_t}\times A\big)}{\Lambda^{\pi_{s_t}}_{(s_t+1):t}\big(\mathcal{N}^\star_{2\epsilon'_{t}}\times \setX\big)},\quad\forall A\in\mathcal{X},\quad\P-a.s
\end{align}

By  \eqref{eq:g_lim}-\eqref{eq:cond_K}, the assumptions of Lemma \ref{lemma:upper_X}  hold for the sequences $(a_\tau)_{\tau\geq 1}$, $(b_\tau)_{\tau\geq 1}$,  $(c_\tau)_{\tau\geq 1}$ and $(\epsilon_\tau)_{\tau\geq 1}$ defined by
\begin{align}\label{eq:seu_lemmas}
a_\tau=s_\tau,\quad b_\tau=\tau,\quad \epsilon_\tau=\epsilon'_{\tau},\quad c_\tau=c'_{\tau},\quad  \tau\geq 1
\end{align}
and thus, by applying this lemma with $\theta=\tilde{\theta}_\star$, it follows that there exist  a $t''\in\mathbb{N}$ and a  sequence $(\xi^+_{t})_{\tau\geq 1}$ of $(0,\infty]$-valued random variables, both  depending on $(\mu_t)_{t\geq 1}$ only through the sequences  $(W^*_{t})_{t\geq 0}$, $(c_{\tau})_{\tau\geq 1}$ and $(\epsilon_{\tau})_{\tau\geq 1}$ and through the random variable $\Gamma^\mu$  defined in \ref{condition:Inf_K}, such  that $\xi_{t}^+=\smallo_\P(1)$ and such that, for all $t\geq t''$,
\begin{align}\label{eq:uu}
\Lambda^{\pi_{s_t}}_{(s_t+1):t}\big(\mathcal{N}^\star_{\epsilon'_{t}}\times A\big)\leq \bar{P}^{\pi_{s_t,\setX}}_{(s_t+1):t}(\tilde{\theta}_\star)\big( \bar{p}_{(s_t+1):t}^{\pi_{s_t,\setX}}(A|\tilde{\theta}_\star)+\xi_{t}^+\big),\quad\forall A\in  \mathcal{T},\quad\P-a.s.
\end{align}

In  addition,  by  \eqref{eq:g_lim}-\eqref{eq:cond_K} the assumptions of Lemma \ref{lemma:lower_X}  hold for the sequences $(a_\tau)_{\tau\geq 1}$, $(b_\tau)_{\tau\geq 1}$, $(c_\tau)_{\tau\geq 1}$ and  $(\epsilon_\tau)_{\tau\geq 1}$ defined in \eqref{eq:seu_lemmas}, and thus  by applying this lemma with $\theta=\tilde{\theta}_\star$, it follows that there exist   a  sequence $(\xi^-_{t})_{\tau\geq 1}$ of $[0,1]$-valued random variables,    depending on $(\mu_t)_{t\geq 1}$ only through the sequences  $(W^*_{t})_{t\geq 0}$, $(c_{\tau})_{\tau\geq 1}$ and $(\epsilon_{\tau})_{\tau\geq 1}$, such  that $\xi_{t}^-=\smallo_\P(1)$ and such that, for all $t\geq 1$,
\begin{align}\label{eq:ll}
\Lambda^{\pi_{s_t}}_{(s_t+1):t}\big(\mathcal{N}^\star_{2\epsilon'_{t}}\times \setX\big)\geq \pi_{s_t}(\mathcal{N}^\star_{\epsilon'_{t}})\exp\Big(-\frac{t-s_t}{c'_{t}}\Big)\bar{P}^{\pi_{s_t,\setX}}_{(s_t+1):t}(\tilde{\theta}_\star)( 1-\xi^-_{t}),\quad\forall t\geq 1,\quad\P-a.s.
\end{align}

By combining \eqref{eq:lower_PI2}, \eqref{eq:bound_AX}, \eqref{eq:uu} and \eqref{eq:ll}, we obtain that for all $t\geq t''':=t'\vee t''$ we have
\begin{align*}
\pi_{t}\big(\mathcal{N}^\star_{\epsilon'_{t}} \times A\big)\leq \frac{\exp\big(-\frac{t-s_t}{c'_{t}}\big)}{(1-\xi_{s_t})(1-\xi^-_{t})}\Big(\bar{p}_{(s_t+1):t}^{\pi_{s_t,\setX}}(A|\tilde{\theta}_\star)+\xi_{t}^+\Big),\quad\forall A\in\mathcal{X},\quad\P-a.s.
\end{align*}
and thus, letting
\begin{align*}
Z_{t}=\Big|\frac{\exp\big(\frac{t-s_t}{c'_{t}}\big)}{(1-\xi_{s_t})(1-\xi^-_{t})}-1\Big|+\frac{\exp\big(\frac{t-s_t}{c'_{t}}\big)}{(1-\xi_{s_t})(1-\xi^-_{t})} \xi^+_{t},\quad\forall t\geq 1 
\end{align*}
it follows that
\begin{align}\label{eq:main_bound_1}
\pi_{t}\big(\mathcal{N}^\star_{\epsilon'_{t}}\times A\big)-\bar{p}_{(s_t+1):t}^{\pi_{s_t,\setX}}(A|\tilde{\theta}_\star)\leq  Z_{t},\quad\forall A\in\mathcal{X},\quad\P-a.s.,\quad\forall t\geq t'''.
\end{align}
 
To complete the proof remark that
\begin{align}\label{eq:main_bound_2}
\pi_{t,\setX}(A)\leq \pi_{t,\Theta}\big(\Theta\setminus\mathcal{N}^\star_{\epsilon'_{t}}\big)+\pi_{t}\big(\mathcal{N}^\star_{\epsilon'_{t}}\times A\big),\quad\forall A\in\mathcal{X},\quad\forall t\geq 1
\end{align}
and thus, by  \eqref{eq:lower_PI}, \eqref{eq:main_bound_1} and \eqref{eq:main_bound_2}, we obtain that
\begin{align}\label{eq:main_bound3}
\sup_{A\in\mathcal{X}}\Big(\pi_{t,\setX}(A)-\bar{p}_{(s_t+1):t}^{\pi_{s_t,\setX}}(A|\tilde{\theta}_\star)\Big)\leq  \xi_{t}+Z_{t},\quad\P-a.s.,\quad\forall t\geq t'''.
\end{align}
It is readily checked that \eqref{eq:main_bound3} is equivalent to
\begin{align*}
\sup_{A\in\mathcal{X}}\Big|\pi_{t,\setX}(A)-\bar{p}_{(s_t+1):t}^{\pi_{s_t,\setX}}(A|\tilde{\theta}_\star)\Big|\leq  \xi_{t}+Z_{t},\quad\P-a.s.,\quad\forall t\geq t'''
\end{align*}
and the proof the theorem is complete upon noting that  $\xi_{t}+Z_{t}=\smallo_\P(1)$ and that both $t'''$ and the sequence $(\xi_{t}+Z_{t})_{t\geq t'''}$  depend on $(\mu_t)_{t\geq 1}$ only through  $(W^*_{t})_{t\geq 0}$, $(\Gamma^\mu_\delta)_{\delta\in (0,\infty)}$, $\Gamma^\mu$, $((v_p,t_p,f_p))_{p\geq 1}$ (resp.~$((k_t,f_{t,r+1}))_{t\geq 1})$ if \ref{condition:mu_seq2} holds) and $(f'_t)_{t\geq 1}$.
\end{proof}

\subsection{Proof of the preliminary results}

\subsubsection{Proof of Lemma \ref{lemma:tech}\label{p-lemma:tech}}
\begin{proof}

For $t\in\{t_2,t_3\}$ let
\begin{align*}
\zeta_{t}\big(\dd(\theta_{t_1-1},\dots,\theta_t,x_{t_1-1},\dots,x_t)\big)= \zeta\big(\dd(\theta_{t_1-1},x_{t_1-1})\big) \prod_{s=t_1}^{t}m_{s,\theta_{s}}(x_s|x_{s-1})\lambda(\dd x_s) K_{\mu_s|\Theta}(\theta_{s-1},\dd\theta_{s})
\end{align*}
and, for all $t\geq 1$, let
\begin{align*}
\tilde{G}_{t,\theta}(x',x)=\frac{G_{t,\theta}(x',x)}{D_{t:t}},\quad \forall(\theta,x,x')\in\Theta\times\setX^2 
\end{align*}
where, for all integers $1\leq s_1\leq s_2$, the random variable $D_{s_1:s_2}$ is as defined in \eqref{eq:Dst}.

Then,
\begin{align*}
P_{t_1: t_2}^\zeta(\Theta)&=\int_{(\Theta\times\setX)^{t_2-t_1+2}} \zeta_{t_2}\big(\dd(\theta_{t_1-1},\dots,\theta_{t_2},x_{t_1-1},\dots,x_{t_2})\big)\prod_{s=t_1}^{t_2} G_{s,\theta_s}(x_{s-1},x_s)\\
&= \bigg(\int_{(\Theta\times\setX)^{t_2-t_1+2}} \zeta_{t_2}\big(\dd(\theta_{t_1-1},\dots,\theta_{t_2},x_{t_1-1},\dots,x_{t_2})\big)\prod_{s=t_1}^{t_2} \tilde{G}_{s,\theta_s}(x_{s-1},x_s)\bigg) D_{t_1:t_2}\\
&\geq  \bigg(\int_{(\Theta\times\setX)^{t_3-t_1+2}} \zeta_{t_3}\big(\dd(\theta_{t_1-1},\dots,\theta_{t_3},x_{t_1-1},\dots,x_{t_3})\big)\prod_{s=t_1}^{t_3} \tilde{G}_{s,\theta_s}(x_{s-1},x_s)\bigg)D_{t_1:t_2}\\
&=P_{t_1: t_3}^\zeta(\Theta) D_{(t_2+1):t_3}^{-1}.
\end{align*}
\end{proof}

\subsubsection{Proof of Lemma \ref{lemma::W}\label{p-lemma::W}}
\begin{proof}
For all $t\geq 2$ let $n_t=r(t-1)+1$ and, for all $\theta\in\Theta$ and $\delta\in [0,\infty)$, let
\begin{align}\label{eq:W_theta}
W_{rt,\theta}(\delta)=\frac{1}{rt}\log \frac{ \int_{\setX^{n_t}}\exp\Big( \delta   \sum_{s= r +1}^{ r t}\varphi_s(x_{s-1}, x_{s})\Big) \eta_r(\dd x_{r})  \prod_{s=r+1}^{r t} Q_{s,\theta}(x_{s-1},\dd x_{s})}{\bar{P}^{\eta_r}_{(r+1):rt}(\theta)}
\end{align}
where, for all integers $1\leq t_1\leq t_2$ and $\eta\in\mathcal{P}(\setX)$, the random variable  $\bar{P}^\eta_{t_1:t_2}(\theta)$ is as defined in \eqref{eq:Pbamodelef}.

By assumption, $\eta_r$ is  $(r,\delta_\star, \tilde{l},(\varphi_t)_{t\geq 1}\big)$-consistent and thus $\P\big(\bar{P}^{\eta_r}_{(r+1):rt}(\theta)\in (0,\infty)\big)=1$ for all $t\geq 2$ and all $\theta\in\Theta$. Therefore, for all $t\geq 2$ and $\theta\in\Theta$ there exists a random probability measure $\eta_{rt,\theta}$ on $(\setX^{n_t},\otimes^{n_t}\mathcal{X})$ such that
\begin{align}\label{eq:eta_theta}
\eta_{rt,\theta}\big(A\big)=\int_A\frac{ \eta_r(\dd x_{r}) \prod_{s=r+1}^{rt} Q_{s,\theta}(x_{s-1},\dd x_{s})}{\bar{P}^{\eta_r}_{(r+1):rt}(\theta)},\quad\forall A\in \otimes^{n_t}\mathcal{X},\quad\P-a.s.
\end{align}
 
To proceed further let $g(\cdot)$ be as in \ref{assume:Model}, $\epsilon_\star= g^{-1}(\delta_\star^2/4)$ and $t\geq 2$. Then, $\P$-a.s., for all $\theta'\in\Theta$ and   $\theta\in B_{\epsilon_\star}(\theta')\cap\Theta$  we have
 \begin{equation}\label{eq:W_1}
 \begin{split}
\frac{1}{rt}&\log W_{rt,\theta}(\delta_\star/2)\\
&=\frac{1}{rt}\log   \int_{\setX^{n_t}}\exp\Big(\frac{\delta_\star}{2} \sum_{s=r+1}^{rt}\varphi_s(x_{s-1}, x_{s})\Big)\eta_{rt,\theta}(\dd(x_{r},\dots,x_{rt})\big)\\
 &\leq \frac{1}{rt}\log \int_{\setX^{n_t}}\exp\Big(\Big\{\frac{\delta_\star}{2}+g\big(\|\theta-\theta'\|\big)\Big\}\sum_{s=r+1}^{rt}\varphi_s(x_{s-1}, x_{s})\Big)\eta_{rt,\theta'}(\dd(x_{r},\dots,x_{rt})\big)\\
 &+ \frac{1}{rt}\log\frac{\bar{P}^{\eta_r}_{(r+1):rt}(\theta')}{\bar{P}^{\eta_r}_{(r+1):rt}(\theta)}\\
  &\leq   \frac{1}{rt}\log \int_{\setX^{n_t}}\exp\Big(  \delta_\star\sum_{s=r+1}^{rt}\varphi_s(x_{s-1}, x_{s})\Big)\eta_{rt,\theta'}(\dd(x_{r},\dots,x_{rt})\big)+ \frac{1}{rt}\log\frac{\bar{P}^{\eta_r}_{(r+1):rt}(\theta')}{\bar{P}^{\eta_r}_{(r+1):rt}(\theta)}\\ 
  &=\frac{1}{rt}\log W_{rt,\theta'}(\delta_\star)+ \frac{1}{rt}\log\frac{\bar{P}^{\eta_r}_{(r+1):rt}(\theta')}{\bar{P}^{\eta_r}_{(r+1):rt}(\theta)}
  \end{split}
 \end{equation}
 where the first inequality holds under  \ref{assume:Model_smooth}  while the second inequality uses the fact that $\delta_\star\in(0,1)$.

To control the last term in \eqref{eq:W_1} remark that, $\P$-a.s., for all $\theta'\in\Theta$ and   $\theta\in B_{\epsilon_\star}(\theta')\cap\Theta$  we have
\begin{equation}\label{eq:W_2}
\begin{split}
 \frac{1}{rt}&\log\frac{\bar{P}^{\eta_r}_{(r+1):rt}(\theta')}{\bar{P}^{\eta_r}_{(r+1):rt}(\theta)}\\
 &\leq \frac{1}{rt}\log   \int_{\setX^{n_t}}\exp\Big(g\big(\|\theta-\theta'\|\big) \sum_{s=r+1}^{rt}\varphi_s(x_{s-1}, x_{s})\Big)\eta_{rt,\theta}(\dd(x_{r},\dots,x_{rt})\big)\\
 &\leq g\big(\|\theta-\theta'\|\big)^{1/2}\frac{1}{rt}\log   \int_{\setX^{n_t}}\exp\Big(g\big(\|\theta-\theta'\|\big)^{1/2} \sum_{s=r+1}^{rt}\varphi_s(x_{s-1}, x_{s})\Big)\eta_{rt,\theta}(\dd(x_{r},\dots,x_{rt})\big)\\
 &\leq  \frac{\delta_{\star}}{rt}\log W_{rt,\theta}(\delta_\star/2)
\end{split}
\end{equation}
where the first inequality holds under   \ref{assume:Model_smooth}  while the second inequality uses  Jensen's inequality.

By combining \eqref{eq:W_1} and \eqref{eq:W_2}  we obtain  that, for all $t\geq 2$,
\begin{align}\label{eq:W2}
\frac{1}{rt}\log W_{rt,\theta}(\delta_{\star}/2)\leq\frac{1}{1-\delta_\star}  \frac{1}{rt}\log W_{rt,\theta'}(\delta_\star),\quad \forall \theta\in B_{\epsilon_\star}(\theta')\cap\Theta,\quad\forall \theta'\in\Theta,    \quad \P-a.s. 
\end{align}

We now let  $\{\theta'_i\}_{i=1}^k\subset \Theta$ be a finite set such that, for all $\theta\in \Theta$, there exists an $i\in\{1,\dots,k\}$ for which $\|\theta-\theta'_{i}\|\leq \epsilon_\star$. Remark  that such a finite set   exists since    $\Theta$  is  compact. Then, by \eqref{eq:W2},  
\begin{equation}\label{eq:Z_boubd}
\begin{split}
\frac{1}{rt}\log W^{\eta_r}_{(r+1):rt} \leq \frac{1}{1-\delta_\star}\, \max_{i\in\{1,\dots,k\}}\frac{1}{rt}\log W_{rt ,\theta'_i}(\delta_\star),   \quad\forall t\geq 2,\quad \P-a.s. 
\end{split}
\end{equation}
where,  recalling that $\eta_r$ is assumed to be  $(r,\delta_\star, \tilde{l},(\varphi_t)_{t\geq 1}\big)$-consistent,
\begin{equation}\label{eq:W_bigO}
 \frac{1}{rt} \log W_{rt ,\theta'_i}( \delta_\star)=\bigO_\P(1),\quad\inf_{t\geq 2}\P\Big(\log W_{rt ,\theta'_i}( \delta_\star)\in [0,\infty)\Big)=1,\quad\forall i\in\{1,\dots,k\}.
\end{equation}
Then, by using \eqref{eq:Z_boubd}-\eqref{eq:W_bigO}, it follows that 
\begin{align}\label{eq:W_bigOO}
\inf_{t\geq 2}\P\big(\log W^{\eta_r}_{(r+1):rt}\in[0,\infty)\big)=1
\end{align}
while, by using  \eqref{eq:W_2} and  \eqref{eq:Z_boubd}-\eqref{eq:W_bigO}, and noting that since $\eta_r$ is  $(r,\delta_\star, \tilde{l},(\varphi_t)_{t\geq 1}\big)$-consistent we have $\inf_{t\geq 2}\P\big(\min_{i\in \{1,\dots,k\}}\log \bar{P}^{\eta_r}_{(r+1):rt}(\theta'_i)>-\infty)=1$, it follows that
\begin{align*}
\inf_{t\geq 2}\P\big(\inf_{\theta\in\Theta}\log \bar{P}^{\eta_r}_{(r+1):rt}(\theta)>-\infty\big)=1.
\end{align*}

To conclude the proof it remains to show that $(rt)^{-1}\log   W^{\eta_r}_{(r+1):rt}=\bigO_\P(1)$. To this aim let $\epsilon\in(0,\infty)$ and note that, by \eqref{eq:W_bigO}, for all $i\in\{1,\dots,k\}$ there exists a  constant  $M_i\in(0,\infty)$   such that
\begin{align*}
\limsup_{t\rightarrow\infty}\P\left(\frac{1}{rt} \log W_{rt ,\theta'_i}( \delta_\star)\geq  M_i\right)\leq \frac{\epsilon}{k}.
\end{align*}
Consequently, letting $M=\max_{i\in\{1,\dots,k\}}M_i/(1-\delta_\star)$ and using \eqref{eq:Z_boubd} and \eqref{eq:W_bigOO},  we have
\begin{equation*}
\begin{split}
\limsup_{t\rightarrow\infty}\P\bigg(\Big|\frac{1}{rt}\log W^{\eta_r}_{(r+1):rt}\Big|  \geq M \bigg)&=\limsup_{t\rightarrow\infty}\P\left(\frac{1}{rt}\log W^{\eta_r}_{(r+1):rt} \geq M \right)\\
&\leq \limsup_{t\rightarrow\infty}\P\left(\max_{i\in\{1,\dots,k\}}\frac{1}{rt}\log W_{rt,\theta'_i}(\delta_{\star})\geq (1-\delta_\star) M\right)\\
&\leq \sum_{i=1}^k  \limsup_{t\rightarrow\infty}\P\left(\frac{1}{rt}\log W_{rt ,\theta'_i}(\delta_\star)\geq M_i\right)\\
&\leq \epsilon 
\end{split}
\end{equation*}
and the proof of the lemma is complete.
\end{proof}

\subsubsection{Proof of Lemma \ref{lemma:Unif_convergence}}\label{p-lemma:Unif_convergence}

\begin{proof}

Since (i) the set $\Theta$ is compact, (ii) the function $l$ is continuous on $\Theta$ under \ref{assume:Model} and (iii) $\eta_r$ is $(r,\delta_\star, \tilde{l},(\varphi_t)_{t\geq 1})$-consistent by assumption, it follows from \citetsup[][Theorem 1]{andrews1992generic} that to prove the result of the lemma it suffices to show that
\begin{equation}\label{eq:SEC}
\begin{split}
\forall\epsilon\in(0,\infty),\quad&\exists\delta\in(0,\infty)\,\,\text{ such that }\\
&\limsup_{t\rightarrow\infty}\P\Big(\sup_{\theta'\in \Theta} \sup_{\theta\in B_{\delta}(\theta')\cap\Theta }\Big|\frac{1}{rt}\log \bar{P}^{\eta_r}_{(r+1):rt}(\theta)-\frac{1}{rt}\log \bar{P}^{\eta_r}_{(r+1):rt}(\theta')\Big|>\epsilon\Big)<\epsilon.
\end{split}
\end{equation}
 
To show \eqref{eq:SEC} let $g(\cdot)$ be as in \ref{assume:Model},  $\eta_{rt,\theta}$ be as defined in \eqref{eq:eta_theta} for all for all $t\geq 2$ and all $\theta\in\Theta$, and  for all $t\geq 2$, $\theta\in\Theta$ and $\delta\in(0,1)$  let $W_{rt,\theta}(\delta)$ be as defined in \eqref{eq:W_theta}.

Then, $\P$-a.s.~and for all $t\geq 2$, for all  $\theta'\in\Theta$ and  $\theta\in B_{\delta_\star}(\theta')\cap\Theta$  such that  $g\big(\|\theta-\theta'\|\big)<1$ we have
\begin{equation}\label{eq::uni1}
\begin{split}
&\frac{1}{rt}\log  \frac{\bar{P}^{\eta_r}_{(r+1):rt}(\theta)}{\bar{P}^{\eta_r}_{(r+1):rt}(\theta')}\\
&\leq \frac{1}{rt}\log\int_{\setX^{r(t-1)+1}} \exp\Big(g\big(\|\theta-\theta'\|\big)\sum_{s=r+1}^{rt}\varphi_s(x_{s-1},x_{s})\Big)\eta_{t,\theta'}\big(\dd(x_{r},\dots,x_{rt})\big)\\
 &\leq g\big(\|\theta-\theta'\|\big)^{1/2}\frac{1}{rt}\log\int_{\setX^{r(t-1)+1}} \exp\Big(g\big(\|\theta-\theta'\|\big)^{1/2}\sum_{s=r+1}^{rt}\varphi_s(x_{s-1},x_{s})\Big)\eta_{t,\theta'}\big(\dd(x_{r},\dots,x_{rt})\big)\\
 &=g\big(\|\theta-\theta'\|\big)^{1/2}\frac{1}{rt}\log  W_{rt,\theta'}\Big(g\big(\|\theta-\theta'\|\big)^{1/2}\Big)
\end{split}    
\end{equation}
where the first inequality holds under \ref{assume:Model_smooth}  while  the second inequality holds by Jensen's inequality.

We now let $\epsilon\in(0,1)$ and $\kappa\leq \delta_\star^2/4$ and   choose a
$\delta_\kappa\in(0,1)$ and a $\gamma_\kappa\in(0,\kappa)$   sufficiently small so  that
\begin{align}\label{eq:kappa_cond}
2 g^{-1}(\delta_\kappa)+\gamma_\kappa\leq g^{-1}(\kappa). 
\end{align}
Next, we let  $\{\theta'_i\}_{i=1}^k\subset \Theta$ be a finite set such that, for all $\theta\in \Theta$, there exists an $i\in\{1,\dots,k\}$ for which $d_{\setR}(\theta,\theta'_{i})\leq  g^{-1}(\delta_\kappa)$. Remark that such a finite set  exists since $\Theta$ is compact. Finally, for all $\theta\in \Theta$ we let $a_\theta\in \{\theta'_i\}_{i=1}^k$   be such that $\|\theta-a_\theta\|\leq   g^{-1}(\delta_\kappa)$. 

Then, $\P$-.a.s.~and  for all $t \geq 2$, for all $\theta\in \Theta$ and $\theta'\in B_{\gamma_\kappa}(\theta)\cap\Theta$, we have
\begin{equation}\label{eq:use_split}
\begin{split}
\Big|\frac{1}{r t}\log \frac{\bar{P}^{\eta_r}_{(r+1):rt}(\theta)}{\bar{P}^{\eta_r}_{(r+1):rt}(\theta')}\Big|&=\Big|\frac{1}{rt}\log \frac{\bar{P}^{\eta_r}_{(r+1):rt}(\theta)}{\bar{P}^{\eta_r}_{(r+1):rt}(a_\theta)}-\frac{1}{rt}\log \frac{\bar{P}^{\eta_r}_{(r+1):rt}(\theta')}{\bar{P}^{\eta_r}_{(r+1):rt}(a_{\theta'})}+\frac{1}{rt}\log \frac{\bar{P}^{\eta_r}_{(r+1):rt}(a_\theta)}{\bar{P}^{\eta_r}_{(r+1):rt}(a_{\theta'})}\Big|\\
&\leq \kappa^{1/2}\bigg(\frac{1}{rt}\log W_{rt,a_{\theta}}\big(\kappa^{1/2}\big)+\frac{1}{rt}\log W_{rt,a_{\theta'}}\big(\kappa^{1/2}\big)+\frac{1}{rt}\log W_{rt,a_{\theta'}}\big(\kappa^{1/2}\big)\bigg)\\
&\leq 3\kappa^{1/2}\max_{i\in\{1,\dots,k\}}\frac{1}{rt}\log W_{rt,\theta'_i}\big(\kappa^{1/2})\\
&\leq 3\kappa^{1/2}\max_{i\in\{1,\dots,k\}}\frac{1}{rt}\log W_{rt,\theta'_i}\big(\delta_\star/2) \\
&\leq  3\kappa^{1/2}\frac{1}{rt}\log W_{(r+1):rt}^{\eta_r}
\end{split}
\end{equation}
where, for all integers $1\leq t_1<t_2$ and $\eta\in\mathcal{P}(\setX)$, the random variable  $W_{t_1:t_2}^\eta$ is as defined in \eqref{eq:W_def} and where the first inequality uses \eqref{eq::uni1} and the fact that, by \eqref{eq:kappa_cond}, we have both $\delta_\kappa<\kappa\leq \delta_\star^2/4$  and
\begin{align*}
d_{\setR}(a_\theta,a_{\theta'})\leq 2 g^{-1}(\delta_\kappa)+\gamma_\kappa\leq g^{-1}(\kappa),\quad\forall \theta\in B_{\Theta,\gamma_\kappa}(\theta')\quad\forall \theta'\in\Theta.
\end{align*}

To proceed further let $M_\epsilon\in (0,\infty)$ be such that
\begin{align}\label{eq:Mepsilon}
\limsup_{t\rightarrow\infty}\P\Big( \frac{1}{rt }\log W_{(r+1):rt}^{\eta_r} \geq  M_\epsilon\Big)< \epsilon.
\end{align}
Remark that such a constant $M_\epsilon$ exists by Lemma  \ref{lemma::W}. Then, using  \eqref{eq:use_split}  with 
\begin{align*}
\kappa=\min\Big(\frac{\epsilon^2}{9 M_\epsilon^2},\,\frac{\delta^2_\star}{4}\Big)
\end{align*}
 and \eqref{eq:Mepsilon}, it follows that 
\begin{align*}
\limsup_{t\rightarrow\infty}\P\Big(\sup_{\theta'\in \Theta} \sup_{\theta\in B_{\gamma_\kappa}(\theta')\cap\Theta}\big|\frac{1}{rt} \log \bar{P}^{\eta_r}_{(r+1):rt}(\theta)-\frac{1}{rt}\log &\bar{P}^{\eta_r}_{(r+1):rt}(\theta')\big|>\epsilon\Big) \\
&\leq \limsup_{t\rightarrow\infty} \P\Big(\frac{1}{rt}\log W_{(r+1):rt}^{\eta_r}\geq M_\epsilon\Big)\\
&<\epsilon.
\end{align*}
 This  shows  \eqref{eq:SEC} and the proof of the lemma is therefore complete.
\end{proof}

\subsubsection{Proof of Lemma \ref{lemma:Zt}\label{p-lemma:Zt}}

\begin{proof}

To prove the lemma it is enough to show that
\begin{align}\label{eq:ZT_show}
 \P\big(\mathfrak{Z}_t\in (0,\infty)\big)=1,\quad\forall t\geq 1
\end{align}
with $(\mathfrak{Z}_t)_{t\geq 1}$  as defined in \eqref{eq:normalizing}.

To show \eqref{eq:ZT_show} remark first that $\mathfrak{Z}_t\leq D_{1:t}$ for all $t\geq 1$, where for all $1\leq t_1\leq t_2$ the random variable $D_{t_1:t_2}$ is as defined in \eqref{eq:Dst}. Therefore, under  \ref{assume:G_bounded},  we have $\P(\mathfrak{Z}_t<\infty)=1$ for all $t\geq 1$.

Letting $r$ be  as in \ref{assume:Model}, we now show that
\begin{align}\label{eq:Zr}
\P(\mathfrak{Z}_t>0)=1,\quad\forall t\in\{sr,\,s\in\mathbb{N}\}.
\end{align}

To this aim,   let $t\in\{sr,\,s\in\mathbb{N}\}$, $\delta_\star\in(0,1)$  be as in \ref{assume:Model}, $\epsilon\in(0, \delta_\star/2)$ be such that $g(\epsilon)< \delta_\star/2 $,  with $g$ as in \ref{assume:Model}, and let
\begin{align}\label{eq:Lt_def}
L_t=\inf_{\theta\in\Theta}\int_{\setX^{t+1}} \exp\Big(-g(\epsilon)\sum_{s=1}^{t} \varphi_{s}(x_{s-1},x_{s})\Big)\chi(\dd x_0)\prod_{s=1}^{t} Q_{s,\theta}(x_{s-1},\dd x_{s}) 
\end{align}
with $(\varphi_s)_{s\geq 1}$ as in  \ref{assume:Model}  and   $f_{t,\epsilon}$ be as defined in \eqref{eq:ft_def}.

Then, we $\P$-a.s.~have 
\begin{equation}\label{eq:lower_z1}
\begin{split}
 \mathfrak{Z}_t&=\int_{(\Theta\times\setX)^{t+1}}\mu_0(\dd\theta_0)\chi(\dd x_0) \prod_{s=1}^{t} Q_{s,\theta_{s}}(x_{s-1},\dd x_{s}) K_{\mu_{s}|\Theta}(\theta_{s-1},\dd\theta_{s})\\
&\geq \int_{(\Theta\times\setX)^{t+1}} f_{t,\epsilon}(\theta_0,\dots,\theta_t)\mu_0(\dd\theta_0)\chi(\dd x_0) \prod_{s=1}^{t} Q_{s,\theta_{s}}(x_{s-1},\dd x_{s}) K_{\mu_{s}|\Theta}(\theta_{s-1},\dd\theta_{s})\\
&\geq L_t \int_{ \Theta^{t+1}} f_{t,\epsilon}(\theta_0,\dots,\theta_t)\,\mu_0(\dd\theta_0) \prod_{s=1}^{t}   K_{\mu_{s}|\Theta}(\theta_{s-1},\dd\theta_{s})\\
&\geq L_t \int_{ \setR^{t+1}} f_{t,\epsilon}(\theta_0,\dots,\theta_t)\,\mu_0(\dd\theta_0) \prod_{s=1}^{t}   \ind_{\Theta}(\theta_s)K_{\mu_{s}}(\theta_{s-1},\dd\theta_{s})
\end{split}
\end{equation}
where the second inequality holds under    \ref{assume:Model_smooth}.

To proceed further  recall that under \ref{assume:Model_ThetaStar} we have   $\mathcal{N}_{\delta_\star}(\tilde{\Theta}_\star)\subseteq \mathring{\Theta}$ and thus, since $\epsilon<\delta_\star/2$, it follows that $\theta+u\in \mathcal{N}_{\delta_\star}(\tilde{\Theta}_\star)\subset \mathring{\Theta}\subset\Theta$ for all $\theta\in \mathcal{N}_{\delta_\star/2}(\tilde{\Theta}_\star)$ and  all $u\in B_\epsilon(0)$. Therefore,
\begin{equation}\label{eq:lower_z2}
\begin{split}
\int_{\setR^{t+1}} f_{t,\epsilon}(\theta_0,\dots,\theta_t)\,\mu_0(\dd\theta_0) \prod_{s=1}^{t}   \ind_{\Theta}(\theta_s) & K_{\mu_{s}}(\theta_{s-1},\dd\theta_{s})\\
&\geq \mu_0\big(\mathcal{N}_{\delta_\star/2}(\tilde{\Theta}_\star)\big)\P\Big(\sum_{i=1}^s U_i\in B_\epsilon(0),\,\forall s\in\{1,\dots,t\} \,\big| \F_t \Big)\\
&\geq  \mu_0\big(\mathcal{N}_{\delta_\star/2}(\tilde{\Theta}_\star)\big)\P\Big( U_s\in B_{\epsilon/t}(0),\,\forall s\in\{1,\dots,t\} \,\big| \F_t  \Big)\\
&=\mu_0\big(\mathcal{N}_{\delta_\star/2}(\tilde{\Theta}_\star)\big)\prod_{s=1}^t \mu_s\big(B_{\epsilon/t}(0)\big) 
\end{split}
\end{equation}
where $\mu_0\big(\mathcal{N}_{\delta_\star/2}(\tilde{\Theta}_\star)\big)>0$ under \ref{condition:mu1}  and where, under \ref{condition:Inf_mu}, $\P(\mu_s\big( B_{\epsilon/t}(0)\big)>0)=1$ for all $s\in\{1,\dots,t\}$. Therefore,
\begin{align*}
\int_{ \setR^{t+1}} f_{t,\epsilon}(\theta_0,\dots,\theta_t)\,\mu_0(\dd\theta_0) \prod_{s=1}^{t}   K_{\mu_{s}}(\theta_{s-1},\dd\theta_{s})>0,\quad\P-a.s.
\end{align*}
and we now show that $\P(L_t>0)=1$.

To do so remark that, using Jensen's inequality,
\begin{equation}\label{eq:lemmLower1}
\begin{split}
 \log L_t &\geq -\log W_{1:t}^\chi+\inf_{\theta\in\Theta}\log  \bar{P}_{1:t}^\chi(\theta)
\end{split}
\end{equation}
 where, for all integers $1\leq t_1\leq t_2$ and $\eta\in\mathcal{P}(\setX)$, the random variable $\bar{P}_{t_1,t_2}^\eta(\theta)$ is as defined in \eqref{eq:Pbamodelef} for all $\theta\in\Theta$ while the random variable $W^{\eta}_{t_1:t_2}$ is as defined in \eqref{eq:W_def}.

Under \ref{assume:Model_stationary} we have
\begin{align*}
\inf_{\theta\in\Theta}\bar{P}_{1:t}^\chi(\theta)\dist \inf_{\theta\in\Theta}\bar{P}_{(r+1):(t+r)}^\chi(\theta),\quad W_{1:t}^\chi\dist W_{(r+1):(t+r)}^\chi
\end{align*}
and thus, since $\chi$ is $(r,\delta_\star,\tilde{l},(\varphi_t)_{t\geq 1})$-consistent under    \ref{assume:Model_init} and   $t+r\in\{sr,\,s\in\mathbb{N}\}$, it follows from Lemmas \ref{lemma::W}-\ref{lemma:Unif_convergence} that 
\begin{align*}
\P\big(\log W_{1:t}^\chi<\infty)=\P\big(\inf_{\theta\in\Theta}\log \bar{P}_{1:t}^\chi(\theta)>-\infty)=1.
\end{align*}
Together with \eqref{eq:lemmLower1}  this shows that $\P(L_t>0)\geq 1$  and the proof of \eqref{eq:Zr} is complete.

To complete the proof of \eqref{eq:ZT_show} it remains to show that, for $r>1$, we have $\P(\mathfrak{Z}_t>0)=1$ for all $t\not\in \{sr,\,s\in\mathbb{N}\}$. To this aim assume that $r>1$ and let $t\not\in \{sr,\,s\in\mathbb{N}\}$ and $\tau\in\mathbb{N}$ be such that $(\tau-1)r<t<\tau r$. Then, by applying  Lemma \ref{lemma:tech} with $\eta=\mu_0\otimes\chi$, $t_1=1$, $t_2=t$ and $t_3=\tau r$, we obtain that $\mathfrak{Z}_t\geq \mathfrak{Z}_{\tau r}  D_{(t+1):\tau t}^{-1}$. Since,   $\P(\mathfrak{Z}_{\tau r}>0)=1$ by \eqref{eq:Zr} while,  under  \ref{assume:G_bounded} we have $\P\big(D_{(t+1):\tau t}^{-1}>0\big)=1$, the proof of \eqref{eq:ZT_show} and thus of the lemma is complete.

\end{proof}

\subsubsection{Proof of Lemma \ref{lemma:gamma}\label{p-lemma:gamma}}

\begin{proof}
Let $D\in\mathcal{X}$ be as in \ref{assume:Model}, $t\geq  r$, $\tau\in\{r,\dots,t\}\cap\{sr,\,s\in\mathbb{N}\}$,
\begin{align*}
&\underline{Z}_{\tau}=\int_{\setX}\Big(\inf_{(\theta,x')\in \Theta \times D} q_{\tau,\theta}(x|x')\Big)\lambda(\dd x),\quad\bar{Z}_{\tau}=\int_{\setX}\Big(\sup_{(\theta,x')\in \Theta \times \setX} q_{\tau,\theta}(x|x') \Big)\lambda(\dd x)
\end{align*}
and let $\underline{\chi}_{\tau}$ and $\bar{\chi}_{\tau}$ be two random probability measures on $(\setX,\mathcal{X})$ such that, $\P$-a.s., for all $A\in\mathcal{X}$ we have
\begin{align*}
\underline{\chi}_{\tau}(A)=\frac{\int_{A}\Big(\inf_{(\theta,x')\in \Theta \times D} q_{\tau,\theta}(x|x') \Big)\lambda(\dd x)}{\underline{Z}_{\tau}},\quad \bar{\chi}_{\tau}(A)=\frac{\int_{A}\Big(\sup_{(\theta,x')\in \Theta \times \setX} q_{\tau,\theta}(x|x') \Big)\lambda(\dd x)}{\bar{Z}_{\tau}}.
\end{align*}
Remark that such random probability measures  $\underline{\chi}_{\tau}$ and $\bar{\chi}_{\tau}$ exist under \ref{assume:Model_stationary} and \ref{assume:Model_G_lower}-\ref{assume:Model_sup}.

Then,  using Lemma \ref{lemma:Zt},  we have
\begin{align}\label{eq:b1}
 \Lambda^{\pi_0}_{1:t}(A\times B)\leq  \bar{Z}_\tau P^{\pi_0}_{1:(\tau-1)}(\Theta) \Lambda^{\bar{\pi}_\tau\otimes\bar{\chi}_\tau}_{(\tau+1):t}(A\times B),\quad\forall (A,B) \in \mathcal{T}\times\mathcal{X},\quad\P-a.s.
\end{align}
and
\begin{align}\label{eq:b2}
 \Lambda^{\pi_0}_{1:t}(A\times B)\geq  \underline{Z}_\tau \Lambda^{\pi_0}_{1:(\tau-1)}(\Theta\times D) \Lambda^{\bar{\pi}_\tau\otimes\underline{\chi}_\tau}_{(\tau+1):t}(A\times B),\quad \forall (A,B) \in \mathcal{T}\times\mathcal{X},\quad\P-a.s.
\end{align}
To proceed further remark that, by using again Lemma \ref{lemma:Zt}, we have
\begin{align*}
\frac{P^{\pi_0}_{1:(\tau-1)}(\Theta) }{\Lambda^{\pi_0}_{1:(\tau-1)}(\Theta\times D)}=\frac{1}{\pi_{(\tau-1),\setX}(D)},\quad\P-a.s.
\end{align*}
and thus,  under \eqref{eq:W_star},
\begin{align}\label{eq:b3}
\Lambda^{\pi_0}_{1:(\tau-1)}(\Theta\times D)\geq W^*_{\tau-1}P^{\pi_0}_{1:(\tau-1)}(\Theta),\quad\P-a.s.
\end{align}
To conclude the proof note that, under \ref{assume:Model_stationary},
\begin{align*}
\underline{Z}_{\tau}\dist \underline{Z}_{r},\quad \bar{Z}_{\tau}\dist \bar{Z}_{r},\quad  \underline{\chi}_{\tau}\dist\underline{\chi}_{r},\quad \bar{\chi}_{\tau}\dist \bar{\chi}_{r}
\end{align*}
where, under \ref{assume:Model_G_lower}-\ref{assume:Model_sup}, $\P\big(\underline{Z}_{r}\in(0,\infty)\big)=\P\big(\bar{Z}_{r}\in(0,\infty)\big)=1$. Therefore, since by assumption $\log W^*_{\tau-1}=\bigO_\P(1)$, it follows from \eqref{eq:b1}-\eqref{eq:b3} that the result of the lemma holds with $(\tilde{\xi}_{rt})_{t\geq 1}$ such that
\begin{align}\label{eq:xi_cor}
\tilde{\xi}_{rt}=\Big(\bar{Z}_{rt}+\frac{1}{\underline{Z}_{rt}}\Big)\exp\Big(-\log W^*_{t-1}\Big),\quad\forall t\geq 1.
\end{align}
The proof is complete.
\end{proof}

\subsubsection{Proof of Lemma \ref{lemma::upper_general}\label{p-lemma::upper_general}}

\begin{proof}

We start by introducing some additional notation. For all $\tau\geq 1$ we let  $\zeta_{a_{\tau}}=\kappa_{a_\tau}\otimes\eta_{a_\tau}$, $d_\tau\in\{sr,\,s\in\mathbb{N}\}$ be such that $d_\tau<b_\tau\leq d_\tau+r$ and $\tilde{f}_\tau=f_{b_\tau-a_\tau,\epsilon_\tau}$ with $f_{t,\epsilon}$ is as defined in \eqref{eq:ft_def} for all $t\in\mathbb{N}$ and $\epsilon\in(0,\infty)$. Next,    we let $\big( (\vartheta_{a_\tau},\tilde{X}_{a_\tau}),\dots, (\vartheta_{b_\tau},\tilde{X}_{b_\tau} )\big)$ be a random variable such that
\begin{align*}
\big( (\vartheta_{a_\tau},\tilde{X}_{a_\tau}),\dots, (\vartheta_{b_\tau},\tilde{X}_{b_\tau} )\big)\big|\F_{b_\tau}\sim \zeta_{a_\tau}\big(\dd(\theta_{a_\tau}, x_{a_\tau})\big)\prod_{s=a_\tau+1}^{b_\tau}m_{s,\theta_{s}}(x_s|x_{s-1})\lambda(\dd x_s) K_{\mu_s|\Theta}(\theta_{s-1},\dd \theta_s)
\end{align*}
and   $\tau'\in\mathbb{N}$ be such that we have both $d_\tau>a_\tau$ and $b_\tau>a_\tau+2$ for all $\tau\geq \tau'$.

With this notation in place, for all $\tau\geq \tau'$ and $A\in\mathcal{T}$   we have
\begin{equation}\label{eq:P_sup1}
\begin{split}
 P^{\zeta_{a_{\tau}}}_{(a_\tau+1):b_\tau} (A) &= \E_{(a_\tau+1):b_\tau}^{\zeta_{a_{\tau}}}\bigg[\ind_A(\vartheta_{b_\tau}) \tilde{f}_{\tau}\big(\vartheta_{a_\tau},\dots,\vartheta_{b_\tau}\big)  \prod_{s=a_\tau+1}^{b_\tau} G_{s,\vartheta_s}(\tilde{X}_{s-1},\tilde{X}_s)\bigg]\\
&+ \E_{(a_\tau+1):b_\tau}^{\zeta_{a_{\tau}}}\bigg[\ind_A(\vartheta_{b_\tau})\Big(1-\tilde{f}_{\tau}\big(\vartheta_{a_\tau},\dots,\vartheta_{b_\tau}\big) \Big) \prod_{s=a_\tau+1}^{b_\tau} G_{s,\vartheta_s}(\tilde{X}_{s-1},\tilde{X}_s)\bigg] 
\end{split}
\end{equation}
and we now study the two terms on the r.h.s.~of \eqref{eq:P_sup1}, starting with the first one.

To this aim, for all $\tau\geq 1$ we let
\begin{equation}\label{eq:xi_tau}
\begin{split}
\xi_{d_\tau}&=\sup_{\theta\in\Theta}\bigg|\frac{1}{d_\tau-a_\tau}\log  \bar{P}^{\eta_{a_\tau}}_{(a_\tau+1):d_\tau}(\theta)-\tilde{l}(\theta)\bigg|+\frac{g(2\epsilon_\tau)^{1/2}}{d_\tau-a_\tau}\log W^{\eta_{a_\tau}}_{(a_\tau+1):d_\tau}\\
&+ \frac{1}{d_\tau-a_\tau}\log^+  D_{(d_\tau+1):(d_\tau+r)}
\end{split}
\end{equation}
with $g(\cdot)$  as in \ref{assume:Model} and where, for all integers $1\leq t_1\leq t_2$ and  $\eta\in\mathcal{P}(\setX)$, the random variable $\bar{P}_{t_1,t_2}^\eta(\theta)$ is as defined in \eqref{eq:Pbamodelef} for all $\theta\in\Theta$ while the random variables   $D_{t_1:t_2}$ and $W^{\eta}_{t_1:t_2}$ are as defined in   \eqref{eq:Dst} and \eqref{eq:W_def}, respectively. Without loss of generality, we assume henceforth that $D_{t,t}\geq 1$ for all $t\geq 1$ and that $\tau'$ is sufficiently large so that  $g(2\epsilon_\tau)<  (\delta_\star/2)^{2}$ for all $\tau\geq \tau'$.

As preliminary computations we show that
\begin{align}\label{eq:conv_xi_tau}
\xi_{d_\tau}=\smallo_\P(1).
\end{align}
To do so remark first that, by assumption, $(\eta_{a_{\tau}})_{\tau\geq 1}$ is such that $\eta_{a_\tau}\dist \eta_r$ for all $\tau\geq 1$ and thus, by   \ref{assume:Model_stationary}, 
\begin{align*}
\bar{P}^{\eta_{a_\tau}}_{(a_\tau+1):d_\tau}(\theta)\dist \bar{P}^{\eta_r}_{(r+1):(d_\tau+r-a_\tau)}(\theta),\quad W^{\eta_{a_\tau}}_{(a_\tau+1):d_\tau} \dist W^{\eta_r}_{(r+1):(d_\tau+r-a_\tau)} ,\quad\forall \tau\geq 1,\quad\forall\theta\in\Theta.
\end{align*}
Therefore, since $\eta_r$ is assumed to be $(r,\delta_\star, \tilde{l},(\varphi_t)_{t\geq 1}\big)$-consistent while $b_\tau+r-a_\tau\in\{sr,\,s\in\mathbb{N}\}$ for all $\tau\geq 1$, it follows from Lemmas \ref{lemma::W}-\ref{lemma:Unif_convergence} that 
\begin{align}\label{eq:eps_1}
\sup_{\theta\in\Theta}\bigg|\frac{1}{d_\tau-a_\tau}\log  \bar{P}^{\eta_{a_\tau}}_{(a_\tau+1):d_\tau}(\theta)-\tilde{l}(\theta)\bigg|+\frac{g(2\epsilon_\tau)^{1/2}}{d_\tau-a_\tau}\log W^{\eta_{a_\tau}}_{(a_\tau+1):d_\tau}=\smallo_\P(1).
\end{align}
In addition, by \ref{assume:Model_stationary}, we have $D_{(d_\tau+1):(d_\tau+r)}\dist D_{(r+1):2r}$ for  all $\tau\geq 1$ and thus, under  \ref{assume:G_bounded},
\begin{align}\label{eq:eps_2}
\frac{1}{d_\tau-a_\tau}\log^+  D_{(d_\tau+1):(d_\tau+r)}=\smallo_\P(1).
\end{align}
Then, \eqref{eq:conv_xi_tau} follows from \eqref{eq:eps_1} and \eqref{eq:eps_2}.

To proceed further remark that, $\P$-a.s., for all $\tau\geq \tau'$ and   $\theta\in\Theta$ we have
\begin{align*}
\log &\int_{\setX^{d_\tau-a_\tau+1}}  \exp\Big( g(2\epsilon_\tau)\sum_{s=a_\tau+1}^{d_\tau}\varphi_s(x_{s-1},x_{s})\Big)\bigg( \frac{\eta_{a_\tau}(\dd x_{a_\tau})  \prod_{s=a_\tau+1}^{d_\tau} Q_{s,\theta}(x_{s-1},\dd x_{s})}{\bar{P}^{\eta_{a_\tau}}_{(a_\tau+1):d_\tau}(\theta)}\bigg)\\
&=\log \int_{\setX^{d_\tau-a_\tau+1}}  \bigg(\exp\Big( g(2\epsilon_\tau)^{\frac{1}{2}} \hspace{-0.1cm}\sum_{s=a_\tau+1}^{d_\tau}\varphi_s(x_{s-1},x_{s})\Big)\bigg)^{g(2\epsilon_\tau)^{\frac{1}{2}}}\bigg( \frac{\eta_{a_\tau}(\dd x_{a_\tau})  \prod_{s=a_\tau+1}^{d_\tau} Q_{s,\theta}(x_{s-1},\dd x_{s})}{\bar{P}^{\eta_{a_\tau}}_{(a_\tau+1):d_\tau}(\theta)}\bigg)\\
&\leq g(2\epsilon_\tau)^{1/2}\log \int_{\setX^{b_\tau-a_\tau+1}}  \exp\Big(\frac{\delta_\star}{2} \sum_{s=a_\tau+1}^{d_\tau}\varphi_s(x_{s-1},x_{s})\Big)\bigg( \frac{\eta_{a_\tau}(\dd x_{a_\tau})  \prod_{s=a_\tau+1}^{d_\tau} Q_{s,\theta}(x_{s-1},\dd x_{s})}{\bar{P}^{\eta_{a_\tau}}_{(a_\tau+1):d_\tau}(\theta)}\bigg) \\
&\leq g(2\epsilon_\tau)^{1/2}\log W^{\eta_{a_\tau}}_{(a_\tau+1):d_\tau}
\end{align*}
where the first inequality holds by   Jensen's inequality.

Therefore, under  \ref{assume:Model_smooth} and using  Tonelli's theorem we $\P$-a.s.~have, for all $\tau\geq \tau'$ and $A\in\mathcal{T}$,  
\begin{equation}\label{eq:1st_term}
\begin{split}
 \E_{(a_\tau+1):b_\tau}^{\zeta_{a_{\tau}}} \bigg[&\ind_A(\vartheta_{b_\tau})\tilde{f}_{\tau} \big(\vartheta_{a_\tau},\dots,\vartheta_{b_\tau}\big)   \prod_{s=a_\tau+1}^{b_\tau} G_{s,\vartheta_s}(\tilde{X}_{s-1},\tilde{X}_s)\bigg]\\
&\leq D_{(d_\tau+1):b_\tau} \exp\Big(g(2\epsilon_\tau)^{1/2}\log W^{\eta_{a_\tau}}_{(a_\tau+1):d_\tau}\Big)\E_{(a_\tau+1):b_\tau}^{\zeta_{a_\tau}}\Big[\ind_A(\vartheta_{b_\tau})\bar{P}_{(a_\tau+1):d_\tau}^{\eta_{a_\tau}}(\vartheta_{b_\tau})\Big]\\
&\leq  D_{(d_\tau+1):b_\tau} \exp\Big(g(2\epsilon_\tau)^{1/2}\log W^{\eta_{a_\tau}}_{(a_\tau+1):d_\tau}+\sup_{\theta\in A}\log \bar{P}_{(a_\tau+1):d_\tau}^{\eta_{a_\tau}}(\theta) \Big)\\
&\leq \exp\Big( (d_\tau-a_\tau)\big(\sup_{\theta\in A}\tilde{l}(\theta)+\xi_{d_\tau}\big)\Big)
\end{split}
\end{equation}
where the first inequality holds under    \ref{assume:Model_smooth}.

To study the second term in \eqref{eq:P_sup1} remark that, under the assumptions of the lemma and using Tonelli's theorem, for all $\tau\geq 1$ and $A\in\mathcal{T}$ we have
\begin{equation}\label{eq:2nd_term}
\begin{split}
 \E_{(a_\tau+1):b_\tau}^{\zeta_{a_{\tau}}}&\bigg[\ind_A(\vartheta_{b_\tau})\Big(1-\tilde{f}_{\tau}\big(\vartheta_{a_\tau},\dots,\vartheta_{b_\tau}\big) \Big)  \prod_{s=a_\tau+1}^{b_\tau} G_{s,\vartheta_s}(\tilde{X}_{s-1},\tilde{X}_s)\Big]\\
 &\leq \E_{(a_\tau+1):b_\tau}^{\zeta_{a_{\tau}}}\bigg[\Big(1-\tilde{f}_{\tau}\big(\vartheta_{a_\tau},\dots,\vartheta_{b_\tau}\big) \Big)  \prod_{s=a_\tau+1}^{b_\tau} G_{s,\vartheta_s}(\tilde{X}_{s-1},\tilde{X}_s)\Big]\\
&\leq D_{(a_\tau+1):b_\tau} \E_{(a_\tau+1):b_\tau}^{\zeta_{a_{\tau}}}\bigg[\Big(1-\tilde{f}_{\tau}(\vartheta_{ a_\tau},\dots,\vartheta_{ b_\tau})\Big)\bigg]\\
&\leq \frac{D_{(a_\tau+1):b_\tau}}{(\Gamma^\mu)^{b_\tau-a_\tau}} \int_{\setR^{b_\tau-a_\tau+1}} \Big(1-\tilde{f}_{\tau}(\theta_{ a_\tau},\dots,\theta_{ b_\tau})\Big)\kappa_{a_\tau}(\dd\theta_{a_\tau})\prod_{s=a_\tau+1}^{b_\tau} K_{\mu_{s}}(\theta_{s-1},\dd\theta_{s})\\
&\leq   \exp\big( (d_\tau-a_\tau) Z_{b_\tau}\big)
\end{split}
\end{equation}
where the random variable $\Gamma^\mu$ is as  defined in \ref{condition:Inf_K} and where
\begin{align*}
Z_{b_\tau}=\frac{1}{d_\tau-a_\tau}\log^+ D_{(a_\tau+1):b_\tau}-\frac{b_\tau-a_\tau}{d_\tau-a_\tau}\Big(\log \Gamma^\mu+c_{\tau}\Big).
\end{align*}

Under \ref{condition:Inf_K} we have that $\P(\Gamma^\mu>0)=1$ while,  under \ref{assume:Model_stationary} and \ref{assume:G_bounded}, we have
\begin{align}\label{eq:convD}
\frac{1}{b_\tau-a_\tau}\log^+ D_{(a_\tau+1):b_\tau}=\bigO_\P(1).
\end{align}
 Therefore, noting that $\lim_{\tau\rightarrow\infty}(b_\tau-a_\tau)/(d_\tau-a_\tau)=1$
and recalling that $\lim_{\tau\rightarrow\infty}c_\tau=\infty$ by assumption, it follows that  $\lim_{\tau\rightarrow\infty}\P(Z_{b_\tau}>M)=0$ for all $M\in\R$.

By combining  \eqref{eq:P_sup1}, \eqref{eq:1st_term} and \eqref{eq:2nd_term}  we obtain that we  $\P$-a.s.~have, for all $A\in\mathcal{T}$ and all $\tau\geq 1$, 
\begin{align*}
 &P^{\kappa_{a_\tau}\otimes\eta_{a_\tau}}_{(a_\tau+1):b_\tau}  (A)\\
& \leq  D_{(a_\tau+1):b_\tau}\wedge   \exp\Big( (d_\tau-a_\tau)\big(\sup_{\theta\in A} \tilde{l}(\theta)+\xi_{d_\tau}\big)\Big)\Big\{1+\exp\Big( (d_\tau-a_\tau)\big(Z_{d_\tau}-\sup_{\theta\in A} \tilde{l}(\theta)-\xi_{d_\tau}\big)\Big)\Big\}
\end{align*}
where
\begin{align*}
\exp\Big( (d_\tau-a_\tau)\big(Z_{b_\tau}-\sup_{\theta\in A} \tilde{l}(\theta)-\xi_{d_\tau}\big)\Big)=\smallo_\P(1),\quad  \lim_{\tau\rightarrow\infty}\frac{b_\tau-a_\tau}{d_\tau-a_\tau}=1 
\end{align*}
and where, under  \ref{assume:Model_stationary} and   \ref{assume:G_bounded}, and recalling that  $\P(D_{t,t}\geq 1)=1$ for all $t\geq 1$,
\begin{align*}
\P\big(D_{(a_\tau+1):b_\tau}<\infty\big)&\leq \P\big( D_{(a_\tau+1):(d_\tau+r)}<\infty\big)=\P\big(D_{(r+1):(d_\tau+2r-a_\tau)}<\infty\big)=1.
\end{align*}
The result of the lemma follows.
\end{proof}

\subsubsection{Proof of Lemma \ref{lemma::lower_general}\label{p-lemma::lower_general}}

\begin{proof}

For all $\tau\geq 1$ we let   $\zeta_{a_{\tau}}=\kappa_{a_\tau}\otimes\eta_{a_\tau}$ and  for all integers $1\leq t_1\leq t_2$ we let $D_{t_1:t_2}$ be as defined in \eqref{eq:Dst}, assuming without loss of generality that   $\P(D_{t_1:t_1}\geq 1)=1$ for all $t_1\geq 1$.

To prove the lemma remark first that, by Lemma \ref{lemma:tech},
 \begin{align}\label{eq:step1}
 P^{\zeta_{a_{\tau}}}_{(a_\tau+1):b_\tau}  (\Theta) \geq D_{(b_\tau+1):e_\tau}^{-1}   P^{\zeta_{a_{\tau}}}_{(a_\tau+1):e_\tau}  (  \Theta)\geq D_{(e_\tau-r+1):e_\tau}^{-1}   P^{\zeta_{a_{\tau}}}_{(a_\tau+1):e_\tau}  (\Theta),\quad\forall \tau\geq 1 
 \end{align}
and we now compute a lower bound for $P^{\zeta_{a_{\tau}}}_{(a_\tau+1):e_\tau}  (\Theta)$, for all $\tau\geq 1$. 

To this aim, for all $\tau\geq 1$ we let
\begin{align*}
\widetilde{W}^{\eta_\tau}_{(a_\tau+1):e_\tau}=\inf_{\theta\in\Theta}\int_{\setX^{e_\tau-a_\tau+1}} \exp\Big( -g(2\epsilon_\tau)\sum_{s=a_\tau+1}^{e_\tau}\varphi_s(x_{s-1},x_{s})\Big) \frac{\eta_{a_{\tau}}(\dd x_{a_\tau})  \prod_{s=a_\tau+1}^{e_\tau} Q_{s,\theta}(x_{s-1},\dd x_{s})}{\bar{P}^{\eta_{a_{\tau}}}_{(a_\tau+1):e_\tau}(\theta)}
\end{align*}
with $g(\cdot)$ as in \ref{assume:Model} and where, for all integers $1\leq t_1\leq t_2$, the random variable $\bar{P}^{\eta}_{t_1:t_2}(\theta)$ is as defined in \eqref{eq:Pbamodelef} for all $\theta\in\Theta$. In addition, for all $\tau\geq 1$  we let $\tilde{f}_\tau=f_{e_\tau-a_\tau,\epsilon_\tau}$, with $f_{t,\epsilon}$ as defined in \eqref{eq:ft_def} for all $t\in\mathbb{N}$ and all $\epsilon\in(0,\infty)$,
\begin{align*}
\xi^-_{e_{\tau}} =  \sup_{\theta\in\Theta}\bigg|\frac{1}{e_\tau-a_\tau}\log  \bar{P}^{\eta_{a_\tau}}_{(a_\tau+1):e_\tau}(\theta)-\tilde{l}(\theta)\bigg|-\frac{1}{e_\tau-a_\tau}\log \widetilde{W}^{\eta_{a_\tau}}_{(a_\tau+1):e_\tau} 
\end{align*}
and $\big( (\vartheta_{a_\tau},\tilde{X}_{a_\tau}),\dots, (\vartheta_{e_\tau},\tilde{X}_{e_\tau} )\big)$ be such that
\begin{align*}
\big( (\vartheta_{a_\tau},\tilde{X}_{a_\tau}),\dots, (\vartheta_{e_\tau},\tilde{X}_{e_\tau} )\big)\big|\F_{e_\tau}\sim \zeta_{a_\tau}\big(\dd(\theta_{a_\tau}, x_{a_\tau})\big)\prod_{s=a_\tau+1}^{e_\tau}m_{s,\theta_{s}}(x_s|x_{s-1})\lambda(\dd x_s)K_{\mu_s|\Theta}(\theta_{s-1},\dd \theta_s). 
\end{align*}
Finally, to simplify the notation  in what follows, we let 
\begin{align*}
p_\tau= \P\Big( \sum_{i=a_\tau+1}^s  U_i \in B_{\epsilon_{\tau}}(0),\,\forall   s\in\{a_\tau+1,\dots, e_\tau\}\big|\,\F_{e_\tau} \Big),\quad\forall \tau\geq 1.
\end{align*}

We first show that $\xi^-_{e_{\tau}}=\smallo_\P(1)$.  To do so,  we let $\tau''\in\mathbb{N}$ be such that   $g(2\epsilon_\tau)\leq \delta_\star/2<1$ for all $\tau\geq \tau''$ and note that, for any $\delta\in(0,\infty)$ and $(0,\infty)$-valued random variable $X$, we have
\begin{align}\label{eq:double_Jensen}
\log \E[\exp(-\delta X)]\geq -\log \E\big[\exp(\delta X)\big]= -\log \E\big[\exp(\delta^{1/2} X)^{\delta^{1/2}}\big]\geq -\delta^{1/2}\log \E\big[\exp(\delta^{1/2} X)\big]
\end{align}
where the two inequalities hold by  Jensen's inequality. Using this latter result we readily obtain that
\begin{align*}
\log \widetilde{W}^{\eta_{a_\tau}}_{(a_\tau+1): e_\tau}&\geq  -g(2\epsilon_\tau)^{1/2}\log  W^{\eta_{a_\tau}}_{(a_\tau+1): e_\tau},\quad\P-a.s.,\quad\forall \tau\geq \tau''
\end{align*}
where, for all integers $1\leq t_1\leq t_2$ and   probability measure $\eta$ on $(\setX,\mathcal{X})$, the random variable $W^{\eta}_{t_1:t_2}$ is as defined in \eqref{eq:W_def}. Therefore, letting $(\xi_{e_{\tau}})_{\tau\geq 1}$ be as defined in \eqref{eq:xi_tau}  (with $d_\tau$ replaced by $e_\tau$ for all $\tau\geq 1$),   it follows that $\P(0\leq \xi^-_{e_{\tau}}\leq \xi_{e_{\tau}} )=1$ for all $\tau\geq \tau''$ and thus, since under the assumptions of the lemma we have $\xi_{e_{\tau}}=\smallo_\P(1)$ by \eqref{eq:conv_xi_tau} (with $d_\tau$ replaced by $e_\tau$ for all $\tau\geq 1$), it follows that $\xi^-_{e_{\tau}}=\smallo_\P(1)$.

To proceed further let $\epsilon\in(0,\delta_\star/2)$ and $\tau_\epsilon''\in\mathbb{N}$ be such that $\epsilon_\tau<\epsilon$ for all $\tau\geq \tau_\epsilon''$. Remark that, under \ref{assume:Model_ThetaStar},   for all $\tau\geq \tau_\epsilon''$  we have   $\theta+u\in \mathcal{N}_{\delta_\star}(\tilde{\Theta}_\star)\subset\mathring{\Theta}$ for all $\theta\in \mathcal{N}_{\epsilon}(\tilde{\Theta}_\star)$ and   all $u\in B_{\epsilon_\tau}(0)$.

Therefore, for all $\tau\geq \tau_\epsilon''$  and using Tonelli's theorem we have, $\P$-a.s.,
\begin{equation}\label{eq:Ise_2}
\begin{split}
P^{\zeta_{a_{\tau}}}_{(a_\tau+1):e_\tau}  (\Theta) &\geq   \E_{(a_\tau+1):e_\tau}^{\zeta_{a_{\tau}}}\bigg[ \ind_{\mathcal{N}_\epsilon(\tilde{\Theta}_\star)}(\vartheta_{a_\tau})\tilde{f}_{\tau}(\vartheta_{ a_\tau},\dots,\vartheta_{ e_\tau})  \prod_{s=a_\tau+1}^{e_\tau} G_{s,\vartheta_s}(\tilde{X}_{s-1},\tilde{X}_s)\bigg]\\
&\geq \widetilde{W}^{\eta_{a_\tau}}_{(a_\tau+1): e_\tau}\Big(  \inf_{\theta\in \mathcal{N}_\epsilon( \tilde{\Theta}_\star)}\bar{P}^{\eta_{a_\tau}}_{(a_\tau+1), e_\tau}(\theta)\Big)   \E_{(a_\tau+1):e_\tau}^{\zeta_{a_{\tau}}}\Big[ \ind_{\mathcal{N}_\epsilon(\tilde{\Theta}_\star)}(\vartheta_{a_\tau}) f_{\tau}(\vartheta_{ a_\tau},\dots,\vartheta_{e_\tau}) \Big]\\
&= p_\tau\widetilde{W}^{\eta_{a_\tau}}_{(a_\tau+1): e_\tau}\Big(  \inf_{\theta\in \mathcal{N}_\epsilon(\tilde{\Theta}_\star)}\bar{P}^{\eta_{a_\tau}}_{(a_\tau+1),e_\tau}(\theta) \Big) \E_{(a_\tau+1):e_\tau}^{\kappa_{a_\tau}\otimes\eta_{a_\tau}}\big[ \ind_{\mathcal{N}_\epsilon(\tilde{\Theta}_\star)}(\vartheta_{a_\tau})\big]\\
&= p_\tau\widetilde{W}^{\eta_{a_\tau}}_{(a_\tau+1): e_\tau} \Big( \inf_{\theta\in  \mathcal{N}_\epsilon(\tilde{\Theta}_\star)}\bar{P}^{\eta_{a_\tau}}_{(a_\tau+1), e_\tau}(\theta) \Big)  \kappa_{a_\tau}\big(\mathcal{N}_\epsilon(\tilde{\Theta}_\star)\big)\\
&\geq \exp\Big((e_\tau-a_\tau)\big(\inf_{\theta\in  \mathcal{N}_\epsilon(\tilde{\Theta}_\star)}\tilde{l}(\theta)-\xi^-_{e_{\tau}}-1/c_\tau\big)\Big)\kappa_{a_\tau}\big(\mathcal{N}_\epsilon(\tilde{\Theta}_\star)\big).
\end{split}
\end{equation}

To conclude the proof of the lemma we let
\begin{align*}
\tilde{\xi}_{b_\tau}=\xi^-_{e_{\tau}}+\frac{1}{c_\tau}+\frac{1}{e_\tau-a_\tau}\log^+ D_{(e_\tau-r+1):e_\tau},\quad\forall \tau\geq 1
\end{align*}
noting that $\tilde{\xi}_{b_\tau}=\smallo_\P(1)$  since  $\xi^-_{e_{\tau}}=\smallo_\P(1)$ while, under \ref{assume:Model_stationary} and   \ref{assume:G_bounded},
\begin{align*}
\frac{1}{e_\tau-a_\tau}\log^+ D_{(e_\tau-r+1):e_\tau}\dist \frac{1}{e_\tau-a_\tau}\log^+  D_{(r+1):2r}=\smallo_\P(1).
\end{align*}
Then, using \eqref{eq:step1} and \eqref{eq:Ise_2}, 
\begin{align*}
   P^{\zeta_{a_{\tau}}}_{(a_\tau+1):b_\tau}  (\Theta)\geq \kappa_{a_\tau}\big(\mathcal{N}_\epsilon(\tilde{\Theta}_\star)\big)\exp\Big((e_\tau-a_\tau)\big(\inf_{\theta\in \mathcal{N}_\epsilon(\tilde{\Theta}_\star)}\tilde{l}(\theta)-\tilde{\xi}_{b_\tau}\Big),\quad\forall \tau\geq \tau_\epsilon'',\quad\P-a.s.
\end{align*}
and the result of the lemma follows  upon noting that  $\lim_{\tau\rightarrow\infty}(e_\tau-a_\tau)/(b_\tau-a_\tau)=1$. 

\end{proof}

\subsubsection{Proof of Theorem \ref{thm:key}\label{p-thm:key}}
\begin{proof}

Let
\begin{align*}
c'_\tau=\frac{e_\tau-a_\tau}{b_\tau-a_\tau}\tilde{c}_\tau,\quad\forall \tau\geq 1
\end{align*}
and remark that, under the assumptions of the theorem, we $\P$-a.s.~have, for all $\tau\geq 1$,
\begin{align*}
\frac{1}{b_\tau-a_\tau} \log \P\Big(\exists  s\in\{a_\tau+1,\dots,&  b_\tau\}:\,  \sum_{i=a_\tau+1}^s  U_i\not\in B_{\tilde{\epsilon}_{\tau}}(0)\big|\,\F_{b_\tau}\Big)\\
&\leq \frac{1}{b_\tau-a_\tau} \log \P\Big(\exists s\in\{a_\tau+1,\dots,  e_\tau\}:\,  \sum_{i=a_\tau+1}^s  U_i\not\in B_{\tilde{\epsilon}_{\tau}}(0)\big|\,\F_{e_\tau}\Big)\\
&\leq -\frac{e_\tau-a_\tau}{b_\tau-a_\tau}\tilde{c}_\tau\\
&=-c'_\tau.
\end{align*}
Therefore, noting that $\lim_{\tau\rightarrow\infty}\tilde{\epsilon}_\tau=\lim_{\tau\rightarrow\infty}(1/c'_\tau)=0$, it follows if we let $c_\tau=c'_\tau$ and $\epsilon_\tau=\tilde{\epsilon}_\tau$ for all $\tau\geq 1$ then the sequences $(a_\tau)_{\tau\geq 1}$, $(b_\tau)_{\tau\geq 1}$, $(\epsilon_\tau)_{\tau\geq 1}$ and $(c_\tau)_{\tau\geq 1}$ satisfy the conditions of Lemma \ref{lemma::upper_general}.

We now show that
\begin{align}\label{eq:Bound_log}
\frac{1}{z}\log (1-p)\geq \frac{1}{\frac{1}{z}\log p},\quad\forall z\geq 1,\quad\forall p\in (0,1).
\end{align}
To this aim let $z\geq 1$ and $p\in(0,1)$. Then
\begin{align*}
\frac{1}{z}\log (1-p)\geq \frac{1}{\frac{1}{z}\log p}&\Leftrightarrow \big(\log p\big)\log (1-p)\leq z^2
\end{align*}
and \eqref{eq:Bound_log} follows from the fact that $z^2\geq 1$ while $\big(\log p\big)\log (1-p)\leq (1-p)p\leq 1$.

Using \eqref{eq:Bound_log}, it follows that under the assumptions of the theorem we $\P$-a.s.~have, for all $\tau\geq 1$,
\begin{align*}
\frac{1}{e_\tau-a_\tau}\log \P\Big( \sum_{i=a_\tau+1}^s  U_i\in B_{\tilde{\epsilon}_\tau}(0),\,\forall   s\in\{a_\tau+1,\dots,  e_\tau\}\big| \,\F_{e_\tau}\Big)\geq -\frac{1}{\tilde{c}_\tau} 
\end{align*}
and thus  if for all $\tau\geq 1$ we let $c_\tau=\tilde{c}_\tau$ and $\epsilon_\tau=\tilde{\epsilon}_\tau$ then the sequences $(a_\tau)_{\tau\geq 1}$, $(b_\tau)_{\tau\geq 1}$, $(\epsilon_\tau)_{\tau\geq 1}$ and $(c_\tau)_{\tau\geq 1}$ satisfy the conditions of Lemma \ref{lemma::lower_general}.

To proceed further remark that, under \ref{assume:Model}, the function $\tilde{l}$ is continuous on $\Theta$ and there exists a non-empty set $A\subset\Theta$ such that $\tilde{l}(\theta)<\tilde{l}(\theta')$ for all $\theta\in A$ and all $\theta'\in\tilde{\Theta}_\star$. Under these two assumptions it is readily checked that  there exists an $\epsilon'\in(0,\infty)$ such that $\Theta\setminus\mathcal{N}_{\epsilon'}(\tilde{\Theta}_\star)\neq\emptyset$  and such that, for all $\epsilon\in(0,\epsilon')$, there is a $\delta_\epsilon\in(0,\infty)$ for which 
\begin{align*}
\inf_{\theta\in\mathcal{N}_{\delta_\epsilon}(\tilde{\Theta}_\star)}\tilde{l}(\theta)>\sup_{\theta\in\Theta\setminus\mathcal{N}_\epsilon(\tilde{\Theta}_\star)} \tilde{l}(\theta).
\end{align*}
Letting $\delta_\star$ be as in \ref{assume:Model} and  $\bar{\epsilon}\in (0,\delta_\star/2)$ be as in Lemma \ref{lemma::lower_general}, we now let $\epsilon\in (0,\epsilon'\wedge\bar{\epsilon}]$ and $\delta_{\epsilon}\in(0,\bar{\epsilon})$ be such that 
\begin{align*}
v_\epsilon:=\inf_{\theta\in\mathcal{N}_{\delta_{\epsilon}}(\tilde{\Theta}_\star)}\tilde{l}(\theta)-\sup_{\theta\in\Theta\setminus\mathcal{N}_\epsilon(\tilde{\Theta}_\star)} \tilde{l}(\theta)>0.
\end{align*}

Next, for all $t\geq 1$ we let $(\tilde{\xi}_{rt},\underline{\chi}_{rt},\bar{\chi}_{rt})$   be as in the statement of Lemma \ref{lemma:gamma}, and  $\tau'$ be such that $a_\tau\geq 2r$ for all $\tau\geq \tau'$. Then,  by combining Lemma \ref{lemma:Zt} and Lemma \ref{lemma:gamma}, we have
\begin{align}\label{eq:bound_piAA}
\pi_{b_\tau,\Theta}(A)\leq \tilde{\xi}^{\,2}_{a_\tau}\,\frac{P^{\bar{\pi}_{a_\tau}\otimes\bar{\chi}_{a_\tau}}_{(a_\tau+1):b_\tau}(A)}{P^{\bar{\pi}_{a_\tau}\otimes\underline{\chi}_{a_\tau}}_{(a_{\tau}+1):b_\tau}(\Theta)},\quad\forall A\in\mathcal{T},\quad\P-a.s.,\quad\forall \tau\geq \tau'
\end{align}
where $\bar{\pi}_t$ is as defined in \eqref{eq:pi_bar} for all $t\geq 1$. Recall  $(\underline{\chi}_{rt})_{t\geq 1}$ and $(\bar{\chi}_{rt})_{t\geq 1}$ depend on $(\mu_t)_{t\geq 1}$ only through the sequence $(W^*_{t})_{t\geq 0}$ and that, for all $t\geq 1$, we have  $\underline{\chi}_{rt}\dist\underline{\chi}_r$ and $\bar{\chi}_{rt}\dist\bar{\chi}_r$ where, under \ref{assume:Model_G_lower}-\ref{assume:Model_sup}, the   random probability measures  $\underline{\chi}_r$ and $\bar{\chi}_r$ are $(r,\delta_\star, \tilde{l}, (\varphi_t)_{t\geq 1})$-consistent, with  $\tilde{l}$ and  $(\varphi_t)_{t\geq 1}$ as in \ref{assume:Model}. In addition, recall that the sequence of $(1,\infty)$-valued random variables $(\tilde{\xi}_{rt})_{t\geq 1}$ depends on $(\mu_t)_{t\geq 1}$ only the sequence $(W^*_{t})_{t\geq 0}$ defined in \eqref{eq:W_star} and that  $\tilde{\xi}_{rt}=\bigO_\P(1)$.

We now let $(\xi^+_{b_\tau})_{\tau\geq 1}$ be the sequence of  $(0,\infty]$-valued random variables defined in Lemma \ref{lemma::upper_general} when $(\epsilon_\tau, c_\tau,\eta_{a_\tau}, \kappa_{a_\tau})=(\tilde{\epsilon}_\tau, c'_\tau,  \bar{\chi}_{a_\tau}, \bar{\pi}_{a_\tau})$ for all $\tau\geq 1$, and   $(\xi^-_{b_\tau})_{\tau\geq 1}$ be the sequence of  $(0,\infty]$-valued random variables defined in Lemma \ref{lemma::lower_general} when $(\epsilon_\tau, c_\tau,\eta_{a_\tau}, \kappa_{a_\tau})=(\tilde{\epsilon}_\tau, \tilde{c}_\tau, \underline{\chi}_{a_\tau},\bar{\pi}_{a_\tau})$ for all $\tau\geq 1$. Remark that the two sequences $(\xi^+_{b_\tau})_{\tau\geq 1}$ and  $(\xi^-_{b_\tau})_{\tau\geq 1}$   depend  on $(\mu_t)_{t\geq 1}$  only through the sequences  $(\tilde{c}_{\tau})_{\tau\geq 1}$,  $(c'_{\tau})_{\tau\geq 1}$ and $(\epsilon_{\tau})_{\tau\geq 1}$ and  through the random variable $\Gamma^\mu$ defined in \ref{condition:Inf_K}, and that they are such that $\xi^+_{b_\tau}=\smallo_\P(1)$ and  $\xi^-_{b_\tau}=\smallo_\P(1)$.

Finally, with $(\xi_{\delta,a_\tau})_{\geq 1}$ be as in the statement of the theorem for all $\delta\in(0,\delta_\star/2)$, we let
\begin{align*}
Z_{\epsilon,b_{\tau}}=\xi^+_{b_{\tau}}+\frac{\log(1+\xi_{b_\tau}^+)}{b_\tau-a_\tau}+\xi^-_{b_{\tau}}+\xi_{\delta_\epsilon,a_\tau}+\frac{2}{b_\tau-a_\tau} \log \tilde{\xi}_{a_\tau},\quad\forall \tau\geq 1
\end{align*}
and  
\begin{align*}
W_{\epsilon, b_\tau}= 1\wedge  \exp\big\{-(b_\tau-a_\tau) (v_\epsilon-Z_{\epsilon,b_{\tau}})\big\},\quad\forall \tau\geq 1.
\end{align*}
Then, using \eqref{eq:bound_piAA} and Lemmas \ref{lemma::upper_general}-\ref{lemma::lower_general}, we have
\begin{align}\label{eq:conv:pi_b}
0\leq \pi_{b_\tau,\Theta}\big(\mathcal{N}_\epsilon(\tilde{\Theta}_\star)^c\big)\leq  W_{\epsilon,b_\tau},\quad\forall \tau\geq 1,\quad \P-a.s.
\end{align}
The sequence $(W_{\epsilon, b_\tau})_{\tau\geq 1}$ of $[0,1]$-valued random variables depends on $(\mu_t)_{t\geq 1}$  only through the  sequences  $(\tilde{c}_{\tau})_{\tau\geq 1}$, $(\epsilon_{\tau})_{\tau\geq 1}$, $(\xi_{\delta_\epsilon,a_\tau})_{\tau\geq 1}$ and  $(W^*_{t})_{t\geq 0}$, and  through the random variable $\Gamma^\mu$,  and thus  by \eqref{eq:conv:pi_b}  to prove the result of the theorem it suffices to show that $W_{\epsilon, b_\tau}=\smallo_\P(1)$.

To this aim note first that, using the aforementioned properties of $(\xi^+_{b_{\tau}})_{\tau\geq 1}$, $(\xi^-_{b_{\tau}})_{\tau\geq 1}$ and $(\tilde{\xi}_{a_{\tau}})_{\tau\geq 1}$, and the fact that, by assumption $\xi_{\delta_\epsilon,a_\tau}=\smallo_\P(1)$, we have $Z_{\epsilon,b_{\tau}}=\smallo_\P(1)$. Therefore, recalling that $v_\epsilon>0$,  for all $\gamma\in(0,\infty)$ we have
\begin{align*}
\limsup_{\tau\rightarrow\infty}\P(W_{\epsilon, b_\tau}\geq \gamma)&\leq \limsup_{\tau\rightarrow\infty}\P\Big(\exp\big(-(b_\tau-a_\tau) v_\epsilon/2\big)\geq \gamma\Big)+\limsup_{\tau\rightarrow\infty}\P\big(Z_{\epsilon,b_{\tau}}\geq v_\epsilon/2\big)=0
\end{align*}
showing that $W_{\epsilon, b_\tau}\PP 0$. The proof of the theorem is complete.
\end{proof}

\subsubsection{Proof of Lemma \ref{lemma:Extend_initit}\label{p-lemma:Extend_initit}}
\begin{proof}

Let $t$, $m$  and $\bar{m}$ be as in the statement of the lemma, $(r,\delta_\star,D)$ be as in \ref{assume:Model}, $k_t\in\{sr,\,s\in\mathbb{N}_0\}$ be such that $k_t< t-\bar{m}\leq k_t+r$ and $k'_t\in\{sr,\,s\in\mathbb{N}_0\}$ be such that $k'_t-r<t-1\leq k'_t$. In addition, for all integers $1\leq t_1\leq t_2$ we let $D_{t_1:t_2}$ be as defined in \eqref{eq:Dst} and
\begin{align*}
\tilde{D}_{t_1:t_2}=\prod_{s=t_1}^{t_2}\inf_{(\theta,x)\in\Theta\times D} Q_{s,\theta}(x,D).
\end{align*}
 
When $\bar{m}=1$ the result of the lemma trivially holds with $\xi'_{t,\bar{m}}=1$ for all $t\geq 1$ and below therefore below we assume that $\bar{m}>1$.

To show the result of the lemma  note first that, by using Lemma \ref{lemma:Zt}, we $\P$-a.s.~have, for all   $A\in\mathcal{T}$
\begin{equation}\label{eq:pp1}
\begin{split}
\bar{\pi}_{t}(A)&\geq \frac{\tilde{D}_{(t-m):(t-1)}}{D_{(t-m):(t-1)}}
\int_{\Theta^{m+2}\times D}\ind_A(\theta_{t}) \pi_{t-m-1}\big(\dd (\theta_{t-m-1},x_{t-m-1})\big)\prod_{s=t-m}^{t}  K_{\mu_s|\Theta}(\theta_{s-1},\dd\theta_s)\\
&\geq \frac{\tilde{D}_{(t-m):(t-1)}}{D_{(t-m):(t-1)}}
\int_{ \Theta^{m+2}\times D}\ind_A(\theta_{t}) \pi_{t-m-1}\big(\dd (\theta_{t-m-1},x_{t-m-1})\big)\prod_{s=t-m}^{t}  K_{\mu_s}(\theta_{s-1},\dd\theta_s)\\
&\geq \frac{\tilde{D}_{(k_t+1):k'_t}}{D_{(k_t+1):k'_t}}
\int_{ \Theta^{m+2}\times D}\ind_A(\theta_{t}) \pi_{t-m-1}\big(\dd (\theta_{t-m-1},x_{t-m-1})\big)\prod_{s=t-m}^{t}  K_{\mu_s}(\theta_{s-1},\dd\theta_s).
\end{split}
\end{equation}

Next, let $\delta\in(0,\delta_\star)$ and note that under \ref{assume:Model_ThetaStar} we have $\theta+u\in\mathcal{N}_{\delta}(\tilde{\Theta}_\star)\subset\mathring{\Theta}_\star$ for all $\theta\in \mathcal{N}_{\delta/2}(\tilde{\Theta}_\star)$ and all $u\in\mathcal{B}_{\delta/2}(0)$. Therefore, under \ref{condition:Inf_mu},
\begin{equation}\label{eq:pp2}
\begin{split}
&\int_{ \Theta^{m+2}\times D}\ind_{\mathcal{N}_{\delta}(\tilde{\Theta}_\star)}(\theta_{t}) \pi_{t-m-1}\big(\dd (\theta_{t-m-1},x_{t-m-1})\big)\prod_{s=t-m}^{t}  K_{\mu_s}(\theta_{s-1},\dd\theta_s)\\
&\geq \bigg(\prod_{s=t-m+1}^t\mu_s\big(B_{\delta/(2m)}(0)\big)\bigg)\int_{ \Theta\times D} \pi_{t-m-1}\big(\dd (\theta_{t-m-1},x_{t-m-1})\big)K_{\mu_{t-m}}\big(\theta_{t-m-1},\mathcal{N}_{\delta/2}(\tilde{\Theta}_\star)\big)\\
&\geq\big(\Gamma^\mu_{\delta/(2m)}\big)^{m}\int_{ \Theta\times D} \pi_{t-m-1}\big(\dd (\theta_{t-m-1},x_{t-m-1})\big)K_{\mu_{t-m}}(\theta_{t-m-1},\mathcal{N}_{\delta/2}(\tilde{\Theta}_\star))\\
&\geq \big(\Gamma^\mu_{\delta/(2\bar{m})}\big)^{\bar{m}}\pi_{t-m-1,\setX}(D)\inf_{\theta\in\Theta}K_{\mu_{t-m}}\big(\theta,\mathcal{N}_{\delta/2}(\tilde{\Theta}_\star)\big)
\end{split}
\end{equation}
with $\Gamma^\mu_\epsilon$ as in \ref{condition:Inf_mu} for all $\epsilon\in(0,\infty)$.

Then, by combining \eqref{eq:pp1} and \eqref{eq:pp2}, we obtain that
\begin{align*}
\bar{\pi}_{t}\big(\mathcal{N}_{\delta}(\tilde{\Theta}_\star)\big)\geq\bigg( \frac{\tilde{D}_{(k_t+1):k'_t}}{D_{(k_t+1):k'_t}}\big(\Gamma^\mu_{\delta/(2\bar{m})}\big)^{\bar{m}}\min_{m\in\{1,\dots,\bar{m}\}}\pi_{t-m-1,\setX}(D)\bigg) \inf_{\theta\in\Theta} K_{\mu_{t-m}}\big(\theta,\mathcal{N}_{\delta/2}(\tilde{\Theta}_\star)\big) 
\end{align*}
and to conclude the proof it remains to show that
\begin{align}\label{eq:mt}
\log^-\bigg( \frac{\tilde{D}_{(k_t+1):k'_t}}{D_{(k_t+1):k'_t}}\big(\Gamma^\mu_{\delta/(2\bar{m})}\big)^{\bar{m}}\min_{m\in\{1,\dots,\bar{m}\}}\pi_{t-m-1,\setX}(D)\bigg)=\bigO_\P(1).
\end{align}
To this aim note that  $\bar{m}\log \Gamma^\mu_{\delta/(2\bar{m})}=\bigO_\P(1)$ under \ref{condition:Inf_mu} while $\log \pi_{t-m-1,\setX}(D)=\bigO_\P(1)$ for all $m\in\{1,\dots,\bar{m}\}$ under \eqref{eq:W_star}. In addition, under \ref{assume:Model_stationary} we have
\begin{align*}
\tilde{D}_{(k_t+1):k'_t} \dist \tilde{D}_{(r+1):(k'_t+r-k_t)},\quad D_{(k_t+1):k'_t}\dist D_{(r+1):(k'_t+r-k_t)}
\end{align*}
and thus, noting that $\sup_{p\geq 1}(k'_t-k_t)<\infty$, we have $\log \tilde{D}_{(k_t+1):k'_t}=\bigO_\P(1)$ under \ref{assume:G_l_E} while we have  $\log D_{(k_t+1):k'_t}=\bigO_\P(1)$ under \ref{assume:G_bounded}. This shows \eqref{eq:mt} and the proof of the lemma is complete.
\end{proof}

\subsubsection{Proof of Lemma \ref{lemma:pred_def}\label{p-lemma:pred_def}}

\begin{proof}

Let $\theta\in\Theta$, $t\geq 2r$ and $\tau\in\{2r,\dots,t\}\cap\{sr,\,s\in\mathbb{N}\}$, and note that to prove the lemma we need to show that
\begin{align*}
\P\bigg(\int_{\setX^{t-\tau+1}}  \pi_{\tau,\setX}(\dd x_{\tau}) \prod_{s=\tau+1}^{t} Q_{s,\theta}(x_{s-1},\dd x_s)\in(0,\infty)\bigg)=1.
\end{align*}
To this aim remark first that
\begin{align*}
\int_{\setX^{t-\tau+1}}  \pi_{\tau,\setX}(\dd x_{\tau}) \prod_{s=\tau+1}^{t} Q_{s,\theta}(x_{s-1},\dd x_s)\leq D_{(\tau+1):t}
\end{align*}
where for all integers $1\leq t_1\leq t_2+1$ the random variable $D_{t_1:t_2}$ is as defined in \eqref{eq:Dst}. Under \ref{assume:G_bounded}, we have $\P\big(D_{(\tau+1):t}<\infty)=1$ and thus to prove the lemma it remains to prove that
\begin{align}\label{eq:Show_XX}
\P\bigg(\int_{\setX^{t-\tau+1}}  \pi_{\tau,\setX}(\dd x_{\tau}) \prod_{s=\tau+1}^{t} Q_{s,\theta}(x_{s-1},\dd x_s)>0\bigg)=1.
\end{align}

To show \eqref{eq:Show_XX} note first that, by Lemma \ref{lemma:Zt},
\begin{align}\label{eq:piX}
\pi_{\tau,\setX}(A)=\frac{  \Lambda^{\pi_0}_{1:\tau}(\Theta\times A)}{ \Lambda^{\pi_0}_{1:\tau}(\Theta\times\setX)},\quad\forall A\in \mathcal{X},\quad \P-a.s.
\end{align}

To proceed further let $\bar{\Lambda}^{\pi_0}_{1:\tau}(\dd x)$ be the random  measure on $(\setX,\mathcal{X})$ such that $\bar{\Lambda}^{\pi_0}_{1:\tau}(A)=\Lambda_{1:\tau}^{\pi_0}(\Theta\times A)$ for all $A\in\mathcal{X}$. In addition, let $k_t\in\{sr,\,s\in\mathbb{N}\}$ be such that $k_t-r<t\leq k_t$ and   let $(\tilde{\xi}_{sr})_{s\geq 1}$ and $(\underline{\chi}_{rs})_{s\geq 1}$ be as in Lemma \ref{lemma:gamma}. Then, using first \eqref{eq:piX} and then Lemma \ref{lemma:gamma}, we $\P$-a.s.~have,
\begin{equation}\label{eq:Show_XXX}
\begin{split}
\int_{\setX^{t-\tau+1}}  \pi_{\tau,\setX}(\dd x_{\tau}) \prod_{s=\tau+1}^{t} Q_{s,\theta}(x_{s-1},\dd x_s)&= \frac{1}{\Lambda^{\pi_0}_{1:\tau}(\Theta\times\setX)}\int_{\setX^{t-\tau+1}}    \bar{\Lambda}^{\pi_0}_{1:\tau}(\dd x_{\tau}) \prod_{s=\tau+1}^{t} Q_{s,\theta}(x_{s-1},\dd x_s)\\
&\geq  \frac{1}{\tilde{\xi}_{\tau}^2}\int_{\setX^{t-\tau+1}}   \underline{\chi}_{\tau}(\dd x_{\tau}) \prod_{s=\tau+1}^{t} Q_{s,\theta}(x_{s-1},\dd x_s)\\
&\geq  \frac{1}{\tilde{\xi}_{\tau}^2\, D_{(t+1):k_t}}\bar{P}_{(\tau+1):k_t}^{\underline{\chi}_\tau}(\theta).  
\end{split}
\end{equation}

To conclude the proof remark that $\P\big(D_{(t+1):k_t}<\infty\big)$ under \ref{assume:G_bounded}. In addition, under \ref{assume:Model_stationary} we have
\begin{align*}
\P\Big(\bar{P}_{(\tau+1):k_t}^{\underline{\chi}_\tau}(\theta)>0\Big)=\P\Big(\bar{P}_{(r+1):(k_t+r-\tau)}^{\underline{\chi}_r}(\theta)>0\Big)
\end{align*}
where $\underline{\chi}_r$ is as in \ref{assume:Model_G_lower}. Under \ref{assume:Model_G_lower}, the random probability measure  $\underline{\chi}_r$ is a $(r,\delta_\star,\tilde{l},(\varphi_t)_{t\geq 1})$-consistent, with $(\delta_\star,\tilde{l},(\varphi_t)_{t\geq 1}))$ as in \ref{assume:Model}, and thus, $\P\big(\bar{P}_{(r+1):(k_t+r-\tau)}^{\underline{\chi}_r}(\theta)>0\big)=1$. Since by Lemma \ref{lemma:gamma} we have $\P(\tilde{\xi}_{\tau}<\infty)=1$, it follows that
\begin{align*}
\P\Big(\frac{1}{\tilde{\xi}_{\tau}^2\,D_{(t+1):k_t}}\bar{P}_{(\tau+1):k_t}^{\underline{\chi}_\tau}(\theta)>0\Big)=1
\end{align*}
which, together with \eqref{eq:Show_XXX}, shows  \eqref{eq:Show_XX}. The proof is complete.
\end{proof}

\subsubsection{Proof of Lemma \ref{lemma:upper_X}\label{p-lemma:upper_X}}

\begin{proof}
For all $\tau\geq 1$ let  $\big( (\vartheta_{a_\tau}, \tilde{X}_{a_\tau}),\dots,(\vartheta_{b_\tau},   \tilde{X}_{b_\tau})\big)$ be a  $(\Theta\times\setX)^{b_\tau-a_\tau+1}$-valued random variable such that
\begin{align*}
\Big( (\vartheta_{a_\tau}, \tilde{X}_{a_\tau}),\dots,(\vartheta_{b_\tau},   \tilde{X}_{b_\tau})\Big) \big|\F_{b_\tau}\sim \pi_{a_\tau}\big(\dd (\theta_{a_\tau},x_{a_\tau})\big)\prod_{s=a_\tau+1}^{b_\tau}m_{s,\theta_{s}}(x_s|x_{s-1})\lambda(\dd x_s)K_{\mu_s|\Theta}(\theta_{s-1},\dd \theta_s) 
\end{align*}
and let $\tilde{f}_\tau=f_{b_\tau-a_\tau,\epsilon_\tau}$, with $f_{t,\epsilon}$ as defined in \eqref{eq:ft_def} for all $t\in\mathbb{N}$ and $\epsilon\in(0,\infty)$. Finally, let   $(\varphi_t)_{t\geq 1}$  and $\delta_\star$ be as in \ref{assume:Model} and  $\tau'\in\mathbb{N}$ be such that $g(3\epsilon_{\tau})\leq\delta_\star/4$ for all $\tau\geq \tau'$. 

To prove the lemma remark first that, for all $\tau\geq 1$ and $A\in\mathcal{X}$, we have
\begin{equation}\label{eq:Thm2_bound0}
\begin{split}
\Lambda^{\pi_{a_\tau}}_{(a_\tau+1):b_\tau}&(\mathcal{N}_{\epsilon_\tau}(\{\theta\})\times A)\\
&=\E_{(a_\tau+1):b_\tau}^{\pi_{a_\tau}}\Big[\ind_{\mathcal{N}_{\epsilon_\tau}(\{\theta\})}(\vartheta_{b_\tau})\ind_{A}(\tilde{X}_{b_\tau})\tilde{f}_{\tau}(\vartheta_{a_\tau},\dots,\vartheta_{b_\tau})\prod_{s=a_\tau+1}^{b_\tau} G_{s,\vartheta_s}(\tilde{X}_{s-1},\tilde{X}_s)\Big]\\
&+\E_{(a_\tau+1):b_\tau}^{\pi_{a_\tau}}\Big[\ind_{\mathcal{N}_{\epsilon_\tau}(\{\theta\})}(\vartheta_{b_\tau})\ind_{A}(\tilde{X}_{b_\tau})\big(1-\tilde{f}_{\tau}(\vartheta_{a_\tau},\dots,\vartheta_{b_\tau})\big)\prod_{s=a_\tau+1}^{b_\tau} G_{s,\vartheta_s}(\tilde{X}_{s-1},\tilde{X}_s)\Big]
\end{split}
\end{equation}
and we now study the two terms on the r.h.s.~of \eqref{eq:Thm2_bound0}, starting with the first one.

To this aim  remark that, under \ref{assume:Model_smooth}, for all $\tau\geq \tau'$ we  $\P$-a.s.~have, for all $A\in\mathcal{X}$,
\begin{equation}\label{eq:Thm2_bound1}
\begin{split}
&\E_{(a_\tau+1):b_\tau}^{\pi_{a_\tau}}\Big[\ind_{\mathcal{N}_{\epsilon_\tau}(\{\theta\})}(\vartheta_{b_\tau})\ind_{A}(\tilde{X}_{b_\tau})\tilde{f}_{\tau}(\vartheta_{a_\tau},\dots,\vartheta_{b_\tau})\prod_{s=a_\tau+1}^{b_\tau} G_{s,\vartheta_s}(\tilde{X}_{s-1},\tilde{X}_s)\Big]\\
&\leq \int_{\setX^{b_\tau-a_\tau+1}} \ind_A(x_{b_\tau})\exp\Big(g(3\epsilon_{\tau})\sum_{s=a_\tau+1}^{b_\tau} \varphi_s(x_{s-1},x_s)\Big)\pi_{a_\tau,\setX}(\dd x_{a_\tau})\prod_{s=a_\tau+1}^{b_\tau} Q_{s,\theta}(x_{s-1},\dd x_s).
\end{split}
\end{equation}
Next, using  first the fact that  $|1-e^x|\leq |x|e^{|x|}$ for all $x\in\R$ and then the fact that $x\leq e^x$  for all $x\in\R$, it follows that for all $\tau\geq \tau'$ we $\P$-a.s.~have, for all $A\in\mathcal{X}$ and letting $r_\tau=b_\tau-a_\tau+1$,
\begin{equation}\label{eq:Thm2_bound21}
\begin{split}
 &\int_{\setX^{r_\tau}} \ind_A(x_{b_\tau})\Big|1-\exp\Big(g(3\epsilon_{\tau})\sum_{s=a_\tau+1}^{b_\tau} \varphi_s(x_{s-1},x_s)\Big)\Big|\pi_{a_\tau,\setX}(\dd x_{a_\tau})\prod_{s=a_\tau+1}^{b_\tau} Q_{s,\theta}(x_{s-1},\dd x_s)\\
 &\leq g(3\epsilon_\tau)\hspace{-0.1cm} \int_{\setX^{r_\tau}}\hspace{-0.13cm} \sum_{s=a_\tau+1}^{b_\tau} \varphi_s(x_{s-1},x_s)\exp\Big(g(3\epsilon_{\tau})\hspace{-0.3cm}\sum_{s=a_\tau+1}^{b_\tau} \varphi_s(x_{s-1},x_s)\Big) \pi_{a_\tau,\setX}(\dd x_{a_\tau})\hspace{-0.3cm}\prod_{s=a_\tau+1}^{b_\tau} Q_{s,\theta}(x_{s-1},\dd x_s)\\
&\leq \frac{4}{\delta_\star}g(3\epsilon_{\tau}) \int_{\setX^{r_\tau}} \exp\Big(\big(g(3\epsilon_{\tau})+\frac{\delta_\star}{4}\big)\sum_{s=a_\tau+1}^{b_\tau} \varphi_s(x_{s-1},x_s)\Big)\pi_{a_\tau,\setX}(\dd x_{a_\tau})\prod_{s=a_\tau+1}^{b_\tau} Q_{s,\theta}(x_{s-1},\dd x_s)\\
&\leq   \frac{4g(3\epsilon_{\tau})}{ \delta_\star} \int_{\setX^{r_\tau}}  \exp\Big(\frac{\delta_\star}{2}\sum_{s=a_\tau+1}^{b_\tau} \varphi_s(x_{s-1},x_s)\Big)\pi_{a_\tau,\setX}(\dd a_\tau)\prod_{s=a_\tau+1}^{b_\tau} Q_{s,\theta}(x_{s-1},\dd x_s)\\
&\leq   \frac{4g(3\epsilon_{\tau})}{ \delta_\star} \bar{P}^{\pi_{a_\tau,\setX}}_{(a_\tau+1):b_\tau}(\theta)\,  W^{\pi_{a_\tau,\setX}}_{(a_\tau+1):b_\tau}
\end{split}
\end{equation}
where  for all integers $1\leq t_1\leq t_2$ and  $\eta\in\mathcal{P}(\setX)$  the random variable $W^\eta_{t_1:t_2}$ is as defined in \eqref{eq:W_def}.

To proceed further for all $\tau\geq 1$ we let $e_\tau\in\{sr,\,s\in\mathbb{N}\}$ be such that $e_\tau-r<b_\tau\leq e_\tau$ and remark that
\begin{align}\label{eq:up_W}
W^{\pi_{a_\tau,\setX}}_{(a_\tau+1):b_\tau}\leq W^{\pi_{a_\tau,\setX}}_{(a_\tau+1):e_\tau},\quad\forall \tau\geq 1.
\end{align}

Next, let  $(\tilde{\xi}_{rt})_{t\geq 1}$, $(\bar{\chi}_{rt})_{t\geq 1}$ and $(\underline{\chi}_{rt})_{t\geq 1}$ be as  in the statement of Lemma \ref{lemma:gamma} and recall that, by Lemma \ref{lemma:Zt},
\begin{align*}
\pi_{a_\tau,\setX}(A)=\frac{\Lambda^{\mu_0\otimes\chi}_{1:a_\tau}(\Theta\times A)}{\Lambda^{\mu_0\otimes\chi}_{1:a_\tau}(\Theta\times \setX)},\quad\forall A\in\mathcal{X},\quad\P-a.s.
\end{align*} 
Then, by using  Lemma \ref{lemma:gamma}, we have
\begin{align}\label{eq:W_upp}
W^{\pi_{a_\tau,\setX}}_{(a_\tau+1):e_\tau}\leq Z_{e_\tau}^{(1)}:=\tilde{\xi}^2_{a_\tau} W^{\bar{\chi}_{a_\tau}}_{(a_\tau+1):e_\tau} \frac{\sup_{\theta'\in\Theta}\bar{P}^{\bar{\chi}_{a_\tau}}_{(a_\tau+1):e_\tau}(\theta')}{\inf_{\theta'\in\Theta}\bar{P}^{\underline{\chi}_{a_\tau}}_{(a_\tau+1):e_\tau}(\theta')},\quad\forall \tau\geq \tau',\quad\P-a.s.
\end{align}
and thus, by using  \eqref{eq:Thm2_bound1}-\eqref{eq:W_upp} and  Lemma \ref{lemma:pred_def}, it follows that  for all $\tau\geq \tau'$ we $\P$-a.s.~have
\begin{equation}\label{eq:thm2_B1}
\begin{split}
\E_{(a_\tau+1):b_\tau}^{\pi_{a_\tau}}\Big[\ind_{\mathcal{N}_{\epsilon_\tau}(\{\theta\})}(\vartheta_{b_\tau})&\ind_{A}(\tilde{X}_{b_\tau})\tilde{f}_{\tau}(\vartheta_{a_\tau},\dots,\vartheta_{b_\tau})\prod_{s=a_\tau+1}^{b_\tau} G_{s,\vartheta_s}(\tilde{X}_{s-1},\tilde{X}_s)\Big]\\
&\leq \bar{P}^{\pi_{a_\tau,\setX}}_{(a_\tau+1):b_\tau}(\theta)\Big( \bar{p}_{(a_\tau+1):b_\tau}^{\pi_{a_\tau,\setX}}(A|\theta)+\frac{4}{\delta_\star} g(3\epsilon_{\tau})Z_{e_\tau}^{(1)}\Big),\quad\forall A\in\mathcal{X}.
\end{split}
\end{equation}

We now study the second term in \eqref{eq:Thm2_bound0}. Using  similar computations as in  \eqref{eq:2nd_term} we $\P$-a.s.~have,  for all $\tau\geq  \tau'$ and $A\in\mathcal{X}$,
\begin{equation}\label{eq:thm2_B2}
\begin{split}
\E_{(a_\tau+1):b_\tau}^{\pi_{a_\tau}}\Big[ \ind_{\mathcal{N}_{\epsilon_\tau}(\{\theta\})}(\vartheta_{b_\tau})&\ind_{A}(\tilde{X}_{b_\tau})\big(1-\tilde{f}_{\tau}(\vartheta_{a_\tau},\dots,\vartheta_{b_\tau})\big)\prod_{s=a_\tau+1}^{b_\tau} G_{s,\vartheta_s}(\tilde{X}_{s-1},\tilde{X}_s)\Big]\\
&\leq \E_{(a_\tau+1):b_\tau}^{\pi_{a_\tau}}\Big[\big(1-\tilde{f}_{\tau}(\vartheta_{a_\tau},\dots,\vartheta_{b_\tau})\big)\prod_{s=a_\tau+1}^{b_\tau} G_{s,\vartheta_s}(\tilde{X}_{s-1},\tilde{X}_s)\Big]\\
&\leq \exp\Big( -(b_\tau-a_\tau)(c_\tau-\log \Gamma^\mu)\Big)\sup_{\theta'\in\Theta}\bar{P}_{(a_\tau+1):b_\tau}^{\pi_{a_\tau,\setX}}(\theta')\\
&\leq \exp\Big(-(b_\tau-a_\tau)(c_\tau-\log \Gamma^\mu)\Big)\bar{P}_{(a_\tau+1):b_\tau}^{\pi_{a_\tau,\setX}}(\theta)\frac{\sup_{\theta'\in\Theta}\bar{P}_{(a_\tau+1):b_\tau}^{\pi_{a_\tau,\setX}}(\theta')}{\inf_{\theta'\in\Theta}\bar{P}_{(a_\tau+1):b_\tau}^{\pi_{a_\tau,\setX}}(\theta')}
\end{split}
\end{equation}
with $\Gamma^\mu$ as defined in \ref{condition:Inf_K}. 

Using again  Lemmas \ref{lemma:Zt}-\ref{lemma:gamma}, we have
\begin{align}\label{eq:TT1}
\frac{\sup_{\theta'\in\Theta}\bar{P}_{(a_\tau+1):b_\tau}^{\pi_{a_\tau,\setX}}(\theta')}{\inf_{\theta'\in\Theta}\bar{P}_{(a_\tau+1):b_\tau}^{\pi_{a_\tau,\setX}}(\theta')}\leq \tilde{\xi}^2_{a_\tau}\frac{\sup_{\theta'\in\Theta}\bar{P}_{(a_\tau+1):b_\tau}^{\bar{\chi}_{a_\tau}}(\theta')}{\inf_{\theta'\in\Theta}\bar{P}_{(a_\tau+1):b_\tau}^{\underline{\chi}_{a_\tau}}(\theta')},\quad\forall \tau\geq \tau',\quad\P-a.s.
\end{align}
where, by applying Lemma \ref{lemma:tech}  with $\mu_t=\delta_{\{0\}}$ for all $t\geq 1$,
\begin{align}\label{eq:TT2}
\bar{P}_{(a_\tau+1):b_\tau}^{\underline{\chi}_{a_\tau}}(\theta')\geq D_{(b_\tau+1):e_\tau}^{-1} \bar{P}_{(a_\tau+1):e_\tau}^{\underline{\chi}_{a_\tau}}(\theta'),\quad\forall \theta\in\Theta,\quad\forall \tau\geq\tau',\quad\P-a.s.
\end{align}
with the  random variable $D_{t_1:t_2}$  as defined in \eqref{eq:Dst} for all integers $1\leq t_1\leq t_2$.

Next,   assuming that $\tau'$ is such that $d_\tau:=e_\tau-r\geq r$ for all $\tau\geq \tau'$, we have
\begin{align}\label{eq:TT3}
\bar{P}_{(a_\tau+1):b_\tau}^{\bar{\chi}_{a_\tau}}(\theta')\leq D_{(d_\tau+1):b_\tau}\bar{P}_{(a_\tau+1):d_\tau}^{\bar{\chi}_{a_\tau}}(\theta'),\quad\forall \theta\in\Theta,\quad\forall \tau\geq\tau',\quad\P-a.s.
\end{align}

Therefore, letting
\begin{align*}
Z^{(2)}_{e_\tau}= \exp\Big(-(b_\tau-a_\tau)(c_\tau-\log \Gamma^\mu)\Big) D_{(d_\tau+1):e_\tau}\tilde{\xi}^2_{a_\tau}\frac{\sup_{\theta'\in\Theta}\bar{P}_{(a_\tau+1):d_\tau}^{\bar{\chi}_{a_\tau}}(\theta')}{\inf_{\theta'\in\Theta}\bar{P}_{(a_\tau+1):e_\tau}^{\underline{\chi}_{a_\tau}}(\theta')},\quad\forall \tau\geq \tau'
\end{align*}
it follows from \eqref{eq:Thm2_bound0} and \eqref{eq:thm2_B1}-\eqref{eq:TT3}  that, $\P$-a.s.,  for all $\tau\geq \tau'$ and $A\in\mathcal{X}$ we have
\begin{equation}\label{eq:var0}
\begin{split}
\Lambda^{\pi_{a_\tau}}_{(a_\tau+1):b_\tau}&(\mathcal{N}_{\epsilon_\tau}(\{\theta\})\times A) \leq \bar{P}^{\pi_{a_\tau,\setX}}_{(a_\tau+1):b_\tau}(\theta)\Big(\bar{p}_{(a_\tau+1):b_\tau}^{\pi_{a_\tau,\setX}}(A|\theta)+\frac{4}{\delta_\star} g(3\epsilon_{\tau})Z_{e_\tau}^{(1)}+Z_{e_\tau}^{(2)} \Big)
\end{split}
\end{equation}
and to complete the proof of the lemma it remains to show that
\begin{align}\label{eq:toShow_end}
g(3\epsilon_{\tau})Z_{e_\tau}^{(1)}+Z_{e_\tau}^{(2)}=\smallo_\P(1).
\end{align}
To this aim remark first that, under \ref{assume:Model_stationary}, for all $\tau\geq \tau'$ we have
\begin{align*}
&\sup_{\theta'\in\Theta}\log \bar{P}^{\bar{\chi}_{a_\tau}}_{(a_\tau+1):e_\tau}(\theta')\dist \sup_{\theta'\in\Theta}\log \bar{P}^{\bar{\chi}_{r}}_{(r+1):(e_\tau+r-a_\tau)}(\theta'),\\
&\sup_{\theta'\in\Theta}\log \bar{P}^{\bar{\chi}_{a_\tau}}_{(a_\tau+1):d_\tau}(\theta')\dist \sup_{\theta'\in\Theta}\log \bar{P}^{\bar{\chi}_{r}}_{(r+1):(d_\tau+r-a_\tau)}(\theta'),\\ &\inf_{\theta'\in\Theta}\log \bar{P}^{\underline{\chi}_{a_\tau}}_{(a_\tau+1):e_\tau}(\theta')\dist \inf_{\theta'\in\Theta}\log \bar{P}^{\underline{\chi}_{r}}_{(r+1):(e_\tau+r-a_\tau)}(\theta')\\
&W^{\bar{\chi}_{a_\tau}}_{(a_\tau+1):e_\tau}\dist W^{\bar{\chi}_{r}}_{(r+1):(e_\tau+r-a_\tau)}
\end{align*}
where $\underline{\chi}_r$  and $\bar{\chi}_r$  are as in  \ref{assume:Model_G_lower} and \ref{assume:Model_sup}, respectively. Since $\underline{\chi}_r$ and $\bar{\chi}_r$   are $(r,\delta_\star,\tilde{l},(\varphi_t)_{t\geq 1})$-consistent under \ref{assume:Model_G_lower}-\ref{assume:Model_sup},  and noting that 
\begin{align*}
\lim_{\tau\rightarrow\infty}\frac{b_\tau-a_\tau}{e_\tau-a_\tau}=\lim_{\tau\rightarrow\infty}\frac{b_\tau-a_\tau}{d_\tau-a_\tau}=1,
\end{align*}
it follows from Lemmas \ref{lemma::W}-\ref{lemma:Unif_convergence} that
\begin{align}
&\frac{1}{b_\tau-a_\tau}\sup_{\theta'\in\Theta}\log \bar{P}^{\bar{\chi}_{a_\tau}}_{(a_\tau+1):e_\tau}(\theta')\PP  \sup_{\theta'\in\Theta} \tilde{l}(\theta'),\label{eq:same_dist11}\\
&\frac{1}{b_\tau-a_\tau}\sup_{\theta'\in\Theta}\log \bar{P}^{\bar{\chi}_{a_\tau}}_{(a_\tau+1):d_\tau}(\theta')\PP  \sup_{\theta'\in\Theta} \tilde{l}(\theta'),\label{eq:same_dist12}\\
&\frac{1}{b_\tau-a_\tau}\inf_{\theta'\in\Theta}\log \bar{P}^{\underline{\chi}_{a_\tau}}_{(a_\tau+1):e_\tau}(\theta')\PP  \inf_{\theta'\in\Theta} \tilde{l}(\theta')\label{eq:same_dist13}\\
&\frac{1}{b_\tau-a_\tau}\log  W^{\bar{\chi}_{a_\tau}}_{(a_\tau+1):b_\tau}=\bigO_\P(1).\label{eq:same_dist}
\end{align}

Under \ref{assume:Model}, the function $\tilde{l}$ is continuous on the compact set $\Theta$, and thus we have   $\inf_{\theta'\in\Theta} \tilde{l}(\theta')>-\infty$ and $\sup_{\theta'\in\Theta} \tilde{l}(\theta')<\infty$. Together with \eqref{eq:same_dist11}-\eqref{eq:same_dist13}, this implies that
\begin{align}\label{eq:same_dist3}
\frac{1}{b_\tau-a_\tau}\log \frac{\sup_{\theta'\in\Theta}\bar{P}_{(a_\tau+1):e_\tau}^{\bar{\chi}_{a_\tau}}(\theta')}{\inf_{\theta'\in\Theta}\bar{P}_{(a_\tau+1):e_\tau}^{\underline{\chi}_{a_\tau}}(\theta')} =\bigO_\P(1),\quad \frac{1}{b_\tau-a_\tau}\log \frac{\sup_{\theta'\in\Theta}\bar{P}_{(a_\tau+1):d_\tau}^{\bar{\chi}_{a_\tau}}(\theta')}{\inf_{\theta'\in\Theta}\bar{P}_{(a_\tau+1):d_\tau}^{\underline{\chi}_{a_\tau}}(\theta')} =\bigO_\P(1) 
\end{align}
and since $\log \tilde{\xi}_{a_\tau}=\bigO_\P(1)$ by Lemma \ref{lemma:gamma}  it follows from \eqref{eq:same_dist} and \eqref{eq:same_dist3} that 
\begin{align}\label{eq:Z_b1}
\frac{1}{b_\tau-a_\tau}\log Z_{e_\tau}^{(1)}=\bigO_\P(1).
\end{align} 
Noting that for all $\epsilon\in(0,\infty)$ we have
\begin{align*}
\P\Big( g(3\epsilon_{\tau})Z_{e_\tau}^{(1)}\geq \epsilon\Big)=\P\Big(\frac{1}{b_\tau-a_\tau}\log Z_{e_\tau}^{(1)}\geq -\frac{1}{b_\tau-a_\tau}\log(1/\epsilon)-\frac{1}{b_\tau-a_\tau}\log g(3\epsilon_\tau)\Big)
\end{align*} 
and recalling that, by assumption $\lim_{\tau\rightarrow\infty}(b_\tau-a_\tau)^{-1}\log g(3\epsilon_\tau)=-\infty$, it follows from \eqref{eq:Z_b1} that 
\begin{align}\label{eq:var2}
 g(3\epsilon_{\tau})Z_{e_\tau}^{(1)}=\smallo_\P(1).
\end{align}

On the other hand,  noting that under \ref{assume:Model_stationary} we have
\begin{align*}
D_{(d_\tau+1):e_\tau}\dist D_{1:(e_\tau-d_\tau)}=D_{1:r},\quad\forall \tau\geq \tau'
\end{align*}
it follows that, under \ref{assume:G_bounded}, $\log D_{(d_\tau+1):e_\tau}=\bigO_\P(1)$.
Then, using \eqref{eq:same_dist3} and the fact $\log \tilde{\xi}_{a_\tau}=\bigO_\P(1)$ while $\lim_{\tau\rightarrow\infty} c_\tau=\infty$ by assumption, we readily obtain that $Z^{(2)}_{e_\tau}=\smallo_\P(1)$. Together with \eqref{eq:var2}, this shows \eqref{eq:toShow_end} and the proof of the lemma is complete.

\end{proof}

\subsubsection{Proof of Lemma \ref{lemma:lower_X}\label{p-lemma:lower_X}}

\begin{proof}

Let   $\big( (\vartheta_{a_\tau}, \tilde{X}_{a_\tau}),\dots,(\vartheta_{b_\tau},   \tilde{X}_{b_\tau})\big)$ and $\tilde{f}_\tau$ be as defined in the proof of Lemma \ref{lemma:upper_X} for all $\tau\geq 1$. In addition,  let $(\varphi_t)_{t\geq 1}$, $g(\cdot)$ and $\delta_\star$ be as in \ref{assume:Model}, and let $\tau'\in\mathbb{N}$ be such that $g(3\epsilon_{\tau})\leq\delta_\star/4$ for all $\tau\geq \tau'$.

To prove the lemma remark first that for all $\tau\geq 1$ we have
\begin{equation}\label{eq:lower_thm2_0}
\begin{split}
\E_{(a_\tau+1):b_\tau}^{\pi_{a_\tau}}&\Big[\ind_{\mathcal{N}_{\epsilon_\tau}(\{\theta\})}(\vartheta_{a_\tau})\tilde{f}_{\tau}(\vartheta_{a_\tau},\dots,\vartheta_{b_\tau})\Big]\\
&=\pi_{a_\tau,\Theta}\big(\mathcal{N}_{\epsilon_\tau}(\{\theta\})\big)\P\Big(\sum_{i=a_\tau+1}^s  U_i\in B_{\epsilon_{\tau}}(0),\,\forall   s\in\{a_\tau+1,\dots,  b_\tau\}\,\big|\, \F_{b_\tau}\Big)\\
&\geq \pi_{a_\tau,\Theta}\big(\mathcal{N}_{\epsilon_\tau}(\{\theta\})\big)\exp\big(-(b_\tau-a_\tau)/c_\tau\big):=K_{\theta,\tau}
\end{split}
\end{equation}
where the  inequality holds under the assumption of the lemma and uses \eqref{eq:Bound_log}.

Then, under \ref{assume:Model_smooth} and using \eqref{eq:lower_thm2_0}, for all $\tau\geq \tau'$  we $\P$-a.s.~have 
\begin{equation}\label{eq:lower_thm2}
\begin{split}
&\Lambda^{\pi_{a_\tau}}_{(a_\tau+1):b_\tau}\big(\mathcal{N}_{2\epsilon_\tau}(\{\theta\})\times \setX\big)\\
&\geq \E_{(a_\tau+1):b_\tau}^{\pi_{a_\tau}}\Big[\ind_{\mathcal{N}_{2\epsilon_\tau}(\{\theta\})}(\vartheta_{b_\tau}) \tilde{f}_{\tau}(\vartheta_{a_\tau},\dots,\vartheta_{b_\tau}) \prod_{s=a_\tau+1}^{b_\tau} G_{s,\vartheta_s}(\tilde{X}_{s-1},\tilde{X}_s)\Big]\\
&\geq \E_{(a_\tau+1):b_\tau}^{\pi_{a_\tau}}\Big[\ind_{\mathcal{N}_{\epsilon_\tau}(\{\theta\})}(\vartheta_{a_\tau})  \tilde{f}_{\tau}(\vartheta_{a_\tau},\dots,\vartheta_{b_\tau}) \prod_{s=a_\tau+1}^{b_\tau} G_{s,\vartheta_s}(\tilde{X}_{s-1},\tilde{X}_s)\Big]\\
&\geq K_{\theta,\tau}   \int_{\setX^{b_\tau-a_\tau+1}} \exp\Big(-g(3\epsilon_\tau)\sum_{s=a_\tau+1}^{b_\tau} \varphi_s(x_{s-1},x_s)\Big)\pi_{a_\tau,\setX}(\dd x_{a_\tau})\prod_{s=a_\tau+1}^{b_\tau} Q_{s,\theta}(x_{s-1},\dd x_s)\\
&\geq K_{\theta,\tau} \bar{P}^{\pi_{a_\tau,\setX}}_{(a_\tau+1):b_\tau}(\theta)\tilde{W}^{\pi_{a_\tau,\setX}}_{(a_\tau+1):b_\tau}
\end{split}
\end{equation}
with 
\begin{align*}
\tilde{W}^{\pi_{a_\tau,\setX}}_{(a_\tau+1):b_\tau}=\inf_{\theta'\in\Theta}\int_{\setX^{b_\tau-a_\tau+1}} \exp\Big(-g(3\epsilon_\tau)\sum_{s=a_\tau+1}^{b_\tau} \varphi_s(x_{s-1},x_s)\Big)\frac{\pi_{a_\tau,\setX}(\dd x_{a_\tau})\prod_{s=a_\tau+1}^{b_\tau} Q_{s,\theta'}(x_{s-1},\dd x_s)}{\bar{P}^{\pi_{a_\tau,\setX}}_{(a_\tau+1):b_\tau}(\theta')}.
\end{align*}

Let $e_\tau\in\{sr,\,s\in\mathbb{N}\}$ be such that $e_\tau-r<b_\tau\leq e_\tau$ and note that
\begin{align}\label{eq:WT}
1\geq\tilde{W}^{\pi_{a_\tau,\setX}}_{(a_\tau+1):b_\tau}\geq \tilde{W}^{\pi_{a_\tau,\setX}}_{(a_\tau+1):e_\tau},\quad\forall\tau\geq 1.
\end{align}

To proceed further, for all integers $1\leq t_1\leq t_2$ and  $\eta\in\mathcal{P}(\setX)$ we let $W^\eta_{t_1:t_2}$ be as defined in \eqref{eq:W_def}. Then, using Jensen's inequality, for all $\tau\geq \tau'$ we have (see the proof of Lemma \ref{lemma::lower_general} for details)
\begin{align}\label{eq:WWW}
1\geq \tilde{W}^{\pi_{a_\tau,\setX}}_{(a_\tau+1):e_\tau}\geq \exp\Big(-g(3\epsilon_\tau)^{1/2}\log W^{\pi_{a_\tau,\setX}}_{(a_\tau+1):e_\tau}\Big)
\end{align}
where, by \eqref{eq:W_upp} and \eqref{eq:Z_b1},
\begin{align}\label{eq:WWWW}
W^{\pi_{a_\tau,\setX}}_{(a_\tau+1):e_\tau}\leq Z_{e_\tau}^{(1)},\quad\forall \tau\geq \tau'
\end{align}
for some sequence of random variables $(Z_{e_\tau}^{(1)})_{\tau\geq 1}$ depending on  $(\mu_t)_{t\geq 0}$ only through $(W^*_t)_{t\geq 1}$ and such that
\begin{align}\label{eq:Z_b11}
\frac{1}{b_\tau-a_\tau}\log Z_{e_\tau}^{(1)}=\bigO_\P(1).
\end{align}

Recalling   that $\log(x)\geq (x-1)/x\geq -1/x$ for all $x\in(0,\infty)$, it follows that under the assumptions of the lemma we have
\begin{align*}
\limsup_{\tau\rightarrow\infty}-\frac{1}{(b_\tau-a_\tau)g(3\epsilon_\tau)^{1/2}}\leq \limsup_{\tau\rightarrow\infty}\frac{\frac{1}{2}\log g(3\epsilon_\tau)}{(b_\tau-a_\tau)}=-\infty
\end{align*}
and thus, by \eqref{eq:Z_b11}, for all $\epsilon\in(0,\infty)$ we have
\begin{align*}
\limsup_{\tau\rightarrow\infty}\P\Big(g(3\epsilon_\tau)^{1/2}|\log Z_{e_\tau}^{(1)}|\geq \epsilon\Big)&=\limsup_{\tau\rightarrow\infty}\P\Big(\Big|\frac{1}{b_\tau-a_\tau} \log Z_{e_\tau}^{(1)}\Big|\geq \frac{\epsilon}{(b_\tau-a_\tau)g(3\epsilon_\tau)}\Big)=0.
\end{align*}
This shows that  $g(3\epsilon_\tau)^{1/2} \log Z_{e_\tau}^{(1)}=\smallo_\P(1)$ and thus, by \eqref{eq:WWW} and \eqref{eq:WWWW}, we have $\tilde{W}^{\pi_{a_\tau,\setX}}_{(a_\tau+1):e_\tau}\PP 1$. Together with \eqref{eq:lower_thm2} and \eqref{eq:WT}, this shows the result of  the lemma.

\end{proof}

\section{Proofs of the results in Section \ref{sec:SSM}\label{p-proof_SSM}}

Proposition \ref{prop:conv_SSM_0} is a direct consequence of Proposition \ref{prop:conv_SSM} stated below, and  in this section we  prove  Theorems \ref{thm:SSM} and \ref{thm:SSM2}.



Throughout we assume, as in Section \ref{sec:SSM},     that the sequence $(Z_t)_{t\geq 1}$ is a stationary and ergodic process, where $Z_t=(Y_{(t-1)\tau+1},\dots,Y_{t\tau})$ for all $t\geq 1$. In addition,   as in Section \ref{sec:SSM}, we assume  that there exist  measurable functions $\{ (f'_{i,\theta})_{\theta\in\Theta}\}_{i=1}^\tau$ and  $\{ (m'_{i+1,\theta})_{\theta\in\Theta}\}_{i=1}^\tau$, and a $\sigma$-finite measure $\lambda(\dd x)$ on $(\setX,\mathcal{X})$, such that
 \begin{align*}
 f_{t,\theta}(y|x')=f'_{i_t,\theta}(y|x'),\quad M_{t+1,\theta}(x',\dd x)=m'_{i_t+1,\theta}(x|x')\lambda(\dd x),\quad\forall (\theta,x',y)\in\Theta\times\setX\times\setY,\quad\forall t\geq 1
 \end{align*}
 with $i_t=t -\tau\lfloor (t-1)/\tau\rfloor$ for all $t\geq 1$. Finally, we let $M'_{1,\theta}(x',\dd x)=M'_{\tau+1,\theta}(x',\dd x)$ for all $(\theta,x')\in\Theta\times\setX$.

\subsection{Additional notation}
Let $\eta\in\mathcal{P}(\setX)$ and $\theta\in\Theta$. Then,  we let
\begin{align*}
l^\eta_{t}(\theta)=\frac{1}{t}\log\int_{\setX^{t+1}}\eta(\dd x_{1})  \prod_{s=1}^{t}f'_{i_s,\theta}(Y_s|x_s)M'_{i_{s}+1,\theta}( x_{s},\dd x_{s+1}) ,\quad\forall t\in\mathbb{N},
\end{align*}
and for all integers $1\leq t_1<t_2$ we let
\begin{align*}
\tilde{p}_{(t_1+1):t_2}^{\eta} (A|\theta)=\frac{\int_{\setX^{t_2-t_1+1}}\ind_A(x_{t_2})\eta(\dd x_{t_1})\prod_{s=t_1+1}^{t_2} f'_{i_s,\theta}(Y_{s}|x_{s})M'_{i_{s-1}+1,\theta}( x_{s-1},\dd x_s)  }{\int_{\setX^{t_2-t_1+1}} \eta(\dd x_{t_1})\prod_{s=t_1+1}^{t_2} f'_{i_s,\theta}(Y_{s}|x_{s})M'_{i_{s-1}+1,\theta}(x_{s-1},\dd x_{s}) },\quad\forall A\in\mathcal{X} 
\end{align*}
and
\begin{align*}
\bar{p}^\eta_{t_1,\theta}(A|Y_{1:t_1})=\frac{\int_{\setX^{t_1+1}}\ind_A(x_{t_1}) \eta(\dd x_1) \prod_{s=1}^{t_1} f'_{i_s,\theta}(Y_s|x_s)M'_{i_s+1,\theta}(x_{s}, \dd x_{s+1})}{\int_{\setX^{t_1+1}}  \eta(\dd x_1) \prod_{s=1}^{t_1} f'_{i_s,\theta}(Y_s|x_s)M'_{i_s+1,\theta}(x_{s}, \dd x_{s+1})},\quad \forall A\in\mathcal{X}. 
\end{align*}
Finally, in what follows, for all $\mu\in\mathcal{P}(\setR)$ the Markov kernel $K_{\mu|\Theta}$ acting on $(\Theta,\mathcal{T})$ is as defined in Section   \ref{sub:SO-FK_model}.

\subsection{Preliminary results}

\begin{lemma}\label{lemma:tech_ssm}
Let $1\leq t_1< t_2$ be two integers, $\eta\in\mathcal{P}(\setX)$ and
\begin{align*} 
D_{(t_1+1):t_2}=\prod_{s=t_1+1}^{t_2}  \sup_{(\theta,x)\in\Theta\times\setX} f'_{i_s,\theta}(Y_s|x).
\end{align*}
 Then, for all $\theta\in\Theta$ we have $l^\eta_{t_1}(\theta)\geq D_{(t_1+1):t_2}^{-1}  l^\eta_{t_2}(\theta)$.
\end{lemma}
\begin{proof}
The proof of the lemma is similar to that of Lemma \ref{lemma:tech} and is therefore omitted to save space.
\end{proof}

The following proposition shows that, under \ref{assumeSSM:K_set}-\ref{assumeSSM:G}, the   function   $l^\eta_{t}(\cdot)$ converges to some function $l(\cdot)$ as $t\rightarrow\infty$.

\begin{proposition}\label{prop:conv_SSM}
Assume that \ref{assumeSSM:K_set}-\ref{assumeSSM:G} hold. Then, there exists a function $l:\Theta\rightarrow\R$ such that $\inf_{\theta\in\Theta}\P\big(\lim_{t\rightarrow\infty}l^{\eta_\theta}_{t}(\theta)=l(\theta)\big)=1$ for any collection $\{\eta_\theta,\,\theta\in\Theta\}$ of random probability  measures on $(\setX,\mathcal{X})$ for which $\inf_{\theta\in\Theta}\P(\eta_\theta(E)>0)=1$  for some compact set $E\in\mathcal{X}$ such that $\lambda(E)>0$.
\end{proposition}
\begin{proof}
See Section \ref{p-prop:conv_SSM}.
\end{proof}

The following proposition is a result on   the forgetting property of the SSM.

\begin{proposition}\label{prop:forget}
Assume that  \ref{assumeSSM:K_set}-\ref{assumeSSM:G} hold. Let $(s_t)_{t\geq 1}$ be a sequence in $\{s\tau,\,s\in\mathbb{N}\}$, such that $\limsup_{t\rightarrow}s_t/t<1$  and such that $\lim_{t\rightarrow\infty}(t-s_t)=\infty$, and let $(\eta_{s_t})_{t\geq 1}$ and $(\chi_{s_t})_{t\geq 1}$ be two sequences of random probability measures on $(\setX,\mathcal{X})$. Assume that there exists a   compact set $E\in\mathcal{X}$ such that  the following holds:
\begin{align*}
&(i) \,\text{ $\exists$ r.v. $(W^-_{s_t})_{t\geq 0}$ s.t. $\inf_{t\geq 1}\P\Big(\eta_{s_t}(E)\geq W^-_{s_t},\,\chi_{s_t}(E)\geq W^-_{s_t}\Big)=1$  and s.t. $\log W^-_{s_t}=\bigO_\P(1)$.}\\
&(ii) \text{ $\exists$   $\eta\in\mathcal{P}(\setX)$ such that $\log \bar{p}_{t,\theta}^\eta(E|Y_{1:t})=\bigO_\P(1)$.}
\end{align*} 
Then, there exists a sequence $(\xi_t^-)_{t\geq 1}$ of $[0,1]$-valued random variables, depending on $(\eta_{s_t})_{t\geq 1}$ and $(\chi_{s_t})_{t\geq 1}$ only through  $(W_{s_t}^-)_{t\geq 1}$, such that  $\xi^-_{t}=\smallo_\P(1)$ and such that
\begin{align*}
\big\|\tilde{p}_{(s_t+1):t}^{\eta_{s_t}}(\dd x|\theta)-\tilde{p}_{(s_t+1):t}^{\chi_{s_t}}(\dd x|\theta)\big\|_{\mathrm{TV}}\leq \xi^-_t,\quad\forall t\geq 1,\quad \P-a.s.
\end{align*}
\end{proposition}
\begin{proof}
See Section \ref{p-prop:forget}.
\end{proof}

The following proposition shows that    SSM \eqref{eq:SSM}  can be seen as two different Feynman-Kac models   in  a random environment.

\begin{proposition}\label{prop:Def_FK_SSM}

Let $\chi'\in \mathcal{P}(\setX)$ and, for  a given Feynman-Kac model  in  a random environment,  let  $\bar{p}^\eta_t(\dd x|\theta)$ be  as defined in  Section \ref{sub:FK_model} for all $\eta\in\mathcal{P}(\Theta)$ and all $t\geq 1$. 
\begin{enumerate}[label=(\roman*)]
\item\label{FK1} Consider   SSM \eqref{eq:SSM} with $\chi_{\theta}=\chi'$ for all $\theta\in\Theta$ and the Feynman-Kac model  of Section \ref{sub:FK_model} with   $\chi=\chi'$  and with $( (\tilde{G}_t,\tilde{m}_t))_{t\geq 1}$    such that, for all $(t,\omega,\theta,x,x')\in\mathbb{N}\times\Omega\times \Theta\times\setX^2$,
\begin{align*}
\tilde{G}_{t}(\omega,\theta, x',x)=f'_{i_t,\theta}(Y^\omega_t|x'),\,\, \tilde{m}_{t}(\omega,\theta,x',x)=m'_{i_t+1,\theta}(x|x').
\end{align*}
Then, for all $t\geq 1$ and $\theta\in\Theta$, we have $\bar{p}^\chi_t(\dd x|\theta)=\int_\setX M_{i_{t+1}+1,\theta}(x_t,\dd x)\bar{p}_{t,\theta}(\dd x_t|Y_{1:t})$.

\item\label{FK2} Consider   SSM \eqref{eq:SSM}  with $\chi_{\theta}(\dd x)=\int_\setX M_{\tau+1,\theta}(x',\dd x)\chi'(\dd x')$ for all $\theta\in\Theta$ and  the Feynman-Kac model  of Section \ref{sub:FK_model} with   $\chi=\chi'$  and with $( (\tilde{G}_t,\tilde{m}_t))_{t\geq 1}$    such that, for all $(t,\omega,\theta,x,x')\in\mathbb{N}\times\Omega\times \Theta\times\setX^2$,
\begin{align*}
\tilde{G}_{t}(\omega,\theta, x',x)=f'_{i_t,\theta}(Y^\omega_t|x),\,\, \tilde{m}_{t}(\omega,\theta,x',x)=m'_{i_{t},\theta}(x|x').
\end{align*}
Then, for all $t\geq 1$ and $\theta\in\Theta$, we have $\bar{p}^\chi_t(\dd x|\theta)=\bar{p}_{t,\theta}(\dd x_t|Y_{1:t})$.

\end{enumerate}
\end{proposition}
\begin{proof}
The result of the proposition directly follows from the definition of $( (\tilde{G}_t,\tilde{m}_t))_{t\geq 1}$.
\end{proof}

The next proposition shows that the two  Feynman-Kac models   in  a random environment defined in Proposition \ref{prop:Def_FK_SSM} verify the assumptions of  Section \ref{sub:assumption}.

\begin{proposition}\label{prop:A1}
Assume that \ref{assumeSSM:K_set}-\ref{assumeSSM:G} hold. Let   $l:\Theta\rightarrow\R$ be as in Proposition \ref{prop:conv_SSM}, $\chi'\in\mathcal{P}(\setX)$ and   assume that  $\Theta_\star:=\argmax_{\theta\in\Theta}l(\theta)$ is such that $\mathcal{N}_\delta(\Theta_\star)\subseteq\Theta$ for some $\delta\in(0,1)$. Then,     
\begin{enumerate}[label=(\roman*)]
\item Under \ref{assumeSSMB:smooth},  if $\chi'(E)>0$ for all  compact set $E\in\mathcal{X}$ such that $\lambda(E)>0$ and if $\chi_{\theta}=\chi'$ for all $\theta\in\Theta$, then Assumptions  \ref{assume:Model}-\ref{assume:G_l_E} hold for $\chi=\chi'$ and for $( (\tilde{G}_t,\tilde{m}_t))_{t\geq 1}$  as defined in  Proposition \ref{prop:Def_FK_SSM}, part \ref{FK1}.

\item  Under \ref{assumeSSM:smooth}, if $\chi_{\theta}(\dd x)=\int_\setX M_{\tau+1,\theta}(x',\dd x)\chi'(\dd x')$ for all $\theta\in\Theta$, then Assumptions \ref{assume:Model}-\ref{assume:G_l_E} hold for    $\chi=\chi'$ and for $( (\tilde{G}_t,\tilde{m}_t))_{t\geq 1}$ as defined in  Proposition \ref{prop:Def_FK_SSM}, part \ref{FK2}.

\end{enumerate}
In both cases, \ref{assume:Model} holds with $\tilde{l}=l$, $r=\tau$ and for any  compact set $D\in\mathcal{X}$ such that $\lambda(D)>0$.
\end{proposition}
\begin{proof}
See Section \ref{p-prop:A1}.
\end{proof}

The next  proposition  shows that     SO-SSM  \eqref{eq:SO-SSM_new}, respectively   SO-SSM \eqref{eq:SO-SSM_new2}, can be seen as a self-organizing version of the Feynman-Kac model   defined in part \ref{FK1}, respectively in part \ref{FK2}, of  Proposition \ref{prop:Def_FK_SSM}.

\begin{proposition}\label{prop:SO-FK_SSM}
 Let $\chi'\in\mathcal{P}(\setX)$,  $\mu_0'\in\mathcal{P}(\Theta)$, $\mu_1=\delta_{\{0\}}$,   $(\mu'_t)_{t\geq 2}$ be a sequence  of random probability measures on $(\setR,\mathcal{R})$ and, for a   given Feynman-Kac model  in  a random environment,  let  $\pi_{t,\Theta}(\dd\theta_t)$ and $\pi_{t,\setX}(\dd x_t)$ be  as defined in  Section \ref{sub:SO-FK_model} for all  $t\geq 1$. 
\begin{enumerate}[label=(\roman*)]
\item   Consider   SO-SSM \eqref{eq:SO-SSM_new}  with $\{\chi_\theta,\,\theta\in\Theta\}$  as in   Proposition \ref{prop:Def_FK_SSM} (part \ref{FK1}),    $\pi_1=\mu_0'$ and with $K_t=K_{\mu'_t|\Theta}$ for all $t\geq 2$, and consider the   self-organizing Feynman-Kac model of  Section \ref{sub:SO-FK_model}  with $\chi=\chi'$, $\mu_0=\mu_0'$,  $(\mu_t)_{t\geq 1}=(\mu'_t)_{t\geq 1}$ and with $( (\tilde{G}_t,\tilde{m}_t))_{t\geq 1}$  as in Proposition \ref{prop:Def_FK_SSM} (part \ref{FK1}). Then, for all $t\geq 1$, we have $\pi_{t,\Theta}(\dd\theta_t)=p_{t,\Theta}(\dd\theta_t|Y_{1:t})$ while $\pi_{t,\setX}(\dd x_t)$   is the conditional distribution of $X_{t+1}$ given $Y_{1:t}$ under SO-SSM \eqref{eq:SO-SSM_new}.

\item\label{PM}    Consider   SO-SSM \eqref{eq:SO-SSM_new2} with $\{\chi_\theta,\,\theta\in\Theta\}$   as in   Proposition \ref{prop:Def_FK_SSM} (part \ref{FK1}),   $\pi_1=\mu_0'$ and with $K_t=K_{\mu'_t|\Theta}$ for all $t\geq 2$, and consider   the self-organizing Feynman-Kac model of  Section \ref{sub:SO-FK_model}  with $\chi=\chi'$, $\mu_0=\mu_0'$,    $(\mu_t)_{t\geq 1}=(\mu'_t)_{t\geq 1}$ and with $( (\tilde{G}_t,\tilde{m}_t))_{t\geq 1}$  as in Proposition \ref{prop:Def_FK_SSM} (part \ref{FK2}). Then, for all $t\geq 1$, we have $\pi_{t,\Theta}(\dd\theta_t)=p_{t,\Theta}(\dd\theta_t|Y_{1:t})$ and $\pi_{t,\setX}(\dd x_t)=p_{t,\setX}(\dd x_t|Y_{1:t})$.
\end{enumerate} 
 
\end{proposition}
\begin{proof}
The result of the proposition  directly follows from Proposition \ref{prop:Def_FK_SSM} and from the definition of  $(\mu_t)_{t\geq 1}$.
\end{proof}

\begin{remark}\label{rem:dist}
If, in Proposition \ref {prop:SO-FK_SSM}, the sequence $(\mu'_t)_{t\geq 1}$ is such that Conditions \ref{condition:Inf_mu}-\ref{condition:mu_extra} hold, or such that  Conditions \ref{condition:Inf_mu}-\ref{condition:Inf_K}, \ref{condition:mu_seq2} and \ref{condition:mu_extra} hold, then this is also the case for the sequence $(\mu_t)_{t\geq 1}$ defined in part \ref{PM} of the proposition. In particular, if $(\mu'_t)_{t\geq 1}$ is such that \ref{condition:mu_seq} holds for some sequence $(t_p)_{p\geq 1}$ then  the resulting sequence  $(\mu_t)_{t\geq 1}$ is such that \ref{condition:mu_seq} also holds for the  same  sequence $(t_p)_{p\geq 1}$.
\end{remark}

\subsection{Proof of Theorem \ref{thm:SSM}\label{app:proof_TSSM}}
\begin{proof}

In addition to prove Theorem \ref{thm:SSM} we also prove Theorem \ref{thm:SSMB} and proof of Theorem \ref{thm:dis} in what follows.

The last part of Theorem \ref{thm:SSM}  is a direct consequence of Proposition \ref{prop:stability} and of Assumptions \ref{assumeSSM:D_set}-\ref{assumeSSM:G}.

The result in the first part of Theorem \ref{thm:SSM} and   the result of  Theorem \ref{thm:dis}  follow  from Propositions \ref{prop:forget}-\ref{prop:SO-FK_SSM}, Remark \ref{rem:dist},  Theorems \ref{thm:main}-\ref{thm:pred} and from the results in Section \ref{p-sub:noise}, upon noting (i) that for each definition of $(K_t)_{t\geq 2}$ considered in Theorem \ref{thm:SSM} or in     Theorem \ref{thm:dis}  there exists a sequence $(\mu_t')_{t\geq 2}$ of  probability measures on $(\setR,\mathcal{R})$ such that $K_t=K_{\mu'_2|\Theta}$ for all $t\geq 2$, (ii) that each of these sequences   $(\mu_t')_{t\geq 2}$ is constructed as in Section \ref{p-sub:noise}, and (iii) that we have, for all integers $0\leq t_1<t_2$ and all $\theta\in\Theta$,
\begin{align*}
\bar{p}_{t_2,\theta}(\dd x_{t_2}|Y_{1:t_2})=\tilde{p}_{(t_1+1):t_2}^{\bar{p}_{t_1,\theta}(\dd x_{t_1}|Y_{1:t_1})}(\dd x_{t_2}| \theta).
\end{align*}

To prove    Theorem \ref{thm:SSMB}  
let $p_{t+1|t,\setX}(\dd x_{t+1}|Y_{1:t})$  denote the conditional distribution of $X_{t+1}$ given $Y_{1:t}$ under SO-SSM \eqref{eq:SO-SSM_new}, and    for all $\theta\in\Theta$ let $\bar{p}_{t+1|t,\theta}(\dd x_{t+1}|Y_{1:t})$ denote the conditional distribution of $X_{t+1}$ given $Y_{1:t}$ under   model \eqref{eq:SSM}. Then, from the argument used above to establish the  first part of Theorem \ref{thm:SSM}, we easily obtain that, under the assumptions of the Theorem \ref{thm:SSMB},
\begin{align*}
 \int_\Theta \theta_t\, p_{t,\Theta}(\dd \theta_t |Y_{1:t})\PP \theta_\star,\quad \|p_{t+1|t,\setX}(\dd x_{t+1}|Y_{1:t})-\bar{p}_{t+1|t,\theta_\star}(\dd x_{t+1}|Y_{1:t})\|_{\mathrm{TV}}\PP 0.
\end{align*}
Then, the result  of Theorem \ref{thm:SSMB}   follows, upon noting that for any $\eta\in\mathcal{P}(\Theta)$ we have
\begin{equation}\label{eq:Split_SSM}
\begin{split}
\Big\|\frac{f'_{i_{t+1},\theta}(Y_{t+1}|x_{t+1}) \eta(\dd x_{t+1})}{\int_{\setX}f'_{i_{t+1},\theta}(Y_{t+1}|x_{t+1}) \eta(\dd x_{t+1})}&-\bar{p}_{t,\theta}(\dd x_{t+1}|Y_{1:t+1})  \Big\|_{\mathrm{TV}}\\
&\leq \frac{2\|\eta-\bar{p}_{t+1|t,\theta}(\dd x_{t+1}|Y_{1:t})\|_{\mathrm{TV}}}{\int_\setX   f'_{i_{t+1},\theta}(Y_{t+1}|x_{t+1})\bar{p}_{t+1|t,\theta}(\dd x_{t+1}|Y_{1:t})}
\end{split}
\end{equation}
where, under the assumptions of the theorem, $\log^-\int_\setX   f'_{i_{t+1},\theta}(Y_{t+1}|x_{t+1})\bar{p}_{t+1|t,\theta}(\dd x_{t+1}|Y_{1:t})=\bigO_\P(1)$. The proof of the theorems is complete.
\end{proof}

\subsection{Proof of Theorem \ref{thm:SSM2}}
\begin{proof}
Algorithm   \ref{algo:online_2} is a particle filter algorithm deployed on     SO-SSM \eqref{eq:SO-SSM_new2} where, for a given number $N$ of particles,  $K_t=K_t^N$ for all $t\geq 2$ and with the sequence   $(K_t^N)_{t\geq 2}$  recursively defined in the course the algorithm. For  all $(N,t)\in\mathbb{N}^2$  let $p_{t,N,\Theta}(\dd \theta_t |Y_{1:t})$  and $p_{t,N,\setX}(\dd x_t |Y_{1:t})$  denote the filtering distribution of $\theta_t$  and the   filtering distribution of $X_t$ under SO-SSM \eqref{eq:SO-SSM_new2}, where the deterministic sequence of Markov kernels $(K_t)_{t\geq 2}$ is replaced by the sequence of random Markov kernels $(K_t^N)_{t\geq 2}$. It is easily checked that for all $N\in\mathbb{N}$  there exists a sequence $(\mu^N_t)_{t\geq 2}$ of random probability measures on $(\setR,\mathcal{R})$ as  defined  in Section \ref{p-sub:noise} and such that $K^N_t=K_{\mu^N_t|\Theta}$ for all $t\geq 2$.  Consequently, by Propositions \ref{prop:forget}-\ref{prop:SO-FK_SSM} and Theorems \ref{thm:main}-\ref{thm:pred} (see also  Remark \ref{rem:dist}), there exists a sequence $(\xi_t)_{t\geq 1}$ of $(0,\infty)$-valued random variables, independent of $N$, such that $\xi_t=\smallo_\P(1)$ and such that for all $N\in\mathbb{N}$ we have, $\P$-a.s.,
\begin{align}\label{eq:con_p}
\Big| \int_\Theta \theta_t\, p_{t,N,\Theta}(\dd \theta_t |Y_{1:t})-\theta_\star\Big|\leq \xi_t,\quad \big\|p_{t,N,\setX}(\dd x_t|Y_{1:t})-\bar{p}_{t,\theta_\star}(\dd x_t|Y_{1:t})\big\|_{\mathrm{TV}}\leq \xi_t,\quad\,\, t\geq 1.
\end{align}

 Then, under   \ref{assumeSSM:D_set}-\ref{assumeSSM:G},  and using Proposition \ref{prop:stability} and    the fact that  by assumption we have $\inf_{(\theta,x)\in\Theta\times\setX}M_{t+1,\theta}(x,E')>0$ for some compact set $E'\in\mathcal{X}$ and all $t\geq 1$, the conclusion of the theorem  follows directly from \eqref{eq:con_p} and from    the calculations used in \citetsup[][Section 11.2.2]{chopin2020introduction2} to study the $L_2$-convergence of particle filter algorithms.

\end{proof}

\subsection{Proof of the preliminary results}

\subsubsection{Proof of Proposition \ref{prop:conv_SSM}\label{p-prop:conv_SSM}}

\begin{proof}
Let $E\in\mathcal{X}$ be a compact set such that $\lambda(E)>0$. Then, under \ref{assumeSSM:K_set}-\ref{assumeSSM:G} and the additional assumptions (a) that  $(Y_t)_{t\geq 1}$ is a stationary and ergodic process and (b) that for all $\theta\in\Theta$ we have $M'_{i+1,\theta}=M'_{2,\theta}$ and $f'_{i,\theta}=f'_{1,\theta}$ for all $i\in\{1,\dots,\tau\}$, it follows from \citetsup[][Proposition 10, part (iii)]{Douc_MLE2} that there exists a function $l':\Theta\rightarrow\R$ such that $\P\big(\lim_{t\rightarrow\infty}l^{\,\tilde{\eta}}_{t\tau}(\theta)=l'(\theta)\big)=1$ for  all $\theta\in\Theta$ and all $\tilde{\eta}\in\mathcal{P}(\setX)$ such that $\tilde{\eta}(E)>0$.

A careful inspection of the proofs in  \citetsup{Douc_MLE2} reveals that the following more general result actually holds:  Under \ref{assumeSSM:K_set}-\ref{assumeSSM:G}  and the above two additional assumptions (a) and (b),   there exists a function $l':\Theta\rightarrow\R$ such that, for any collection $\{\tilde{\eta}_\theta,\,\theta\in\Theta\}$ of random probability measures on $(\setX,\mathcal{X})$ such that $\P(\tilde{\eta}_\theta(E)>0)=1$ for all $\Theta\in\Theta$,  we have
\begin{align}\label{eq:conv_1_lik}
\lim_{t\rightarrow\infty}l^{\,\tilde{\eta}_\theta}_{t\tau}(\theta)=l'(\theta),\quad\P-a.s.,\quad\forall\theta\in\Theta.
\end{align}

Another careful inspection of the proofs in  \citetsup{Douc_MLE2}  reveals that, under  \ref{assumeSSM:K_set}-\ref{assumeSSM:G}, the assumption that $(Z_t)_{t\geq 1}$  is a stationary and ergodic process and the assumption  that $(f_{t,\theta},M_{t+1,\theta})=(f'_{i_t,\theta}, M'_{i_t+1,\theta})$ for all $(t,\theta)\in\mathbb{N}\times\Theta$, the above two additional assumptions (a) and (b) are not needed for \eqref{eq:conv_1_lik} to hold.

To prove the result of the proposition we let  $\{\eta_\theta,\,\theta\in\Theta\}$ be such that $\inf_{\theta\in\Theta}\P(\eta_\theta(E)>0)=1$. Then, by \eqref{eq:conv_1_lik}, we have 
\begin{align}\label{eq:conv_eta}
 \lim_{t\rightarrow\infty}l_{t\tau}^{\eta_\theta}(\theta)=l'(\theta),\quad\P-a.s., \quad \forall\theta\in\Theta
\end{align}
which proves that the result of the proposition holds with $l(\cdot)=l'(\cdot)$ when $\tau=1$.

To complete the proof of the proposition assume that $\tau>1$ and for all integers $1\leq t_1\leq t_2$ let
\begin{align*}
D_{t_1:t_2}=\prod_{s=t_1}^{t_2}\Big(1\vee \sup_{(\theta,x)\in\Theta\times\setX} f'_{i_s,\theta}(Y_s|x)\Big).
\end{align*}
In addition, let $(s_t)_{t\geq 1}$ be  a  strictly increasing sequence in $\mathbb{N}\setminus \{s\tau,\,s\in\mathbb{N}\}$ such that $s_{1}>\tau$, and  for all $t\geq 1$   let $k_t\in\mathbb{N}_0$   be such that $k_t\tau<s_t<(k_t+1)\tau$.

Then, for all $t\geq 1$ and $\theta\in\Theta$  we have 
\begin{align}\label{eq:s1}
-\frac{1}{s_t}\log D_{(k_t\tau+1):(k_t+1)\tau}+\frac{(k_t+1)\tau}{s_t} l_{(k_t+1)\tau}^{\eta_\theta}(\theta)\leq l_{s_t}^{\eta_\theta}(\theta)\leq\frac{1}{s_t}\log D_{(k_t\tau+1):(k_t+1)\tau}+\frac{k_t\tau}{s_t} l_{k_t\tau}^{\eta_\theta}(\theta)
\end{align}
where the first inequality uses   Lemma \ref{lemma:tech_ssm}. Noting that $\lim_{t\rightarrow\infty} (k_t+1)\tau/s_t=\lim_{t\rightarrow\infty} k_t\tau/s_t=1$, it follows from \eqref{eq:conv_eta} that   
\begin{align}\label{eq:s2}
\lim_{t\rightarrow\infty}\frac{(k_t+1)\tau}{s_t} l_{(k_t+1)\tau}^{\eta_\theta}(\theta)=\lim_{t\rightarrow\infty}\frac{k_t\tau}{s_t} l_{k_t\tau}^{\eta_\theta}(\theta)= l'(\theta),\quad\P-a.s. \quad\forall\theta\in\Theta.
\end{align}
On the other hand, since by assumption $(Z_t)_{t\geq 1}$ is a stationary process, $\big(D_{(k_t\tau+1):(k_t+1)\tau}\big)_{t\geq 1}$ is a stationary   process and, under \ref{assumeSSM:G}, we have $\E[\log D_{(k_1\tau+1):(k_1+1)\tau}]<\infty$. Therefore, by \citetsup[][Lemma 7]{Douc_MLE2}, 
\begin{align}\label{eq:s3}
\lim_{t\rightarrow\infty}\frac{1}{s_t}\log D_{(k_t\tau+1):(k_t+1)\tau}=0,\quad\P-a.s.
\end{align}
By combining \eqref{eq:s1}-\eqref{eq:s3}, we obtain that
 \begin{align*}
 \lim_{t\rightarrow\infty}l_{s_t}^{\eta_\theta}(\theta)=l'(\theta),\quad\P-a.s.\quad\forall \theta\in\Theta
 \end{align*}
which, together with \eqref{eq:conv_eta} and recalling that $(s_t)_{t\geq 1}$ is an  arbitrary strictly increasing sequence  $\mathbb{N}\setminus \{s\tau,\,s\in\mathbb{N}\}$, shows that the result of the proposition holds with $l(\cdot)= l'(\cdot)$.
\end{proof}

\subsubsection{Proof of Proposition \ref{prop:forget}\label{p-prop:forget}}

\begin{proof}
For all integers $0\leq t_1<t_2$ and probability measure $\eta\in\mathcal{P}(\setX)$ we let 
\begin{align*}
\mathfrak{p}^{\eta}_{(t_1+1):t_2}&(A|\theta)=\frac{\int_{\setX^{(t_2-t_1)+1}}\ind_A(x_{t_2\tau+1})\eta(\dd x_{t_1\tau+1})\prod_{s=t_1+1}^{t_2}L_\theta[Z_s](x_{(s-1)\tau+1},\dd x_{s\tau+1})}{\int_{\setX^{(t_2-t_1)+1}}\eta(\dd x_{t_1\tau+1})\prod_{s=t_1+1}^{t_2}L_\theta[Z_s](x_{(s-1)\tau+1},\dd x_{s\tau+1})},\quad\forall A\in\mathcal{X}
\end{align*}
and, for all  elements $z_{t_1+1},\dots,z_{t_2}$  of $\setY^{\tau}$, we let
\begin{align*}
\mathfrak{p}^{\eta}\big[\{z_{i}\}_{i=t_1+1}^{t_2}\big]&(A|\theta)\\
&=\frac{\int_{\setX^{(t_2-t_1)+1}}\ind_A(x_{t_2\tau+1})\eta(\dd x_{t_1\tau+1})\prod_{s=t_1+1}^{t_2}L_\theta[z_s](x_{(s-1)\tau+1},\dd x_{s\tau+1})}{\int_{\setX^{(t_2-t_1)+1}}\eta(\dd x_{t_1\tau+1})\prod_{s=t_1+1}^{t_2}L_\theta[z_s](x_{(s-1)\tau+1},\dd x_{s\tau+1})},\quad\forall A\in\mathcal{X}.
\end{align*}
In addition, we let $K\subset\setY^{\tau}$ be as in \ref{assumeSSM:K_set} and $0<\gamma^-<\gamma^+\leq 1$ be such that
\begin{align}\label{eq:b_def}
b:=\max\big(1-\gamma^-,(1+\gamma^+)/2\big)<\P(K).
\end{align}
Remark that such constants  $\gamma^-$ and $\gamma^+$ exist since $\P(K)>2/3$ by assumption.

We now let $(s_t)_{t\geq 1}$ and $(W_{s_t}^-)_{t\geq 1}$ be as in the statement of the proposition, $(m_t)_{t\geq 1}$ be the sequence  in $\mathbb{N}_0$ such that $s_t=\tau m_t$ for all $t\geq 1$, and for all $t\geq 1$ we let  $k_t\in \mathbb{N}_0$  be such that $k_t\tau< t\leq \tau (k_t+1)$. In addition, we let $t'\in\mathbb{N}$ be such that $k_t>m_t\geq 1$ for all $t\geq t'$, and   $(\mu_{s_t})_{t\geq 1}$ and $(\mu'_{s_t})_{t\geq 1}$ be two sequences of   random probability measures on $(\setX,\mathcal{X})$ such that 
\begin{align}\label{eq:mu_cond_k}
\big(\mu'_{s_t}(E)\vee \mu_{s_t}(E)\big)\geq \underline{W}_{\,s_t},\quad \underline{W}_{\,s_t} =W_{s_t}^-\inf_{x\in E}M'_{\tau+1,\theta}(x,E),\quad\forall t\geq 1,\quad\P-a.s.
\end{align}
with  $E\in\mathcal{X}$ as in the statement of the proposition.

Then, a careful inspection of the proofs of \citetsup{Douc_MLE2} reveals that, under the assumptions of the proposition, the conclusion of \citetsup[Lemma 13]{Douc_MLE2} holds and thus,  for all  $\kappa\in(0,\infty)$ there exists a $\rho_\kappa\in(0,1)$ such that, for all $t\geq t'$ and sequence $\{z_i\}_{i=m_t+1}^{k_t}$ such that $(k_t-m_t)^{-1}\sum_{i=m_t+1}^{k_t}\ind_K(z_i)\geq b$ (with $b$ as defined above), we have, for all $\theta\in\Theta$ and letting $\gamma=(\gamma^+-\gamma^-)/2$,
\begin{equation}\label{eq:bound_forget}
\begin{split}
\sup_{A\in\mathcal{X}}\Big(\mathfrak{p}^{\mu_{s_t}}\big[\{z_i\}_{i=m_t+1}^{k_t}\big](A|\theta)- \mathfrak{p}^{\mu_{s_t}'}&\big[\{z_i\}_{i=m_t+1}^{k_t}\big]( A|\theta)\Big)\\
&\leq   2\rho_\kappa^{\lfloor (k_t-m_t)\gamma\rfloor}+2\frac{\kappa^{\lfloor (k_t-m_t)\gamma\rfloor}}{\mu_{s_t}(E)\mu_{s_t}'(E)}\prod_{s=m_t+1}^{k_t} D^2_{z_s} \\
&\leq   2\rho_\kappa^{\lfloor (k_t-m_t)\gamma\rfloor}+2\frac{\kappa^{\lfloor (k_t-m_t)\gamma\rfloor}}{\underline{W}_{\,s_t}^2}\prod_{s=m_t+1}^{k_t} D^2_{z_s}\\
&=: B_{k_t\tau}\big[\{z_i\}_{i=m_t+1}^{k_t}\big]
\end{split}
\end{equation}
where
\begin{align*}
D_z=\frac{\sup_{(\theta,x)\in\Theta\times\setX} L_\theta[z](x,\setX)}{\inf_{(\theta,x)\in\Theta\times E} L_\theta[z](x,E)},\quad\forall z\in\setY^{\tau}.
\end{align*}
Remark that  $D:=\E[\log^+ D_{Z_1}]<\infty$ under \ref{assumeSSM:D_set} and \ref{assumeSSM:G}  and let
\begin{align*}
W_{k_t\tau}=
\begin{cases}
1, & \frac{1}{k_t-m_t}\sum_{s=m_t+1}^{k_t}\log^+ D_{Z_s}\leq 2D\\
0, &\text{otherwise}
\end{cases},\quad\forall t\geq t'.
\end{align*}

We now  let $\epsilon\in(0,\infty)$ and  note that,  under \ref{assumeSSM:D_set} and   the assumptions of the proposition,  there exists a constant $C_\epsilon\in(0,\infty)$   such that $\limsup_{t\rightarrow\infty}\P\big(\log^- \underline{W}_{\,s_t}>C_\epsilon\big)\leq \epsilon$. Then, for all $t\geq t'$ we let
\begin{align*}
W_{\epsilon,k_t\tau}=
\begin{cases}
1, &\text{if } W_{k_t\tau}=1,\,\log^-\underline{W}_{\,s_t}\leq C_\epsilon,\, \text{ and }\frac{1}{k_t-m_t}\sum_{i=m_t+1}^{k_t}\ind_K(Z_i)\geq b\\
0, &\text{otherwise}
\end{cases}
\end{align*}
with $b$ as defined in \eqref{eq:b_def}. Then, by using \eqref{eq:bound_forget} with $\kappa=\exp\big(-(1+8D)/\gamma\big)$, we obtain that, for all $t\geq t'$,
\begin{equation}\label{eq:bound_forget2}
\begin{split}
\E\bigg[\sup_{A\in\mathcal{X}}\Big(\mathfrak{p}^{\mu_{s_t}}&\big[\{Z_i\}_{i=m_t+1}^{k_t}\big](A|\theta) - \mathfrak{p}^{\mu_{s_t}'}\big[\{Z_i\}_{i=m_t+1}^{k_t}\big](A|\theta)\Big)\bigg]\\
&\leq \P(W_{\epsilon,k_t\tau}=0)+ 2\rho_\kappa^{\lfloor (k_t-m_t)\gamma\rfloor}+2\exp(2C_\epsilon) \kappa^{\lfloor (k_t-m_t)\gamma\rfloor}\exp\big((k_t-m_t)4D \big)\\
&\leq   \P(W_{k_t\tau}=0)+\P\bigg(\frac{1}{k_t-m_t}\sum_{i=m_t+1}^{k_t}\ind_K(Z_i)< b\bigg)+\epsilon\\
&+ 2\rho_\kappa^{\lfloor (k_t-m_t)\gamma\rfloor}+2\exp\bigg(2C_\epsilon+\frac{1+8D}{\gamma}\bigg) \exp\Big(-(k_t-m_t)(1+4 D)\Big).
\end{split}
\end{equation}
Under the assumptions of the proposition  we have, using Birkhoff's theorem,
\begin{align*}
\lim_{t\rightarrow\infty}\P(W_{k_t\tau}=0)=\lim_{t\rightarrow\infty}\P\bigg(\frac{1}{k_t-m_t}\sum_{i=m_t+1}^{k_t}\ind_K(Z_i)< b\bigg)=0
\end{align*}
and thus, by \eqref{eq:bound_forget2}, 
\begin{equation*}
\limsup_{t\rightarrow\infty}\E\bigg[\sup_{A\in\mathcal{X}}\Big(\mathfrak{p}^{\mu_{s_t}}\big[\{Z_i\}_{i=m_t+1}^{k_t}\big](A|\theta) - \mathfrak{p}^{\mu_{s_t}'}\big[\{Z_i\}_{i=m_t+1}^{k_t}\big](A|\theta)\Big)\bigg]\leq \epsilon.
\end{equation*}
Since $\epsilon\in(0,\infty)$ is arbitrary and the upper bound in \eqref{eq:bound_forget2} depends on $(\mu_{s_t})_{t\geq 1}$ and $(\mu_{s_t}')_{t\geq 1}$ only through $(W^-_{s_t})_{t\geq 1}$, it follows   that there exists a sequence $(\xi_t)_{t\geq 1}$ of $[0,1]$-valued random variables, depending on $(\mu_{s_t})_{t\geq 1}$ and $(\mu_{s_t}')_{t\geq 1}$ only through $(W^-_{s_t})_{t\geq 1}$, such that $\xi_t=\smallo_\P(1)$ and such that
\begin{align*}
\sup_{A\in\mathcal{X}}\Big(\mathfrak{p}^{\mu_{s_t}}[\{Z_i\}_{i=m_t+1}^{k_t}](A|\theta)-(\mathfrak{p}^{\mu'_{s_t}}[\{Z_i\}_{i=m_t+1}^{k_t}](A|\theta)\Big)\leq \xi_{k_t\tau},\quad\forall t\geq 1,\quad\P-a.s.
\end{align*}
which readily implies that
\begin{align}\label{eq:first_res}
\big\|\mathfrak{p}^{\mu_{s_t}}_{(m_t+1):k_t}(\dd x|\theta)-\mathfrak{p}^{\mu_{s_t}'}_{(m_t+1):k_t}(\dd x|\theta)\big\|_{\mathrm{TV}}\leq \xi_{k_t\tau},\quad\forall t\geq 1,\quad\P-a.s.
\end{align}

To proceed further recall that   under  \ref{assumeSSM:G}  we $\P$-a.s.~have $\sup_{x\in\setX} f'_{i,\theta}(Y_i|x)<\infty$ for all $i\in\{1,\dots,\tau\}$ and let
\begin{align*}
\tilde{f}_{t,\theta}(Y_t|x_t)=\frac{f'_{i_t,\theta}(Y_t|x_t)}{\sup_{x\in\setX} f'_{i_t,\theta}(Y_t|x)},\quad\forall x_t\in\setX,\quad\forall t\geq 1.
\end{align*}
Then, for all integers $1\leq t_1<t_2$ and $t\geq \tau t_2+1$,  and all probability measure $\eta\in\mathcal{P}(\setX)$, we let  
\begin{align*}
l^{\eta}_{(t_1,t_2),t}(\theta)=\int_{\setX^{t-\tau t_2+1}}\mathfrak{p}_{(t_1+1):t_2}^{\eta_\theta}(\dd x_{t_2\tau+1}|\theta)\prod_{s=t_2\tau+1}^{t} \tilde{f}_{s,\theta}(Y_{s}|X_{s})M'_{i_{s}+1,\theta}(x_s,\dd x_{s+1}) 
\end{align*}
with  $\eta_\theta(\dd x)=\int_\setX\eta(\dd x')M'_{\tau+1,\theta}(x',\dd x)$. Remark that, under \ref{assumeSSM:G} and the assumptions that $(Z_t)_{t\geq 1}$ is a stationary process, for all integers $1\leq t_1<t_2$ and $t\geq \tau t_2+1$ we $\P$-a.s.~have, for all $A\in\mathcal{X}$,
\begin{equation}\label{eq:use_f1}
\begin{split}
\tilde{p}_{(t_1\tau+1):t}^{\eta} (A&|\theta)\\ 
&=\frac{\int_{\setX^{t-t_2\tau+1}}\ind_A(x_{t})\mathfrak{p}_{(t_1+1):t_2}^{\eta_\theta}(\dd x_{t_2\tau+1}|\theta)\prod_{s=t_2\tau+1}^{t} \tilde{f}_{s,\theta}(Y_{s}|X_{s})M'_{i_{s}+1,\theta}(x_s,\dd x_{s+1}) }{l^{\eta_\theta}_{(t_1,t_2),t}(\theta)}.
\end{split}
\end{equation}
Finally, recall  that for any two probability measures $\eta_1,\eta_2\in\mathcal{P}(\setX)$  we have
\begin{align}\label{eq:use_f2}
\|\eta_1-\eta_2\|_{\mathrm{TV}}=\sup_{f:\setX\rightarrow[0,1]}\Big|\int_{\setX} f(x)\big(\eta_1(\dd x)-\eta_2(\dd x)\big)\Big|
\end{align}
and, for all $t\in\mathbb{N}$, let
\begin{align*}
\eta_{\theta,s_t}=\int_\setX\eta_{s_t}(\dd x')M'_{\tau+1,\theta}(x',\dd x),\quad  \chi_{\theta,s_t}(\dd x)= \int_\setX\chi_{s_t}(\dd x')M'_{\tau+1,\theta}(x',\dd x).
\end{align*}

Then, using \eqref{eq:use_f1}-\eqref{eq:use_f2}, we $\P$-a.s.~have, for all $t\geq t'$ and $A\in\mathcal{X}$,
\begin{equation}\label{eq:split_forget}
\begin{split}
\big|\tilde{p}_{(s_t+1):t}^{\eta_{s_t}}(A|\theta)&-\tilde{p}_{(s_t+1):t}^{\chi_{s_t}}(A|\theta)\big|\\
&\leq  \frac{2\big\|\mathfrak{p}_{(m_t+1):k_t}^{\eta_{\theta,s_t}}(\dd x|\theta)-\mathfrak{p}_{(m_t+1):k_t}^{\chi_{\theta,s_t}}(\dd x|\theta)\big\|_{\mathrm{TV}}}{\big(l_{(m_t,k_t),t}^{\eta_{\theta,s_t}}(\theta)\big)\big(l_{(m_t,k_t),t}^{\chi_{\theta,s_t}}(\theta)\big)}\\
&\leq  \frac{2\big\|\mathfrak{p}_{(m_t+1):k_t}^{\eta_{\theta,s_t}}(\dd x|\theta)-\mathfrak{p}_{(m_t+1):k_t}^{\chi_{\theta,s_t}}(\dd x|\theta)\big\|_{\mathrm{TV}}}{\big(l_{(m_t,k_t),(k_t+1)\tau}^{\eta_{\theta,s_t}}(\theta)\big)\big(l_{m_t,k_t,(k_t+1)\tau}^{\chi_{\theta,s_t}}(\theta)\big)}\\
&\leq \frac{2\big\|\mathfrak{p}_{(m_t+1):k_t}^{\eta_{\theta,s_t}}(\dd x|\theta)-\mathfrak{p}_{(m_t+1):k_t}^{\chi_{\theta,s_t}}(\dd x|\theta)\big\|_{\mathrm{TV}}\Big(\prod_{s=k_t\tau+1}^{(k_t+1)\tau}\sup_{x\in\setX}f'_{i_s,\theta}(Y_s|x)\Big)^2}{\big(\mathfrak{p}_{(m_t+1):k_t}^{\eta_{\theta,s_t}}(E|\theta)\big)\big(\mathfrak{p}_{(m_t+1):k_t}^{\chi_{\theta,s_t}}(E|\theta)\big)\big(\inf_{x\in E}L_\theta[Z_{k_t+1}](x,E)\big)^2}.
\end{split}
\end{equation}

To conclude the proof of the proposition remark first that, under its assumptions, \eqref{eq:mu_cond_k} holds when $\mu_{s_t}=\chi_{\theta,s_t}$ and $\mu'_{s_t}=\eta_{\theta,s_t}$ for all $t\geq 1$  and thus, by \eqref{eq:first_res},
\begin{align}\label{eq:C1}
\big\|\mathfrak{p}_{(m_t+1):k_t}^{\eta_{\theta,s_t}}(\dd x|\theta)-\mathfrak{p}_{(m_t+1):k_t}^{\chi_{\theta,s_t}}(\dd x|\theta)\big\|_{\mathrm{TV}}\leq \xi_{k_t\tau},\quad\forall t\geq 1,\quad\P-a.s.
\end{align}

In addition, by assumption, there exists a  probability measures $\eta\in\mathcal{P}(\setX)$ for which     $\log \bar{p}_{t,\setX}^\eta(E|Y_{1:t})=\bigO_\P(1)$ and thus, under \ref{assumeSSM:D_set},   $\log \mathfrak{p}_{1:m_t}^\eta(E|\theta)=\bigO_\P(1)$. Without loss of generality  we assume henceforth that $\log \mathfrak{p}_{1:m_t}^\eta(E|\theta)\geq \underline{W}_{\,s_t}$ for all $t\geq 1$. Then, by applying \eqref{eq:first_res} with $\mu_{s_t}=\eta_{\theta,s_t}$ and $\mu'_{s_t}=\mathfrak{p}_{1:m_t}^{\eta}(\dd x|\theta)$ for all $t\geq 1$, we obtain that
\begin{align}\label{eq:BB_p1}
\big|\mathfrak{p}_{(m_t+1):k_t}^{\eta_{\theta, s_t}}(E|\theta)-\mathfrak{p}_{1:k_t}^\eta(E|\theta)\big|\leq \xi_{k_t\tau},\quad\forall t\geq t',\quad\P-a.s.
\end{align}
Similarly,  by applying \eqref{eq:first_res} with $\mu_{s_t}=\chi_{\theta,s_t}$ and $\mu_{s_t}'=\mathfrak{p}_{1:m_t}^{\eta}(\dd x|\theta)$ for all $t\geq 1$ we obtain that  
\begin{align}\label{eq:BB_p2}
\big|\mathfrak{p}_{(m_t+1):k_t}^{\chi_{\theta,s_t}}(E|\theta)-\mathfrak{p}_{1:k_t}^\eta(E|\theta)\big|\leq \xi_{k_t\tau},\quad\forall t\geq t',\quad\P-a.s.
\end{align}

We now let
\begin{align*}
V_{k_t\tau}=\max\big(0,\mathfrak{p}_{1:k_t}^\eta(E|\theta)-\xi_{k_t\tau}\big),\quad\forall t\geq 1
\end{align*}
and show that $\log V_{k_t\tau}=\bigO_\P(1)$. To this aim let $\epsilon\in(0,1)$ and remark that since $\log W^-_{k_t\tau}=\bigO_\P(1)$ there exists a constant $C_\epsilon\in(0,\infty)$ such that
\begin{align*}
\liminf_{t\rightarrow\infty}\P\big(\log W^-_{k_t\tau}> -C_\epsilon\big)\geq 1-\epsilon.
\end{align*}
Then,  recalling that $\xi_{k_t\tau}=\smallo_\P(1)$, it follows that
\begin{align*}
\liminf_{t\rightarrow\infty}\P\big(\log V_{k_t\tau}\geq -C_\epsilon-\log 2\big)&=\liminf_{t\rightarrow\infty}\P\big(V_{k_t\tau}\geq  0.5e^{-C_\epsilon}\big)\\
&\geq \liminf_{t\rightarrow\infty}\P\big(V_{k_t\tau}\geq  0.5e^{-C_\epsilon},\,\xi_{k_t\tau}\leq 0.5e^{-C_\epsilon},\, W^-_{k_t\tau}\geq   e^{-C_\epsilon} \big)\\
&=\liminf_{t\rightarrow\infty}\P\big(\xi_{k_t\tau}\leq 0.5e^{-C_\epsilon},\, W^-_{k_t\tau}\geq   e^{-C_\epsilon} \big)\\
&\geq \liminf_{t\rightarrow\infty}\P\big(\log W^-_{k_t\tau}\geq   -C_\epsilon  \big)\\
&\geq 1-\epsilon
\end{align*}
where the penultimate inequality holds by Frechet's inequality.  Since $\epsilon\in(0,\infty)$ is arbitrary this shows that $\log V_{k_t\tau}=\bigO_\P(1)$. 

Finally, let
\begin{align*}
\tilde{\xi}_{k_t\tau}=1\vee \frac{2\xi_{k_t\tau}\Big(\prod_{s=k_t\tau+1}^{(k_t+1)\tau}\sup_{x\in\setX}f'_{i_s,\theta}(Y_s|x)\Big)^2}{V_{k_t\tau}^2 \big(\inf_{x\in E}L_\theta[Z_{k_t+1}](x,E)\big)^2},\quad\forall t\geq 1
\end{align*}
and note that  $\tilde{\xi}_{k_t\tau}=\smallo_\P(1)$ since (i) $\xi_{k_t\tau}=\smallo_\P(1)$, (ii)  $\log V_{k_t\tau}=\bigO_\P(1)$ as shown above, and (iii) we have, under \ref{assumeSSM:D_set}-\ref{assumeSSM:G},
\begin{align*}
\log^-\inf_{x\in E}L_\theta[Z_{k_t+1}](x,E)=\bigO_\P(1),\quad \log^+\prod_{s=k_t\tau+1}^{(k_t+1)\tau}\sup_{x\in\setX}f'_{i_s,\theta}(Y_s|x)=\bigO_\P(1).
\end{align*}
Remark also that $(\tilde{\xi}_{k_t\tau})_{t\geq 1}$ depends on  $(\eta_{s_t})_{t\geq 1}$ and $(\chi_{s_t})_{t\geq 1}$ only through  $(W_{s_t}^-)_{t\geq 1}$ and that, by \eqref{eq:split_forget}-\eqref{eq:BB_p2}, for all $t\geq t'$ we $\P$-a.s.~have
\begin{align*}
\big\|\tilde{p}_{(s_t+1):t}^{\eta_{s_t}}(\dd x|\theta) -\tilde{p}_{(s_t+1):t}^{\chi_{s_t}}(\dd x|\theta)\big\|_{\mathrm{TV}}&\leq\tilde{\xi}_{k_t\tau}.
\end{align*}
The proof of the proposition is complete.
\end{proof}

\subsubsection{Proof of Proposition \ref{prop:A1}\label{p-prop:A1}} 
 
\begin{proof}

We start by showing the first part of the proposition.

Since $(Z_t)_{t\geq 1}$ is a stationary and ergodic process, it directly follows from Assumptions  \ref{assumeSSM:D_set}-\ref{assumeSSM:G} that  \ref{assume:G_bounded}-\ref{assume:G_l_E} hold, and to prove the second part of the proposition it remains to show that \ref{assume:Model} holds as well.

To this aim, for  all $t\geq 1$ and $\theta\in\Theta$ we let
\begin{align}\label{eq:q_prime}
q'_{t,\theta}(x|x')=f'_{i_t,\theta}(Y_t|x')  m'_{i_t+1,\theta}(x|x'),\quad\forall (x,x')\in\setX^2 
\end{align}
and
\begin{align*}
 Q'_{t,\theta}(x',\dd x)=q'_{t,\theta}(x|x',y)\lambda(\dd x)\quad\forall (x',y)\in\setX\times\setY. 
\end{align*}
In addition, to simplify the notation in what follows, for all integers $0\leq t_1< t_2$ and $\eta\in\mathcal{P}(\setX)$ we let
\begin{align*}
\tilde{l}^\eta_{(t_1+1):t_2}(\theta)=\frac{1}{(t_2-t_1)}\log\int_{\setX^{(t_2-t_1)+1}}\eta(\dd x_{t_1})  \prod_{s=t_1+1}^{t_2}Q'_{s,\theta}( x_{s-1},\dd x_s),\quad\forall\theta\in\Theta.
\end{align*}

Next, we let $g'(\cdot)$, $\delta\in(0,\infty)$, $\{\varphi'_i\}_{i=1}^\tau$ and $F(\cdot)$ be as in \ref{assumeSSMB:smooth}, and we let  $\eta_{\tau}$ be an arbitrary random probability measure on $(\setX,\mathcal{X})$ such that $\P(\eta_\tau(E)>0)=1$ for some compact set $E\in\mathcal{X}$ for which we have $\lambda(E)>0$. Then, since $(Y_t)_{t\geq 1}$ is assumed to be  $\tau$-stationary, under \ref{assumeSSM:K_set}-\ref{assumeSSM:G}  it follows from Proposition \ref{prop:conv_SSM}   that
\begin{align}\label{eq:convLik2}
\lim_{t\rightarrow\infty}\tilde{l}^{\eta_{\tau}}_{(\tau+1):(t+1)\tau}(\theta)=l(\theta),\quad\P-a.s.,\quad\forall \theta\in\Theta 
\end{align}
while,  under \ref{assumeSSM:D_set},
\begin{align}\label{eq:init_E}
\P\Big(\tilde{l}^{\eta_\tau}_{(\tau+1):\tau(t+1)}(\theta)>-\infty\Big)=1,\quad\forall \theta\in\Theta,\quad\forall t\geq 1.
\end{align}

We now show that, for $(\tilde{G}_t)_{t\geq 1}$ and $(\tilde{m}_t)_{t\geq 1}$ as defined in the first part of the proposition, Assumption \ref{assume:Model} holds with $g=g'$, $r=\tau$, $\tilde{l}=l$, $ \varphi_t(\cdot,\cdot)=\varphi'_{i_t}(Y_t,\cdot,\cdot)$ for all $t\geq 1$, and with   $D=E$ where   $E\in\mathcal{X}$ is an arbitrary compact set such that $\lambda(E)>0$.

We start by showing that, as required by \ref{assume:Model}, the function $\theta\mapsto l(\theta)$ is continuous on $\Theta$. To this aim let $\chi'(\dd x)$ be as in the statement of the proposition, $\epsilon\in(0,\infty)$ be such that $g'(\epsilon)^{1/2}\leq \min(1/2,\delta)$  and  let $(\theta,\theta')\in\Theta^2$  be  such that $\|\theta-\theta'\|\leq \epsilon$.   Then, for all $t\geq 2$,
\begin{align*}
 \tilde{l}^{\chi'}_{1:t\tau}(\theta)&-\tilde{l}^{\chi'}_{1:t\tau}(\theta')\\
&\leq \frac{1}{t\tau}\log \int_{\setX^{t\tau+1}}\hspace{-0.4cm}\exp\bigg( g'\big(\|\theta-\theta'\|)  \sum_{s= 1}^{t\tau}\varphi'_{i_t}(Y_s,x_{s-1}, x_{s})\bigg)\frac{\chi'(\dd x_{0})  \prod_{s=1}^{t\tau} Q'_{ s,\theta'}( x_{s-1},\dd x_s)}{\exp(t\tau\, \tilde{l}^{\chi'}_{1:t\tau}(\theta'))}\\
&\leq \frac{g'\big(\|\theta-\theta'\|)^{\frac{1}{2}}}{t\tau}\log  \int_{\setX^{t\tau+1}}\hspace{-0.4cm}\exp\bigg(\hspace{-0.1cm}\delta  \sum_{s= 1}^{t\tau}\varphi'_{i_t}(Y_s,x_{s-1}, x_{s})\bigg)\frac{\chi'(\dd x_{0})  \prod_{s=1}^{t\tau} Q'_{s,\theta'}(x_{s-1}, \dd x_s)}{\exp(t\tau \tilde{l}^{\chi'}_{1:t\tau}(\theta'))}\\
&\leq  g'\big(\|\theta-\theta'\|)^{1/2}\Big( \frac{1}{t\tau}\sum_{s=1}^{t}\log F(Z_s)-\tilde{l}^{\chi'}_{1:t\tau}(\theta')\Big)
\end{align*}
where the first inequality holds under \ref{assumeSSMB:smooth}, the second inequality holds by Jensen's inequality and the last inequality holds again by \ref{assumeSSMB:smooth}. On the other hand,
\begin{align*}
\tilde{l}^{\chi'}_{1:t\tau}(\theta)&-\tilde{l}^{\chi'}_{1:t\tau}(\theta')\\
&\geq \frac{1}{t\tau}\log \int_{\setX^{t\tau+1}}\hspace{-0.4cm}\exp\bigg(\hspace{-0.2cm}-g'\big(\|\theta-\theta'\|)  \sum_{s= 1}^{t\tau}\varphi'_{i_t}(Y_s,x_{s-1}, x_{s})\bigg)\frac{\chi'(\dd x_{0})  \prod_{s=1}^{t\tau} Q'_{s,\theta'}(x_{s-1},\dd x_s)}{\exp(t\tau \tilde{l}^{\chi'}_{1:t\tau}(\theta'))}\\
&\geq -\frac{g'\big(\|\theta-\theta'\|)^{\frac{1}{2}}}{t\tau}\log  \int_{\setX^{t\tau+1}}\hspace{-0.4cm}\exp\bigg(\hspace{-0.1cm}\delta  \sum_{s= 1}^{t\tau}\varphi'_{i_t}(Y_s,x_{s-1}, x_{s})\bigg)\frac{\chi'(\dd x_{0})  \prod_{s=1}^{t\tau} Q_{s,\theta'}(x_{s-1},\dd x_s)}{\exp(t\tau \tilde{l}^{\chi'}_{1:t\tau}(\theta'))}\\
&\geq  -g'\big(\|\theta-\theta'\|)^{1/2}\Big( \frac{1}{t\tau}\sum_{s=1}^{t}\log F(Z_s)-\tilde{l}^{\chi'}_{t\tau}(\theta')\Big)
\end{align*}
where the first inequality holds under \ref{assumeSSMB:smooth}, the second inequality holds by \eqref{eq:double_Jensen} and the last inequality holds again by \ref{assumeSSMB:smooth}. Therefore,  for all $(\theta,\theta')\in\Theta^2$    such that $\|\theta-\theta'\|\leq \epsilon$ we have
\begin{align}\label{eq:Lyp}
|\tilde{l}^{\chi'}_{1:t\tau}(\theta)-\tilde{l}^{\chi'}_{1:t\tau}(\theta')|\leq g'\big(\|\theta-\theta'\|)^{1/2}\Big( \frac{1}{t\tau}\sum_{s=1}^{t}\log^+ F(Z_t)+|\tilde{l}^{\chi'}_{1:t\tau}(\theta')|\Big).
\end{align}
Under the assumptions of the first part of the proposition, for all $\theta\in\Theta$ we have $\chi'(E)>0$ and thus, by Proposition \ref{prop:conv_SSM},
\begin{align}\label{eq:Lyp2}
\lim_{t\rightarrow\infty}\tilde{l}_{t\tau}^{\chi'}(\theta)=l(\theta),\quad\P-a.s.,\quad\forall \theta\in\Theta.
\end{align}
Therefore, we $\P$-a.s.~have, for all $(\theta,\theta')\in\Theta^2$ such that $\|\theta-\theta'\|\leq \epsilon$,
\begin{align*}
|l(\theta)-l(\theta')|&=\big|\lim_{t\rightarrow\infty}\tilde{l}^{\chi'}_{1:t\tau}(\theta)-\lim_{t\rightarrow\infty}\tilde{l}^{\chi'}_{1:t\tau}(\theta')\big|\\
&= \lim_{t\rightarrow\infty}\big|\tilde{l}^{\chi'}_{1:t\tau}(\theta)-\tilde{l}^{\chi'}_{1:t\tau}(\theta')\big|\\
&\leq g'(\|\theta-\theta'\|)^{1/2}\bigg(\limsup_{t\rightarrow\infty}\frac{1}{t\tau}\sum_{s=1}^t\log F(Z_s)+\limsup_{t\rightarrow\infty} |\tilde{l}^{\chi'}_{1:t\tau}(\theta')|\bigg)\\
&= g'(\|\theta-\theta'\|)^{1/2}\bigg(\frac{\E\big[\log F(Z_1)\big]}{\tau}+|l(\theta')|\bigg)
\end{align*}
where the first equality holds by \eqref{eq:Lyp2},  the inequality holds by \eqref{eq:Lyp} and the last equality holds under \ref{assumeSSMB:smooth} by Birkhoff's theorem and uses  \eqref{eq:Lyp2}. This shows that the function $l$ is $\P$-a.s.~continuous on $\Theta$ and thus, since $l$ is non-random, it follows that $l$ is continuous on $\Theta$.

Next, under \ref{assumeSSMB:smooth} and using  \eqref{eq:convLik2}-\eqref{eq:init_E} with $\eta_{\tau}=\chi'$,  it follows that  $\chi'$ is a 
 $(\tau, \delta', \tilde{l},(\varphi_{t})_{t\geq 1})$-consistent probability measure on $(\setX,\mathcal{X})$ for any $\delta'\in(0,\delta]$, showing that \ref{assume:Model_init} holds for $\chi=\chi'$. Remark also that \ref{assume:Model_smooth} holds by \ref{assumeSSMB:smooth}.
 
To proceed further let
 \begin{align*}
 \bar{p}_{\tau}(x)=\sup_{(\theta,x')\in \Theta \times \setX} q'_{\tau,\theta}(x|x'),\quad\forall x\in\setX
 \end{align*}
 and note that, under \ref{assumeSSM:D_set}, we $\P$-a.s.~have $\int_{\setX}\bar{p}_{\tau}(x)\lambda(\dd x)>0$ while, under  \ref{assumeSSMB:smooth},
 \begin{align*}
 \int_{\setX}\bar{p}_{\tau}(x)\lambda(\dd x)<\infty,\quad\P-a.s.
 \end{align*}
  Therefore, there exists a random probability measure  $\bar{\chi}_{\tau}$ on $(\setX,\mathcal{X})$ such that
\begin{equation}\label{eq:chi_ssm}
\bar{\chi}_{\tau}(A) =\int_A \frac{\sup_{(\theta,x')\in \Theta \times \setX} q'_{\tau,\theta}(x|x', ) }{\int_{\setX}\big\{\sup_{(\theta,x')\in \Theta \times \setX}q'_{\tau,\theta}(x|x')\big\}\lambda(\dd x)}\lambda(\dd x),\quad\forall A\in\mathcal{X},\quad \P-a.s.
\end{equation}   
Remark that  $\bar{\chi}_{\tau}(E)>0$ under \ref{assumeSSM:D_set} and thus, under \ref{assumeSSMB:smooth} and by applying \eqref{eq:convLik2}-\eqref{eq:init_E}  with $\eta_{\tau}=\bar{\chi}_{\tau}$,  it follows   that   $\bar{\chi}_{\tau}$ is a 
 $(\tau, \delta', l,(\varphi_{t})_{t\geq 1})$-consistent probability measure on $(\setX,\mathcal{X})$ for any $\delta'\in(0,\delta]$, and thus  \ref{assume:Model_sup} holds.
 
Similarly, letting  
\begin{align*}
 \underline{p}_{\tau}(x)=\inf_{(\theta,x')\in \Theta \times E} q'_{\tau,\theta}(x|x'),\quad\forall x\in\setX,
 \end{align*}
it follows that, under \ref{assumeSSM:D_set}  we $\P$-a.s.~have  $\int_{\setX}\underline{p}_{\tau}(x)\lambda(\dd x)>0$ while, using  \ref{assumeSSMB:smooth} for the second inequality, 
\begin{align*}
\int_{\setX}\underline{p}_{\tau}(x)\lambda(\dd x)\leq \int_{\setX}\bar{p}_{\tau}(x)\lambda(\dd x)<\infty,\quad\P-a.s.
\end{align*}
Therefore, there exists a random probability measure $\underline{\chi}_{\tau}$  on $(\setX,\mathcal{X})$ such that
\begin{equation}\label{eq:chi_ssm2}
\underline{\chi}_{\tau}(A) =\int_A \frac{\inf_{(\theta,x')\in \Theta \times E} q'_{\tau,\theta}(x|x')}{\int_{\setX}\big\{\inf_{(\theta,x')\in \Theta \times E} q'_{\tau,\theta}(x|x')\big\}\lambda(\dd x)}\lambda(\dd x),\quad\forall A\in\mathcal{X},\quad \P-a.s.
\end{equation} 
Remark that  $\underline{\chi}_{\tau}(E)>0$ under \ref{assumeSSM:D_set} and thus, under \ref{assumeSSMB:smooth} and by applying  \eqref{eq:convLik2}-\eqref{eq:init_E} with $\eta_{\tau}=\underline{\chi}_{\tau}$,  it follows  $\underline{\chi}_{\tau}$ is a  $(\tau,\delta', l,(\varphi_{t})_{t\geq 1})$-consistent probability measure on $(\setX,\mathcal{X})$ for any $\delta'\in(0,\delta]$, and thus \ref{assume:Model_G_lower} holds. The proof of the first part of the proposition is complete.

The proof of second part of the proposition is omitted to save space, since its result can   be readily  established by repeating the above computations where, for all $t\geq 1$ and $(\theta, x')\in\Theta\times\setX$, instead of being as defined in \eqref{eq:q_prime} the function $q'_{t,\theta}(\cdot|x')$ is defined by
\begin{align*}
q'_{t,\theta}(x|x')=f'_{i_t,\theta}(Y_t|x)  m'_{i_{t},\theta}(x|x'),\quad x\in\setX
\end{align*}
and by using \ref{assumeSSM:smooth} instead of \ref{assumeSSMB:smooth} and by noting that, for this definition of $(q_t)_{t\geq 1}$, the results in \eqref{eq:convLik2} and \eqref{eq:init_E} hold for any random probability measures $\eta_\tau$ on $(\setX,\mathcal{X})$.
\end{proof}

\section{Proofs of the results in Section \ref{sec:MLE}\label{p-proof_MLE}}

Throughout this subsection $(Y_t)_{t\geq 1}$ and  SSM \eqref{eq:SSM} are assumed to be as defined in Section \ref{sec:MLE}. In order to apply the theory developed in Appendix \ref{sec:theory}   to prove Theorem \ref{thm:MLE}, we let  $\tilde{\chi}=\tilde{\chi}_{\tilde{\theta}}$ for some arbitrary $\tilde{\theta}\in\Theta$.

\subsection{Additional notation}

For all $(\theta,x')\in\Theta\times\setX$ we let
\begin{align*}
 M''_{1,\theta}(x',\dd x)=\tilde{\chi}_\theta(\dd x),\,\, M''_{T+1,\theta}(x',\dd x)=\tilde{\chi}(\dd x)\text{ and }M''_{t,\theta}(x',\dd x)=\tilde{M}_{t,\theta}(x',\dd x),\quad\forall t\in\{2,\dots,T\}
\end{align*}
and, for all $(\theta,x,y)\in\Theta\times\setX\times\setY$, we let $f''_{t,\theta}(y|x)=\tilde{f}_{t,\theta}(y|x)$ for all $t\in\{1,\dots,T\}$ and $f''_{T+1,\theta}(y|x)=1$. In addition, for all $t\geq 1$ we let  $i''_t=t+(T+1)\lfloor (t-1)/(T+1)\rfloor$ and
\begin{align*}
\tilde{y}''_{(t-1)(T+1)+s}=
\begin{cases}
\tilde{y}_{s}, & s\in\{1,\dots,T\}\\
\tilde{y}_1, &s=T+1
\end{cases},\quad\forall s\in\{1,\dots,T+1\}.
\end{align*}

 Finally, in what follows, for all $\mu\in\mathcal{P}(\setR)$ the Markov kernel $K_{\mu|\Theta}$ acting on $(\Theta,\mathcal{T})$ is as defined in Section   \ref{sub:SO-FK_model} while, for a given  Feynman-Kac  in random environments, $L^{\eta}_{(t_1+1):t_2}(\setX|\theta)$ is as defined in Section \ref{sub:FK_model} for all $\eta\in\mathcal{P}(\setX)$, all $\theta\in\Theta$ and all integers $0\leq t_1<t_2$.

\subsection{Preliminary results}

The following   proposition  defines the  Feynman-Kac models that we will use to study   SO-SSM \eqref{eq:SO-SSM_new2}.

\begin{proposition}\label{prop:Def_FK_MLE}
Let $\chi'\in \mathcal{P}(\setX)$. Then, the Feynman-Kac model  of Section \ref{sub:FK_model} with $\chi=\chi'$ and with   $( (\tilde{G}_t,\tilde{m}_t))_{t\geq 1}$    defined by 
\begin{align*}
\tilde{G}_{t}(\omega,\theta, x',x)=f''_{i''_t,\theta}(\tilde{y}''_{t}|x),\,\, \tilde{m}_{t}(\omega,\theta,x',x)=m''_{i''_t,\theta}(x|x'),\,\,\forall (t,\omega,\theta,x,x')\in\mathbb{N}\times\Omega\times \Theta\times\setX^2
\end{align*}
is such that $L_{(T+2):(T+1)t}^{\chi}(\setX|\theta)=\tilde{L}_T(\theta)^{t-1}$ for all $\theta\in\Theta$ and all $t\geq 2$.
\end{proposition}
\begin{proof}
The result of the proposition directly follows from the definition of $( (\tilde{G}_t,\tilde{m}_t))_{t\geq 1}$.
\end{proof}

To following proposition shows that the Feynman-Kac model  defined in Proposition  \ref{prop:Def_FK_MLE} is such that Assumption \ref{assume:Model} holds.

\begin{proposition}\label{prop:A1_MLE}
Let $\Theta_T:=\argmax_{\theta\in\Theta}\tilde{L}_T(\theta)$ and let $( (\tilde{G}_t,\tilde{m}_t))_{t\geq 1}$ be as defined in Proposition \ref{prop:Def_FK_MLE}. Assume that there exists a $\delta\in(0,\infty)$ such that $\mathcal{N}_\delta(\Theta_T)\subseteq\Theta$ and that Assumption \ref{assumeSSMB:smooth_MLE} holds.   Then, for any $\chi\in\mathcal{P}(\setX)$, Assumption   \ref{assume:Model} holds with $\tilde{l}(\cdot)=T^{-1}\log\tilde{L}_T(\cdot)$, $D=\setX$ and with $r=T+1$.

\end{proposition}
\begin{proof}

Remark first that under  \ref{assumeSSMB:smooth_MLE} the mapping $\theta\mapsto\log \tilde{L}_T(\theta)$ is  continuous and thus, since $\Theta$ is a compact set, $\Theta_T\subseteq\Theta$. 

To show the result of the proposition remark that Assumption \ref{assume:Model_stationary} follows from the definition of $( (\tilde{G}_t,\tilde{m}_t))_{t\geq 1}$ while \ref{assume:Model_ThetaStar} holds by assumption. Assumption \ref{assume:Model_smooth} holds by \ref{assumeSSMB:smooth_MLE} with $g=g'$ and with $(\varphi_t)_{t\geq 1}$  defined by
\begin{align*}
\varphi_t(x',x)=
\begin{cases}
\varphi'_{i''_t}\big(x',x),&  t\not\in\{s(T+1),\,s\in\mathbb{N}\}\\
1,  &t\in\{s(T+1),\,s\in\mathbb{N}\}
\end{cases},\quad (t,x',x)\in\mathbb{N}\times\setX^2.
\end{align*} 
Next, for the so-defined functions $(\varphi_t)_{t\geq 1}$,  under  \ref{assumeSSMB:smooth_MLE} and by   Proposition \ref{prop:Def_FK_MLE}, and since  under  \ref{assumeSSMB:smooth_MLE} we have $\tilde{L}_T(\theta)\in (0,\infty)$  for all $\theta\in\Theta$, it readily  follows that any  random probability measure $\eta$ on $(\setX,\mathcal{X})$ is $(T+1,\delta_\star,\tilde{l}, (\varphi_t)_{t\geq 1})$-consistent for $\delta_\star\in(0,1)$ small enough, with $\tilde{l}$ as in the statement of the proposition.  Hence,  \ref{assume:Model_init}   holds for any $\chi\in\mathcal{P}(\setX)$. To proceed further denote by $\tilde{p}(\cdot)$ the density of $\tilde{\chi}(\dd x)$ w.r.t.~$\lambda(\dd x)$ (recall that $\tilde{\chi}=\tilde{\chi}_{\tilde{\theta}}$ for some $\tilde{\theta}\in\Theta$, and thus  $\tilde{\chi}$ is absolutely continuous w.r.t.~$\lambda(\dd x)$) and   remark that for the considered Feynman-Kac model we have $q_{T+1,\theta}(x|x')=\tilde{p}(x)$ for all $(\theta,x',x)\in\Theta\times\setX^2$. Therefore, Assumption  \ref{assume:Model_G_lower}  holds with $D=\setX$ and with $\underline{\chi}_r(\dd x)=\tilde{\chi}(\dd x)$  while \ref{assume:Model_sup} holds   $\bar{\chi}_r(\dd x)=\tilde{\chi}(\dd x)$. The proof of the proposition is complete.
 
\end{proof}
The next  proposition  shows that   SO-SSM \eqref{eq:SO-SSM_new2} can be studied by analyzing a particular self-organizing version of the  Feynman-Kac model defined in  Proposition \ref{prop:Def_FK_MLE}.

\begin{proposition}\label{prop:SO-FK_MLE2}
 Let $\chi'\in\mathcal{P}(\setX)$,  $\mu_0'\in\mathcal{P}(\Theta)$, $(\eta_t)_{t\geq 1}$  be a sequence  of random probability measures on $(\setR,\mathcal{R})$ and    $(\mu'_t)_{t\geq 1}$ be such that $\mu'_{s(T+1)+1}=\delta_{\{0\}}$ for all $s\in\mathbb{N}_0$ and such that
\begin{align*}
\mu'_{s(T+1)+i}=\eta_{sT+i}\quad\forall i\in\{2,\dots,T+1\},\quad\forall s\in\mathbb{N}_0.
\end{align*}
In addition, for all $t\geq 1$ let $\pi_{t,\Theta}(\dd\theta)$   be as defined in Section \ref{sub:SO-FK_model}  with $\mu_0=\mu_0'$, with $( (\tilde{G}_t,\tilde{m}_t))_{t\geq 1}$  as defined in   Proposition \ref{prop:Def_FK_MLE}, with  $\chi=\chi'$ and with $(\mu_t)_{t\geq 1}=(\mu'_t)_{t\geq 1}$. Then, for   SO-SSM \eqref{eq:SO-SSM_new2} with $K_t=K_{\eta_t|\Theta}$ for all $t\geq 1$ and with $\mu_0=\mu_0'$,  there exists a sequence $(s_t)_{t\geq 1}$ in $\mathbb{N}$ such that $\lim_{t\rightarrow\infty}s_t=\infty$ and such that  $p_{t,\Theta}(\dd\theta_t| Y_{1:t})=\pi_{s_t,\Theta}(\dd\theta)$ for all $t\in\mathbb{N}$.
\end{proposition}
\begin{proof}
The result of the proposition directly follows from Proposition \ref{prop:Def_FK_MLE}  and from the definition of  $(\mu_t)_{t\geq 1}$.
\end{proof}

\begin{remark}\label{rem:MLE2}
If, in Proposition \ref{prop:SO-FK_MLE2}, the sequence $(\eta_t)_{t\geq 1}$ is such that Conditions \ref{condition:Inf_mu}-\ref{condition:mu_extra} hold, or such that  Conditions \ref{condition:Inf_mu}-\ref{condition:Inf_K} and \ref{condition:mu_seq2}-\ref{condition:mu_extra} hold, then this is also the case for the sequence $(\mu_t)_{t\geq 1}$ defined in   the proposition. In particular, if $(\eta_t)_{t\geq 1}$  is such that \ref{condition:mu_seq} holds for some sequence $(t_p)_{p\geq 1}$ in $\{sT+1,\, s\in\mathbb{N}\}$ then  the sequence  $(\mu_t)_{t\geq 1}$ is such that \ref{condition:mu_seq}  holds for some sequence $(t_p)_{p\geq 1}$ in $\{s(T+1),\, s\in\mathbb{N}\}$.
\end{remark}

\subsection{Proof of Theorem \ref{thm:MLE}}
\begin{proof}

Theorem \ref{thm:MLE} is a direct consequence of Propositions \ref{prop:A1_MLE}-\ref{prop:SO-FK_MLE2}, Remark \ref{rem:MLE2}, of Theorems \ref{thm:main}-\ref{thm:main2} and of the results in Section \ref{p-sub:noise}, upon noting (i) that for each definition of $(K_t)_{t\geq 1}$ considered in Theorem \ref{thm:MLE} there exists a sequence $(\mu_t')_{t\geq 1}$ of  probability measures on $(\setR,\mathcal{R})$ such that $K_t=K_{\mu'_t|\Theta}$ for all $t\geq 1$ and (ii) that each of these sequences   $(\mu_t')_{t\geq 1}$ is constructed as in Section \ref{p-sub:noise}.

\end{proof}

\subsection{Proof of Theorem \ref{thm:MLE2}}
\begin{proof}
Each  algorithm  for which Theorem \ref{thm:MLE2} applies is a particle filter deployed on  SO-SSM \eqref{eq:SO-SSM_new2} where, for a given number $N$ of particles,  $K_t=K_t^N$ for all $t\geq 1$ and with the sequence   $(K_t^N)_{t\geq 1}$  recursively defined in the course the algorithm.

Consider one of these algorithms and, for   all $N\in\mathbb{N}$ and $t\geq 1$ let $p_{t,N,\Theta}(\dd \theta_t |Y_{1:t})$     denote the filtering distribution of $\theta_t$  in model SO-SSM  \eqref{eq:SO-SSM_new2} when $(K_t)_{t\geq 1}=(K_t^N)_{t\geq 1}$. It is easily checked that for all $N\in\mathbb{N}$  there exists a sequence $(\mu^N_t)_{t\geq 1}$ of random probability measures on $(\setR,\mathcal{R})$ defined as  in Section \ref{p-sub:noise} and such that $K^N_t=K_{\mu^N_t|\Theta}$ for all $t\geq 1$.  Consequently, by Propositions  \ref{prop:A1_MLE}-\ref{prop:SO-FK_MLE2} and Theorems \ref{thm:main}-\ref{thm:main2} (see also   Remark \ref{rem:MLE2}) there exists a sequence $(\xi_t)_{t\geq 1}$ of $(0,\infty)$-valued random variables, independent of $N$, such that $\xi_t=\smallo_\P(1)$ and such that for all $N\in\mathbb{N}$ we have, $\P$-a.s.,
\begin{align}\label{eq:con_p2}
\Big\| \int_\Theta \theta_t\, p_{t,N,\Theta}(\dd \theta_t |Y_{1:t})-\tilde{\theta}_T\Big\|\leq \xi_t,\quad\forall t\geq 1.
\end{align}
Then, under   \ref{assumeSSMB:smooth_MLE},  and using Proposition \ref{prop:stability} and    the fact that  by assumption  there is a   compact set $E'\in\mathcal{X}$ such that $\inf_{(\theta,x)\in\Theta\times\setX}\tilde{M}_{t,\theta}(x,E')>0$ for all $t\in\{2,\dots,T\}$, the conclusion of the theorem  follows directly from \eqref{eq:con_p2} and from    the calculations used in \citetsup[][Section 11.2.2]{chopin2020introduction2} to study the $L_2$-convergence of particle filter algorithms.

\end{proof}

\bibliographystylesup{apalike}
\bibliographysup{references}
\end{document}